\documentclass{article}

\RequirePackage{amsthm,amsmath,amsfonts,amssymb}

\usepackage{dsfont}
\usepackage{enumitem}
\usepackage[T1]{fontenc}
\usepackage[margin=1.25in]{geometry}
\usepackage[normalem]{ulem} 

\usepackage[square,sort,comma,numbers]{natbib}

\usepackage{hyperref}
\RequirePackage{graphicx}

\newcommand{\defeq}{\mathrel{\mathop:}=}
\newcommand*\Laplace{\mathop{}\!\mathbin\bigtriangleup}

\newcommand{\myemph}[1]{  
\textit{#1}.
}

\makeatletter
\renewcommand*\env@matrix[1][\arraystretch]{%
  \edef\arraystretch{#1}%
  \hskip -\arraycolsep
  \let\@ifnextchar\new@ifnextchar
  \array{*\c@MaxMatrixCols c}}
\makeatother

\newtheorem{theorem}{Theorem}[section]
\newtheorem{lemma}[theorem]{Lemma}

\newtheorem{assumption}{Assumption}
\newtheorem{proposition}[theorem]{Proposition}

\newtheorem*{theorem*}{Theorem}
\newtheorem*{proposition*}{Proposition}
\newtheorem*{corollary*}{Corollary}
\newtheorem*{lemma*}{Lemma}

\theoremstyle{definition}
\newtheorem{definition}[theorem]{Definition}
\newtheorem{remark}[theorem]{Remark}

\DeclareMathOperator*{\esssup}{ess \, sup}
\DeclareMathOperator*{\supp}{supp}

\newcommand{\ignore}[1]{}

\newcommand{\llbracket}{[\![}
\newcommand{\rrbracket}{]\!]}

\title{A functional law of large numbers for a spatial model of Muller's ratchet}

\author{João Luiz de Oliveira Madeira\thanks{Department of Statistics, University of Oxford, UK} \and Marcel Ortgiese\footnote{Department of Mathematical Sciences, University of Bath, UK} \and Sarah Penington\footnotemark[2]}

\date{}

\begin{document}




\maketitle

\begin{abstract}
The spatial Muller's ratchet is a model introduced by Foutel-Rodier and Etheridge to study the impact of cooperation and competition on the fitness of an expanding asexual population. The model is an interacting particle system consisting of particles performing symmetric random walks that reproduce and die with rates that depend on the local number of particles. For each particle, we keep track of the number of deleterious mutations that it carries, and after each birth event, with some positive probability, the offspring particle can acquire an additional mutation that gives it a lower reproduction rate than its parent. 
We show that, under an appropriate scaling, the process converges weakly to the solution of an infinite system of partial differential equations (PDEs), confirming non-rigorous computations of Foutel-Rodier and Etheridge. 
Combining the weak convergence with analytical results for the limiting PDE system, we derive quantitative lower and upper bounds on the proportion of particles with mutations that hold with high probability for the particle system. A key obstacle is the absence of uniform bounds on the number of particles per site, together with the presence of infinitely many types of particle and the nonlinear interactions. To address this, we establish a new tightness criterion for interacting particle systems in general~$L_p$ spaces based only on local properties of the dynamics.
\end{abstract}


\tableofcontents

\section{Introduction}
\label{Paper02_introduction}

In this article, we rigorously establish the scaling limit of a spatial reaction-diffusion system of interacting particles with countably infinitely many mutation types. The system is a generalisation of a model introduced by Foutel-Rodier and Etheridge in~\cite{foutel2020spatial} to investigate the impact of cooperation and competition on the propagation of deleterious mutations through an asexual population expanding in space, where individual birth and death rates depend on the local population density. Before formally defining the interacting particle system in Section~\ref{Paper02_model_description}, in this section we outline the biological motivation for the model, as well as the (non-rigorous) derivations in~\cite{foutel2020spatial} that we prove in this article.

\subsection{Motivation} \label{Paper02_biological_motivation}

As a central mechanism of evolution, when an individual passes its genetic material to its offspring, a mutation may occur and the genetic material is modified. This modification may have an effect on the fitness of the individual: a mutation may be deleterious (reducing fitness), advantageous (increasing fitness), or neutral (having no effect on fitness).

Although fitness is a key factor in determining the survival probability of an individual, in a spatially structured population this probability also depends on the local population density~\cite{foutel2020spatial}. Higher density may either decrease survival probability due to competition or increase it due to cooperation.

It is known that rare neutral mutations can propagate across large regions during range expansion, a phenomenon called \emph{gene surfing}, first described by Edmonds et al.~\cite{edmonds2004mutations}. Roughly speaking, if competition dominates over cooperation and so the per-capita growth rate is high at low population densities, then a neutral mutation appearing at the front of an expanding population (i.e.~at the leading edge of the expanding population) can spread locally with high probability, since the population density at the front is usually low. Since colonisation of new habitat is driven by individuals near the front, such mutations can become prevalent over large spatial regions~\cite{foutel2020spatial}. 

Gene surfing of neutral mutations has been studied through simulations of discrete particle systems, partial differential equations (PDEs), and stochastic PDEs (see~\cite{edmonds2004mutations,hallatschek2008gene, roques2012allee}), and rigorously in PDE-based models~\cite{garnier2012inside, roques2012allee}. We also refer the reader to~\cite{roques2012allee} for a discussion of the possible different definitions of gene surfing of neutral mutations. In contrast, the gene surfing of deleterious mutations has received less attention.

To address the question of whether deleterious mutations can also surf, Foutel-Rodier and Etheridge introduced the \emph{spatial Muller's ratchet}~\cite{foutel2020spatial}, a spatial extension of a mechanism first proposed by Muller~\cite{muller1964relation} to explain the evolution of recombination and sexual reproduction. In asexual reproduction, chromosomes are inherited as indivisible units, so the number of mutations along an ancestral lineage can only increase. Since the majority of mutations in biological systems are deleterious~\cite{bao2022mutations}, this leads to a gradual decline in overall population fitness~\cite{etheridge2009often}. In Muller's ratchet models, when all individuals in the most-adapted class in the population acquire at least one additional deleterious mutation, the minimum mutational load in the population increases — this is known as a ‘click' of the ratchet. From a mathematical perspective, Muller's ratchet has been extensively studied in populations without spatial structure~\cite{casanova2022quasi,pfaffelhuber2012muller,haigh1978accumulation,etheridge2009often,mariani2020metastability}.

The model in~\cite{foutel2020spatial} shares the same key feature of deleterious mutations, but also includes a spatial component and density-dependent birth and death rates. We now describe the model. Let $N \in \mathbb{N}$ denote a scaling parameter. The population is subdivided into demes indexed by~$L_N^{-1}\mathbb{Z}$, where $L_N > 0$ is a space renormalisation parameter. For $x \in L_N^{-1}\mathbb{Z}$, $t \geq 0$ and $k \in \mathbb{N}_{0}$, let $\eta^N_{k}(t,x)$ indicate the number of particles (representing individuals in the population) carrying exactly~$k$ mutations living in deme $x$ at time $t$, and let $\| \eta^N(t,x) \|_{\ell_{1}} \defeq \sum_{k = 0}^{\infty} \eta_{k}^N(t,x)$. Each particle (independently) migrates at rate~$m_N$, jumping to one of the two neighbouring demes with equal probability. Moreover, a particle located in deme $x \in L_N^{-1}\mathbb{Z}$ at time $t \geq 0$ and carrying $k \in \mathbb{N}_{0}$ mutations gives birth to a new particle at rate
\begin{equation*}
    r(1-s)^{k}\Bigg(B\frac{\| \eta^N(t-,x) \|_{\ell_{1}}}{N} + 1\Bigg),
\end{equation*}
and dies at rate
\begin{equation*}
   r\frac{\| \eta^N(t-,x) \|_{\ell_{1}}}{N}\Bigg(B\frac{\| \eta^N(t-,x) \|_{\ell_{1}}}{N} + 1\Bigg),
\end{equation*}
where~$r \in (0,\infty)$ is a Malthusian growth parameter, $s \in (0,1)$ represents the impact on fitness of each deleterious mutation, and the parameter $B \in [0, \infty)$ captures the relative effect of cooperation and competition in the local population dynamics. Here, larger values of $B$ indicate a greater role of cooperation in the dynamics. Let $\mu \in (0,1)$. After a birth event, an offspring particle is added to the same deme as its parent: with probability $1 - \mu$, the offspring inherits the same number of mutations as its parent, and with probability~$\mu$, it accumulates an additional mutation.

Foutel-Rodier and Etheridge
consider a scaling regime in which $L_N$ is of order $N$ and $\lim_{N \rightarrow \infty} \frac{m_N}{L_N^2} = m \in (0, \infty)$. By performing a (non-rigorous) generator calculation, they conjecture that the number of particles in each deme (rescaled by $1/N$) converges as $N \rightarrow \infty$
to the solution of the following system of PDEs: For each $k \in \mathbb{N}_0$, for $t > 0$,
\begin{equation} \label{Paper02_conjectured_PDE_heuristics}
\begin{split}
\partial_t u_k 
     = \frac{m}{2} \Laplace u_k & + r(B\| u \|_{\ell_1} + 1)(u_k(1-s)^{k}(1- \mu) \\
    &  + \mathds{1}_{\{k \geq 1\}}u_{k-1}(1-s)^{k-1}\mu - u_{k} \| u \|_{\ell_1}). 
\end{split}
\raisetag{-1cm}
\end{equation}
Foutel-Rodier and Etheridge also predict~\cite[Equation~(9)]{foutel2020spatial} that the (critical) spreading speed into an empty habitat of a population governed by the dynamics of~\eqref{Paper02_conjectured_PDE_heuristics} depends on the parameter $B \in [0, \infty)$, and is given by $c^*>0$, where:
\begin{equation} \label{Paper02_conjecture_spreading_speed_foutel_rodier_etheridge}
    c^* \defeq \left\{\begin{array}{lcl}
        \sqrt{2mr(1-\mu)} \quad & \textrm{if} & \displaystyle B \leq \frac{2}{1-\mu},  \\
          \displaystyle \frac{B(1-\mu) + 2}{2}\sqrt{\frac{mr}{B}} \quad & \textrm{if} & \displaystyle B > \frac{2}{1-\mu}.
    \end{array}\right.
\end{equation}
The change in the expression for the spreading speed at $B = \frac{2}{1-\mu}$ reflects a transition from \emph{pulled} to \emph{pushed} expansion waves~\cite{roques2012allee,garnier2012inside}. Heuristically, in the case of pulled waves, the critical speed is the same as for the system of PDEs linearised around the state $0$, i.e.~the wave is pulled by its leading edge. On the other hand, in pushed waves, the critical speed is determined not only by the leading edge, but by the whole front, i.e.~the wave is pushed from behind~\cite{garnier2012inside}. Pulled expansion waves are observed when the highest per-capita growth rate is achieved at very low population density, i.e.~in particular when the population dynamics is of \emph{Fisher-KPP} type~\cite{roques2012allee}. On the other hand, when the per-capita growth rate is negative at very low population densities, i.e.~when the dynamics displays a \emph{strong Allee} effect, pushed expansion waves are observed~\cite{roques2012allee}. When the maximum growth rate is achieved at an intermediate population density, but the population still exhibits growth even at low population densities, then the dynamics shows a \emph{weak Allee} effect; in this case, the wave can be either pulled or pushed, depending on the model parameters. From~\eqref{Paper02_conjecture_spreading_speed_foutel_rodier_etheridge}, Foutel-Rodier and Etheridge conclude that for $B \in [0, 2/(1-\mu))$, the dynamics of the system of PDEs~\eqref{Paper02_conjectured_PDE_heuristics} exhibits pulled behaviour, while for $B > 2/(1 - \mu)$, the system exhibits pushed behaviour.

From a mathematical perspective, there are several challenges in rigorously proving the conjectures of Foutel-Rodier and Etheridge in~\cite{foutel2020spatial}. First, since we allow for the existence of particles carrying an arbitrarily large number of mutations, we must keep track of infinitely many types of particles. Moreover, there are no \textit{a priori} bounds on the number of particles per deme, and both birth and death rates are unbounded. Furthermore, the fact that particles at the same deme with different numbers of mutations interact with each other leads to ‘non-local interactions' in both the interacting particle system and the system of PDEs~\eqref{Paper02_conjectured_PDE_heuristics}.

In fact, it is highly non-trivial to show that the particle system even exists (when started with an infinite number of particles), not only for the reasons mentioned above, but also because the system turns out to be non-monotone. 
We address this question in our companion article~\cite{madeira2025existence}, where we construct a general version of the interacting particle system introduced in~\cite{foutel2020spatial} rigorously, and {we prove (local) moment bounds on the particle density. In our other companion article~\cite{madeira2025PDE_surfing}, we establish existence and uniqueness of mild solutions to the system of PDEs.
Moreover, we derive the spreading speed in the Fisher-KPP regime (that is, for $B \leq 1/(1-\mu)$) and establish
the asymptotic shape of the mutational profile.}

{In the present article, using the results of~\cite{madeira2025existence,madeira2025PDE_surfing}, we rigorously derive the hydrodynamic limit of the general spatial Muller's ratchet model. The main probabilistic innovation is a new tightness criterion in suitable~$L_p$ spaces for interacting particle systems, which relies only on local properties of the dynamics, and therefore applies to processes with infinitely many particle types and unbounded birth and death rates. Using the weak convergence together with the PDE analysis from~\cite{madeira2025PDE_surfing}, we further derive quantitative bounds on the mutational profile of the particle system.}

\medskip

\noindent \textbf{Structure of the article.} In Section~\ref{Paper02_model_description}, we define the spatial Muller's ratchet model, and in Sections~\ref{Paper02_sec:subsection_law_large_numbers} and~\ref{Paper02_sec:subsection_asymptotic_behaviour_PDE} we state the main results of this article. In Section~\ref{Paper02_example_pop_dynamics}, we mention some biologically relevant examples. In Section~\ref{Paper02_related_work_introd}, we briefly review some techniques commonly used to establish hydrodynamic limit results for interacting particle systems, and explain why they are not directly applicable to our setting. In Section~\ref{Paper02_subsection_heuristics}, we informally present the key ideas behind the proofs of our results. In Section~\ref{Paper02_subsection_preliminaries}, we review some concepts that will be used in the article, including the notion of solutions to the (infinite) system of PDEs. The main results of our companion article~\cite{madeira2025existence} that will be used in this article are stated in Section~\ref{Paper02_subsection_preliminaries_existence_moment_bouds}.

The proofs of our results can be found in Sections~\ref{Paper02_section_greens_function}--\ref{Paper02_uniqueness_weak_solutions_section}. In Section~\ref{Paper02_section_greens_function}, we use a Green's function representation of the random walk semigroup to establish regularity properties of the population density. We use these properties to derive tightness in the space of sequences of Radon measures in Section~\ref{Paper02_tightness}. Tightness of the process in a suitable generalisation of the usual $L_p$ space is established in Section~\ref{Paper02_section_limiting_process_L1_construction}. We complete the proofs of 
the main results of this article in Section~\ref{Paper02_uniqueness_weak_solutions_section}. The appendix contains standard results and technical estimates used in Sections~\ref{Paper02_section_greens_function}--\ref{Paper02_uniqueness_weak_solutions_section}.

\medskip
\noindent\textbf{Notation.} Throughout the article we will use the following notation. Let $\ell_{1}$ denote the space of summable real-valued sequences, i.e. 
\begin{equation*}
    \ell_{1} = \ell_{1}\left(\mathbb{N}_{0}\right) \defeq \bigg\{z = \left(z_{k}\right)_{k \in \mathbb{N}_{0}} \in \mathbb{R}^{\mathbb{N}_{0}}: \, \| z \|_{\ell_1} \defeq \sum_{k = 0}^{\infty}\vert z_{k} \vert < \infty \bigg\}.
\end{equation*}
Moreover, let $\ell_{1}^{+}$ denote the subset of $\ell_{1}$ consisting of the summable sequences with non-negative entries. For any 
$K \in \mathbb{N}$, we let $\llbracket K \rrbracket \defeq \{1,2,\ldots, K\}$. Let $\lambda$ be the Lebesgue measure on $\mathbb R^d$ for any $d\in \mathbb N$. Let
\begin{equation*}
    \mathcal{C}_{c}(\mathbb{R}) \defeq \left\{\phi: \mathbb{R} \rightarrow \mathbb{R}: \, \phi \textrm{ is continuous and has compact support} \right\}.
\end{equation*}
Let $\mathcal{M}(\mathbb{R})$ be the space of non-negative locally finite Radon measures on $\mathbb{R}$. For $\varphi \in \mathcal{C}_c(\mathbb R)$ and $\rho \in \mathcal{M}(\mathbb{R})$, we let
\[
\langle \rho, \varphi \rangle \defeq \int_{\mathbb R} \varphi(x) \, \rho(dx).
\]
We equip $\mathcal{M}(\mathbb{R})$ with the vague topology to turn this space into a Polish space, and let~$d_{\textrm{vague}}$ denote a corresponding metric. 
For a complete and separable metric space $(\mathcal{S}, d_{\mathcal{S}})$, we let $\mathcal{D}([0,\infty); \mathcal{S})$ denote the space of $\mathcal{S}$-valued càdlàg paths, and we equip $\mathcal{D}([0,\infty); \mathcal{S})$ with the $J_1$-Skorokhod metric.

For a metric space $(\mathcal{S}, d_{\mathcal{S}})$, we let $\mathcal{C}(\mathcal{S}; \mathbb{R})$ denote the space of continuous real-valued functions on $\mathcal{S}$. Let $\mathcal{C}^{1}(\mathbb{R})$ denote the set of continuously differentiable real-valued functions defined on $\mathbb R$. Let $\mathcal{C}^{1,2}((0,\infty) \times \mathbb{R}; \mathbb{R})$ denote the set of real-valued functions on $(0,\infty) \times \mathbb{R}$ that are continuously differentiable in the first coordinate and twice continuously differentiable in the second coordinate.

For a set $\mathcal{S}$, suppose $f,g: \mathcal{S} \rightarrow (0,\infty)$ are functions on $\mathcal{S}$. We write $f \lesssim g$ to indicate that there exists a constant $C > 0$ such that for every $x \in \mathcal{S}$, $f(x) \leq Cg(x)$. If the constant $C$ depends on a parameter $p$, then we write $f \lesssim_p g$. We say that sequences of positive real numbers $(a_N)_{N \in \mathbb{N}}$ and $(b_N)_{N \in \mathbb{N}}$ satisfy $a_N = \Theta(b_N)$ as $N\rightarrow \infty$ if
$$0 < \liminf_{N \rightarrow \infty} \frac{a_N}{b_N} \leq \limsup_{N \rightarrow \infty} \frac{a_N}{b_N} < \infty.$$
For $z \in \mathbb R$, we write
$
z^+ \defeq z \vee 0
$ for the positive part of $z$.

\section{Model definition and main results} \label{Paper02_model_description}

Let $N \in \mathbb{N}$ be a scaling parameter and $L_{N} > 0$ a space renormalisation parameter depending on $N$. The model consists of particles moving on the rescaled one-dimensional lattice $L_{N}^{-1}\mathbb{Z}$, where we call each point $x \in L_{N}^{-1}\mathbb{Z}$ a \emph{deme}. Each particle (representing an individual in the population that carries a unique chromosome) is characterised by two features: the number of deleterious mutations that it carries and its spatial location. 
For each $t \geq 0$, $x \in L_{N}^{-1}\mathbb{Z}$ and $k \in \mathbb N_0$, let $\eta^{N}_{k}(t,x)$  denote the number of particles at deme $x$ carrying exactly $k$ mutations at time $t$. Note that for each deme $x \in L_{N}^{-1}\mathbb{Z}$, we can characterise the set of particles at this deme at time $t$ by the sequence $\eta^{N}(t,x) \defeq (\eta^{N}_{k}(t,x))_{k \in \mathbb{N}_{0}}$. The total number of particles living at deme $x$ at time $t$ is given by
$\| \eta^{N}(t,x) \|_{\ell_1} = \sum_{k = 0}^{\infty} \eta^{N}_{k}(t,x)$.

Let $m_{N} > 0$ be the migration rate, and let $(s_{k})_{k \in \mathbb N_0}$ be a sequence of fitness parameters, where $s_k \geq 0$ denotes the fitness of a particle carrying $k$ mutations. Let $q_{+},q_{-}: \mathbb{R}_{+} \rightarrow \mathbb{R}_{+}$ be non-negative functions, and let $\mu \in [0,1]$ be the mutation probability. We will precisely state our assumptions on $m_{N}$, $(s_{k})_{k \in \mathbb{N}_{0}}$, $q_{+}$ and $q_{-}$ later in this section. The spatial Muller's ratchet process $(\eta^{N}(t))_{t \geq 0} = (\eta_{k}^{N}(t,x): \, k \in \mathbb{N}_{0}, \, x \in L_{N}^{-1}\mathbb{Z})_{t \geq 0}$ can be described informally as follows:
\begin{itemize}
    \item \textit{Migration events}: For each $t \geq 0$ and $x \in L_{N}^{-1}\mathbb{Z}$, each particle living at deme $x$ at time $t$ independently jumps at rate $m_{N}$ to a uniformly chosen deme from $\{x - L_{N}^{-1}, \, x + L_{N}^{-1}\}$.
    \item \textit{Reproduction events}: For each $t \geq 0$, $x \in L_{N}^{-1}\mathbb{Z}$ and $k \in \mathbb{N}_{0}$, each particle carrying $k$ mutations at deme $x$ at time $t$ reproduces independently at rate $\displaystyle s_{k}q_{+}(\| \eta^{N}(t-,x) \|_{\ell_{1}} / N)$. After such an event, a new offspring particle will be added at deme $x$. With probability $1 - \mu$, the offspring particle will carry $k$ mutations, and with probability $\mu$ it will carry $k + 1$ mutations. Note that in both cases, the parent particle remains alive and keeps the same number of mutations.
    \item \textit{Death events}: For each $t \geq 0$ and $x \in L_{N}^{-1}\mathbb{Z}$, each particle at deme $x$ at time $t$ dies independently at rate $\displaystyle q_{-}(\| \eta^{N}(t-,x) \|_{\ell_{1}} / N)$, regardless of its number of mutations. When a particle dies, it is simply removed from the process.
\end{itemize}

We will now state our conditions on $m_{N}$, $L_{N}$, $(s_{k})_{k \in \mathbb{N}_{0}}$, $q_{+}$ and $q_{-}$. The migration rate $m_{N}$ and the space renormalisation parameter $L_{N}$ will satisfy the following assumptions.

\begin{assumption}[Space renormalisation]\label{Paper02_scaling_parameters_assumption}
The positive parameters $m_{N}$ and $L_{N}$ are such that
\begin{enumerate}
    \item[(i)] $\displaystyle \frac{m_{N}}{L_{N}^{2}} \rightarrow m \in (0, \infty)$ as $N \rightarrow \infty$.
    \item[(ii)] $L_{N} = \Theta(N)$ as $N \rightarrow \infty$. 
\end{enumerate}
\end{assumption}

We assume that all the mutations in our model are deleterious, i.e.~that $(s_k)_{k \in \mathbb{N}_0}$ is decreasing, so that any mutation decreases the rate of reproduction. More precisely, we assume:

\begin{assumption} [Fitness parameters]\label{Paper02_assumption_fitness_sequence}
The sequence of fitness parameters $\left(s_{k}\right)_{k \in \mathbb{N}_{0}}$ satisfies the following conditions:
\begin{enumerate}
\item[(i)] $s_{0} = 1$.
\item[(ii)] $s_{k} \geq 0$ for all $k \in \mathbb{N}_{0}$.
\item[(iii)] $\left(s_{k}\right)_{k \in \mathbb{N}_{0}}$ is monotonically non-increasing, i.e.~$s_{k} \geq s_{k+1}$ for all $k \in \mathbb{N}_{0}$.
\item[(iv)] $\lim_{k \rightarrow \infty} s_{k} = 0$.
\end{enumerate}
\end{assumption}

Now, since in our model we do not impose \textit{a priori} bounds on the number of particles per deme, we make the following assumptions on the birth and death rates.

\begin{assumption} [Birth and death polynomial rates]\label{Paper02_assumption_polynomials}
Suppose that $q_{+}, q_{-}: [0,\infty) \rightarrow [0,\infty)$ are polynomials such that $0 \leq \deg q_{+} < \deg q_{-}$.
\end{assumption}

By Assumption~\ref{Paper02_assumption_polynomials}, the leading coefficient of the polynomial $q_{+}$ is non-negative, and the leading coefficient of $q_-$ is strictly positive. Note also that the scaling parameter $N$ can be thought of as being proportional to the local carrying capacity of the population. Indeed, since $\deg q_{+} < \deg q_{-}$, when the number of particles at a deme is much larger than $N$, then death events happen at a higher rate than birth events, providing a local regulation mechanism for the population density.

The initial configurations of particles in ${\eta}^N$ will be chosen in such a way that they converge as $N \rightarrow \infty$ to a function $f: \mathbb{R} \rightarrow \ell_{1}^{+}$ satisfying the following assumptions. Let $\lambda$ denote the Lebesgue measure on $\mathbb R$.

\begin{assumption} [Initial condition] \label{Paper02_assumption_initial_condition}
Suppose $f = \left(f_{k}\right)_{k \in \mathbb{N}_{0}}: \mathbb{R} \rightarrow \ell_{1}^{+}$ satisfies the following conditions:
\begin{enumerate}[label=(\roman*)]
    \item $f$ is continuous $\lambda$-almost everywhere, i.e.~there exists $\mathcal{N}^{(1)} \subset \mathbb{R}$ such that $\lambda(\mathcal{N}^{(1)}) = 0$ and $f$ is continuous on $\mathbb{R} \setminus \mathcal{N}^{(1)}$.
    \item $f \in L_{\infty}(\mathbb{R}; \ell_{1})$, i.e.~$\esssup_{x \in \mathbb{R}} \| f(x) \|_{\ell_1} < \infty$.
    \item There exists $\mathcal{N}^{(2)} \subset \mathbb{R}$ such that $\lambda(\mathcal{N}^{(2)}) = 0$ and
    \begin{equation*}
        \lim_{k \rightarrow \infty} \; \sup_{x \in \mathbb{R} \setminus \mathcal{N}^{(2)}} \; \sum_{j \geq k} f_{j}(x) = 0.
    \end{equation*}
\end{enumerate}
\end{assumption}
We will clarify what we mean by $L_{p}$ spaces of $\ell_{1}$-valued functions in Section~\ref{Paper02_subsection_preliminaries}. For a function $f = (f_k)_{k \in \mathbb N_0}$ satisfying Assumption~\ref{Paper02_assumption_initial_condition}, for every $N \in \mathbb{N}$, we will define the initial condition of our process ${\eta}^N$ as the configuration $\boldsymbol{\eta}^{N} = (\eta^N_k(x))_{k \in \mathbb N_0, \, x \in L_N^{-1}\mathbb Z} \in (\mathbb{N}_{0}^{\mathbb{N}_{0}})^{L_{N}^{-1}\mathbb{Z}}$ such that for all $x \in L_{N}^{-1}\mathbb{Z}$ and $k \in \mathbb{N}_{0}$,
\begin{equation} \label{Paper02_initial_condition}
    {\eta}^{N}_{k}(x) \defeq \bigg\lfloor L_{N}N\int_{x - \frac 12 L_{N}^{-1}}^{x + \frac 12 L_{N}^{-1}} f_{k}(y) \, dy \bigg\rfloor.
\end{equation}
For $f$ satisfying Assumption~\ref{Paper02_assumption_initial_condition}, for every $N \in \mathbb{N}$, the initial configuration~$\boldsymbol{\eta}^{N} \in (\mathbb{N}_{0}^{\mathbb{N}_{0}})^{L_{N}^{-1}\mathbb{Z}}$ given by~\eqref{Paper02_initial_condition} satisfies
\begin{equation} \label{Paper02_immediate_estimates_definition_initial_condition}
    \sup_{x \in L_{N}^{-1}\mathbb{Z}} \; \| \eta^{N}(x) \|_{\ell_{1}} < \infty \quad \textrm{and} \quad \lim_{k \rightarrow \infty} \; \sup_{x \in L_{N}^{-1}\mathbb{Z}} \; \sum_{j \geq k} \; \eta^{N}_{j}(x) = 0.
\end{equation}
In particular, for every~$N \in \mathbb{N}$, the initial configuration $\boldsymbol{\eta}^N$ satisfies $\boldsymbol{\eta}^N \in \mathcal{S}^N$, where
\begin{align} \label{Paper02_definition_state_space_formal}
    \quad \mathcal{S}^{N} \defeq \bigg\{\boldsymbol{\xi} = (\xi_k(x))_{k \in \mathbb{N}_0, \, x \in L^{-1}_N\mathbb{Z}} \in (\mathbb{N}_{0}^{\mathbb{N}_{0}})^{L_{N}^{-1}\mathbb{Z}}: \; \vert \vert \vert \boldsymbol{\xi} \vert \vert \vert_{\mathcal{S}^{N}} \defeq \sum_{x \in L_{N}^{-1}\mathbb{Z}} \frac{\| \xi(x) \|_{\ell_{1}}}{(1 + \vert x \vert)^{2}} < \infty\bigg\}.
\raisetag{-0.7cm}
\end{align}
We define a metric~$d_{\mathcal{S}^{N}}$ on~$\mathcal{S}^{N}$ by
\begin{equation} \label{Paper02_definition_semi_metric_state_space}
\begin{aligned}
   d_{\mathcal{S}^{N}}: \mathcal{S}^{N} \times \mathcal{S}^{N} & \rightarrow [0, \infty) \\
    (\boldsymbol{\zeta}, \boldsymbol{\xi}) & \mapsto  d_{\mathcal{S}^{N}}(\boldsymbol{\zeta}, \boldsymbol{\xi}) \defeq \vert \vert \vert \boldsymbol{\zeta} - \boldsymbol{\xi} \vert \vert \vert_{\mathcal{S}^{N}} = \sum_{x \in L_{N}^{-1}\mathbb{Z}} \frac{\| \zeta(x) - \xi(x) \|_{\ell_{1}}}{(1 + \vert x \vert)^{2}}.
\end{aligned}
\end{equation}
Then the space~$(\mathcal{S}^{N},d_{\mathcal{S}^{N}})$ is a complete and separable metric space (see Proposition~\ref{Paper02_topological_properties_state_space} below). In our companion article~\cite{madeira2025existence},
we prove that, under Assumptions~\ref{Paper02_assumption_fitness_sequence} and~\ref{Paper02_assumption_polynomials}, for each $N \in \mathbb N$, for an initial configuration $\boldsymbol \eta^N$ satisfying~\eqref{Paper02_immediate_estimates_definition_initial_condition}, there exists a càdlàg $\mathcal{S}^{N}$-valued strong Markov process~$(\eta^{N}(t))_{t \geq 0}$, with $\eta^N(0) = \boldsymbol{\eta}^{N}$ almost surely, given by the unique weak limit of a sequence of $\mathcal{S}^{N}$-valued Markov processes, and such that $(\eta^N(t))_{t \geq 0}$ evolves according to the informal description at the start of this section (see~\cite[Theorem~2.2]{madeira2025existence} for a full statement of this result; we will formally state a version of the result that will suffice for our purposes in this article in Theorem~\ref{Paper02_thm_existence_uniqueness_IPS_spatial_muller} below). We also refer the reader to~\cite[Section~3.3]{madeira2025existence} for an explanation of why~\eqref{Paper02_immediate_estimates_definition_initial_condition} is required. We call the process~$(\eta^{N}(t))_{t \geq 0} = (\eta^{N}_k(t,x))_{k \in \mathbb N_0, \, x \in L_N^{-1}\mathbb Z, \, t \geq 0}$  the \emph{spatial Muller's ratchet}. 

For $x \in L_N^{-1}\mathbb Z$ and $k \in \mathbb{N}_0$, let $\boldsymbol{e}_{k}^{(x)} \in (\mathbb{N}_{0}^{\mathbb{N}_{0}})^{L_{N}^{-1}\mathbb{Z}}$ denote the configuration consisting of a single particle carrying exactly $k$ mutations at deme $x$. In order to define the infinitesimal generator of the Markov process $(\eta^N(t))_{t \geq 0}$, let
\begin{align}
    & \mathcal{C}_*(\mathcal{S}^N; \mathbb{R}) \defeq \bigg\{\phi \in \mathcal{C}(\mathcal{S}^N; \mathbb{R}): \sup_{\substack{x \in L_N^{-1}\mathbb{Z}, \\ k \in \mathbb{N}_{0}}} \sup_{\boldsymbol{\zeta} \in \mathcal{S}^N} \bigg( \Big\vert \phi\Big(\boldsymbol{\zeta} + \boldsymbol{e}^{(x)}_{k}\Big) - \phi(\boldsymbol{\zeta}) \Big\vert \nonumber \\
    & \quad +  \mathds{1}_{\{\zeta_{k}(x)>0\}}\Big\vert \phi\Big(\boldsymbol{\zeta} - \boldsymbol{e}^{(x)}_{k}\Big) - \phi(\boldsymbol{\zeta}) \Big\vert \label{Paper02_definition_Lipschitz_function} \\ & \quad + \mathds{1}_{\{\zeta_{k}(x)>0\}} \sum_{a = 1}^{2} \Big\vert \phi\Big(\boldsymbol{\zeta} + \boldsymbol{e}^{(x+ (-1)^{a}L_N^{-1})}_{k} - \boldsymbol{e}^{(x)}_{k}\Big) - \phi(\boldsymbol{\zeta}) \Big\vert \bigg) (1 + \vert x \vert)^{2(1 + \deg q_{-})} < \infty \bigg\} \nonumber.
\end{align}
The set $\mathcal{C}_*(\mathcal{S}^N; \mathbb{R})$ can be thought of as a set of functions whose discrete derivatives decrease polynomially fast in distance from the origin.
By~\cite[Theorem~2.2]{madeira2025existence}, for every~$N \in \mathbb{N}$, the infinitesimal generator~$\mathcal{L}^N$ of~$(\eta^{N}(t))_{t \geq 0}$ satisfies, for all $\phi \in \mathcal{C}_*(\mathcal{S}^{N}; \mathbb{R})$,
\begin{equation} \label{Paper02_generator_foutel_etheridge_model}
    \mathcal{L}^{N}\phi(\boldsymbol{\xi}) = \frac{m_{N}}{2}\mathcal{L}^{N}_{m}\phi(\boldsymbol{\xi}) + \mathcal{L}^{N}_{r}\phi(\boldsymbol{\xi}) \quad \forall \boldsymbol{\xi} \in \mathcal{S}^{N},
\end{equation}
where $\mathcal{L}^{N}_{m}$ corresponds to the migration process and $\mathcal{L}^{N}_{r}$ corresponds to the birth-death process and are defined as follows: for any~$\phi \in \mathcal{C}_{*}(\mathcal{S}^{N}, \mathbb{R})$ and~$\boldsymbol{\xi} = (\xi_k(x))_{k \in \mathbb N_0, \, x \in L_N^{-1}\mathbb Z} \in \mathcal{S}^{N}$,
{\begin{align} \label{Paper02_infinitesimal_generator}
     (\mathcal{L}^{N}_{m}\phi)(\boldsymbol{\xi}) & \defeq \sum_{x \in L_{N}^{-1}\mathbb{Z}} \; \sum_{k = 0}^{\infty} \; \sum_{z \in \{-L_N^{-1}, L_N^{-1}\}} \xi_{k}(x) \left(\phi\Big(\boldsymbol{\xi} + \boldsymbol{e}_{k}^{(x +z)} - \boldsymbol{e}_{k}^{(x)}\Big) - \phi(\boldsymbol{\xi}) \right), \notag \\
     (\mathcal{L}^{N}_{r}\phi)(\boldsymbol{\xi}) 
    &  \defeq \sum_{x \in L_{N}^{-1}\mathbb{Z}} \, \sum_{k = 0}^{\infty} \bigg(q_{+}\left(\frac{\| \xi(x) \|_{\ell_{1}}}{N}\right)\Big(s_{k}(1 - \mu)\xi_{k}(x) + \mathds{1}_{\{k \geq 1\}} s_{k-1} \mu \xi_{k-1}(x)\Big) \\ 
    & \quad \quad  \cdot \left(\phi\Big(\boldsymbol{\xi} + \boldsymbol{e}_{k}^{(x)}\Big) - \phi(\boldsymbol{\xi})\right) + q_{-}\left(\frac{\| \xi(x) \|_{\ell_{1}}}{N}\right)  \xi_{k}(x)\left(\phi\Big(\boldsymbol{\xi}- \boldsymbol{e}_{k}^{(x)}\Big) - \phi(\boldsymbol{\xi})\right)\bigg) \notag.
\end{align}}

\subsection{Functional law of large numbers} \label{Paper02_sec:subsection_law_large_numbers}

We now prepare to state our main result concerning convergence of the particle system to a system of PDEs. For $N \in \mathbb N$ and $k \in \mathbb N_0$, we define the approximate density process $(u^N_k(t, \cdot))_{t \geq 0}$ given by $u^{N}_{k}(t, \cdot): \mathbb{R} \rightarrow [0,\infty)$ for $t \geq 0$
by setting
\begin{equation} \label{Paper02_approx_pop_density_definition}
    u^{N}_{k}(t,x) \defeq \frac{\eta^{N}_{k}(t,x)}{N}\quad \textrm{ for all } x \in L_{N}^{-1}\mathbb{Z},
\end{equation}
and linearly interpolating $u^{N}_{k}(t,\cdot)$ between the demes $x \in L_{N}^{-1}\mathbb{Z}$. Since $(\eta^N(t))_{t \geq 0}$ is an $\mathcal{S}^N$-valued càdlàg process, we have that for every $x \in L_{N}^{-1}\mathbb{Z}$ and any $t \geq 0$, $\sum_{k = 0}^{\infty} u^{N}_{k}(t,x)$ is almost surely finite.

Our main result will describe the limiting behaviour of the system defined above as $N \rightarrow \infty$. To deal with the infinitely many types of particles in the system, we will interpret 
$((u^{N}_{k}(t, \cdot))_{k \in \mathbb{N}_{0}})_{t \geq 0}$ as a Markov process taking values in the space of sequences of non-negative Radon measures $\mathcal{M}(\mathbb{R})^{\mathbb N_0}$. 

By the definition of the state space~$\mathcal{S}^{N}$ in~\eqref{Paper02_definition_state_space_formal}, the definition of $u^{N}$ in~\eqref{Paper02_approx_pop_density_definition}, and the fact that $(\eta^{N}(t))_{t \geq 0}$ is an $\mathcal{S}^N$-valued process, we conclude that for any $k \in \mathbb N_0$ and $t \geq 0$, almost surely, for any compact set $\mathcal{K} \subset \mathbb{R}$ we have $u^N_k(t, \cdot) \in \mathcal{C}(\mathcal{K}; \mathbb{R})$. Therefore, for any $k \in \mathbb{N}_{0}$ and $t \geq 0$, $(u^N_k(t, x))_{x \in \mathbb{R}}$ is almost surely the density of a Radon measure with respect to the Lebesgue measure. With a slight abuse of notation, we denote this measure by $u^{N}_{k}(t)$, and for $\phi \in \mathcal{C}_{c}(\mathbb{R})$, we let
\begin{equation} \label{Paper02_definition_u_N_as_radon_measure}
    \left\langle u^{N}_{k}(t), \phi \right\rangle \defeq \int_{\mathbb{R}} \phi(x) u^{N}_{k}(t,x) \, dx.
\end{equation}
For each $N \in \mathbb N$ and $k \in \mathbb N_0$, since $(\eta^N(t))_{t \geq 0}$ is an $\mathcal S^N$-valued càdlàg process, we have $(u^{N}_k(t))_{t \geq 0} \in \mathcal{D}\Big([0, \infty), \Big(\mathcal{M}(\mathbb{R}), d_{\textrm{vague}}\Big)\Big)$ almost surely, where we recall that $d_{\textrm{vague}}$ is a metric on $\mathcal{M}(\mathbb{R})$ which induces the vague topology. Recall that $\Big(\mathcal{M}(\mathbb{R}), d_{\textrm{vague}}\Big)$ is a complete and separable metric space (see e.g.~\cite[Proposition~III.1.9.14]{bourbaki2004measures}). 

The process $(u^N(t))_{t \geq 0} = ((u^N_k(t))_{k \in \mathbb N_0})_{t\ge 0}$ can then be interpreted as a càdlàg process with sample paths in $\mathcal{M}(\mathbb{R})^{\mathbb{N}_{0}}$, i.e.~the space of sequences of non-negative Radon measures. In order to characterise the convergence of $u^N$ in $\mathcal{M}(\mathbb{R})^{\mathbb{N}_{0}}$ as $N \rightarrow \infty$, we introduce the metric $d$ given by
\begin{equation} \label{Paper02_metric_product_topology}
\begin{aligned}
    {d}: \mathcal{M}(\mathbb{R})^{\mathbb{N}_{0}} \times \mathcal{M}(\mathbb{R})^{\mathbb{N}_{0}} & \rightarrow [0,2] \\ (u,v) & \mapsto d(u,v) \defeq \sum_{k = 0}^{\infty} 2^{-k} \Big(d_{\textrm{vague}}\left(u_{k},v_{k}\right) \wedge 1\Big).
\end{aligned}
\end{equation}
Then $(\mathcal{M}(\mathbb{R})^{\mathbb{N}_{0}}, d)$ is a complete and separable metric space equipped with the product topology (see e.g.~the comment before Proposition~3.4.6 in~\cite{ethier2009markov}). In particular, a sequence $(v^{n})_{n \in \mathbb{N}} = ((v^n_k)_{k \in \mathbb N_0})_{n \in \mathbb N}$ of elements of $\mathcal{M}(\mathbb{R})^{\mathbb{N}_{0}}$ converges to $v = (v_k)_{k \in \mathbb N_0} \in \mathcal{M}(\mathbb{R})^{\mathbb{N}_{0}}$ in the topology of $(\mathcal{M}(\mathbb{R})^{\mathbb{N}_{0}}, d)$ if and only if $v^{n}_{k} \rightarrow v_{k}$ as $n \rightarrow \infty$ in $(\mathcal{M}(\mathbb{R}), d_{\textrm{vague}})$ for all $k \in \mathbb{N}_{0}$.

In \cite{foutel2020spatial}, Foutel-Rodier and Etheridge 
derived non-rigorously via a generator calculation  
that for the sequence of initial conditions $(\boldsymbol{\eta}^N)_{N \in \mathbb N}$ given by~\eqref{Paper02_initial_condition},
as $N \rightarrow \infty$, the sequence of processes $(u^{N})_{N \in \mathbb{N}}$ should converge to the solution $u = (u_k)_{k \in \mathbb N_0}$ of an infinite system of PDEs given by
\begin{equation} \label{Paper02_PDE_scaling_limit}
\begin{aligned}
    \partial_{t} u_{k} & = \frac{m}{2} \Laplace u_{k} + F_{k}(u) \quad 
    \forall k \in \mathbb N_0, \, t > 0, \\
    u(0,\cdot) & = f(\cdot),
\end{aligned}
\end{equation}
where $f: \mathbb{R} \rightarrow \ell_{1}^{+}$ is the function in~\eqref{Paper02_initial_condition}, and $F = (F_k)_{k \in \mathbb{N}_0}: \ell_{1}^{+} \rightarrow \ell_{1}$ is given by, for all $u = (u_k)_{k \in \mathbb{N}_0} \in \ell_1^+$,
\begin{equation} \label{Paper02_reaction_term_PDE}
    F_{k}(u) = q_{+}(\| u \|_{\ell_1})\left(s_{k}(1-\mu)u_{k} + \mathds{1}_{\{k \geq 1\}}s_{k-1}\mu u_{k-1}\right) - q_{-}(\| u \|_{\ell_1})u_{k} \quad 
    \forall \, k \in \mathbb N_0.
\end{equation}

Note that,  
in the system of PDEs above, $\displaystyle \frac{m}{2} \Laplace u_{k}$ corresponds to migration of particles carrying exactly $k$ mutations; $q_{+}(\| u\|_{\ell_{1}})s_{k}(1 - \mu)u_{k}$ is the rescaled rate at which particles carrying exactly $k$ mutations give birth to new particles carrying exactly $k$ mutations; $q_{-}(\| u\|_{\ell_{1}}) u_{k}$ is the rescaled rate at which particles carrying $k$ mutations die; and $q_{+}(\| u\|_{\ell_{1}})s_{k-1}\mu u_{k-1}$ is the rescaled rate at which particles carrying exactly $k-1$ mutations give birth to new particles carrying $k$ mutations (when $k \geq 1$).
Our main result rigorously confirms the conjectured convergence to the system of PDEs.  

\begin{theorem} \label{Paper02_deterministic_scaling_foutel_rodier_etheridge}
Suppose that $(m_N)_{N \in \mathbb N}$, $(L_N)_{N \in \mathbb N}$, $(s_k)_{k \in \mathbb N_0}$, $q_+$, $q_-$ and $f$ satisfy Assumptions~\ref{Paper02_scaling_parameters_assumption},~\ref{Paper02_assumption_fitness_sequence},~\ref{Paper02_assumption_polynomials} and~\ref{Paper02_assumption_initial_condition}. For $N \in \mathbb N$, define $\boldsymbol \eta^N$ as in~\eqref{Paper02_initial_condition}, and let $(\eta^N(t))_{t \geq 0}$ denote the càdlàg $\mathcal{S}^N$-valued strong Markov process with generator $\mathcal{L}^N$ defined in~\eqref{Paper02_generator_foutel_etheridge_model} and~\eqref{Paper02_infinitesimal_generator} with $\eta^N(0) = \boldsymbol{\eta}^N$ almost surely.
Then, as $N \rightarrow \infty$, the approximate density process $(u^{N}(t))_{t \geq 0}$ defined in~\eqref{Paper02_approx_pop_density_definition} and~\eqref{Paper02_definition_u_N_as_radon_measure} converges in distribution on $\mathcal{D}\left([0,\infty), (\mathcal{M}(\mathbb{R})^{\mathbb{N}_{0}}, d)\right)$ with respect to the $J_1$-topology  to a continuous-time $\mathcal{M}(\mathbb{R})^{\mathbb{N}_{0}}$-valued process $(u(t))_{t \geq 0} = ((u_k(t))_{k \in \mathbb N_0})_{t \geq 0}$, which satisfies the following conditions:
\begin{enumerate}[label = (\roman*)]
    \item For every $k \in \mathbb{N}_{0}$ and $t \geq 0$, $u_{k}(t)$ is  absolutely continuous with respect to the Lebesgue measure, with density denoted by $\left(u_{k}(t,x)\right)_{x \in \mathbb{R}}$.
    \item The family of sequences of densities $(u(t, \cdot))_{t \geq 0} = ((u_k(t, \cdot))_{k \in \mathbb N_0})_{t \geq 0}$ is a non-negative mild solution to the system of PDEs~\eqref{Paper02_PDE_scaling_limit}, and satisfies, for all $T > 0$,
    \[
    \esssup_{(t,x) \in [0, T] \times \mathbb R} \|u(t,x)\|_{\ell_1} < \infty.
    \]
    \item The family of sequences of densities $(u(t, \cdot))_{t \geq 0} = ((u_k(t, \cdot))_{k \in \mathbb N_0})_{t \geq 0}$ is such that for all $k \in \mathbb N_0$, $(u_k(t, \cdot))_{t \geq 0}$ is equal almost everywhere on $[0, \infty) \times \mathbb R$ to a map $\hat{u}_k: [0, \infty) \times \mathbb R \rightarrow [0, \infty)$ such that $\hat u_{k} \in \mathcal{C}^{1,2}((0,\infty) \times \mathbb{R}; \mathbb{R}).$
\end{enumerate}
Finally, $(u(t,\cdot))_{t \geq 0}$ is the unique weak solution to the system of PDEs~\eqref{Paper02_PDE_scaling_limit} satisfying conditions~(i)-(iii).
\end{theorem}

We will carefully define what it means to say that $u$ is a weak solution or a mild solution of~\eqref{Paper02_PDE_scaling_limit} in Section~\ref{Paper02_subsection_preliminaries}.

\begin{remark}
    We note that the arguments used in the proof of Theorem~\ref{Paper02_deterministic_scaling_foutel_rodier_etheridge} remain valid under more general assumptions. In particular, Assumption~\ref{Paper02_assumption_fitness_sequence}(iii) is not essential for our proof of the functional law of large numbers. Furthermore, instead of requiring the per-capita birth and death rates $q_+, q_- : [0,\infty)\to[0,\infty)$ to be polynomials satisfying Assumption~\ref{Paper02_assumption_polynomials}, it suffices that $q_+$ and $q_-$ are non-negative locally Lipschitz functions such that:
\begin{enumerate}
    \item[(1)] $ \liminf_{u\to\infty} \displaystyle \frac{q_-(u)}{u\,q_+(u)+1} > 0$.
    \item[(2)]  $\lim_{u\to\infty} q_-(u) = \infty$.
    \item[(3)] There exists $u^*\in(0,\infty)$ for which $q_-(u^{(1)}) > q_-(u^{(2)})$ whenever $u^{(1)} > u^{(2)} > u^*$.
    \item[(4)] There exists a non-negative polynomial $p: [0,\infty) \rightarrow [0, \infty)$ such that $q_-(u) \leq p(u)$, for all $u \in [0, \infty)$.
\end{enumerate}
The same arguments also yield the functional law of large numbers for the spatial Muller's ratchet defined on the rescaled $d$-dimensional lattice $L_N^{-1}\mathbb{Z}^d$ for any spatial dimension $d\in\mathbb{N}$.
\end{remark}

\subsection{The asymptotic behaviour of the particle system} \label{Paper02_sec:subsection_asymptotic_behaviour_PDE}

As an application of our law of large numbers result, we can control 
the long-term behaviour of the mutational profile under certain conditions on the reaction term~\eqref{Paper02_reaction_term_PDE} of the system of PDEs~\eqref{Paper02_PDE_scaling_limit}. 
To introduce these conditions, we first recall the terminology usually used to describe one-dimensional reaction-diffusion equations.
Consider the PDE
\begin{equation} \label{Paper02_one_dimensional_RD}
    \partial_{t}U = \frac{m}{2}\Laplace U + g(U) \quad \textrm{for } t > 0,
\end{equation}
where $g \in \mathcal{C}^{1}(\mathbb{R})$, $g(0) = g(1) = 0$ and $\int_{0}^{1} g(U) \, dU > 0$. We say that the reaction term $g$ is \textit{monostable} if $g'(0) > 0$, $g'(1) < 0$ and $g(x) > 0$ for all $x \in (0,1)$. We adapt this notion to the system~\eqref{Paper02_PDE_scaling_limit} as follows.

\begin{definition}[Monostable reaction terms] \label{Paper02_assumption_monostable_condition}
We say that the reaction term $F=(F_{k})_{k \in \mathbb{N}_{0}}: \ell_1^+ \rightarrow \ell_1$ defined in~\eqref{Paper02_reaction_term_PDE} is \emph{monostable} if the sequence of fitness parameters $(s_{k})_{k \in \mathbb{N}_{0}}$ satisfies Assumption~\ref{Paper02_assumption_fitness_sequence} and is strictly decreasing, the mutation rate satisfies $\mu \in (0,1)$, and the functions $q_{+},q_{-}: [0,\infty) \rightarrow [0,\infty)$ satisfy Assumption~\ref{Paper02_assumption_polynomials} and the following conditions:
\begin{enumerate}[label = (\roman*)]
    \item $q_{+}(U) \geq q_{-}(U) \quad \forall \, U \in [0,1]$.
    \item $q_{+}(U) > 0 \; \forall \, U \in [0,1]$.
    \item $q_{+}(1) = q_{-}(1)$.
    \item $q'_{+}(1) < q'_{-}(1)$.
    \item $(1-\mu)q_{+}(0) - q_{-}(0) > 0$.
\end{enumerate}
\end{definition}

\ignore{Note that if the reaction term $g(U) = U(q_+(U) - q_-(U))$ of the one-dimensional PDE
\begin{equation} \label{Paper02_correspondent_one_dimensional_system}
    \partial_{t} U = \frac{m}{2} \Laplace U + U(q_{+}(U) - q_{-}(U))
\end{equation}
is monostable, then conditions (i), (iii) and (iv) of Definition~\ref{Paper02_assumption_monostable_condition} are automatically satisfied. Moreover, if the reaction term of~\eqref{Paper02_correspondent_one_dimensional_system} is monostable, then $q_{+}(0) >q_{-}(0)$, and therefore condition (v) of Definition~\ref{Paper02_assumption_monostable_condition} must hold for sufficiently small $\mu$. We also must have $q_+(U) > 0$ for all $U \in (0,1)$, but the additional condition $q_+(1) > 0$ in condition (ii) is not implied by the monostability of the reaction term of~\eqref{Paper02_correspondent_one_dimensional_system}; this is a technical assumption required in our proofs in the companion article~\cite{madeira2025PDE_surfing}. This condition, however, is not restrictive, since it is reasonable to assume in biological models that the per-capita reproduction and death rates are both strictly positive in high population density.}

Next, we assume some control on the proportions of particles carrying mutations in the initial condition. Recall that we denote the Lebesgue measure on $\mathbb R$ by $\lambda$.

\begin{assumption}[Control on the initial prevalence of mutations] \label{Paper02_assumption_initial_condition_modified}
Let $f = \left(f_{k}\right)_{k \in \mathbb{N}_{0}}: \mathbb{R} \rightarrow \ell_{1}^{+}$ be a function satisfying Assumption~\ref{Paper02_assumption_initial_condition}, and the following additional conditions:
\begin{enumerate}[label = (\roman*)]
    \item $\displaystyle \| f \|_{L_{\infty}(\mathbb{R}; \ell_{1})} \defeq \esssup_{x \in \mathbb{R}} \, \| f(x) \|_{\ell_{1}} \leq 1$.
   \item $\lambda(\{x \in \mathbb R: \, f_0(x) > 0\}) > 0$.
   \item There exists a sequence $\left(\hat{\pi}_{k}\right)_{k \in \mathbb{N}_{0}} \in \ell_{1}^{+}$ such that for each $k \in \mathbb{N}$ and $\lambda$-almost every $x \in \mathbb R$,
   \begin{equation*}
       0 \leq f_{k}(x) \leq \hat{\pi}_{k} f_{0}(x).
   \end{equation*}
\end{enumerate}
\end{assumption}

For $\mu \in (0,1)$, let $(\alpha_{k})_{k \in \mathbb{N}_{0}} = \Big(\alpha_{k}(\mu, (s_j)_{j \in \mathbb N_0})\Big)_{k \in \mathbb{N}_{0}}$ be given by $\alpha_{0} \defeq 1$ and
\begin{equation} \label{Paper02_definition_alpha_k}
    \alpha_{k} \defeq \prod_{i=1}^{k} \frac{\mu s_{i-1}}{(1-\mu) (1 - s_{i})} \quad \; \forall \, k \in \mathbb{N}.
\end{equation}
The sequence~$(\alpha_k)_{k \in \mathbb N_0}$ defined in~\eqref{Paper02_definition_alpha_k} gives the proportions of the population carrying different numbers of mutations in a stationary solution of the system of PDEs~\eqref{Paper02_PDE_scaling_limit} (see~\cite[Remark~2.3]{madeira2025PDE_surfing}). Let
\begin{equation} \label{Paper02_definition_max_min_birth_polynomial}
\mathfrak{Q}_{\min} \defeq \displaystyle \min_{U \in [0,1]} q_{+}(U) \quad \textrm{and} \quad \mathfrak{Q}_{\max} \defeq \displaystyle \max_{U \in [0,1]} q_{+}(U).
\end{equation}
Note that, by Definition~\ref{Paper02_assumption_monostable_condition}, if the reaction term $F = (F_{k})_{k \in \mathbb{N}_{0}}$ is monostable, then $\alpha_k > 0$ for all $k \in \mathbb N_0$, and $\mathfrak{Q}_{\max} \geq \mathfrak{Q}_{\min} > 0$. In the monostable regime, we use Theorem~\ref{Paper02_deterministic_scaling_foutel_rodier_etheridge}, together with the results about the solutions to the system of PDEs~\eqref{Paper02_PDE_scaling_limit} in our companion article~\cite{madeira2025PDE_surfing}, to establish the following result about 
the proportions of particles carrying different numbers of mutations in the spatial Muller's ratchet. Recall from the end of Section~\ref{Paper02_introduction} that for $K \in \mathbb N$, we let $\llbracket K \rrbracket = \{1, \ldots, K\}$.

\begin{theorem} \label{Paper02_long_time_behaviour_IPS}
   Suppose that $(m_N)_{N\in \mathbb N}$, $(L_N)_{N\in \mathbb N}$, $(s_k)_{k \in \mathbb N_0}$, $q_+$, $q_-$ and $f$ satisfy Assumptions~\ref{Paper02_scaling_parameters_assumption},~\ref{Paper02_assumption_fitness_sequence},~\ref{Paper02_assumption_polynomials} and~\ref{Paper02_assumption_initial_condition_modified}, and that the reaction term $F = (F_{k})_{k \in \mathbb{N}_{0}}$ defined in~\eqref{Paper02_reaction_term_PDE} is monostable in the sense of Definition~\ref{Paper02_assumption_monostable_condition}. For $N \in \mathbb N$, define $\boldsymbol\eta^N$ as in~\eqref{Paper02_initial_condition}, and let $(\eta^N(t))_{t \geq 0}$ denote the càdlàg $\mathcal{S}^N$-valued strong Markov process with generator $\mathcal L^N$ defined in~\eqref{Paper02_generator_foutel_etheridge_model} and~\eqref{Paper02_infinitesimal_generator} with $\eta^N(0) = \boldsymbol{\eta}^N$ almost surely. Let $(u^N(t))_{t \geq 0} = ((u^N_k(t))_{k \in \mathbb N_0})_{t \geq 0}$ denote the approximate density process defined in~\eqref{Paper02_approx_pop_density_definition} and~\eqref{Paper02_definition_u_N_as_radon_measure}. Then for any compact interval with strictly positive length $\mathcal{I} \subset \mathbb{R}$, any $K \in \mathbb{N}$ and any $\delta, \varepsilon > 0$, there exists $T_{\delta, K} > 0$ such that for any $T \geq T_{\delta,K}$, there exists $N_{\delta, \varepsilon,\mathcal{I}, K, T} \in \mathbb{N}$ such that for $N \geq N_{\delta, \varepsilon, \mathcal{I}, K, T}$ and $k \in \llbracket K \rrbracket$,
   \begin{equation*}
   \begin{aligned}
   \mathbb{P}\bigg(\left(\alpha_{k}\Big(\tfrac{\mathfrak{Q}_{\min}}{\mathfrak{Q}_{\max}}\right)^{k}-\delta\Big) u^{N}_{0}(T)(\mathcal{I}) \leq u^{N}_{k}(T)(\mathcal{I})  \leq  \Big(\alpha_{k}\left(\tfrac{\mathfrak{Q}_{\max}}{\mathfrak{Q}_{\min}}\right)^{k} +\delta\Big) u^{N}_{0}(T)(\mathcal{I}) \bigg) \geq 1 - \varepsilon,
   \end{aligned}
   \end{equation*}
where the sequence $(\alpha_{k})_{k \in \mathbb{N}_{0}}$ is given by~\eqref{Paper02_definition_alpha_k}, and $\mathfrak{Q}_{\min}$ and $\mathfrak{Q}_{\max}$ are given by~\eqref{Paper02_definition_max_min_birth_polynomial}. 
\end{theorem}

\subsection{Application to different settings of population dynamics} \label{Paper02_example_pop_dynamics}

In this subsection, we show how our results apply to different models of population growth. Since our definition of the spatial Muller's ratchet $\eta^N$ is rather general, it can be used to model asexual populations under different biological assumptions. For instance, taking for some $r >0$ and $B \geq 0$,
\begin{equation} \label{Paper02_foutel_etheridge_birth_death_rates}
      q_{+}(U) \defeq r(BU + 1) \quad \textrm{ and } \quad q_{-}(U) \defeq r(BU+1)U \quad \forall U \geq 0, 
\end{equation}
we recover the same birth and death rates as in the model introduced by Foutel-Rodier and Etheridge in~\cite{foutel2020spatial} (recall that we described this model in Section~\ref{Paper02_biological_motivation}).
Our Theorem~\ref{Paper02_deterministic_scaling_foutel_rodier_etheridge} shows that for $q_+$ and $q_-$ given by~\eqref{Paper02_foutel_etheridge_birth_death_rates}, under Assumptions~\ref{Paper02_scaling_parameters_assumption} and~\ref{Paper02_assumption_fitness_sequence} on $(m_N)_{N\in \mathbb N}$, $(L_N)_{N\in \mathbb N}$ and $(s_k)_{k \in \mathbb N_0}$ and suitable assumptions on the initial condition,
the spatial Muller's ratchet $\eta^N$ converges as $N \rightarrow \infty$ to the solution of the corresponding system of PDEs~\eqref{Paper02_PDE_scaling_limit} (which is the same as~\eqref{Paper02_conjectured_PDE_heuristics} in the special case $s_k = (1-s)^k$ for all $k \in \mathbb N_0$). For the rates in~\eqref{Paper02_foutel_etheridge_birth_death_rates}, the conditions~(i)--(v) on~$q_+$ and~$q_-$ in Definition~\ref{Paper02_assumption_monostable_condition} are satisfied for all $B \geq 0$, $r > 0$ and $\mu \in (0,1)$. Hence, when the fitness sequence $(s_k)_{k \in \mathbb N_0}$ is strictly decreasing, we 
can deduce from Theorem~\ref{Paper02_long_time_behaviour_IPS} quantitative bounds on the proportions of the population carrying $k \in \mathbb{N}_0$ mutations. 

Another interesting example is to take
\begin{equation*}
    q_{+}(U) \defeq U(B+1) \quad \textrm{and} \quad q_{-}(U) \defeq U^{2} + B \quad \forall \, U \geq 0, \textrm{ for some } B \in \left(0, 1/2\right).
\end{equation*}
Then the function $U \mapsto q_{+}(U) - q_{-}(U) = (1-U)(U-B)$ is negative for $U \in [0,B)$. We say that the population (neglecting the effect of deleterious mutations) exhibits a strong Allee effect \cite{garnier2012inside}. Biologically, this means that cooperation is fundamental for population growth, 
as the population shrinks at low population densities. In this case, our Theorem~\ref{Paper02_deterministic_scaling_foutel_rodier_etheridge} again shows the convergence of the spatial Muller's ratchet to the solution of the corresponding system of PDEs~\eqref{Paper02_PDE_scaling_limit}. However, the asymptotic behaviour of the particle system~$\eta^N$ is not determined by Theorem~\ref{Paper02_long_time_behaviour_IPS}, because the reaction term $F = (F_{k})_{k \in \mathbb{N}_{0}}$ given by~\eqref{Paper02_reaction_term_PDE} is not monostable in the sense of Definition~\ref{Paper02_assumption_monostable_condition}. We highlight that the main obstruction to establishing Theorem~\ref{Paper02_long_time_behaviour_IPS} in this case is determining the asymptotic behaviour of solutions to the system of PDEs~\eqref{Paper02_PDE_scaling_limit} when the reaction term is bistable.

\subsection{Related work} \label{Paper02_related_work_introd}

The study of the convergence of interacting particle systems to reaction–diffusion partial differential equations has been a very active area of research in probability theory since at least the 1980s. For spatial birth–death processes, there are several approaches to proving convergence to a hydrodynamic limit, depending on the assumptions on the initial condition, the spatial domain, and the dynamics of the discrete model. In this subsection, we briefly review some of this literature and explain why the methods developed previously cannot be directly applied to the proof of Theorem~\ref{Paper02_deterministic_scaling_foutel_rodier_etheridge}.

Roughly speaking, the literature on hydrodynamic scaling limits of spatial birth–death processes can be divided into five main categories:
\begin{enumerate}
    \item[(A)] Works that analyse models in which there is an \textit{a priori} bound on the local number of particles, so that in the limit the population density is always uniformly bounded (see for instance \cite{durrett1994particle,mueller1995stochastic,durrett2016genealogies}).
    \item[(B)] Works that assume the spatial domain of the process is compact, e.g.~the torus $\mathbb{T}^{d} \defeq [0,1]^{d}$ (see for instance \cite{arnold1980consistency,arnold1980deterministic,kotelenez1986law,kotelenez1988high,blount1991comparison,blount1992law,blount1993limit,blount1994density,feng1996hydrodynamic}).
    \item[(C)] Works that apply the method of correlation functions to derive the limit \cite{boldrighini1987collective,demasi1991mathematical, perrut2000hydrodynamic,tendron2024non}.
    \item[(D)] Works that use the relative entropy method \cite{mourragui1996comportement,perrut2000hydrodynamic,jara2018non}.
    \item[(E)] Works that analyse models with finitely many particles in the whole domain and bounded reproduction rates, so that there are bounds on the moments of the total mass~\cite{etheridge2023looking,flandoli2021kpp}.
\end{enumerate}

To the best of our knowledge, there are no existing works that address spatial birth–death processes with infinitely many particle types, unbounded reproduction rates, and an unbounded number of particles per site simultaneously. Consequently, the techniques developed in the works mentioned above cannot be applied to our model without significant adaptation.

Works in category (A) typically construct the birth–death interacting particle system using a countable collection of Poisson processes. Since, in these works, the number of particles per deme is bounded by some integer $N$, the approximate population density is uniformly bounded, see e.g.~\cite{durrett1994particle, mueller1995stochastic, durrett2016genealogies}. This uniform bound means that convergence can be studied in the space of continuous functions, which is strictly stronger than convergence in the space of Radon measures. Since no uniform density bound holds in our setting, these techniques cannot be applied directly.

To handle the absence of uniform bounds, Arnold and Theodosopulu~\cite{arnold1980consistency, arnold1980deterministic}, Blount~\cite{blount1991comparison, blount1992law, blount1993limit, blount1994density}, and Kotelenez~\cite{kotelenez1986law, kotelenez1988high} developed techniques in the 1980s and 1990s that apply to spatial birth–death processes on the one-dimensional torus $\mathbb{T} = [0,1]$. These works (category (B) in the list above) exploit the construction of an explicit orthonormal basis of the Hilbert space $L_{2}([0,1])$ to analyse weak convergence. Importantly, their techniques rely on properties of martingales taking values in Hilbert spaces, that are not shared by martingales taking values in general Banach spaces. Since our model involves infinitely many particle types, the solution of the limiting system of PDEs~\eqref{Paper02_PDE_scaling_limit} is most naturally understood as an $\ell_{1}$-valued function on $[0,\infty) \times \mathbb{R}$. Consequently, it is not natural to study convergence in Hilbert spaces, and these techniques do not transfer directly to our setting.

In category (C), Boldrighini et al.~were the first to rigorously analyse a spatial birth–death process with polynomial birth and death rates on the whole real line \cite{boldrighini1987collective, demasi1991mathematical}. They consider a single-type particle system on the lattice $\varepsilon \mathbb{Z}$, where $\varepsilon > 0$ is a spatial scaling parameter, and study its behaviour as $\varepsilon \to 0$. Importantly, their setting does not involve scaling the initial number of particles per deme by $1/\varepsilon$ to produce a mean-field limit. Instead, they use correlation functions, a form of duality function, to bound local moments and show that the limit satisfies a so-called BBGKY hierarchy consistent with the target PDE. More recently, Tendron also used the BBGKY hierarchy to study the scaling limit of branching processes with local competition~\cite{tendron2024non}.

The BBGKY hierarchy becomes intractable for systems with infinitely many particle types (see \cite[Section~4]{boldrighini1987collective}), which means this method is not feasible for our purposes. Nonetheless, we adapt some ideas from \cite{boldrighini1987collective, demasi1991mathematical}, in particular the use of correlation functions, to obtain the estimates in our companion article~\cite[Theorem~2.3]{madeira2025existence} that are used in the proof of the hydrodynamic limit result in this article.

The relative entropy method, used by Mourragui~\cite{mourragui1996comportement} and others~\cite{perrut2000hydrodynamic, jara2018non} (category (D)), requires the solution of the limiting PDE to be uniformly bounded away from zero. In our setting, the limiting population density $u(t,x)$ takes values in $\ell_1$, so for any $t \ge 0$ and $x \in \mathbb{R}$ we have $\lim_{k \to \infty} u_k(t,x) = 0$. It is therefore not clear how to apply the relative entropy method to our model.

In a recent article closely related to our work, Etheridge et al.~\cite{etheridge2023looking} study a spatial birth–death process on $\mathbb{R}^{d}$ with non-local interactions. The authors prove weak convergence of the process under different scaling regimes, either retaining the non-local interaction term in the limit or not. A notable advantage of the approach in~\cite{etheridge2023looking} is that it does not require \emph{a priori} bounds on the density of particles per unit region. By assuming the reproduction rate is uniformly bounded and that the initial configuration contains only finitely many particles, the authors derive precise moment estimates for the total mass of the process, which they use to establish tightness of the sequence of interacting particle systems and to characterise the limit. Moreover, to prove uniqueness of the limiting PDE in the space of measures, they apply~\cite[Theorem~3.5]{kurtz1999particle}, which requires the reproduction and death rates to be uniformly bounded and globally Lipschitz continuous.

In a similar vein, Flandoli and Huang~\cite{flandoli2021kpp} analyse the scaling limit of Brownian particles in~$\mathbb{R}^{d}$ with a uniformly bounded reproduction rate, proving convergence to the Fisher–KPP equation without non-local interaction. Their technique also relies on the uniform bound on the reproduction rate and on having a finite initial number of particles. In contrast, in our setting, in general the reproduction rate is not uniformly bounded, and neither the reproduction nor death rates are globally Lipschitz continuous. We also allow for infinitely many particles at time~$0$. Consequently, we cannot rely on these methods to prove convergence of the spatial Muller's ratchet to the solution of the system of PDEs~\eqref{Paper02_PDE_scaling_limit} without substantial adaptations.

We would also like to highlight some previous work on the hydrodynamic limit of coagulation-fragmentation processes \cite{lang1980smoluchowski, hammond2006kinetic, hammond2007kinetic}. Although in these works the authors study particle systems with infinitely many types, there is no creation of mass, since only coagulation, fragmentation and diffusion events are allowed. Hence, the total mass is always uniformly bounded, and $L_{\infty}$ bounds can be obtained on the population densities. This is not possible in our setting, which makes our proof significantly different. 

Finally, there is also previous work on the scaling limits of interacting particle systems with infinitely many types of particles but without spatial dependence~\cite{barbour2008laws, barbour2012law, barbour2012central, rath2009mean, yeo2018frozen}. These systems arise naturally from questions involving graph dynamics and epidemic models. In this class of systems, it is often possible to describe the deterministic scaling limit in terms of ordinary differential equations taking values in the space of sequences $\ell_{1}$. The fact that we analyse a population with spatial structure, however, makes our approach very different. 

\subsection{Overview of the proof} \label{Paper02_subsection_heuristics}

Before introducing the main ideas of the proof, we highlight the principal challenges in analysing the spatial Muller's ratchet model. First, the lack of \textit{a priori} bounds on the local number of particles — together with unbounded birth and death rates — prevent us from establishing uniform estimates in space; in particular, we are not able to bound~$\mathbb{E}\Big[\| u^N(t,\cdot) \|_{L_\infty(\mathbb{R};\ell_1)}\Big]$. Second, letting the `type' of a particle denote the number of mutations that it carries, having infinitely many types of particles requires us to embed the stochastic process in non-standard Polish spaces to study its scaling limit. Finally, the dynamics is non-local in type space, since particles of different types in the same deme interact with each other (in the sense that birth and death rates depend on the total number of particles in the deme). As we explain below, the key step in overcoming these challenges is to combine local moment estimates with regularity properties derived from the Green’s function representation of the random walk semigroup, to show tightness in suitable $\ell_1$-valued $L_p$ spaces in which uniqueness and smoothness of the solution to the system of PDEs in~\eqref{Paper02_PDE_scaling_limit} can be established.

\ignore{\noindent\textbf{Outline of proof of Theorem~\ref{Paper02_deterministic_scaling_foutel_rodier_etheridge}.}}
The proof of Theorem~\ref{Paper02_deterministic_scaling_foutel_rodier_etheridge} can be roughly divided into two steps: the proof of tightness of the sequence of processes $(u^{N})_{N \in \mathbb{N}}$ in $\mathcal{D}([0,\infty), (\mathcal{M}(\mathbb{R})^{\mathbb{N}_{0}}, d))$, and the characterisation and proof of uniqueness of the limiting process. An important ingredient comes from  our companion article~\cite[Theorem~2.3]{madeira2025existence}, where we 
establish the following estimate:
for every $p \geq 1$ and $T > 0$, and $f = (f_k)_{k \in \mathbb N_0}: \mathbb R \rightarrow \ell_1$ satisfying Assumption~\ref{Paper02_assumption_initial_condition}, defining $\boldsymbol{\eta}^N$ as in~\eqref{Paper02_initial_condition} and letting $\eta^N(0) = \boldsymbol \eta^N$ for every $N \in \mathbb N$,
\begin{equation} \label{Paper02_heuristics_correlation_functions_iii}
     \sup_{N \in \mathbb{N}} \; \sup_{t \leq T} \; \sup_{x \in L_N^{-1}\mathbb{Z}} \mathbb{E}\left[ \| u^N(t,x) \|^p_{\ell_1}\right] \lesssim_{p,T}  \| f \|_{L_{\infty}(\mathbb{R};\ell_1)}^p + 1.
\end{equation}
{The estimate relies on the fact that by Assumption~\ref{Paper02_assumption_fitness_sequence}, $s_k \leq 1$ $\forall k \in \mathbb N_0$, and that by Assumption~\ref{Paper02_assumption_polynomials}, the birth and death polynomials are such that $0 \leq \deg q_+ < \deg q_-$. The proof in~\cite{madeira2025existence} uses the method of correlation functions introduced in~\cite{boldrighini1987collective, demasi1991mathematical}.}
For  ease of reference, we will state this result precisely in Theorem~\ref{Paper02_bound_total_mass} below. Estimate~\eqref{Paper02_heuristics_correlation_functions_iii} is a ‘local' estimate in the sense that the supremum over space is outside the expectation; as mentioned at the start of this subsection, we are not able to bound $\mathbb{E}\Big[\| u^N(t,\cdot) \|_{L_\infty(\mathbb{R};\ell_1)}\Big]$.

To establish tightness of $(u^N)_{N \in \mathbb N}$, and to characterise any subsequential limit, we will need to establish more regularity properties of the sequence $(u^N)_{N \in \mathbb{N}}$. This is carried out in Section~\ref{Paper02_section_greens_function}, where we will apply a Green's function representation of the migration semigroup to write the population density in each deme in terms of a martingale. This technique is widely used for establishing scaling limits of spatial birth-death processes with an almost surely bounded number of particles per site~\cite{durrett2016genealogies,mueller1995stochastic}. More precisely, we will establish in Lemma~\ref{Paper02_lemma_little_u_k_N_hat_semimartingale} that for every $N \in \mathbb{N}$,~$x \in L_N^{-1}\mathbb{Z}$,~$\mathcal{I} \subseteq \mathbb{N}_{0}$ and~$T \geq 0$, there exists a càdlàg martingale $\Big(M^{N,T,x}_{\mathcal{I}}(t)\Big)_{t \in [0,T]}$ with $M^{N,T,x}_{\mathcal{I}}(0) = 0$ such that
\begin{equation} \label{Paper02_heuristics_greens_function}
\begin{aligned}
    \sum_{k \in \mathcal{I}} u^{N}_{k}(T,x) & = \sum_{y \in L_N^{-1}\mathbb Z} \mathbb{P}_{0}(X^N(T) = y - x)  \sum_{k \in \mathcal{I}} u^{N}_{k}(0,y) + M^{N,T,x}_{\mathcal{I}}(T) \\ & \quad \quad + \int_{0}^{T} \sum_{y \in L_N^{-1}\mathbb Z} \mathbb{P}_{0}(X^N(T-t) = y-x) \sum_{k \in \mathcal{I}} F_k(u^N(t-,y)) \, dt,
\end{aligned}
\end{equation}
where the reaction term $F = (F_k)_{k \in \mathbb{N}_0}$ is given by~\eqref{Paper02_reaction_term_PDE}, and $(X^N(t))_{t \geq 0}$ is a simple symmetric random walk on $L^{-1}_N\mathbb Z$ with total jump rate $m_N$, and $\mathbb{P}_0$ is the probability measure under which $X^N(0) = 0$. Combining~\eqref{Paper02_heuristics_greens_function}, an expression for $\left \langle M^{N,T,x}_{\mathcal I} \right \rangle$ and classical martingale estimates will allow us to derive weak equicontinuity properties for $(u^N)_{N \in \mathbb N}$ in Lemma~\ref{Paper02_lemma_increments_liitle_u_hat_n_k}. Moreover, by combining~\eqref{Paper02_heuristics_greens_function} with the fact that by Assumption~\ref{Paper02_assumption_fitness_sequence}, $s_k \rightarrow 0$ as $k \rightarrow \infty$, we will be able to uniformly bound the expectation of the sum over large $k$ of $u^N_k(t,x)$ in
Lemma~\ref{Paper02_lemma_control_local_density_high_number_mutations}. These properties, together with~\eqref{Paper02_heuristics_correlation_functions_iii} and standard stochastic analysis arguments, will be used in our proof of tightness in Section~\ref{Paper02_tightness}.

A major challenge in the proof of Theorem~\ref{Paper02_deterministic_scaling_foutel_rodier_etheridge} is that, since we embed our process in the space of $\mathcal{M}(\mathbb{R})^{\mathbb{N}_{0}}$-valued càdlàg functions, tightness alone does not imply that any subsequential limit has sample paths lying in the subset of $\mathcal{M}(\mathbb R)^{\mathbb N_0}$ given by sequences of measures that are absolutely continuous with respect to Lebesgue measure, nor does it guarantee that the birth and death rate polynomials of the local population density converge in the space of measures. Moreover, even after establishing the existence of a sequence of densities for the limiting measure-valued process, we note that, because tightness is proved in the space of $\mathcal{M}(\mathbb{R})^{\mathbb{N}_{0}}$-valued càdlàg functions, proving uniqueness of the limit requires us to establish uniqueness of solutions to the system of PDEs~\eqref{Paper02_PDE_scaling_limit} in $\mathcal{M}(\mathbb{R})^{\mathbb{N}_{0}}$. This is highly non-trivial due to the non-linear reaction term in~\eqref{Paper02_reaction_term_PDE} that is neither uniformly bounded nor globally Lipschitz continuous, and because, as explained at the start of this subsection, we do not have any guarantee that the limiting solution remains uniformly bounded in space over finite time intervals (see~\cite{brezis1979uniqueness,kurtz1999particle,carrillo2024dissipative} and the references therein for a discussion of the challenges in establishing uniqueness of measure-valued solutions to non-linear PDEs).

To overcome these difficulties, we will also show tightness of $(u^N)_{N\in \mathbb{N}}$ in the function space $L_{4\deg q_-}([0,T] \times \mathbb R, \hat \lambda; \ell_1)$, where for $p \ge 1$,
\begin{equation*}
\begin{aligned}
    & {L}_{p}([0,T] \times \mathbb{R}, \hat{\lambda}; \ell_{1}) \\ & \quad \defeq \bigg\{v: [0,T] \times \mathbb{R} \rightarrow \ell_{1}: \; \| v \|_{{L}_{p}([0,T] \times \mathbb{R}, \hat{\lambda}; \ell_{1})} \defeq \bigg(\int_{0}^{T} \int_{\mathbb{R}} \frac{\| v(t,x) \|_{\ell_{1}}^{p}}{1 + \vert x \vert^{2}} \, dx \, dt\bigg)^{1/p} < \infty\bigg\},
\end{aligned}
\end{equation*}
and $\hat{\lambda}$ is a measure on $[0, \infty) \times \mathbb{R}$ which is absolutely continuous with respect to the Lebesgue measure $\lambda$, and given by
\begin{equation*}
    \hat{\lambda}(dt \; dx) \defeq \frac{\mathds{1}_{\{t \in [0,T]\}}}{1 + \vert x \vert^{2}} \lambda (dt \; dx).
\end{equation*}
We refer the reader to Section~\ref{Paper02_PDES_sequence_spaces} for a brief overview of a rigorous definition of measurable $\ell_{1}$-valued functions on $\mathbb{R}$. The advantage of working in ${L}_{4\deg q_-}([0,T] \times \mathbb{R}, \hat{\lambda}; \ell_{1})$ instead of $\mathcal{D}\Big([0, \infty), \mathcal{M}(\mathbb{R})^{\mathbb{N}_0}\Big)$ is that, as in the usual $L_p$ spaces, convergence in ${L}_{4\deg q_-}([0,T] \times \mathbb{R}, \hat{\lambda}; \ell_{1})$ implies convergence of the densities almost everywhere on the time-space box $[0,T] \times \mathbb{R}$. This idea generalises methods that rely on using Hilbert spaces to characterise the hydrodynamic limit of interacting particle systems as in~\cite{arnold1980consistency,arnold1980deterministic,kotelenez1986law,kotelenez1988high,blount1991comparison,blount1992law,blount1993limit,blount1994density,feng1996hydrodynamic}.
However, we are not aware of any previous results establishing tightness criteria for interacting
particle systems by considering the densities of particles as $L_p$-valued random variables.

By using Díaz and Mayoral's characterisation of compact subsets of Bochner spaces~\cite[Theorem~3.2]{diaz1999compactness} (for which an elementary proof can be found in~\cite{van2014compactness}), we will derive a tightness criterion for ${L}_{4\deg q_-}([0,T] \times \mathbb{R}, \hat{\lambda}; \ell_{1})$ in Lemma~\ref{Paper02_easy_tightness_criterion_L_p_spaces}. The criterion is simple, and requires the following conditions:
\begin{enumerate}[label = (\roman*)]
        \item The~$(8\deg q_-)$-moments of~$\big((u^{N}(t,x))_{t \in [0,T], \, x \in \mathbb{R}}\big)_{N \in \mathbb{N}}$ are uniformly bounded, i.e.
        \begin{equation*}
            \sup_{N \in \mathbb{N}} \; \sup_{t \in [0,T]} \; \sup_{x \in \mathbb{R}} \; \mathbb{E}\Big[\| u^{N}(t,x) \|_{\ell_{1}}^{8\deg q_-}\Big] < \infty.
        \end{equation*}
        \item \textit{Bounds on the expectation of the tail of $u^N(t,x)$ in type space:} 
        \begin{equation*}
            \lim_{k \rightarrow \infty} \; \sup_{N \in \mathbb{N}} \; \sup_{t \in [0,T]} \; \sup_{x \in \mathbb{R}} \; \mathbb{E}\bigg[\sum_{j= k}^\infty \, u^{N}_{j}(t,x) \bigg] = 0.
        \end{equation*}
        \item \textit{Integral weak equicontinuity property:}  There exist~$l_{1},l_{2},C_{T} > 0$ such that for every~$\gamma \in (-1,1)$ and~$i \in \{1,2\}$,
        \begin{equation*}
            \sup_{N \in \mathbb{N}} \; \sup_{k \in \mathbb{N}_{0}} \mathbb{E}\Big[\| \theta^{(i)}_{\gamma}u^{N}_{k} - u^{N}_{k} \|_{{L}_{1}([0,T] \times \mathbb{R}, \hat{\lambda}; \mathbb{R})}\Big] \leq C_{T} \vert \gamma \vert^{l_{i}},
        \end{equation*}
    \end{enumerate}
where for any function~$g: [0,T] \times \mathbb{R} \rightarrow \mathbb{R}$,~$\gamma \in \mathbb{R}$, and for all $(t,x) \in [0,T] \times \mathbb R$,
\begin{equation*}
    (\theta^{(1)}_{\gamma}g)(t,x) \defeq g(t + \gamma,x) \cdot \mathds 1_{\{t + \gamma \in [0,T]\}} \quad \textrm{and} \quad  (\theta^{(2)}_{\gamma}g)(t,x) \defeq g(t,x + \gamma).
\end{equation*}
Tightness of $(u^N)_{N \in \mathbb{N}}$ in ${L}_{4\deg q_-}([0,T] \times \mathbb{R}, \hat{\lambda}; \ell_{1})$ will then follow from this criterion by combining~\eqref{Paper02_heuristics_correlation_functions_iii} with the regularity properties derived from the Green's function representation in~\eqref{Paper02_heuristics_greens_function}.

Let $v = (v_k(t,x))_{k \in \mathbb{N}_0, \, t \in [0,T], \, x \in \mathbb{R}}$ be any subsequential limit of $(u^N)_{N \in \mathbb{N}}$ in the function space ${L}_{4\deg q_-}([0,T] \times \mathbb{R}, \hat{\lambda}; \ell_{1})$. By Skorokhod's representation theorem, we can construct $(u^N)_{N \in \mathbb{N}}$ and $v$ on the same probability space in such a way that the subsequence converges almost surely. Then, by taking the limit as $N \rightarrow \infty$ on both sides of the Green's function representation in~\eqref{Paper02_heuristics_greens_function}, and using the fact that the random walk semigroup converges to the heat kernel under Brownian scaling, letting $\{P_t\}_{t \geq 0}$ denote the semigroup of Brownian motion run at speed $m$, we will be able to establish in Lemma~\ref{Paper02_weak_convergence_mild_solution} that
 \begin{equation*}
       \mathbb{E}\Bigg[\int_{0}^{T} \int_{\mathbb{R}} \, \sum_{k = 0}^\infty \frac{1}{1 + \vert x \vert^{2}} \left\vert v_{k}(t,x) - (P_{t}f_{k})(x) - \int_{0}^{t} \Big(P_{t-\tau}F_{k}(v(\tau, \cdot))\Big)(x) \,d\tau\right\vert \, dx \, dt\Bigg] = 0,
   \end{equation*}
i.e.~that $v$ is a mild solution to the system of PDEs~\eqref{Paper02_PDE_scaling_limit}. Importantly, mild solutions to nonlinear PDEs are better characterised than measure-valued weak solutions. In fact, we establish uniqueness in ${L}_{4\deg q_-}([0,T] \times \mathbb{R}, \hat{\lambda}; \ell_{1})$ for $v$ using standard PDE arguments in our companion article~\cite[Proposition~6.1]{madeira2025PDE_surfing}. The only limitation of using the space ${L}_{4\deg q_-}([0,T] \times \mathbb{R}, \hat{\lambda}; \ell_{1})$ is that we lose information about any specific instant of time, since for any $t \in (0,T]$, the set $\{t\} \times \mathbb{R}$ has null Lebesgue measure. This, however, is not an important limitation, since by the tightness in $\mathcal{D}([0, \infty), \mathcal{M}(\mathbb{R})^{\mathbb{N}_0})$, we will establish in Lemma~\ref{Paper02_relation_measure_valued_L_P_time_space_boxes_solutions} that any limiting process $(u(t))_{t \geq 0} = ((u_k(t))_{k \in \mathbb{N}_{0}})_{t \geq 0}$ in $\mathcal{D}([0, \infty), \mathcal{M}(\mathbb{R})^{\mathbb{N}_0})$ satisfies the following identity in distribution, for every~$H \in \mathbb{N}$,~$(t_{h})_{h \in \llbracket H \rrbracket} \in [0,T)^{H}$,~$(\varphi_{h})_{h \in \llbracket H \rrbracket} \in \mathcal{C}_{c}(\mathbb{R})^{H}$ and~$(k_{h})_{h \in \llbracket H \rrbracket} \in (\mathbb{N}_{0})^{H}$:
    \begin{equation*}
       (\langle u_{k_{h}}(t_{h}), \, \varphi_{h} \rangle)_{h \in \llbracket H \rrbracket} \overset{d}{=} \Bigg(\lim_{t' \downarrow t_{h}} \frac{1}{t' - t_{h}}\int_{t_{h}}^{t'} \int_\mathbb R v_{k_{h}}(\tau, x) \varphi_{h} (x) \, dx \, d\tau\Bigg)_{h \in \llbracket H \rrbracket}.
    \end{equation*}
Together with the uniqueness of mild solutions, this relation implies uniqueness for the measure-valued process, and will complete the proof of our convergence result (Theorem~\ref{Paper02_deterministic_scaling_foutel_rodier_etheridge}).

It is important to note that although we work in the general setting of $\ell_1$-valued functions, the framework extends naturally to the analysis of scaling limits of interacting particle systems with finitely many particle types. In this case, condition~(ii) in the tightness criterion is trivial. Moreover, our method does not require the particle density to be almost surely bounded, nor does it require the birth and death rates to be uniformly bounded, or the initial number of particles to be finite. The framework is therefore flexible enough to apply to a broad class of interacting particle systems, including models related to those studied in~\cite{etheridge2023looking,flandoli2021kpp}.

\section{Preliminaries} \label{Paper02_subsection_preliminaries}

In this section, we first state in Section~\ref{Paper02_subsection_preliminaries_existence_moment_bouds} the main results from the companion article~\cite{madeira2025existence} that will be used in this article. Then, in Section~\ref{Paper02_PDES_sequence_spaces}, we introduce the notions of weak and mild $\ell_1$-valued solutions to partial differential equations, and state the main results from~\cite{madeira2025PDE_surfing} concerning solutions to the system of PDEs~\eqref{Paper02_PDE_scaling_limit} that are required for our analysis.

\subsection{Moment bound and martingale problem for the spatial Muller's ratchet} \label{Paper02_subsection_preliminaries_existence_moment_bouds}

In this subsection, we give an overview of the properties of the spatial Muller's ratchet proved in our companion article~\cite{madeira2025existence}. We start by stating some topological properties of the state space~$\mathcal{S}^{N}$, which can be found in~\cite[Proposition~A.5]{madeira2025existence}.

\begin{proposition} \label{Paper02_topological_properties_state_space}
    For each~$N \in \mathbb{N}$, the metric space $(\mathcal{S}^{N}, d_{\mathcal{S}^{N}})$ defined in~\eqref{Paper02_definition_state_space_formal} is complete and separable. Moreover, a subset $\mathcal{K} \subset \mathcal{S}^{N}$ is relatively compact in the topology induced by $d_{\mathcal{S}^{N}}$ if and only if all the following conditions are satisfied:
    \begin{enumerate}[label = (\roman*)]
        \item $\mathcal{K}$ is $\vert \vert \vert \cdot \vert \vert \vert_{\mathcal{S}^{N}}$-bounded, i.e.~$\sup_{\boldsymbol{\xi} \in \mathcal{K}} \vert \vert \vert \boldsymbol{\xi} \vert \vert \vert_{\mathcal{S}^{N}} < \infty$.
        \item For any $\varepsilon > 0$, there exists $R_{\varepsilon} > 0$ such that
        \begin{equation*}
            \sup_{\boldsymbol{\xi} = (\xi(y))_{y \in L_N^{-1}\mathbb Z} \in \mathcal{K}} \; \sum_{\{x \in L_{N}^{-1}\mathbb{Z}: \, \vert x \vert \geq R_{\varepsilon}\}} \frac{\| \xi(x) \|_{\ell_{1}}}{(1 + \vert x \vert)^{2}} \leq \varepsilon.
        \end{equation*}
        \item For each $x \in L_{N}^{-1}\mathbb{Z}$, there exists $J^{(x)} \in \mathbb{N}_{0}$ such that for any $k \in \mathbb N_0$ with $k \geq J^{(x)}$, $\xi_{k}(x) = 0$ for all~$\boldsymbol{\xi} = (\xi_j(y))_{j \in \mathbb N_0, \, y \in L_N^{-1}\mathbb Z} \in \mathcal{K}$.
    \end{enumerate}
\end{proposition}

We now prepare to formally introduce the process $(\eta^{N}(t))_{t \geq 0}$ that we call the spatial Muller's ratchet. For any fixed $N \in \mathbb{N}$, under Assumptions~\ref{Paper02_assumption_fitness_sequence} and~\ref{Paper02_assumption_polynomials}, for an initial configuration $\boldsymbol \eta^N$ satisfying~\eqref{Paper02_immediate_estimates_definition_initial_condition}, the process $(\eta^N(t))_{t \geq 0}$ with $\eta^N(0) = \boldsymbol \eta^N$ almost surely will be given by the weak limit of a sequence of approximating Feller processes~$(\eta^{N,n})_{n \in \mathbb{N}}$. More precisely, let $(\lambda_{n})_{n \in \mathbb{N}} \subset (0, \infty)$ and $(K_{n})_{n \in \mathbb{N}} \subseteq \mathbb{N}$ denote increasing sequences such that
\begin{equation*}
    \lim_{n \rightarrow \infty} \, \lambda_{n} = \infty \quad \textrm{and} \quad  \lim_{n \rightarrow \infty} \, K_{n} = \infty.
\end{equation*}
For each $n \in \mathbb{N}$, let $\Lambda_{N,n} \defeq L^{-1}_N\mathbb{Z} \, \cap \, [-\lambda_n, \lambda_n] $; the process $(\eta^{N,n}(t))_{t \geq 0}$ will be an $\mathcal S^N$-valued càdlàg Markov process. Informally, the process $\eta^{N,n}$ can be described as follows: particles outside the box $\Lambda_{N,n}$ are frozen. All particles living inside $\Lambda_{N,n}$ can
migrate and die (at the rates in the informal description of $\eta^N$ at the start of Section~\ref{Paper02_model_description}), but particles on the boundary of $\Lambda_{N,n}$ can only migrate to a deme inside $\Lambda_{N,n}$. Finally, only particles that carry $K_n$ or fewer mutations are able to reproduce (again, at the rates in the informal description at the start of Section~\ref{Paper02_model_description}). Since the number of demes in $\Lambda_{N,n}$ is finite, and only particles carrying at most $K_n$ mutations can reproduce, by ignoring frozen particles, the process $(\eta^{N,n}(t))_{t \geq 0}$ can be thought of as a Markov process with finitely many types. A formal definition of the process $(\eta^{N,n}(t))_{t \geq 0}$ in terms of its infinitesimal generator is given in the companion article~\cite[Section~2]{madeira2025existence}. In~\cite[Theorem~2.2]{madeira2025existence}, we show that for $\boldsymbol \eta^N$ satisfying~\eqref{Paper02_immediate_estimates_definition_initial_condition},
conditioning on $\eta^{N,n}(0) = \boldsymbol \eta^N$ for every $n \in \mathbb N$, the sequence of processes $((\eta^{N,n}(t))_{t \geq 0})_{n \in \mathbb{N}}$ converges weakly as $n \rightarrow \infty$. More precisely, define a set of initial configurations $\mathcal S^N_0 \subset \mathcal S^N$ as follows (see~\cite[Assumption~3]{madeira2025existence}):
\begin{equation*}
\begin{aligned}
     \mathcal{S}^N_{0} \defeq \bigg\{\boldsymbol{\xi} = (\xi_j(x))_{j \in \mathbb N_0, \, x \in L_N^{-1}\mathbb  Z}\in \mathcal{S}^N: \; & \sup_{x \in L^{-1}_N\mathbb{Z}} \| \xi(x) \|_{\ell_{1}} < \infty \; \\ 
     & \qquad \textrm{and} \; \lim_{k \rightarrow \infty} \, \sup_{x \in L^{-1}_N\mathbb{Z}} \, \sum_{j = k}^{\infty} \, \xi_{j}(x) = 0\bigg\}.
\end{aligned}
\end{equation*}
Then, by~\eqref{Paper02_immediate_estimates_definition_initial_condition}, for $f: \mathbb R \rightarrow \ell_1^+$ satisfying Assumption~\ref{Paper02_assumption_initial_condition}, the initial configuration $\boldsymbol{\eta}^N$ given by~\eqref{Paper02_initial_condition} is an element of $\mathcal{S}^N_0$, and so~\cite[Theorem~2.2]{madeira2025existence} yields the following result. Recall the definition of $\mathcal C_*(\mathcal S^N, \mathbb R)$ in~\eqref{Paper02_definition_Lipschitz_function}, and note that by~\cite[Lemma~6.3]{madeira2025existence}, for $\phi \in \mathcal C_*(\mathcal S^N, \mathbb R)$ and $\mathcal{L}^{N}\phi$ as defined in~\eqref{Paper02_generator_foutel_etheridge_model}-\eqref{Paper02_infinitesimal_generator}, $\mathcal{L}^{N}\phi$ is well-defined and finite on $\mathcal S^N$.

\begin{theorem} \label{Paper02_thm_existence_uniqueness_IPS_spatial_muller}
    Suppose $N \in \mathbb N$, $L_N, \, m_N > 0$, $\mu \in [0,1]$, and $(s_k)_{k \in \mathbb N_0}$, $q_+$, $q_-$ and $f$ satisfy Assumptions~\ref{Paper02_assumption_fitness_sequence},~\ref{Paper02_assumption_polynomials} and~\ref{Paper02_assumption_initial_condition}. Define $\boldsymbol \eta^N \in \mathcal S^N_0$ as in~\eqref{Paper02_initial_condition}. Then conditioning on $\eta^{N,n}(0) = \boldsymbol{\eta}^N$ for every $n \in \mathbb{N}$, the sequence of processes $((\eta^{N,n}(t))_{t \geq 0})_{n \in \mathbb{N}}$ converges weakly with respect to the $J_1$-topology on $\mathcal{D}([0, \infty), \mathcal{S}^{N})$ as $n \rightarrow \infty$ to an $\mathcal S^N$-valued càdlàg Markov process $(\eta^{N}(t))_{t \geq 0}$ with $\eta^N(0) = \boldsymbol{\eta}^N$ almost surely. Moreover, $(\eta^{N}(t))_{t \geq 0}$ is a strong Markov process with respect to its right-continuous natural filtration $\{\mathcal{F}^{\eta^{N}}_{t+}\}_{t \geq 0}$. Furthermore, for any $\phi \in \mathcal C_*(\mathcal{S}^{N}, \mathbb{R})$, there exists a càdlàg $\{\mathcal{F}^{\eta^{N}}_{t+}\}_{t \geq 0}$-square integrable martingale $(M^{\phi}(t))_{t \geq 0}$ such that for all $T \geq 0$,
    \begin{equation*}
        \phi(\eta^{N}(T)) = \phi(\boldsymbol{\eta}^N) + M^{\phi}(T) + \int_{0}^{T} (\mathcal{L}^{N}\phi)(\eta^{N}(t-)) \, dt,
    \end{equation*}
    where $\mathcal{L}^{N}\phi$ is defined in~\eqref{Paper02_generator_foutel_etheridge_model}-\eqref{Paper02_infinitesimal_generator}.
\end{theorem}

We will refer to the process~$(\eta^{N}(t))_{t \geq 0}$ as the spatial Muller's ratchet. For $\boldsymbol \eta^N \in \mathcal S^N_0$, we let $\mathbb P_{\boldsymbol \eta^N}$ denote the probability measure under which $(\eta^N(t))_{t \geq 0}$ is the spatial Muller's ratchet with $\eta^N(0) = \boldsymbol \eta^N$ almost surely, and $\mathbb E_{\boldsymbol \eta^N}$ the corresponding expectation. We will also use the following moment bound from our companion article~\cite{madeira2025existence}, which we state here for ease of reference. Note that for $f: \mathbb R \rightarrow \ell_1^+$ satisfying Assumption~\ref{Paper02_assumption_initial_condition}, $N \in \mathbb N$, and $\boldsymbol \eta^N$ as defined in~\eqref{Paper02_initial_condition}, we have
\[
\sup_{x \in L_N^{-1}\mathbb Z} \, \| \eta^N (x) \|_{\ell_1} \leq N \esssup_{x \in \mathbb R} \| f(x) \|_{\ell_1} < \infty.
\]
Hence, the following result is a direct consequence of~\cite[Theorem~2.3]{madeira2025existence}.

\begin{theorem} \label{Paper02_bound_total_mass}
Suppose $q_+$ and $q_-$ satisfy Assumption~\ref{Paper02_assumption_polynomials}, and $f$ satisfies Assumption~\ref{Paper02_assumption_initial_condition}. Then for any $p \geq 1$, there is a non-decreasing function $C_{p}: [0, \infty) \rightarrow [1, \infty)$ such that the following holds. For any $N \in \mathbb N$, $L_N,m_N > 0$, $\mu \in [0,1]$ and $(s_k)_{k \in \mathbb N_0}$ satisfying Assumption~\ref{Paper02_assumption_fitness_sequence}, for $\boldsymbol \eta^N$ as defined in~\eqref{Paper02_initial_condition} and for all $t \geq 0$,
\begin{equation*}
    \sup_{x \in L_{N}^{-1}\mathbb{Z}} \mathbb{E}_{\boldsymbol \eta^N}\left[\| u^{N}(t,x) \|_{\ell_{1}}^{p}\right] \leq C_{p}(t),
\end{equation*}
where $u^N(t,x)$ is defined in~\eqref{Paper02_approx_pop_density_definition}.
\end{theorem}

\subsection{Partial differential equations in $\ell_{1}$} \label{Paper02_PDES_sequence_spaces}

In this subsection, we define $L_{p}$ spaces of $\ell_{1}$-valued functions, and weak and mild solutions to the system of PDEs~\eqref{Paper02_PDE_scaling_limit}. We also state the main results of our companion article~\cite{madeira2025PDE_surfing} that will be used in this paper. Let $(S, \Sigma, \nu)$ be a measure space. We say a function $g = (g_k)_{k \in \mathbb N_0}: S \rightarrow \ell_{1}$ is (Bochner) measurable with respect to $(S, \Sigma, \nu)$ if $g_{k}: S \rightarrow \mathbb{R}$ is measurable for every $k \in \mathbb{N}_{0}$. Note that in this case, $\| g \|_{\ell_{1}}: S \rightarrow [0, \infty)$ is the pointwise limit of a sequence of measurable functions, i.e.~$\| g(x) \|_{\ell_{1}} = \lim_{n \rightarrow \infty} \sum_{k = 0}^{n} \vert g_{k}(x) \vert \; \forall \, x \in S$, and so $\| g \|_{\ell_{1}}$ is also measurable. Following \cite[Chapters 1 and 2]{hytonen2016analysis}, for $r \in [1, \infty]$, we let
\begin{equation*}
    L_{r}(S; \ell_{1}) \defeq \left\{g: S \rightarrow \ell_{1}: \; g \textrm{ is measurable and } \| g \|_{L_r(S;\ell_1)} \defeq \Big \| \, \| g \|_{\ell_{1}} \Big\|_{L_{r}(\nu)} < \infty\right\}.
\end{equation*}
As usual, two functions define the same element of~$L_{r}(S; \ell_{1})$ if they are equal $\nu$-almost everywhere in~$S$. Note that $L_r(S; \ell_1)$ is a Banach space (see e.g.~\cite[Proposition~1.2.29]{hytonen2016analysis}). For $d \in \mathbb N$, for the Lebesgue measure space $(\mathbb{R}^{d}, \mathcal{R}^{(d)}, \lambda)$, for $S \in \mathcal{R}^{(d)}$, we say that a measurable function $g: S \rightarrow \ell_{1}$ is an element of $L_{1,\textrm{loc}}(S; \ell_{1})$ if for all compact sets $\mathcal{K} \subseteq S$, we have
\begin{equation*}
    \int_{\mathcal{K}}  \| g(x) \|_{\ell_{1}} \; dx < \infty.
\end{equation*}

We are now ready to properly define the meaning of a weak solution to the system of PDEs~\eqref{Paper02_PDE_scaling_limit}. We will be thinking of weak solutions in the sense of distributions, rather than in the sense of elements of Sobolev spaces (see~\cite[Chapter~2]{hytonen2016analysis} for the difference between these concepts).

\begin{definition}[Weak solution] \label{Paper02_weak_solution_distributional_sense}
    We say that a measurable function $u \in L_{1, \textrm{loc}}([0, \infty) \times \mathbb{R}; \ell_{1})$ is a weak solution to the system of PDEs~\eqref{Paper02_PDE_scaling_limit} if the following conditions are satisfied:
    \begin{enumerate}[label = (\roman*)]
        \item $u(0,x) = f(x)$ for $\lambda$-almost every $x \in \mathbb{R}$.
        \item $u(T, \cdot) \in L_{1,\textrm{loc}}(\mathbb{R}; \ell_{1})$ for all $T \geq 0$.
        \item $F(u): [0,\infty) \times \mathbb{R} \rightarrow \ell_{1}$ is an element of $L_{1, \textrm{loc}}([0,\infty) \times \mathbb{R}; \ell_{1})$, where $F = (F_{k})_{k \in \mathbb{N}_{0}}$ is the reaction term defined in~\eqref{Paper02_reaction_term_PDE}.
        \item For any $\varphi \in \mathcal{C}^{1,2}_c([0,\infty) \times \mathbb R; \mathbb R)$, i.e.~for any continuous function $\varphi: [0,\infty) \times \mathbb{R} \rightarrow \mathbb{R}$ with compact support, continuously differentiable in time and twice continuously differentiable in space, and any $T > 0$,
   \begin{equation} \label{Paper02_coordinate_wise_integral}
    \begin{aligned}
             & \int_{\mathbb{R}} u(T,x) \varphi(T,x) \, dx \\
         & \quad =  \int_{\mathbb{R}} u(0,x) \varphi(0,x) \, dx + \int_{0}^{T} \int_{\mathbb{R}} u(t,x) \partial_{t} \varphi(t,x) \, dx \, dt
        \\ & \qquad + \int_{0}^{T} \int_{\mathbb{R}} \frac{m}{2} u(t,x) \Laplace \varphi(t,x) \, dx \, dt + \int_{0}^{T} \int_{\mathbb{R}} F(u(t,x)) \varphi(t,x) \, dx \, dt.
    \end{aligned}
    \end{equation}
    \end{enumerate}
\end{definition}

Note that~\eqref{Paper02_coordinate_wise_integral} holds if and only if it holds in a coordinate-wise manner, i.e.~if and only if for all $k \in \mathbb{N}_{0}$,
\begin{equation} \label{Paper02_coordinate_wise_weak_sol}
\begin{aligned}
     & \int_{\mathbb{R}} u_{k}(T,x) \varphi(T,x) \, dx \\ 
     & \quad = \int_{\mathbb{R}} u_{k}(0,x) \varphi(0,x) \, dx + \int_{0}^{T} \int_{\mathbb{R}} u_{k}(t,x) \partial_{t} \varphi(t,x) \, dx \, dt
     \\ & \qquad + \int_{0}^{T} \int_{\mathbb{R}} \frac{m}{2} u_{k}(t,x) \Laplace \varphi(t,x) \, dx \, dt + \int_{0}^{T} \int_{\mathbb{R}} F_{k}(u(t,x)) \varphi(t,x) \, dx \, dt.
\end{aligned}
\end{equation}
We refer the reader to~\cite[Chapter~1]{hytonen2016analysis} for a proof of this equivalence.

We will also need the definition of mild solutions to the system of PDEs~\eqref{Paper02_PDE_scaling_limit}. Let $\{P_{t}\}_{t \geq 0}$ denote the semigroup of Brownian motion run at speed $m$. The following notation will be useful: for $t > 0$ and $x \in \mathbb R$, we let
\begin{equation}
\label{Paper02_Gaussian_kernel}
p(t,x) \defeq \frac{1}{\sqrt{2\pi mt}} e^{-x^2/(2mt)}.
\end{equation}
Analogously to the definition of the action of $\{P_{t}\}_{t \geq 0}$ on real-valued functions, with a slight abuse of notation we can define its action on $\ell_{1}$-valued functions as follows. Let $v: \mathbb{R} \rightarrow \ell_{1}$ be a measurable function. For all $(t,x) \in [0, \infty) \times \mathbb R$ such that $(P_{t}\| v \|_{\ell_{1}}) (x) < \infty$,
we can define
\begin{equation}
\label{Paper02_semigroup_BM_action_ell_1_functions}
    \left(P_{t}v\right)(x) \defeq \left\{\begin{array}{lc}
        \displaystyle \int_{\mathbb{R}} p(t,x-y) v(y) \; dy & \textrm{ if } t > 0, \\
        v(x) & \textrm{ if } t = 0.
    \end{array}\right.
\end{equation}

\begin{definition}[Mild solution] \label{Paper02_definition_mild_solution}
    We say that a measurable function $u: [0,\infty) \times \mathbb{R} \rightarrow \ell_{1}$ is a (global) mild solution to the system of PDEs~\eqref{Paper02_PDE_scaling_limit} if and only if the following conditions are satisfied:
    \begin{enumerate}[label = (\roman*)]
    \item For $\lambda$-almost every $(T,x) \in [0,\infty) \times \mathbb{R}$, $\left(P_{T-t}\| F(u(t,\cdot)) \|_{\ell_{1}}\right) (x) < \infty$ for $\lambda$-almost every $t \in [0,T]$.
    \item For $\lambda$-almost every $(T,x) \in [0,\infty) \times \mathbb{R}$,
    \begin{equation} \label{Paper02_mild_solution_weaker_version}
        u(T,x) = \left(P_{T} f\right)(x) + \int_{0}^{T} \Big(P_{T-t} F(u(t,\cdot))\Big)(x) \, dt.
    \end{equation}
    \end{enumerate}
\end{definition}

\begin{remark} \label{Paper02_remark_general_mild_sol}
Although our definition of a mild solution is compatible with the integral formulations used in the $L_p$-regularity literature (see e.g.~\cite[Chapter 5]{pruss2016moving}), classical semigroup approaches to deterministic PDEs typically require mild solutions to be continuous in time and to satisfy the variation-of-constants formula~\eqref{Paper02_mild_solution_weaker_version} pointwise in time (see e.g.~\cite[Section 4.1]{pazy2012semigroups}). In contrast, we work with a weaker formulation in which~\eqref{Paper02_mild_solution_weaker_version} holds almost everywhere in time and in space. This choice reflects the fact that our construction of solutions takes place in $L_p$-based function spaces, where time and space continuity is not available \textit{a priori}.
\end{remark}

We show in~\cite[Lemma~5.4]{madeira2025PDE_surfing} that, under mild assumptions, a mild solution to the system of PDEs~\eqref{Paper02_PDE_scaling_limit} is also a weak solution. Since we will use this result in this article, we state it for ease of reference.

\begin{lemma} \label{Paper02_lem_equiv_mild_weak_sol}
Let $u: [0, \infty) \times \mathbb R \rightarrow \ell_1^+$ denote a measurable function. Suppose that $u$ is a mild solution to the system of PDEs~\eqref{Paper02_PDE_scaling_limit} in the sense of Definition~\ref{Paper02_definition_mild_solution}, and that it satisfies the following conditions:
\begin{enumerate}[label = (\roman*)]
    \item $u(0,x) = f(x)$ for $\lambda$-almost every $x \in \mathbb R$.
    \item For all $T > 0$, $u \big|_{[0,T] \times \mathbb R} \in L_{\infty}([0,T] \times \mathbb R; \ell_1)$,
    where $ u \big|_{[0,T] \times \mathbb R}$ denotes the restriction of $u$ to $[0,T] \times \mathbb R$.
    \item For all $T > 0$, $u$ satisfies identity~\eqref{Paper02_mild_solution_weaker_version} for $\lambda$-almost every $x \in \mathbb R$.
\end{enumerate}
Then $u$ is a weak solution to the system of PDEs~\eqref{Paper02_PDE_scaling_limit} in the sense of Definition~\ref{Paper02_weak_solution_distributional_sense}.
\end{lemma}

We prove uniqueness of mild solutions to the system of PDEs~\eqref{Paper02_PDE_scaling_limit} in~\cite[Proposition~6.1]{madeira2025PDE_surfing} using standard PDE techniques. Since this result will be used in the proof of Theorem~\ref{Paper02_deterministic_scaling_foutel_rodier_etheridge}, we state it below for ease of reference.
The function space $L_{r}([0,T] \times \mathbb{R}, \hat{\lambda}; \ell_{1})$ mentioned in Section~\ref{Paper02_subsection_heuristics} above is defined in
\eqref{Paper02_measure_time_space_box_i}-\eqref{Paper02_functional_space_L1_l1} below.

\begin{proposition} \label{Paper02_thm_uniqueness_weak_solutions}
    Suppose $m > 0$, $\mu \in [0,1]$, $(s_k)_{k \in \mathbb N_0}$ satisfies Assumption~\ref{Paper02_assumption_fitness_sequence}, $q_+,q_-: [0, \infty) \rightarrow [0, \infty)$ satisfy Assumption~\ref{Paper02_assumption_polynomials}, and $f \in L_{\infty}(\mathbb R; \ell_1)$ satisfies Assumption~\ref{Paper02_assumption_initial_condition}. For~$T > 0$, let $v^{(1)}, v^{(2)} \in L_{4\deg q_{-}}([0,T] \times \mathbb{R}, \hat{\lambda}; \ell_{1})$ be $\lambda$-almost everywhere non-negative mild solutions to the system of PDEs~\eqref{Paper02_PDE_scaling_limit}. 
    Then~$v^{(2)}$ is a version of~$v^{(1)}$ in~$L_{4\deg q_-}([0,T] \times \mathbb{R}, \hat{\lambda}; \ell_{1})$, i.e.~$v^{(1)}$ and $v^{(2)}$ satisfy 
    \begin{equation*}
    \| v^{(1)} - v^{(2)} \|_{{L}_{4\deg q_{-}}([0,T] \times \mathbb{R}, \hat \lambda; \ell_{1})} = 0.
    \end{equation*}
\end{proposition}

We also establish in~\cite[Proposition~6.2]{madeira2025PDE_surfing} the following regularity property of mild solutions to the system of PDEs~\eqref{Paper02_PDE_scaling_limit}, which will also be used in the proof of Theorem~\ref{Paper02_deterministic_scaling_foutel_rodier_etheridge}.

\begin{proposition} \label{Paper02_smoothness_mild_solution}
    Under the assumptions of Proposition~\ref{Paper02_thm_uniqueness_weak_solutions}, there exists a unique function $v = (v_k)_{k \in \mathbb N_0}: [0, \infty) \times \mathbb R \rightarrow \ell_1^+$ such that $v \in \mathcal C((0,\infty) \times \mathbb  R; \ell_1^+)$ and satisfies the following conditions:
    \begin{enumerate}[label = (\roman*)]
       \item $v = (v_k)_{k \in \mathbb N_0}$ is a global non-negative mild solution to the system of PDEs~\eqref{Paper02_PDE_scaling_limit} in the sense of Definition~\ref{Paper02_definition_mild_solution}, and it satisfies~\eqref{Paper02_mild_solution_weaker_version} for all $(t,x) \in [0,\infty) \times \mathbb R$, i.e.~for any $k \in \mathbb N_0$ and~$(t,x) \in [0, \infty) \times \mathbb R$,
       \begin{equation} \label{Paper02_coordinate_wise_strong_duhamel}
           v_k(t,x) = (P_tf_k)(x) + \int_0^t \Big(P_{t-\tau} F_k(v(\tau, \cdot))\Big)(x) \, d\tau.
       \end{equation}
        \item For all $T > 0$,
        \[
        v \vert_{[0,T] \times \mathbb R} \in L_{4\deg q_-}([0,T] \times \mathbb R, \hat \lambda; \ell_1) \cap L_\infty([0,T] \times \mathbb R; \ell_1).
        \]
        \item For all $T > 0$, $\sup_{t \in [0,T]} \; \| {v}(t, \cdot) \|_{L_{\infty}(\mathbb{R}; \ell_{1})} < \infty$.
        \end{enumerate}
        Moreover, the unique function $v = (v_k)_{k \in \mathbb N_0}: [0, \infty) \times \mathbb R \rightarrow \ell_1^+$ satisfies the following conditions:
        \begin{enumerate}
        \item[(iv)] $v_k \in \mathcal{C}^{1,2}((0, \infty) \times \mathbb R; \mathbb R)$ for every $k \in \mathbb N_0$.
        \item[(v)] The map $(0, \infty) \times \mathbb R \ni (t,x) \mapsto \| v(t,x) \|_{\ell_1}$ is in $\mathcal C^{1,2}((0,\infty) \times \mathbb R; \mathbb R)$.
    \end{enumerate}
\end{proposition}

We 
refer to the function~$v = (v_k)_{k \in \mathbb N_0}: [0, \infty) \times \mathbb R \rightarrow \ell_1^+$ given in Proposition~\ref{Paper02_smoothness_mild_solution} as the \emph{continuous mild solution} to the system of PDEs~\eqref{Paper02_PDE_scaling_limit}. The next result, which can be found in~\cite[Theorem~2.7]{madeira2025PDE_surfing}, describes the asymptotic behaviour of the continuous mild solution to the system of PDEs~\eqref{Paper02_PDE_scaling_limit} when the reaction term in~\eqref{Paper02_reaction_term_PDE} is monostable in the sense of Definition~\ref{Paper02_assumption_monostable_condition}, and will be used in the proof of Theorem~\ref{Paper02_long_time_behaviour_IPS}.

\begin{theorem} \label{Paper02_prop_control_proportions}
    Suppose that $(s_k)_{k \in \mathbb N_0}$, $q_+$ and $q_-$ satisfy Assumptions~\ref{Paper02_assumption_fitness_sequence} and~\ref{Paper02_assumption_polynomials}, that $\mu \in (0,1)$, that $m > 0$, that the reaction term $F = (F_{k})_{k \in \mathbb{N}_{0}}$ defined in~\eqref{Paper02_reaction_term_PDE} is monostable in the sense of Definition~\ref{Paper02_assumption_monostable_condition},  and that $f$ satisfies Assumption~\ref{Paper02_assumption_initial_condition_modified}. Let $u = (u_{k})_{k \in \mathbb{N}_{0}}$ be the unique continuous mild solution to the system of PDEs~\eqref{Paper02_PDE_scaling_limit}. Define $(\alpha_k)_{k \in \mathbb N_0}$ as in~\eqref{Paper02_definition_alpha_k}, and $\mathfrak{Q}_{\min}$,~$\mathfrak{Q}_{\max}$ as in~\eqref{Paper02_definition_max_min_birth_polynomial}. Then, for all $x \in \mathbb R$ and $T,y \in (0, \infty)$,
    \[
    \int_{x - y}^{x+y} u_0(T,z) \, dz > 0. 
    \]
    Moreover, there exist functions $\underline{\pi} = \left(\underline{\pi}_{k}\right)_{k \in \mathbb{N}_0}: [0, \infty) \rightarrow \ell_1^+$ and $\overline{\pi} = \left(\overline{\pi}_{k}\right)_{k \in \mathbb{N}_0}: [0,\infty) \rightarrow \ell_{1}^{+}$ such that the following limits hold in $\ell_1$:
    \begin{equation*}
        \lim_{T \rightarrow \infty} \underline{\pi}(T) = \bigg(\alpha_{k}\left(\frac{\mathfrak{Q}_{\min}}{\mathfrak{Q}_{\max}}\right)^{k}\bigg)_{k \in \mathbb{N}_0} \quad \textrm{ and } \quad  \lim_{T \rightarrow \infty} \overline{\pi}(T) = \bigg(\alpha_{k}\left(\frac{\mathfrak{Q}_{\max}}{\mathfrak{Q}_{\min}}\right)^{k}\bigg)_{k \in \mathbb{N}_0},
    \end{equation*}
    and such that 
    for all $k \in \mathbb{N}$, $T > 0$ and $x \in \mathbb{R}$, 
    \begin{equation*} 
        \underline{\pi}_{k}(T) u_{0}(T,x) \leq u_{k}(T,x) \leq \overline{\pi}_{k}(T) u_{0}(T,x).
    \end{equation*}
\end{theorem}

\section{Green's function representation}
\label{Paper02_section_greens_function}

Recall the definition of the approximate density process $u^N$ in~\eqref{Paper02_approx_pop_density_definition}. In this section, we will apply a technique used in~\cite{mueller1995stochastic,durrett2016genealogies} based on a Green's function representation to obtain uniform estimates on $u^N_k(t,x)$ for large $k$, and estimates on both space and time increments of $(u^{N})_{N \in \mathbb{N}}$. We begin by briefly highlighting the main differences between our setting and that of~\cite{mueller1995stochastic,durrett2016genealogies}.

Mueller and Tribe~\cite{mueller1995stochastic},  and later Durrett and Fan~\cite{durrett2016genealogies}, studied the density of particles in a scaled version of the voter model. Since in their model, under a suitable rescaling, the local density of particles is uniformly bounded by $1$, a Green's function representation combined with random walk estimates allows them to prove H\"{o}lder estimates on the expectation of the particle density. In our case, we have two main challenges. First, the particle density is not uniformly bounded, so although we will be able to derive estimates for the space and time increments, we will not be able to establish H\"{o}lder estimates for these increments. Second, since we must keep track of infinitely many types of particles, we will prove estimates that depend on the subset of types of particles considered (here, the ‘type' of a particle refers to the number of mutations that it carries). This will provide us with some control over $u^N_k(t,x)$ for large $k$. We will repeatedly use the random walk estimates stated in Lemma~\ref{Paper02_random_walks_lemma} in Section~\ref{Paper02_subsection_rw_estimates} of the appendix.

For any set of indices $\mathcal{I} \subseteq \mathbb{N}_{0}$, $N \in \mathbb N$, $T \geq 0$, and $x \in \mathbb{R}$, let
\begin{equation} \label{Paper02_local_mass_subset_indices_definition}
    U^N_{\mathcal{I}} (T,x) \defeq \sum_{k \in \mathcal{I}} u^N_{k}(T,x).
\end{equation}
In particular, for any~$k \in \mathbb{N}_{0}$, we have $U^{N}_{\{k\}} \equiv u^{N}_{k}$. Our strategy will be to write~$U^{N}_{\mathcal{I}} (T,x)$ for each $T \geq 0$ and $x \in L_{N}^{-1}\mathbb{Z}$ in terms of an appropriate martingale problem.

For $N \in \mathbb N$, let $(X^{N}(t))_{t \geq 0}$ denote a simple symmetric random walk on $L_{N}^{-1}\mathbb{Z}$ with total jump rate $m_{N}$. We denote its infinitesimal generator by $\frac{m_{N}}{2}\hat{\mathcal{L}}^{N}_{m}$, where for any bounded function $g: L_{N}^{-1}\mathbb{Z} \rightarrow \mathbb{R}$ and for all $x \in L_{N}^{-1}\mathbb{Z}$,
\begin{equation} \label{Paper02_inf_gen_rw}
    \frac{m_{N}}{2}\hat{\mathcal{L}}^{N}_{m} g(x) \defeq \frac{m_{N}}{2} \left(g(x+L_{N}^{-1}) + g(x-L_{N}^{-1}) - 2g(x)\right).
\end{equation}
Let $\{P^{N}_{t}\}_{t \geq 0}$ be the semigroup associated with $ \frac{m_{N}}{2}\hat{\mathcal{L}}^{N}_{m}$. By Assumption~\ref{Paper02_scaling_parameters_assumption}(i), we have $ \frac{m_{N}}{L_{N}^{2}} \rightarrow m$ as $N \rightarrow \infty$, and so $\{P^{N}_{t}\}_{t \geq 0}$ converges as $N \rightarrow \infty$ to the semigroup $\left\{P_{t}\right\}_{t \geq 0}$ associated with Brownian motion on $\mathbb{R}$ run at total rate $m$ (recall that we defined this before~\eqref{Paper02_semigroup_BM_action_ell_1_functions}). Following Durrett and Fan~\cite{durrett2016genealogies}, we define the normalised transition density $p^{N}$ by letting, for all $N \in \mathbb N$, $t \geq 0$ and~$x \in L_N^{-1}\mathbb Z$,
\begin{equation} \label{Paper02_definition_p_N_test_function_green}
    p^{N}(t,x) \defeq L_{N} \mathbb{P}_{0}\left(X^{N}(t) = x\right) = L_{N}\mathbb{P}\left(X^{N}(t) = x \Big\vert X^{N}(0) = 0 \right).
\end{equation}
Then, for any fixed $N \in \mathbb{N}$, $T \geq 0$ and $x \in L_{N}^{-1}\mathbb{Z}$, let $\phi^{N,T,x}: [0,\infty) \times L_{N}^{-1}\mathbb{Z} \rightarrow \left[0,L_{N}\right]$ be given by
\begin{equation} \label{Paper02_green_function_test_function}
    \phi^{N, T, x}(t, y) = \left\{\begin{array}{ll} p^{N}\left(T-t,y - x\right) &  \textrm{if } t \in [0,T], \\ 0 & \textrm{otherwise.}\end{array}\right.
\end{equation}
Observe that, by standard large deviation results on Poisson processes and continuous-time random walks, see e.g.~\cite[Corollary~A.10]{madeira2025existence}, for~$y \in L_{N}^{-1}\mathbb{Z}$ such that~$\vert y - x \vert \geq e^2m_NTL_N^{-1}$ and~$t \in [0,T]$,
\begin{equation} \label{Paper02_trivial_large_deviation_estimatre_CTRW}
    \phi^{N,T,x}(t,y) \leq L_{N}e^{-L_{N} \vert y -x \vert}.
\end{equation}
For $N \in \mathbb N$ and $g,h: L_{N}^{-1}\mathbb{Z} \rightarrow \mathbb{R}$, we write
\begin{equation} \label{Paper02_definition_discrete_integral_over_space}
    \langle g,h \rangle_{N} \defeq \frac{1}{L_{N}} \sum_{x \in L_{N}^{-1}\mathbb{Z}} g(x)h(x),
\end{equation}
whenever the sum on the right-hand side is well defined. Using the notation defined above, we now work towards the Green's function representation. We start by proving a simple consequence of the moment bound in Theorem~\ref{Paper02_bound_total_mass}.

\begin{lemma} \label{Paper02_simple_estimate_greens_function_representation_moments}
    Suppose $q_+$ and $q_-$ satisfy Assumption~\ref{Paper02_assumption_polynomials}, and $f$ satisfies Assumption~\ref{Paper02_assumption_initial_condition}. Then, for any~$l,r \geq 1$ and~$T \geq 0$, there exists~$C_{l,r,T} > 0$ such that for any $N \in \mathbb N$, $L_N,m_N>0$, $\mu \in [0,1]$ and $(s_k)_{k \in \mathbb N_0}$ satisfying Assumption~\ref{Paper02_assumption_fitness_sequence}, for $\boldsymbol \eta^N$ as defined in~\eqref{Paper02_initial_condition},
    \begin{equation*}
        \sup_{\substack{\mathcal{I} \subseteq \mathbb{N}_{0}}} \; \sup_{x \in L_{N}^{-1}\mathbb{Z}} \; \sup_{t_{1},t_{2} \in [0,T]} \; \mathbb{E}_{\boldsymbol \eta^N}\left[\Big\langle U^{N}_{\mathcal{I}}(t_{2}, \cdot)^{l}, \phi^{N,T,x}(t_{1}, \cdot) \Big\rangle_{N}^{r} \right] \leq C_{l,r,T}.
    \end{equation*}
\end{lemma}

\begin{proof}
    By the definitions in~\eqref{Paper02_local_mass_subset_indices_definition},~\eqref{Paper02_green_function_test_function},~\eqref{Paper02_definition_p_N_test_function_green} and~\eqref{Paper02_definition_discrete_integral_over_space}, observe that for any $N \in \mathbb{N}$, $\mathcal{I} \subseteq \mathbb{N}_{0}$, $t_{1},t_{2} \in [0,T]$ and $x \in L_{N}^{-1}\mathbb{Z}$,
    \begin{equation} \label{Paper02_step_i_trivial_bound_kernel_RW_moments_ppulation_density}
    \begin{aligned}
        \Big\langle U^{N}_{\mathcal{I}}(t_{2}, \cdot)^{l}, \phi^{N,T,x}(t_{1}, \cdot) \Big\rangle_{N} \leq \sum_{y \in L_{N}^{-1}\mathbb{Z}} \, \| u^{N}(t_{2}, y) \|_{\ell_{1}}^{l} \cdot \mathbb{P}_{0}\Big(X^{N}(T-t_{1}) = y -x \Big).
    \end{aligned}
    \end{equation}
    By Jensen's inequality, it follows that
    \begin{equation} \label{Paper02_step_iii_trivial_bound_kernel_RW_moments_ppulation_density}
    \begin{aligned}
         \Big\langle U^{N}_{\mathcal{I}}(t_{2}, \cdot) ^{l}, \phi^{N,T,x}(t_{1}, \cdot) \Big\rangle_{N}^{r} \leq \sum_{y \in L_{N}^{-1}\mathbb{Z}} \,\| u^{N}(t_{2}, y) \|_{\ell_{1}}^{lr} \mathbb{P}_{0}\Big(X^{N}(T-t_{1}) = y -x \Big).
    \end{aligned}
    \end{equation}
    By taking expectations on both sides of~\eqref{Paper02_step_iii_trivial_bound_kernel_RW_moments_ppulation_density} and applying Theorem~\ref{Paper02_bound_total_mass}, the result follows.
\end{proof}

It will be convenient to introduce, for any~$N \in \mathbb{N}$,~$T \geq 0$, $\mathcal{I} \subseteq \mathbb{N}_{0}$ and $x \in L_{N}^{-1}\mathbb{Z}$, the function $g^{N,T,x}_{\mathcal{I}}:  [0,T] \times \mathcal{S}^{N} \rightarrow [0,\infty)$ given by, for all~$t \in [0,T]$ and~$\boldsymbol{\xi} = (\xi_k(y))_{k \in \mathbb N_0, \, y \in L_{N}^{-1}\mathbb Z} \in \mathcal{S}^{N}$,
\begin{equation} \label{Paper02_green_function_representation_writing_as_lipschitz_map}
    g^{N,T,x}_{\mathcal{I}}(t, \boldsymbol{\xi}) \defeq \Bigg\langle \sum_{k \in \mathcal{I}} \frac{\xi_{k}(\cdot)}{N}, \phi^{N,T,x}(t, \cdot) \Bigg\rangle_{N}.
\end{equation}
By the definition of the metric space~$(\mathcal{S}^{N},d_{\mathcal{S}^{N}})$ in~\eqref{Paper02_definition_state_space_formal} and~\eqref{Paper02_definition_semi_metric_state_space}, and by~\eqref{Paper02_trivial_large_deviation_estimatre_CTRW},~$g^{N,T,x}_{\mathcal{I}}$ is well defined as a real-valued continuous map on~$[0,T] \times \mathcal{S}^{N}$. We now characterise the action of the infinitesimal generator~$\mathcal{L}^{N}$ (defined in~\eqref{Paper02_generator_foutel_etheridge_model} and~\eqref{Paper02_infinitesimal_generator}) on~$g^{N,T,x}_{\mathcal{I}}$. Before stating this result, we recall that $F =(F_{k})_{k \in \mathbb{N}_{0}}$ is the reaction term defined in~\eqref{Paper02_reaction_term_PDE}, and we define the modified reaction term $F^+=(F^{+}_{k})_{k \in \mathbb{N}_{0}}: \ell_{1}^{+} \rightarrow \ell_{1}^{+}$ as follows. For all $u = (u_{j})_{j \in \mathbb N_0} \in \ell_{1}^{+}$ and $k \in \mathbb{N}_{0}$, we let
\begin{equation} \label{Paper02_reaction_predictable_bracket_process}
    F^{+}_{k}(u) \defeq q_{+}(\| u \|_{\ell_1})\left(s_{k}(1-\mu)u_{k} + \mathds{1}_{\{k \geq 1\}}s_{k-1}\mu u_{k-1}\right) + q_{-}(\| u \|_{\ell_1})u_{k}.
\end{equation}
Recalling~\eqref{Paper02_definition_p_N_test_function_green}, for $t \geq 0$ and $x \in L_N^{-1}\mathbb Z$, we let
\begin{equation} \label{Paper02_discrete_grad}
    \nabla_{L_N} p^N(t,x) \defeq L_N \left(p^N(t,x + L_N^{-1}) - p^N(t,x)\right).
\end{equation}
We also recall the definition of~$\mathcal{C}_*(\mathcal{S}^{N}; \mathbb{R})$ in~\eqref{Paper02_definition_Lipschitz_function}.

\begin{lemma} \label{Paper02_properties_greens_function_muller_ratchet}
    Suppose $N \in \mathbb N$, $L_N,m_N > 0$, $\mu \in [0,1]$, $(s_k)_{k \in \mathbb N_0}$ satisfies Assumption~\ref{Paper02_assumption_fitness_sequence}, $q_+$ and $q_-$ satisfy Assumption~\ref{Paper02_assumption_polynomials}, $f$ satisfies Assumption~\ref{Paper02_assumption_initial_condition}, and $\boldsymbol \eta^N$ is defined as in~\eqref{Paper02_initial_condition}. Then, for any~$x \in L_{N}^{-1}\mathbb{Z}$,~$T \geq 0$ and~$\mathcal{I} \subseteq \mathbb{N}_{0}$, the map~$g^{N,T,x}_{\mathcal{I}}: [0,T] \times \mathcal{S}^{N} \rightarrow [0, \infty)$ defined in~\eqref{Paper02_green_function_representation_writing_as_lipschitz_map} satisfies the following properties:
    \begin{enumerate}[label = (\roman*)]
        \item The map~$[0,T] \times \mathcal{S}^N \ni (t, \boldsymbol \xi) \mapsto g^{N,T,x}_{\mathcal I}(t,\boldsymbol{\xi})$ is in $\mathcal C([0,T] \times \mathcal S^N;\mathbb R)$. Moreover, the map~$(0,T) \times \mathcal{S}^{N} \ni (t,\boldsymbol{\xi}) \mapsto \left(\frac{\partial}{\partial t} g^{N,T,x}_{\mathcal{I}}(\cdot, \boldsymbol{\xi})\right)(t)$ is in~$\mathcal{C}((0,T) \times \mathcal{S}^{N}; \mathbb{R})$, and is given by, for~$t \in (0,T)$ and~$\boldsymbol{\xi} \in \mathcal{S}^{N}$,
        \begin{equation} \label{Paper02_formula_derivative_time_greens_function}
            \left(\frac{\partial}{\partial t} g^{N,T,x}_{\mathcal{I}}(\cdot, \boldsymbol{\xi})\right)(t) = \frac{m_{N}}{2}\Big(2g^{N,T,x}_{\mathcal{I}} - g^{N,T,x - L_{N}^{-1}}_{\mathcal{I}} - g^{N,T,x + L_{N}^{-1}}_{\mathcal{I}}\Big)(t,\boldsymbol{\xi}).
        \end{equation}
        \item For any $t \in [0,T]$, the map~$\mathcal S^N \ni \boldsymbol{\xi} \mapsto g^{N,T,x}_{\mathcal{I}}(t, \boldsymbol{\xi})$ is in $\mathcal{C}_*(\mathcal{S}^{N};\mathbb{R})$. Moreover, the map~$[0,T] \times \mathcal{S}^{N} \ni (t,\boldsymbol{\xi}) \mapsto \mathcal{L}^{N}\Big(g^{N,T,x}_{\mathcal{I}}(t,\cdot)\Big)(\boldsymbol{\xi})$ is in $\mathcal{C}([0,T] \times \mathcal{S}^{N}; \mathbb{R})$, and is given by, for~$t \in [0,T]$ and~$\boldsymbol{\xi} = (\xi_k(y))_{k \in \mathbb N_0, \, y \in L_N^{-1}\mathbb Z} \in \mathcal{S}^{N}$,
        \begin{equation} \label{Paper02_definition_action_generator_greens_function}
        \begin{split}
             \mathcal{L}^{N}\Big(g^{N,T,x}_{\mathcal{I}}(t,\cdot)\Big)(\boldsymbol{\xi})  & = \frac{m_{N}}{2}\Big(g^{N,T,x - L_{N}^{-1}}_{\mathcal{I}} + g^{N,T,x + L_{N}^{-1}}_{\mathcal{I}} - 2g^{N,T,x}_{\mathcal{I}}\Big)(t,\boldsymbol{\xi}) \\ & \qquad + \bigg\langle \sum_{k \in \mathcal{I}} \, F_{k}\left(\frac{\xi(\cdot)}{N}\right), \, \phi^{N,T,x}(t, \cdot) \bigg\rangle_{N}.
        \end{split}
        \end{equation}
        \item 
        The map~$[0,T] \times \mathcal S^N \ni (t, \boldsymbol{\xi}) \mapsto \mathcal{L}^{N}\Big((g^{N,T,x}_{\mathcal{I}})^{2}(t, \cdot)\Big)(\boldsymbol{\xi})$ is in $\mathcal{C}([0,T] \times \mathcal{S}^{N}; \mathbb{R})$, and is given by, for~$t \in [0,T]$ and~$\boldsymbol{\xi} = (\xi_k(y))_{k \in \mathbb N_0, y \in L_N^{-1}\mathbb Z} \in \mathcal{S}^{N}$,
        \begin{equation} \label{Paper02_definition_action_generator_on_square_greens_function}
        \begin{split}
            & \mathcal{L}^{N}\Big((g^{N,T,x}_{\mathcal{I}})^{2}(t, \cdot)\Big)(\boldsymbol{\xi}) \\ & \quad = 2g^{N,T,x}_{\mathcal{I}}(t, \boldsymbol{\xi})\mathcal{L}^{N}\Big(g^{N,T,x}_{\mathcal{I}}(t, \cdot)\Big)(\boldsymbol{\xi}) + \frac{1}{N L_{N}^{2}} \sum_{y \in L_{N}^{-1}\mathbb{Z}} p^{N}(T-t,y-x)^{2} \sum_{k \in \mathcal{I}} F^{+}_{k}\left(\frac{\xi(y)}{N}\right) \\ & \quad \quad + \frac{m_{N}}{2NL^{4}_{N}} \sum_{y \in L_{N}^{-1}\mathbb{Z}} \left( \nabla_{L_N}p^{N}(T-t,y - L_{N}^{-1} - x)^2 + \nabla_{L_N}p^{N}(T-t,y - x)^2 \right) \sum_{k \in \mathcal{I}} \frac{\xi_{k}(y)}{N}.
        \end{split}
        \end{equation}
        \item
        The following estimate holds:
        \begin{equation} \label{Paper02_general_estimates_moments_green_function_and_derivatives_time_generator}
        \begin{aligned}
            \sup_{t_{1},t_{2} \in [0,T]} \; & \bigg(\mathbb{E}_{\boldsymbol \eta^N}\bigg[\Big(g^{N,T,x}_{\mathcal{I}}\Big)^{2}(t_{1}, \eta^{N}(t_{2})) + \left(\frac{\partial}{\partial t} g^{N,T,x}_{\mathcal{I}}(\cdot, \eta^{N}(t_{2}))\right)^{2}(t_{1}) \bigg] 
            \\ & \quad + \mathbb{E}_{\boldsymbol \eta^N}\bigg[\Big(g^{N,T,x}_{\mathcal{I}}\Big)^{2}(t_{1}, \eta^{N}(t_{2}))\left(\frac{\partial}{\partial t} g^{N,T,x}_{\mathcal{I}}(\cdot, \eta^{N}(t_{2}))\right)^{2}(t_{1})\Bigg] 
            \\ & \quad + \mathbb{E}_{\boldsymbol \eta^N}\bigg[\Big(\mathcal{L}^{N}(g^{N,T,x}_{\mathcal{I}}(t_{1}, \cdot))\Big)^{2}(\eta^{N}(t_{2})) \\
            & \hspace{3cm} + \Big(\mathcal{L}^{N}(g^{N,T,x}_{\mathcal{I}})^{2}(t_{1}, \cdot)\Big)^{2}(\eta^{N}(t_{2}))\bigg] \bigg) < \infty.
        \end{aligned}
        \end{equation}

        \item For any fixed~$t_{1} \in [0,T]$, the process~$(M(t))_{t \geq 0}$ given by, for~$t \geq 0$,
        \begin{equation*}
            M(t) \defeq \Big(g^{N,T,x}_{\mathcal{I}}(t_{1}, \eta^{N}(t))\Big)^{2} - \Big(g^{N,T,x}_{\mathcal{I}}(t_{1}, \eta^{N}(0))\Big)^{2} - \int_{0}^{t}  \mathcal{L}^{N}\Big((g^{N,T,x}_{\mathcal{I}})^{2}(t_{1}, \cdot)\Big)(\eta^{N}(\tau-)) \, d\tau,
        \end{equation*}
        is a càdlàg martingale with respect to the filtration~$\{\mathcal{F}^{\eta^{N}}_{t+}\}_{t \geq 0}$.

        \item 
        For any~$T_{1},T_{2} \in [0,T]$ and~$x_{1},x_{2} \in L_{N}^{-1}\mathbb{Z}$, the map~$[0,T_{1} \wedge T_{2}] \times \mathcal{S}^{N} \ni (t, \boldsymbol{\xi})  \, \mapsto \mathcal{L}^{N}\Big((g^{N,T_{1},x_{1}}_{\mathcal{I}} - g^{N,T_{2},x_{2}}_{\mathcal{I}})^{2}(t, \cdot)\Big)(\boldsymbol{\xi})$ is in $\mathcal{C}([0,T_{1} \wedge T_{2}] \times \mathcal{S}^{N}; \mathbb{R})$ and is given by, for~$t \in [0,T_{1} \wedge T_{2}]$ and~$\boldsymbol{\xi} = (\xi_k(y))_{k \in \mathbb N_0, \, y \in L_N^{-1}\mathbb Z}\in \mathcal{S}^{N}$,
        \begin{equation} \label{Paper02_identity_martingale_difference_increments_generator}
        \begin{split}
            & \mathcal{L}^{N}\Big((g^{N,T_{1},x_{1}}_{\mathcal{I}} - g^{N,T_{2},x_{2}}_{\mathcal{I}})^{2}(t, \cdot)\Big)(\boldsymbol{\xi}) \\ & \quad = 2\Big(g^{N,T_{1},x_{1}}_{\mathcal{I}} - g^{N,T_{2},x_{2}}_{\mathcal{I}}\Big)(t, \boldsymbol{\xi}) \mathcal{L}^{N} \Big((g^{N,T_{1},x_{1}}_{\mathcal{I}} - g^{N,T_{2},x_{2}}_{\mathcal{I}})(t, \cdot)\Big)(\boldsymbol{\xi}) \\ & \quad \quad \quad + \frac{m_{N}}{2NL_{N}^{4}} \sum_{y \in L_{N}^{-1}\mathbb{Z}} \Big(\delta^{N,T_{1},T_{2},x_{1},x_{2}}_{-}(t,y) + \delta^{N,T_{1},T_{2},x_{1},x_{2}}_{+}(t,y) \Big) \sum_{k \in \mathcal{I}} \frac{\xi_{k}(y)}{N} \\ & \quad \quad \quad + \frac{1}{NL_{N}^{2}} \sum_{y \in L_{N}^{-1}\mathbb Z} \; \Big(p^{N}(T_{1} - t, y - x_{1}) - p^{N}(T_{2} - t, y - x_{2})\Big)^{2} \sum_{k \in \mathcal{I}} F^{+}_{k}\left(\frac{\xi(y)}{N}\right),
        \end{split}
        \raisetag{-2.3cm}
        \end{equation}
        where
        \begin{align} 
             & \delta^{N,T_{1},T_{2},x_{1},x_{2}}_{-}(t,y) \notag \\ 
             & \quad \defeq \Big( \nabla_{L_N} p^{N}(T_1-t,y - L_{N}^{-1} - x_{1}) - \nabla_{L_N} p^{N}(T_2-t,y - L_{N}^{-1} - x_{2})\Big)^2 \label{Paper02_migration_backward_difference_green}, \\ & \delta^{N,T_{1},T_{2},x_{1},x_{2}}_{+}(t,y) \notag \\
             & \quad \defeq \Big( \nabla_{L_N} p^{N}(T_1-t,y - x_{1}) - \nabla_{L_N} p^{N}(T_2-t,y - x_{2})\Big)^2 \label{Paper02_migration_forward_difference_green}.
        \end{align}
        Furthermore,
        \begin{equation} \label{Paper02_square_gen_diff_square_green}
            \sup_{t_1,t_2 \in [0, T_1 \wedge T_2]} \mathbb{E}_{\boldsymbol \eta^N} \left[ \left(\mathcal{L}^{N}\Big((g^{N,T_{1},x_{1}}_{\mathcal{I}} - g^{N,T_{2},x_{2}}_{\mathcal{I}})^{2}(t_1, \cdot)\Big)(\eta^N(t_2))\right)^2\right] < \infty.
        \end{equation}\end{enumerate}
\end{lemma}

Since the proof of Lemma~\ref{Paper02_properties_greens_function_muller_ratchet} follows from standard arguments, we postpone it until Section~\ref{Paper02_appendix_section_proof_auxiliary_no_technical_lemmas} in the appendix. Importantly, Lemma~\ref{Paper02_properties_greens_function_muller_ratchet} will now allow us to prove a Green's function representation of the spatial Muller's ratchet. Recall from after~\eqref{Paper02_inf_gen_rw} that for~$N \in \mathbb{N}$, we let~$\{P^{N}_{t}\}_{t \geq 0}$ be the semigroup associated to the simple symmetric random walk on~$L_{N}^{-1}\mathbb{Z}$ with total jump rate~$m_{N}$. For any set of indices $\mathcal{I} \subseteq \mathbb{N}_{0}$ and $N \in \mathbb{N}$, define $\widehat{U}^{N}_{\mathcal{I}}: [0,\infty) \times \mathbb{R} \rightarrow \mathbb{R}$ by letting, for~$T \geq 0$ and~$x \in L_{N}^{-1}\mathbb{Z}$,
\begin{equation} \label{Paper02_definition_little_u_k_N_hat}
    \widehat{U}^{N}_{\mathcal{I}}(T,x) \defeq {U}^{N}_{\mathcal{I}}(T,x) - \sum_{k \in \mathcal{I}} P^{N}_{T}u^{N}_{k}(0,\cdot)(x),
\end{equation}
and linearly interpolating between the demes $x \in L_{N}^{-1}\mathbb{Z}$ to extend the definition of $\widehat{U}^{N}_{\mathcal{I}}(T,x)$ to all $x \in \mathbb{R}$. We are now ready to characterise $\widehat{U}^{N}_{\mathcal{I}}$ in terms of a semimartingale, giving us our Green's function representation.

\begin{lemma}[Green's function representation] \label{Paper02_lemma_little_u_k_N_hat_semimartingale}
Suppose $N \in \mathbb{N}$ and the conditions of Lemma~\ref{Paper02_properties_greens_function_muller_ratchet} are satisfied. Then for any $x \in L_{N}^{-1}\mathbb{Z}$, $T \geq 0$ and $\mathcal{I} \subseteq \mathbb{N}_{0}$, under $\mathbb P_{\boldsymbol \eta^N}$, there exist a càdlàg square integrable martingale $(M^{N,T,x}_{\mathcal{I}}(t))_{t \in [0,T]}$ with respect to the filtration~$\{\mathcal{F}^{\eta^{N}}_{t+}\}_{t \geq 0}$ with $M^{N,T,x}_{\mathcal{I}}(0) = 0$, and a finite variation process $(A^{N,T,x}_{\mathcal{I}}(t))_{t \in [0,T]}$ such that
\begin{equation} \label{Paper02_definition_general_formula_u_hat}
    \widehat{U}^{N}_{\mathcal{I}}(T,x) = M^{N,T,x}_{\mathcal{I}}(T) + A^{N,T,x}_{\mathcal{I}}(T).
\end{equation}
Moreover, for~$t \in [0,T]$,
\begin{equation} \label{Paper02_finite_variation_process_uN}
     A^{N,T,x}_{\mathcal{I}}(t) = \frac{1}{L_{N}} \, \sum_{y \in L_{N}^{-1}\mathbb{Z}} \, \int_{0}^{t} \sum_{k \in \mathcal{I}} \; F_{k}(u^{N}(\tau-, y)) p^{N}(T-\tau,y - x)\, d\tau,
\end{equation}
and the predictable bracket process of $(M^{N,T,x}_{\mathcal{I}}(t))_{t \in [0,T]}$ is given by, for $t \in [0,T]$,
\begin{equation} \label{Paper02_predictable_bracket_process_uN_k_hat}
\begin{aligned}
     & \left\langle M^{N,T,x}_{\mathcal{I}}\right\rangle (t) \\
     & \quad = \frac{1}{NL_{N}^{2}} \, \sum_{y \in L_{N}^{-1}\mathbb{Z}} \, \int_{0}^{t} p^{N}(T-\tau,y - x)^{2} \sum_{k \in \mathcal{I}} \, F_{k}^{+}(u^{N}(\tau-, y)) \, d\tau
     \\ & \quad \; + \frac{m_{N}}{2NL_{N}^{4}} \, \sum_{y \in L_{N}^{-1}\mathbb{Z}} \, \int_{0}^{t} \left( \nabla_{L_N} p^{N}(T-\tau,y - L_{N}^{-1} - x)^2 + \nabla_{L_N} p^{N}(T-\tau,y - x)^2\right) 
     \\ & \hspace{10cm} \cdot \sum_{k \in \mathcal{I}} u^{N}_{k}(\tau-, y)  \, d\tau.
\end{aligned}
\end{equation}
\end{lemma}

\begin{proof}
The result will follow from an application of Lemma~\ref{Paper02_properties_greens_function_muller_ratchet} and a standard stochastic chain rule formula (see Lemmas~\ref{Paper02_general_integration_by_parts_formula} and~\ref{Paper02_integration_by_parts_time_predictable_bracket_process} in the appendix). Indeed, let $N \in \mathbb{N}$, $x \in L_{N}^{-1}\mathbb{Z}$, $T \geq 0$ and $\mathcal{I} \subseteq \mathbb{N}_{0}$ be fixed. By the definition of~$g^{N,T,x}_{\mathcal{I}}$ in~\eqref{Paper02_green_function_representation_writing_as_lipschitz_map} and by~\eqref{Paper02_green_function_test_function},~\eqref{Paper02_local_mass_subset_indices_definition} and~\eqref{Paper02_approx_pop_density_definition}, we almost surely have
\begin{equation} \label{Paper02_using_green_function_density_process}
    U^{N}_{\mathcal{I}}(T,x) = g^{N,T,x}_{\mathcal{I}}(T, \eta^{N}(T)).
\end{equation}
We now claim that the map~$g^{N,T,x}_{\mathcal{I}}: [0,T] \times \mathcal{S}^{N} \rightarrow \mathbb{R}_{+}$ satisfies the conditions of Lemmas~\ref{Paper02_general_integration_by_parts_formula} and~\ref{Paper02_integration_by_parts_time_predictable_bracket_process} with $\mathcal S = \mathcal S^N$, $\eta = \eta^N$ and $\mathcal L = \mathcal L^N$. Indeed, checking the conditions for Lemma~\ref{Paper02_general_integration_by_parts_formula}, Lemma~\ref{Paper02_properties_greens_function_muller_ratchet} tells us that $g^{N,T,x}_{\mathcal{I}} \in \mathcal{C}([0,T] \times \mathcal S^N, \mathbb R)$, and~$g^{N,T,x}_{\mathcal{I}}$ satisfies condition~(i) by estimate~\eqref{Paper02_general_estimates_moments_green_function_and_derivatives_time_generator} and Jensen's inequality, condition~(ii) by Lemma~\ref{Paper02_properties_greens_function_muller_ratchet}(i), condition~(iii) by Lemma~\ref{Paper02_properties_greens_function_muller_ratchet}(ii), and condition~(iv) by~estimate~\eqref{Paper02_general_estimates_moments_green_function_and_derivatives_time_generator}. Then checking the conditions of Lemma~\ref{Paper02_integration_by_parts_time_predictable_bracket_process}, $g^{N,T,x}_{\mathcal{I}}$ satisfies condition~(i) by Lemma~\ref{Paper02_properties_greens_function_muller_ratchet}(iii) and~(v), and condition~(ii) by estimate~\eqref{Paper02_general_estimates_moments_green_function_and_derivatives_time_generator}. Therefore, by Lemma~\ref{Paper02_general_integration_by_parts_formula}, there exists a càdlàg martingale~$(M^{N,T,x}_{\mathcal{I}}(t))_{t \in [0,T]}$ with respect to the filtration $\{\mathcal F^{\eta^N}_{t +}\}_{t \geq 0}$ with $M^{N,T,x}_{\mathcal I}(0) = 0$ such that
\begin{equation} \label{Paper02_intermediate_semimartingale_formulation}
\begin{aligned}
& g^{N,T,x}_{\mathcal{I}}(T,\eta^{N}(T)) \\
& \quad =  g^{N,T,x}_{\mathcal{I}}(0,\eta^{N}(0)) + M^{N,T,x}_{\mathcal{I}}(T)
\\ & \; \quad + \int_{0}^{T} \left(\left(\frac{\partial}{\partial t} g^{N,T,x}_{\mathcal{I}}(\cdot, \eta^{N}(t-))\right)(t) + \mathcal{L}^{N}\Big(g^{N,T,x}_{\mathcal I}(t, \cdot )\Big) (\eta^{N}(t-))\right) \, dt
\\ & \quad = g^{N,T,x}_{\mathcal{I}}(0,\eta^{N}(0)) + M^{N,T,x}_{\mathcal{I}}(T) + \int_{0}^{T} \Bigg\langle \sum_{k \in \mathcal{I}} \, F_{k}(u^{N}(t-, \cdot)),\phi^{N,T,x}(t, \cdot) \Bigg\rangle_{N} \, dt
\\ & \quad = g^{N,T,x}_{\mathcal{I}}(0,\eta^{N}(0)) + M^{N,T,x}_{\mathcal{I}}(T) + A^{N,T,x}_{\mathcal{I}}(T),
\end{aligned}
\end{equation}
where in the second equality we used identities~\eqref{Paper02_formula_derivative_time_greens_function} and~\eqref{Paper02_definition_action_generator_greens_function}, and in the third equality we used~\eqref{Paper02_finite_variation_process_uN} and~\eqref{Paper02_green_function_test_function}. Moreover, by~\eqref{Paper02_green_function_representation_writing_as_lipschitz_map},~\eqref{Paper02_green_function_test_function} and the definition of $\{P^N_t\}_{t \geq 0}$ after~\eqref{Paper02_inf_gen_rw}, we have
\begin{equation} \label{Paper02_brownian_semigroup_initial_condition_heat_kernel}
    g^{N,T,x}_{\mathcal{I}}(0,\eta^{N}(0)) = \sum_{k \in \mathcal{I}} \, P^{N}_{T}u^{N}_{k}(0, \cdot)(x).
\end{equation}
Applying~\eqref{Paper02_definition_little_u_k_N_hat},~\eqref{Paper02_using_green_function_density_process} and~\eqref{Paper02_brownian_semigroup_initial_condition_heat_kernel} to~\eqref{Paper02_intermediate_semimartingale_formulation}, and then rearranging terms, we get
\begin{equation*}
    \widehat{U}^{N}_{\mathcal{I}}(T,x) = M^{N,T,x}_{\mathcal{I}}(T) + A^{N,T,x}_{\mathcal{I}}(T).
\end{equation*}

It remains to compute the predictable bracket process of~$(M^{N,T,x}_{\mathcal{I}}(t))_{t \in [0,T]}$. By Lemma~\ref{Paper02_integration_by_parts_time_predictable_bracket_process} in the appendix, we have, for~$t \in [0,T]$,
\begin{equation*}
\begin{aligned}
    \left\langle M^{N,T,x}_{\mathcal{I}}\right\rangle(t) &= \int_{0}^{t} \Big(\mathcal{L}^{N}\Big((g^{N,T,x}_{\mathcal{I}})^{2}(\tau, \cdot)\Big)(\eta^{N}(\tau-)) \\
    & \hspace{2.5cm} - 2g^{N,T,x}_{\mathcal{I}}(\tau,\eta^{N}(\tau-))\mathcal{L}^{N}\Big(g^{N,T,x}_{\mathcal{I}}(\tau, \cdot)\Big)(\eta^{N}(\tau-)) \Big) d\tau.
\end{aligned}
\end{equation*}
Hence, by~\eqref{Paper02_definition_action_generator_on_square_greens_function}, we get~\eqref{Paper02_predictable_bracket_process_uN_k_hat}, as claimed.
\end{proof}

We now apply Lemma~\ref{Paper02_lemma_little_u_k_N_hat_semimartingale} to obtain first moment estimates on the increments of $\widehat{U}^{N}_{\mathcal{I}}$ over time and space.

\begin{lemma} \label{Paper02_lemma_increments_liitle_u_hat_n_k}
Suppose $(m_N)_{N \in \mathbb N}$, $(L_N)_{N \in \mathbb N}$, $(s_k)_{k \in \mathbb N_0}$, $q_+$, $q_-$ and $f$ satisfy Assumptions~\ref{Paper02_scaling_parameters_assumption},~\ref{Paper02_assumption_fitness_sequence},~\ref{Paper02_assumption_polynomials} and~\ref{Paper02_assumption_initial_condition}. For each $N \in \mathbb N$, define $\boldsymbol \eta^N$ as in~\eqref{Paper02_initial_condition}. Then for any $T \geq 0$, there exists $C_{T} > 0$ such that for all $0 \leq T_{1} \leq T_{2} \leq T$, $N \in \mathbb{N}$ and $x_{1}, x_{2} \in L_{N}^{-1}\mathbb{Z}$,
\begin{equation} \label{Paper02_increments_space_time_expectation_bound_fixed_N}
\begin{aligned}
    & \sup_{\substack{\mathcal{I} \subseteq \mathbb{N}_{0}}} \; \mathbb{E}_{\boldsymbol \eta^N}\left[\left\vert \widehat{U}^{N}_{\mathcal{I}}(T_{1}, x_{1}) - \widehat{U}^{N}_{\mathcal{I}}(T_{2}, x_{2}) \right\vert\right] \\
    & \quad \leq C_{T} \left(\vert x_{1} - x_{2} \vert^{1/2} + \vert x_{1} - x_{2} \vert + \left(T_{2} - T_{1}\right)^{1/4}\right).
\end{aligned}
\end{equation}
\end{lemma}

\begin{proof}
    By the triangle inequality, we have that almost surely,
    \begin{equation} \label{Paper02_triangle_inequality_first_moment_increments_little_u_N_hat}
    \begin{aligned}
        & \left\vert \widehat{U}^{N}_{\mathcal{I}}(T_{1}, x_{1}) - \widehat{U}^{N}_{\mathcal{I}}(T_{2}, x_{2}) \right\vert \\
        & \quad \leq \left\vert \widehat{U}^{N}_{\mathcal{I}}(T_{2}, x_{1}) - \widehat{U}^{N}_{\mathcal{I}}(T_{2}, x_{2}) \right\vert + \left\vert \widehat{U}^{N}_{\mathcal{I}}(T_{2}, x_{1}) - \widehat{U}^{N}_{\mathcal{I}}(T_{1}, x_{1}) \right\vert.
    \end{aligned}
    \end{equation}
    We will bound the two terms on the right-hand side of~\eqref{Paper02_triangle_inequality_first_moment_increments_little_u_N_hat} separately. 

\medskip

\myemph{Step~$(1)$: Space increments}
We claim that the space increments satisfy the following bound: there exists $C_T^{(1)} > 0$ such that for all $0 \leq T_2 \leq T$, $N \in \mathbb N$, $x_1,x_2 \in L_N^{-1}\mathbb Z$, and $\mathcal I \subseteq \mathbb N_0$,
\begin{equation} \label{Paper02_spatial_increment_triangle_inequality_intermediate_final_bound}
    \mathbb{E}_{\boldsymbol{\eta}^N}\left[\left\vert \widehat{U}^{N}_{\mathcal{I}}(T_{2}, x_{1}) - \widehat{U}^{N}_{\mathcal{I}}(T_{2}, x_{2}) \right\vert \right] \leq C^{(1)}_T \left(\vert x_{1} - x_{2} \vert + {\vert x_{1} - x_{2} \vert}^{1/2}\right).
\end{equation}
We now prove~\eqref{Paper02_spatial_increment_triangle_inequality_intermediate_final_bound}. Observe first that by Lemma~\ref{Paper02_lemma_little_u_k_N_hat_semimartingale} and the triangle inequality, for $0 \leq T_2 \leq T$, $N \in \mathbb N$, $x_1,x_2 \in L_N^{-1}\mathbb Z$ and $\mathcal I \subseteq \mathbb N_0$,
\begin{equation} \label{Paper02_spatial_increment_triangle_inequality}
\begin{aligned}
    & \left\vert \widehat{U}^{N}_{\mathcal{I}}(T_{2}, x_{1}) - \widehat{U}^{N}_{\mathcal{I}}(T_{2}, x_{2}) \right\vert \\
    & \quad \leq \left\vert M^{N,T_{2},x_{1}}_{\mathcal{I}}(T_{2}) - M^{N,T_{2},x_{2}}_{\mathcal{I}}(T_{2}) \right\vert + \left\vert A^{N,T_{2},x_{1}}_{\mathcal{I}}(T_{2}) - A^{N,T_{2},x_{2}}_{\mathcal{I}}(T_{2}) \right\vert.
\end{aligned}
\end{equation}
We will again bound the terms on the right-hand side of~\eqref{Paper02_spatial_increment_triangle_inequality} separately. For the second term on the right-hand side, by~\eqref{Paper02_finite_variation_process_uN} from Lemma~\ref{Paper02_lemma_little_u_k_N_hat_semimartingale} and the triangle inequality, we conclude that
\begin{equation}
\label{Paper02_bounding_space_increment_finite_variation_process}
\begin{aligned}
    & \mathbb{E}_{\boldsymbol \eta^N}\left[\left\vert A^{N,T_{2},x_{1}}_{\mathcal{I}}(T_{2}) - A^{N,T_{2},x_{2}}_{\mathcal{I}}(T_{2})\right\vert \right]
    \\ & \quad \leq \frac{1}{L_{N}} \sum_{y \in L_{N}^{-1}\mathbb{Z}} \, \int_{0}^{T_{2}} \left\vert p^{N}\left(T_{2}-t,y - x_{1}\right) - p^{N}\left(T_{2}-t,y - x_{2}\right)\right\vert \\
    & \hspace{6cm} \cdot \sum_{k \in \mathcal{I}} \mathbb{E}_{\boldsymbol \eta^N}\left[\left\vert F_{k}(u^{N}(t-, y))\right\vert\right] \, dt
    \\ & \quad \leq \frac{1}{L_{N}} \sum_{y \in L_{N}^{-1}\mathbb{Z}} \, \int_{0}^{T_{2}} \left\vert p^{N}\left(T_{2}-t,y - x_{1}\right) - p^{N}\left(T_{2}-t,y - x_{2}\right)\right\vert \\
    & \hspace{6cm} \cdot \sum_{k \in \mathcal{I}} \mathbb{E}_{\boldsymbol \eta^N}\left[ F_{k}^{+}(u^{N}(t-, y))\right] \, dt,
\end{aligned}
\end{equation}
where in the second inequality we used the definitions of~$(F_{k})_{k \in \mathbb{N}_{0}}$ in~\eqref{Paper02_reaction_term_PDE} and $(F^{+}_{k})_{k \in \mathbb{N}_{0}}$ in~\eqref{Paper02_reaction_predictable_bracket_process} together with the triangle inequality. Now, notice that there exists $C^{(1)}_{q_+,q_-,f}(T) > 0$ such that for any $\mathcal{I} \subseteq \mathbb{N}_{0}$, $t \in [0, T]$ and $y \in L_{N}^{-1}\mathbb{Z}$,
\begin{equation} \label{Paper02_good_behaviour_indices_modified_reaction_term}
\begin{aligned}
    & \sum_{k \in \mathcal{I}} \, \mathbb{E}_{\boldsymbol \eta^N}\left[ F_{k}^{+}(u^{N}(t-, y))\right] \\ & \, =  \sum_{k \in \mathcal{I}} \mathbb{E}_{\boldsymbol \eta^N}\Big[q_{+}\left(\| u^{N}(t-,y) \|_{\ell_{1}}\right)\left(s_{k}(1-\mu)u^{N}_{k}(t-,y) + \mathds{1}_{\{k \geq 1\}}s_{k-1}\mu u^{N}_{k-1}(t-,y)\right)  \\ & \quad \quad \quad \quad \quad + q_{-}\left(\| u^{N}(t-,y) \|_{\ell_{1}}\right)u^{N}_{k}(t-,y)\Big] \\ & \, \leq \mathbb{E}_{\boldsymbol \eta^N}\Big[\| u^{N}(t-,y) \|_{\ell_{1}}\left(q_{+}\left(\| u^{N}(t-,y) \|_{\ell_{1}}\right) + q_{-}\left(\| u^{N}(t-,y) \|_{\ell_{1}}\right)\right)\Big] \\ & \leq {C}^{(1)}_{q_{+}, q_{-},f}(T),
\end{aligned}
\end{equation}
where the first inequality follows since $s_{k} \leq 1$ for every~$k \in \mathbb{N}_{0}$ by Assumption~\ref{Paper02_assumption_fitness_sequence}, and the second inequality follows from Theorem~\ref{Paper02_bound_total_mass}. Hence, applying~\eqref{Paper02_good_behaviour_indices_modified_reaction_term} and then estimate~\eqref{Paper02_random_walk_estimates_ii} from Lemma~\ref{Paper02_random_walks_lemma} in the appendix to~\eqref{Paper02_bounding_space_increment_finite_variation_process}, we get that there exists $C^{(2)}_{q_+,q_-,f}(T) > 0$ such that for any $0 \leq T_2 \leq T$, $N \in \mathbb N$, $x_1,x_2 \in L_N^{-1}\mathbb Z$ and $\mathcal I \subseteq \mathbb N_0$,
\begin{equation} \label{Paper02_bounding_space_increment_finite_variation_process_II}
\begin{aligned}
    & \mathbb{E}_{\boldsymbol \eta^N}\left[\left\vert A^{N,T_{2},x_{1}}_{\mathcal{I}}(T_{2}) - A^{N,T_{2},x_{2}}_{\mathcal{I}}(T_{2})\right\vert \right] \\ & \quad \leq C^{(1)}_{q_+,q_-,f}(T) \frac{1}{L_{N}} \sum_{y \in L_{N}^{-1}\mathbb{Z}} \int_{0}^{T_{2}} \left\vert p^{N}\left(T_{2}-t,y - x_{1}\right) - p^{N}\left(T_{2}-t,y - x_{2}\right)\right\vert dt \\
    & \quad \leq C^{(2)}_{q_+,q_-,f}(T) \vert x_{1} - x_{2} \vert.
\end{aligned}
\end{equation}

Moving now to the first term on the right-hand side of~\eqref{Paper02_spatial_increment_triangle_inequality}, notice that since by Lemma~\ref{Paper02_lemma_little_u_k_N_hat_semimartingale}, both processes $(M^{N,T_{2},x_{1}}_{\mathcal{I}}(t))_{t \in [0,T_2]}$ and $(M^{N,T_{2},x_{2}}_{\mathcal{I}}(t))_{t \in [0,T_2]}$ are martingales with respect to the same filtration~$\{\mathcal{F}^{\eta^{N}}_{t+}\}_{t \in 
[0,T_2]}$, their difference is also a martingale with respect to this filtration. We now bound the predictable bracket process of the difference.
By Lemma~\ref{Paper02_properties_greens_function_muller_ratchet}(vi) and the same argument as in the proof of~\eqref{Paper02_predictable_bracket_process_uN_k_hat} in Lemma~\ref{Paper02_lemma_little_u_k_N_hat_semimartingale}, with~$g^{N,T,x}_{\mathcal{I}}$ replaced by~$g^{N,T_{2},x_{1}}_{\mathcal{I}} - g^{N,T_{2},x_{2}}_{\mathcal{I}}$, we conclude that
\begin{equation} \label{Paper02_first_bound_martingale_term_space_increments_i} 
\begin{aligned}
    & \mathbb{E}_{\boldsymbol \eta^N}\left[\left\langle M^{N,T_{2},x_{1}}_{\mathcal{I}} - M^{N,T_{2},x_{2}}_{\mathcal I} \right\rangle (T_{2})\right]
    \\
    & \quad = \frac{1}{NL_{N}^{2}} \sum_{y \in L_{N}^{-1}\mathbb{Z}} \int_{0}^{T_{2}} \left(p^{N}\left(T_{2}-t,y - x_{1}\right) - p^{N}\left(T_{2}-t,y - x_{2}\right)\right)^{2} \\
    & \hspace{7cm} \cdot \sum_{k \in \mathcal{I}} \mathbb{E}_{\boldsymbol \eta^N}\left[F_{k}^{+}(u^{N}(t-,y))\right]\, dt
    \\ & \quad \quad \quad + \frac{m_{N}}{2NL_{N}^{4}} \sum_{y \in L_{N}^{-1}\mathbb{Z}} \int_{0}^{T_{2}} \Big(\delta^{N,T_{2},T_{2},x_{1},x_{2}}_{-}(t,y) + \delta^{N,T_{2},T_{2},x_{1},x_{2}}_{+}(t,y) \Big) \\
    & \hspace{7cm} \cdot \sum_{k \in \mathcal{I}} \mathbb{E}_{\boldsymbol \eta^N}\left[u^{N}_{k}(t-,y)\right] \, dt,
\end{aligned}
\end{equation}
where the terms~$\delta^{N,T_{2},T_{2},x_{1},x_{2}}_{-}(t,y)$ and~$\delta^{N,T_{2},T_{2},x_{1},x_{2}}_{+}(t,y)$ are defined in~\eqref{Paper02_migration_backward_difference_green} and~\eqref{Paper02_migration_forward_difference_green} respectively. Using the elementary inequality~\eqref{Paper02_elem_ineq_binomial_power_p} with $p =2$, we have that for~$t \in [0,T_{2}]$ and~$y \in L_{N}^{-1}\mathbb{Z}$,
\begin{equation} \label{Paper02_auxiliary_elementary_space_increment}
\begin{aligned}
    & \delta^{N,T_{2},T_{2},x_{1},x_{2}}_{-}(t,y) \\
    & \quad \leq 2 \nabla_{L_N} p^{N}(T_{2} - t, y - L_N^{-1} - x_{1})^2 + 2 \nabla_{L_N} p^{N}(T_{2} - t, y - L_N^{-1} - x_{2})^2, \\ \textrm{and }\quad  & \delta^{N,T_{2},T_{2},x_{1},x_{2}}_{+}(t,y) \leq 2 \nabla_{L_N} p^{N}(T_{2} - t, y - x_{1})^2 + 2 \nabla_{L_N} p^{N}(T_{2} - t, y - x_{2})^2.
\end{aligned}
\end{equation}
Therefore, substituting into~\eqref{Paper02_first_bound_martingale_term_space_increments_i}, using~\eqref{Paper02_good_behaviour_indices_modified_reaction_term} and the random walk estimate~\eqref{Paper02_random_walk_estimates_v} from Lemma~\ref{Paper02_random_walks_lemma} in the appendix to bound the first term on the right-hand side, and using Theorem~\ref{Paper02_bound_total_mass},~\eqref{Paper02_auxiliary_elementary_space_increment},~\eqref{Paper02_discrete_grad} and~\eqref{Paper02_random_walk_estimates_v} again to bound the second term on the right-hand side, there exists $C^{(1)}_{q_+,q_-,f,m}(T)>0$ such that for any $0 \leq T_2 \leq T$, $N \in \mathbb N$, $x_1,x_2 \in L_N^{-1}\mathbb Z$ and $\mathcal I \subseteq \mathbb N_0$,
\begin{equation} \label{Paper02_first_bound_martingale_term_space_increments_ii}
    \mathbb E_{\boldsymbol \eta^N}\left[\left\langle M^{N,T_{2},x_{1}}_{\mathcal{I}} - M^{N,T_{2},x_{2}}_{\mathcal I} \right\rangle (T_{2})\right] 
     \leq C^{(1)}_{q_+,q_-,f,m}(T)\left(\frac{1}{NL_N} \vert x_1 - x_2 \vert + \frac{m_N}{NL_N} \cdot L_N^{-1}\right).
\end{equation}
Observe that there exists a constant $K > 0$ such that for any càdlàg martingale~$(M(t))_{t \geq 0}$ with $M(0) = 0$ and any~$t \geq 0$, by Jensen's inequality and then by the Burkholder-Davis-Gundy (BDG) inequality,
\begin{equation} \label{Paper02_BDG_inequality}
\begin{aligned}
    \mathbb{E}\Big[\sup_{\tau \leq t} \vert M(\tau) \vert \Big]^2  \leq  \mathbb{E}\Big[\sup_{\tau \leq t} \vert M(\tau) \vert^{2} \Big] \leq K \mathbb{E}\Big[ [M](t) \Big] = K \mathbb{E}\Big[ \langle M \rangle(t) \Big].
\end{aligned}
\end{equation}
Therefore, by~\eqref{Paper02_first_bound_martingale_term_space_increments_ii} and since $m_N/L_N^2 \rightarrow m \in (0, \infty)$ and $L_N = \Theta(N)$ as $N \rightarrow \infty$ by Assumption~\ref{Paper02_scaling_parameters_assumption}, there exists $C^{(2)}_{q_+,q_-,f,m}(T) > 0$ such that for any $0 \leq T_2 \leq T$, $N \in \mathbb N$, $x_1 \neq x_2 \in L_N^{-1}\mathbb Z$ and $\mathcal I \subseteq \mathbb N_0$,
\begin{equation} \label{Paper02_estimate_martingale_spatial_increment}
    \mathbb{E}_{\boldsymbol{\eta}^N}\left[\left\vert \left(M^{N,T_{2},x_{1}}_{\mathcal{I}} - M^{N,T_{2},x_{2}}_{\mathcal{I}}\right) (T_{2}) \right\vert\right] \leq C^{(2)}_{q_+,q_-,f,m}(T)\vert x_{1} - x_{2} \vert^{1/2}.
\end{equation}
Hence, combining~\eqref{Paper02_spatial_increment_triangle_inequality},~\eqref{Paper02_bounding_space_increment_finite_variation_process_II} and~\eqref{Paper02_estimate_martingale_spatial_increment}, we obtain~\eqref{Paper02_spatial_increment_triangle_inequality_intermediate_final_bound}, as claimed.

\medskip

\myemph{Step~$(2)$: Time increments}
We will now bound the expectation of the second term on the right-hand side of~\eqref{Paper02_triangle_inequality_first_moment_increments_little_u_N_hat}; we claim that there exists $C^{(2)}_T>0$ such that for all $0 \leq T_1 \leq T_2 \leq T$, $N \in \mathbb N$, $x_1 \in L_N^{-1}\mathbb Z$ and $\mathcal I \subseteq \mathbb N_0$,
\begin{equation} \label{Paper02_temporal_increment_triangle_inequality_intermediate_final_bound}
    \mathbb{E}_{\boldsymbol{\eta}^N}\left[\left\vert \widehat{U}^{N}_{\mathcal{I}}(T_{2}, x_{1}) - \widehat{U}^{N}_{\mathcal{I}}(T_{1}, x_{1}) \right\vert \right] \leq C^{(2)}_T(T_{2} - T_{1})^{1/4}.
\end{equation}
To prove the claim, we begin by noting that by Lemma~\ref{Paper02_lemma_little_u_k_N_hat_semimartingale} and the triangle inequality, almost surely
\begin{equation} \label{Paper02_time_increment_triangle_inequality}
\begin{aligned}
    & \left\vert \widehat{U}^{N}_{\mathcal{I}}(T_{2}, x_{1}) - \widehat{U}^{N}_{\mathcal{I}}(T_{1}, x_{1}) \right\vert \\ & \quad \leq \left\vert A^{N,T_{2},x_{1}}_{\mathcal{I}}(T_{2}) - A^{N,T_{2},x_{1}}_{\mathcal{I}}(T_{1}) \right\vert + \left\vert A^{N,T_{2},x_{1}}_{\mathcal{I}}(T_{1}) - A^{N,T_{1},x_{1}}_{\mathcal{I}}(T_{1}) \right\vert \\ & \quad \quad \quad + \left\vert M^{N,T_{2},x_{1}}_{\mathcal{I}}(T_{2}) - M^{N,T_{2},x_{1}}_{\mathcal{I}}(T_{1}) \right\vert + \left\vert M^{N,T_{2},x_{1}}_{\mathcal{I}}(T_{1}) - M^{N,T_{1},x_{1}}_{\mathcal{I}}(T_{1}) \right\vert.
\end{aligned}
\end{equation}
We will tackle each of the terms on the right-hand side of~\eqref{Paper02_time_increment_triangle_inequality} separately. For the first term, note that for all $k \in \mathbb{N}_{0}$ and all $u \in \ell_{1}^{+}$, $\vert F_{k}(u) \vert \leq F^{+}_{k}(u)$ by~\eqref{Paper02_reaction_term_PDE} and~\eqref{Paper02_reaction_predictable_bracket_process}. Therefore, by~\eqref{Paper02_finite_variation_process_uN} in Lemma~\ref{Paper02_lemma_little_u_k_N_hat_semimartingale}, and by Fubini's theorem, and then by~\eqref{Paper02_good_behaviour_indices_modified_reaction_term}, for $0 \leq T_1 \leq T_2 \leq T$, $N \in \mathbb N$, $x_1 \in L_N^{-1}\mathbb Z$ and $\mathcal I \subseteq \mathbb N_0$,
\begin{equation}
\label{Paper02_bounding_time_increment_finite_variation_process_larger_time}
\begin{aligned}
    & \mathbb{E}_{\boldsymbol{\eta}^N}\left[\left\vert A^{N,T_{2},x_{1}}_{\mathcal{I}}(T_{2}) - A^{N,T_{2},x_{1}}_{\mathcal{I}}(T_{1})\right\vert \right] \\ & \quad \leq \frac{1}{L_{N}} \sum_{y \in L_{N}^{-1}\mathbb{Z}} \int_{T_{1}}^{T_{2}} p^{N}\left(T_{2}-t,y - x_{1}\right) \cdot \sum_{k \in \mathcal{I}} \mathbb{E}_{\boldsymbol \eta^N}\Big[F^{+}_{k}(u^{N}(t-,y))\Big] \, dt \\
    & \quad \leq C^{(1)}_{q_+,q_-,f}(T) \cdot \frac{1}{L_{N}} \sum_{y \in L_{N}^{-1}\mathbb{Z}} \int_{T_{1}}^{T_{2}} p^{N}\left(T_{2}-t,y - x_{1}\right) \, dt \\
    & \quad = C^{(1)}_{q_+,q_-,f}(T)(T_{2} - T_{1}),
\end{aligned}
\end{equation}
where for the last identity we applied the fact that, by the definition of $p^{N}$ in~\eqref{Paper02_definition_p_N_test_function_green},
\begin{equation} \label{Paper02_simple_identity_total_sum_pN}
    \frac{1}{L_{N}} \sum_{y \in L_{N}^{-1}\mathbb{Z}} p^{N}\left(\tau,y - x_{1}\right) = 1 \quad \forall \, \tau \geq 0.
\end{equation}

For the second term on the right-hand side of~\eqref{Paper02_time_increment_triangle_inequality}, we again apply~\eqref{Paper02_finite_variation_process_uN} from Lemma~\ref{Paper02_lemma_little_u_k_N_hat_semimartingale} and use that $\vert F_k(u) \vert \leq F^{+}_k(u)$ $\forall \, k \in \mathbb N_0$, $u \in \ell_1^+$, and then use estimate~\eqref{Paper02_good_behaviour_indices_modified_reaction_term} and the random walk estimate~\eqref{Paper02_random_walk_estimates_i} from Lemma~\ref{Paper02_random_walks_lemma} in the appendix, obtaining that there exists $C^{(3)}_{q_+,q_-,f,m}(T) > 0$ such that for any $0 \leq T_1 \leq T_2 \leq T$, $N \in \mathbb N$, $x_1 \in L_N^{-1}\mathbb Z$ and $\mathcal I \subseteq \mathbb N_0$, 
\begin{equation}
\label{Paper02_bounding_time_increment_finite_variation_process_smaller_time}
\begin{aligned}
    & \mathbb{E}_{\boldsymbol{\eta}^N}\left[\left\vert A^{N,T_{2},x_{1}}_{\mathcal{I}}(T_{1}) - A^{N,T_{1},x_{1}}_{\mathcal{I}}(T_{1})\right\vert \right]
    \\ & \quad \leq \frac{1}{L_{N}} \sum_{y \in L_{N}^{-1}\mathbb{Z}} \int_{0}^{T_{1}} \left\vert p^{N}\left(T_{2}-t,y - x_{1}\right) - p^{N}\left(T_{1}-t,y - x_{1}\right)\right\vert 
    \\ & \hspace{6cm} \cdot\sum_{k \in \mathcal{I}} \mathbb{E}_{\boldsymbol{\eta}^N}\Big[F^{+}_{k}(u^{N}(t-,y))\Big] \, dt
    \\
    & \quad \leq C^{(3)}_{q_+,q_-,f,m}(T) (T_{2} - T_{1})^{1/2}.
\end{aligned}
\end{equation}

Now tackling the third term on the right-hand side of~\eqref{Paper02_time_increment_triangle_inequality}, by~\eqref{Paper02_predictable_bracket_process_uN_k_hat} from Lemma~\ref{Paper02_lemma_little_u_k_N_hat_semimartingale} combined with~\eqref{Paper02_good_behaviour_indices_modified_reaction_term} and Theorem~\ref{Paper02_bound_total_mass}, there exists $C^{(3)}_{q_+,q_-,f}(T) > 0$ such that for any $0 \leq T_1 \leq T_2 \leq T$, $N \in \mathbb N$, $x_1 \in L_N^{-1}\mathbb Z$ and $\mathcal I \subseteq \mathbb N_0$,
\begin{equation} \label{Paper02_intermediate_step_star_control_fluctiation_time_increment_i}
\begin{aligned}
    & \mathbb{E}_{\boldsymbol{\eta}^N}\left[\left\langle M^{N,T_{2},x_{1}}_{\mathcal{I}} \right\rangle (T_{2}) - \left\langle M^{N,T_{2},x_{1}}_{\mathcal{I}}\right\rangle (T_{1}) \right]
    \\ & \quad \leq C^{(3)}_{q_+,q_-,f}(T) \Bigg( \frac{1}{NL_{N}^{2}} \, \sum_{y \in L_{N}^{-1}\mathbb{Z}} \, \int_{T_{1}}^{T_{2}} p^{N}(T_{2}-\tau,y - x_{1})^{2} \, d\tau
    \\ & \quad \quad + \frac{m_{N}}{2NL_{N}^{4}} \, \sum_{y \in L_{N}^{-1}\mathbb{Z}} \, \int_{T_{1}}^{T_{2}} \Big( \nabla_{L_N} p^{N}(T_{2}-\tau,y - L_N^{-1} - x_{1})^2 \\
    & \hspace{5cm} + \nabla_{L_N} p^{N}(T_{2}-\tau,y - x_{1})^2 \Big) d\tau\Bigg).
\end{aligned}
\end{equation}
We will bound each term on the right-hand side of~\eqref{Paper02_intermediate_step_star_control_fluctiation_time_increment_i} separately. For the first term, by~\eqref{Paper02_simple_identity_total_sum_pN} and the fact that $p^N(t,z) \leq L_N$ $\forall \, t \geq 0, \, z \in L_N^{-1}\mathbb Z$, we can write
\begin{equation} \label{Paper02_intermediate_step_star_control_fluctiation_time_increment_ii}
    \frac{1}{NL_{N}^{2}} \, \sum_{y \in L_{N}^{-1}\mathbb{Z}} \, \int_{T_{1}}^{T_{2}} p^{N}(T_{2}-\tau,y - x_{1})^{2} \, d\tau \leq \frac{T_{2} - T_{1}}{N}.
\end{equation}
For the second term on the right-hand side of~\eqref{Paper02_intermediate_step_star_control_fluctiation_time_increment_i}, observe that by Fubini's theorem, and then by the Cauchy-Schwarz inequality,
\begin{equation*}
\begin{aligned}
    & \frac{m_{N}}{2NL_{N}^{4}} \, \sum_{y \in L_{N}^{-1}\mathbb{Z}} \, \int_{T_{1}}^{T_{2}} \nabla_{L_N} p^{N}(T_{2}-\tau,y - x_{1})^{2}  \, d\tau
    \\ & \quad = \frac{m_{N}}{2NL_{N}^{4}} \int_{0}^{T_{2}} \, \sum_{y \in L_{N}^{-1}\mathbb{Z}} \, \mathds{1}_{\{T_{1} \leq \tau \leq T_{2}\}} \left\vert \nabla_{L_N} p^{N}(T_{2}-\tau,y - x_{1}) \right\vert^{3/2}
    \\ & \hspace{6.5cm} \cdot \left\vert \nabla_{L_N} p^{N}(T_{2}-\tau,y - x_{1}) \right\vert^{1/2} \,  d\tau
    \\ & \quad \leq  \frac{m_{N}}{2NL_{N}^{4}} \int_{0}^{T_{2}} \, \Bigg(\sum_{y \in L_{N}^{-1}\mathbb{Z}} \, \mathds{1}_{\{T_{1} \leq \tau \leq T_{2}\}} \left\vert \nabla_{L_N} p^{N}(T_{2}-\tau,y - x_{1}) \right\vert^{3}\Bigg)^{1/2} \\ & \quad \quad \quad \quad \quad \quad \quad \quad \quad \quad \quad \quad \quad \cdot \Bigg(\sum_{y \in L_{N}^{-1}\mathbb{Z}} \, \left\vert \nabla_{L_N} p^{N}(T_{2}-\tau,y - x_{1}) \right\vert \Bigg)^{1/2} \, d\tau.
\end{aligned}
\end{equation*}
Using the definition of $\nabla_{L_N}p^N$ in~\eqref{Paper02_discrete_grad}, the fact that $p^N(t,z) \leq L_N$ $\forall \, t \geq 0, \, z \in L_N^{-1}\mathbb Z$, and the elementary inequality $(\vert a \vert + \vert b \vert)^3 \leq 4(\vert a\vert^3 + \vert b\vert^3) \; \forall \, a,b \in \mathbb R$, and then by~\eqref{Paper02_simple_identity_total_sum_pN}, it follows that
\begin{equation} \label{Paper02_intermediate_step_star_control_fluctiation_time_increment_iii}
\begin{split}
    & \frac{m_{N}}{2NL_{N}^{4}} \, \sum_{y \in L_{N}^{-1}\mathbb{Z}} \, \int_{T_{1}}^{T_{2}} \nabla_{L_N} p^{N}(T_{2}-\tau,y - x_{1})^2  \, d\tau \\ 
    & \quad \leq \frac{m_{N}}{2NL_{N}^{4}} \int_{0}^{T_{2}} \, \Bigg(4L_{N}^{5}\sum_{y \in L_{N}^{-1}\mathbb{Z}} \, \mathds{1}_{\{T_{1} \leq \tau \leq T_{2}\}} \\
    & \hspace{5cm} \cdot \Big(p^{N}(T_{2}-\tau,y - x_{1}) + p^{N}(T_{2}-\tau,y + L_{N}^{-1} - x_{1}) \Big)\Bigg)^{1/2} \\ 
    & \hspace{7cm} \cdot \Bigg(\sum_{y \in L_{N}^{-1}\mathbb{Z}} \, \left\vert \nabla_{L_N}
    p^{N}(T_{2}-\tau,y - x_{1}) \right\vert \Bigg)^{1/2} \, d\tau \\ 
    & \quad \leq \frac{\sqrt{2} m_{N}}{NL_{N}^{3/2}} \int_{0}^{T_{2}} \mathds{1}_{\{T_{1} \leq \tau \leq T_{2}\}} \Bigg(\sum_{y \in L_{N}^{-1}\mathbb{Z}} \, \left\vert \nabla_{L_N}
    p^{N}(T_{2}-\tau,y - x_{1}) \right\vert \Bigg)^{1/2} \, d\tau \\ & \quad \leq \frac{\sqrt{2} m_{N}}{NL_{N}^{3/2}} \Bigg(\int_{0}^{T_{2}} \mathds{1}_{\{T_{1} \leq \tau \leq T_{2}\}} \, d\tau\Bigg)^{1/2} \cdot \Bigg(\int_{0}^{T_{2}} \sum_{y \in L_{N}^{-1}\mathbb{Z}} \, \left\vert \nabla_{L_N} p^{N}(T_{2}-\tau,y - x_{1}) \right\vert \, d\tau \Bigg)^{1/2},
\end{split}
\raisetag{-5cm}
\end{equation}
where the last inequality follows by the Cauchy-Schwarz inequality. Therefore, by~\eqref{Paper02_random_walk_estimates_ii} from Lemma~\ref{Paper02_random_walks_lemma} in the appendix, we can write
\[
\frac{m_N}{2NL_N^4} \sum_{y \in L_N^{-1}\mathbb Z} \int_{T_1}^{T_2} \nabla_{L_N} p^N(T_2 - \tau, y-x_1)^2 \, d\tau \lesssim_{m,T} \frac{m_N}{NL_N^{3/2}} (T_2 - T_1)^{1/2} L_N^{1/2}.
\]
Since the same bound holds when $y$ is replaced by $y - L_N^{-1}$ on the left-hand side, by~\eqref{Paper02_intermediate_step_star_control_fluctiation_time_increment_i} and~\eqref{Paper02_intermediate_step_star_control_fluctiation_time_increment_ii} it follows that
\begin{equation} \label{Paper02_intermediate_step_star_control_fluctiation_time_increment_iv}
\begin{aligned}
    & \mathbb{E}_{\boldsymbol{\eta}^N}\left[\left\langle M^{N,T_{2},x_{1}}_{\mathcal{I}} \right\rangle (T_{2}) - \left\langle M^{N,T_{2},x_{1}}_{\mathcal{I}}\right\rangle (T_{1}) \right] \\
    & \quad \lesssim_{m,T} C^{(3)}_{q_+,q_-,f}(T) \left(\frac{T_2 - T_1}{N} + \frac{m_N}{NL_N} (T_2 - T_1)^{1/2}\right).
\end{aligned}
\end{equation}

By Assumption~\ref{Paper02_scaling_parameters_assumption}, we have $L_{N} = \Theta(N)$ and $m_{N}/L_{N}^{2} \rightarrow m \in (0, \infty)$ as $N \rightarrow \infty$. Therefore, by~\eqref{Paper02_intermediate_step_star_control_fluctiation_time_increment_iv}, and since $\left(M^{N,T_{2},x_{1}}_{\mathcal{I}}(t) - M^{N,T_{2},x_{1}}_{\mathcal{I}}(T_{1})\right)_{t \in [T_{1}, T_{2}]}$ is a martingale, by applying~Jensen's and the BDG inequalities as in~\eqref{Paper02_BDG_inequality}, we conclude that there exists $C^{(4)}_{q_+,q_-,f,m}(T) > 0$ such that for any $0 \leq T_1 \leq T_2 \leq T$, $N \in \mathbb N$, $x_1 \in L_N^{-1}\mathbb Z$ and $\mathcal I \subseteq \mathbb N_0$,
\begin{equation} \label{Paper02_estimate_martingale_time_increment_larger_time}
    \mathbb{E}_{\boldsymbol \eta^N}\left[\left\vert M^{N,T_{2},x_{1}}_{\mathcal{I}}(T_{2}) - M^{N,T_{2},x_{1}}_{\mathcal{I}}(T_{1}) \right\vert\right] \leq C^{(4)}_{q_+,q_-,f,m}(T)(T_{2} - T_{1})^{1/4}.
\end{equation}

It remains to bound the expectation of the fourth term on the right-hand side of~\eqref{Paper02_time_increment_triangle_inequality}. Since the processes~$(M^{N,T_{2},x_{1}}_{\mathcal{I}}(t))_{t \in [0,T_{2}]}$ and~$(M^{N,T_{1},x_{1}}_{\mathcal{I}}(t))_{t \in [0,T_{1}]}$ are martingales with respect to the filtration~$\{\mathcal{F}^{{\eta}^{N}}_{t+}\}_{t \geq 0}$, the càdlàg process $(M^{N,T_{2},x_{1}}_{\mathcal{I}}(t) - M^{N,T_{1},x_{1}}_{\mathcal{I}}(t))_{t \in [0,T_{1}]}$ is also a martingale with respect to the same filtration. By the same argument as for~\eqref{Paper02_first_bound_martingale_term_space_increments_i}, i.e.~by Lemma~\ref{Paper02_properties_greens_function_muller_ratchet}(vi) and the same argument as in the proof of~\eqref{Paper02_predictable_bracket_process_uN_k_hat} in Lemma~\ref{Paper02_lemma_little_u_k_N_hat_semimartingale}, with~$g^{N,T,x}_{\mathcal{I}}$ replaced by~$g^{N,T_{2},x_{1}}_{\mathcal{I}} - g^{N,T_{1},x_{1}}_{\mathcal{I}}$, we conclude that
\begin{equation} \label{Paper02_increments_time_xi}
\begin{aligned}
    & \mathbb{E}_{\boldsymbol \eta^N}\left[\left\langle M^{N,T_{2},x_{1}}_{\mathcal{I}} - M^{N,T_{1},x_{1}}_{\mathcal{I}} \right\rangle (T_{1})\right] \\
    & \quad = \frac{1}{NL_{N}^{2}} \sum_{y \in L_{N}^{-1}\mathbb{Z}} \int_{0}^{T_{1}} \left(p^{N}\left(T_{2}-t,y - x_{1}\right) - p^{N}\left(T_{1}-t,y - x_{1}\right)\right)^{2} \\
    & \hspace{6cm} \cdot \sum_{k \in \mathcal{I}} \mathbb{E}_{\boldsymbol{\eta}^N}\left[F_{k}^{+}(u^{N}(t-,y))\right]\, dt
    \\ & \quad \quad + \frac{m_{N}}{2NL_{N}^{4}} \sum_{y \in L_{N}^{-1}\mathbb{Z}} \int_{0}^{T_{1}} \Big(\delta^{N,T_{2},T_{1},x_{1},x_{1}}_{-}(t,y) + \delta^{N,T_{2},T_{1},x_{1},x_{1}}_{+}(t,y) \Big) \\
    & \hspace{6cm} \cdot \sum_{k \in \mathcal{I}} \mathbb{E}_{\boldsymbol \eta^N}\left[u^{N}_{k}(t-,y)\right] \, dt,
\end{aligned}
\end{equation}
where the terms~$\delta^{N,T_{2},T_{1},x_{1},x_{1}}_{-}(t,y)$ and~$\delta^{N,T_{2},T_{1},x_{1},x_{1}}_{+}(t,y)$ are defined in~\eqref{Paper02_migration_backward_difference_green} and~\eqref{Paper02_migration_forward_difference_green} respectively. Using~\eqref{Paper02_discrete_grad} and that $(a + b)^2 \leq 2a^2 + 2b^2$ $\forall \, a,b \in \mathbb R$, for $t \in [0,T_1]$ and $y \in L_N^{-1} \mathbb Z$, we have
\begin{equation} \label{Paper02_increments_time_x}
\begin{split}
    \delta^{N,T_{2},T_{1},x_{1},x_{1}}_{-}(t,y) & \leq 2L_N^2 \Big(p^N(T_1-t, y-x_1) - p^N(T_2-t, y-x_1)\Big)^2 \\
     & \quad \; + 2L_N^2 \Big(p^N(T_1-t, y - L_N^{-1} -x_1) - p^N(T_2-t, y - L_N^{-1} -x_1)\Big)^2, \\
     \delta^{N,T_{2},T_{1},x_{1},x_{1}}_{+}(t,y) & \leq 2L_N^2 \Big(p^N(T_1-t, y-x_1) - p^N(T_2-t, y-x_1)\Big)^2 \\ & \quad \; + 2L_N^2 \Big(p^N(T_1-t, y + L_N^{-1} -x_1) - p^N(T_2-t, y + L_N^{-1} -x_1)\Big)^2.
\end{split}
\raisetag{-1.7cm}
\end{equation}
Hence, substituting into~\eqref{Paper02_increments_time_xi}, using~\eqref{Paper02_good_behaviour_indices_modified_reaction_term} and the random walk estimate~\eqref{Paper02_random_walk_estimates_v} from Lemma~\ref{Paper02_random_walks_lemma} in the appendix to bound the first term on the right-hand side, and using Theorem~\ref{Paper02_bound_total_mass},~\eqref{Paper02_increments_time_x} and~\eqref{Paper02_random_walk_estimates_v} again to bound the second term, there exists $C^{(5)}_{q_+,q_-,f,m}(T) > 0$ such that for any $0 \leq T_1 \leq T_2 \leq T$, $N \in \mathbb N$, $x_1 \in L_N^{-1}\mathbb Z$ and $\mathcal I \subseteq \mathbb N_0$,
\begin{equation*} 
\begin{aligned}
    & \mathbb{E}_{\boldsymbol \eta^N}\left[\left\langle M^{N,T_{2},x_{1}}_{\mathcal{I}} - M^{N,T_{1},x_{1}}_{\mathcal{I}} \right\rangle (T_{1})\right] \\
    & \quad \leq C^{(5)}_{q_+,q_-,f,m}(T) \left(\frac{({T_{2} - T_{1}})^{1/2}}{NL_{N}} + \frac{m_N}{NL_N}({T_{2} - T_{1}})^{1/2}\right).
\end{aligned}
\end{equation*}
Hence, by applying Jensen's and the BDG inequalities as in~\eqref{Paper02_BDG_inequality}, and since $m_N/L_N^2 \rightarrow m \in (0,\infty)$ and $L_N = \Theta(N)$ as $N \rightarrow \infty$ by Assumption~\ref{Paper02_scaling_parameters_assumption}, there exists $C^{(6)}_{q_+,q_-,f,m}(T)> 0$ such that for any $0 \leq T_1 \leq T_2 \leq T$, $N \in \mathbb N$, $x_1 \in L_N^{-1}\mathbb Z$ and $\mathcal I \subseteq \mathbb N_0$,
\begin{equation} \label{Paper02_estimate_martingale_time_increment_lower_times}
    \mathbb{E}_{\boldsymbol \eta^N}\left[\left\vert (M^{N,T_{2},x_{1}}_{\mathcal{I}} - M^{N,T_{1},x_{1}}_{\mathcal{I}}) (T_{1}) \right\vert\right] \leq C^{(6)}_{q_+,q_-,f,m}(T) \left(T_{2} - T_{1}\right)^{1/4}.
\end{equation}

Combining~\eqref{Paper02_time_increment_triangle_inequality},~\eqref{Paper02_bounding_time_increment_finite_variation_process_larger_time},~\eqref{Paper02_bounding_time_increment_finite_variation_process_smaller_time},~\eqref{Paper02_estimate_martingale_time_increment_larger_time} and~\eqref{Paper02_estimate_martingale_time_increment_lower_times}, our claim~\eqref{Paper02_temporal_increment_triangle_inequality_intermediate_final_bound} holds. Therefore, combining our bounds on space increments~\eqref{Paper02_spatial_increment_triangle_inequality_intermediate_final_bound} and time increments~\eqref{Paper02_temporal_increment_triangle_inequality_intermediate_final_bound} with~\eqref{Paper02_triangle_inequality_first_moment_increments_little_u_N_hat}, the proof is complete.
\end{proof}

By combining standard random walk estimates with Lemma~\ref{Paper02_lemma_increments_liitle_u_hat_n_k}, we can now derive an integral form of equicontinuity for~$\Big((U^{N}_{\mathcal{I}}(t,x))_{t \in [0,T], \, x \in \mathbb{R}}\Big)_{N \in \mathbb{N}}$ that will be useful for the characterisation of the limiting process in Section~\ref{Paper02_tightness_L_P_l_1spaces}.

\begin{lemma}
\label{Paper02_spatial_increment_bound_usual}
Under the conditions of Lemma~\ref{Paper02_lemma_increments_liitle_u_hat_n_k}, for any~$T > 0$, there exists~$C_{T} > 0$ such that for any $N \in \mathbb N$, $\mathcal I \subseteq \mathbb N_0$, $\gamma_1 \in (-1,1)$, $\gamma_2 \in \mathbb R$ and $x \in \mathbb R$,
\begin{equation*}
\begin{aligned}
    & \mathbb{E}_{\boldsymbol \eta^N}\Bigg[\int_{0}^{T} \left\vert U^{N}_{\mathcal{I}}\left(t,x\right) - U^{N}_{\mathcal{I}}\left(t + \gamma_{1},x+\gamma_{2}\right) \cdot \mathds 1_{\{t + \gamma_1 \in [0,T]\}} \right\vert \, dt\Bigg] \\
    & \quad \leq C_{T}\left(\left\vert \gamma_{1}\right\vert^{1/4} + \left\vert \gamma_{2}\right\vert^{1/2} + \vert \gamma_2 \vert\right).
\end{aligned}
\end{equation*}
\end{lemma}

\begin{proof}
The proof adapts ideas of Durrett and Fan from~\cite[Lemma~5]{durrett2016genealogies}. Recall from~\eqref{Paper02_local_mass_subset_indices_definition} and~\eqref{Paper02_approx_pop_density_definition} that~$U^{N}_{\mathcal{I}}(t,x)$ is defined via linear interpolation in space. We first observe that for $N \in \mathbb N$, $\mathcal I \subseteq \mathbb N_0$, $\gamma_1 \in (-1,1)$, $\gamma_2 \in \mathbb R$ and $x \in \mathbb R$, we have almost surely
\begin{equation} \label{Paper02_lem_equic_int_i}
\begin{aligned}
    & \int_{0}^{T} \left\vert U^{N}_{\mathcal{I}}\left(t,x\right) - U^{N}_{\mathcal{I}}\left(t + \gamma_{1},x+\gamma_{2}\right) \cdot \mathds 1_{\{t + \gamma_1 \in [0,T]\}} \right\vert \, dt \\ & \quad =  \int_{0 \vee (- \gamma_1)}^{T \wedge (T - \gamma_1)} \left\vert U^{N}_{\mathcal{I}}\left(t,x\right) - U^{N}_{\mathcal{I}}\left(t + \gamma_{1},x+\gamma_{2}\right) \right\vert \, dt + \int_{0}^{0 \vee (- \gamma_1)} U^{N}_{\mathcal{I}}\left(t,x\right) \, dt \\ & \quad \quad + \int_{T \wedge (T- \gamma_1)}^{T} U^{N}_{\mathcal{I}}\left(t,x\right) \, dt.
\end{aligned}
\end{equation}
We will bound the expectation of the terms on the right-hand side of~\eqref{Paper02_lem_equic_int_i} separately. For the second and third terms, by~\eqref{Paper02_local_mass_subset_indices_definition} and~\eqref{Paper02_approx_pop_density_definition}, and then by Theorem~\ref{Paper02_bound_total_mass}, there exists $C^{(1)}_{q_+,q_-,f}(T) > 0$ such that for any $N \in \mathbb{N}$, $\mathcal I \subseteq \mathbb N_0$, $\gamma_1 \in (-1,1)$ and $x \in \mathbb R$,
\begin{equation} \label{Paper02_lem_equic_int_ii}
\begin{aligned}
    & \mathbb{E}_{\boldsymbol \eta^N}\Bigg[\int_{0}^{0 \vee (- \gamma_1)} U^{N}_{\mathcal{I}}\left(t,x\right) \, dt + \int_{T \wedge (T- \gamma_1)}^{T} U^{N}_{\mathcal{I}}\left(t,x\right) \, dt \Bigg] \\ & \quad = \int_{0}^{0 \vee (- \gamma_1)} \mathbb E_{\boldsymbol \eta^N}\left[U^{N}_{\mathcal{I}}\left(t,x\right)\right] \, dt + \int_{T \wedge (T- \gamma_1)}^{T} \mathbb E_{\boldsymbol \eta^N}\left[U^{N}_{\mathcal{I}}\left(t,x\right)\right] \, dt \\ & \quad \leq C^{(1)}_{q_+,q_-,f}(T) \vert \gamma_1 \vert.
\end{aligned}
\end{equation}

It remains to bound the expectation of the first term on the right-hand side of~\eqref{Paper02_lem_equic_int_i}. For $N \in \mathbb N$, $\mathcal I \subseteq \mathbb N_0$, $\gamma_1 \in (-1,1)$, $t \in [0 \vee (- \gamma_1), T \wedge (T - \gamma_1)]$, and $y_1,y_2 \in  L_N^{-1}\mathbb Z$, by~\eqref{Paper02_definition_little_u_k_N_hat} and the triangle inequality we have
\begin{equation} \label{Paper02_lem_equic_int_iii}
\begin{aligned}
    & \left\vert U^N_{\mathcal I}(t, y_1) - U^N_{\mathcal I}(t + \gamma_1,y_2) \right\vert \\ & \quad \leq \left\vert \widehat U^N_{\mathcal I}(t, y_1) - \widehat U^N_{\mathcal I}(t + \gamma_1,y_2) \right\vert + \left\vert P^N_t U^N_{\mathcal I}(0, \cdot)(y_1) - P^N_{t+\gamma_1} U^N_{\mathcal I}(0,\cdot)(y_2) \right\vert.
\end{aligned}
\end{equation}
For the second term on the right-hand side of~\eqref{Paper02_lem_equic_int_iii}, recall from after~\eqref{Paper02_inf_gen_rw} that $\{P^N_t\}_{t \geq 0}$ is the semigroup associated to the simple symmetric random walk on $L_N^{-1}\mathbb Z$ with total jump rate $m_N$, and recall the definition of $p^N$ in~\eqref{Paper02_definition_p_N_test_function_green}. Since~$\boldsymbol \eta^N$ is given by~\eqref{Paper02_initial_condition}, where $f$ satisfies Assumption~\ref{Paper02_assumption_initial_condition}, we have $\| \eta^N(y) \|_{\ell_1} \leq N \| f \|_{L_\infty(\mathbb R; \ell_1)}$ $\forall \, y \in L_N^{-1}\mathbb Z$, and so
\begin{equation} \label{Paper02_lem_equic_int_iv}
\begin{aligned}
    & \mathbb E_{\boldsymbol \eta^N} \Bigg[\int_{0 \vee (-\gamma_1)}^{T \wedge (T - \gamma_1)} \left\vert P^N_t U^N_{\mathcal I}(0, \cdot)(y_1) - P^N_{t+\gamma_1} U^N_{\mathcal I}(0,\cdot)(y_2) \right\vert \, dt \Bigg] \\ & \quad \leq \int_{0 \vee (-\gamma_1)}^{T \wedge (T - \gamma_1)} \frac{1}{L_N} \sum_{y \in L_N^{-1}\mathbb Z} \| f \|_{L_{\infty}(\mathbb R; \ell_1)} \left\vert p^N(t, y - y_1) - p^N(t + \gamma_1, y-y_2)\right\vert \, dt \\ & \quad \lesssim_{T,m} \| f \|_{L_{\infty}(\mathbb R; \ell_1)} \left(\vert y_1 - y_2 \vert + \vert \gamma_1 \vert^{1/2} \right),
\end{aligned}
\end{equation}
where the second inequality follows from~\eqref{Paper02_random_walk_estimates_i} and~\eqref{Paper02_random_walk_estimates_ii} in Lemma~\ref{Paper02_random_walks_lemma} in the appendix. By~\eqref{Paper02_lem_equic_int_i}-\eqref{Paper02_lem_equic_int_iv} and Lemma~\ref{Paper02_lemma_increments_liitle_u_hat_n_k}, it follows that for $T > 0$, there exists $C^{(1)}_{T} > 0$ such that for any $N \in \mathbb N$, $\mathcal I \subseteq \mathbb N_0$, $\gamma_1 \in (-1,1)$ and $y_1, y_2 \in L_N^{-1}\mathbb Z$,
\begin{equation} \label{Paper02_lem_equic_int_v}
\begin{aligned}
    & \mathbb{E}_{\boldsymbol \eta^N}\Bigg[\int_{0}^{T} \left\vert U^{N}_{\mathcal{I}}\left(t,y_1\right) - U^{N}_{\mathcal{I}}\left(t + \gamma_{1},y_{2}\right) \cdot \mathds 1_{\{t + \gamma_1 \in [0,T]\}} \right\vert \, dt\Bigg] \\ & \quad \leq C^{(1)}_{T}\left(\left\vert \gamma_{1}\right\vert^{1/4} + \left\vert y_1 - y_2\right\vert^{1/2} + \vert y_1 - y_2 \vert\right).
\end{aligned}
\end{equation}

From now on in the proof, for $y \in \mathbb R$ we write $\lfloor y \rfloor_N \defeq L_N^{-1} \lfloor L_N y \rfloor \in L_N^{-1}\mathbb Z$. Take $\gamma_1 \in (-1,1)$ and $\gamma_2 \in \mathbb R$; for $N \in \mathbb N$, $\mathcal I \subseteq \mathbb N_0$, $t \in [0,T]$ and $x \in \mathbb R,$
\begin{equation} \label{Paper02_lem_equic_int_vi}
\begin{aligned}
    & \left\vert U^N_{\mathcal I} (t,x) - U^N_{\mathcal I}(t + \gamma_1, x + \gamma_2) \cdot \mathds 1_{\{t + \gamma_1 \in [0,T]\}}\right\vert \\ & \quad \leq \left\vert U^N_{\mathcal I} (t,x) - U^N_{\mathcal I}(t, x + \gamma_2)\right\vert + \left\vert U^N_{\mathcal I} (t,x + \gamma_2) - U^N_{\mathcal I}(t + \gamma_1, x + \gamma_2)\cdot \mathds 1_{\{t + \gamma_1 \in [0,T]\}}\right\vert.
\end{aligned}
\end{equation}
We will control the terms on the right-hand side of~\eqref{Paper02_lem_equic_int_vi} separately. For the first term, we consider the cases $\vert \gamma_2 \vert < L_N^{-1}$ and $\vert \gamma_2 \vert \geq L_N^{-1}$ separately.

First, suppose $\vert \gamma_2 \vert < L_N^{-1}$. Then we must have $x + \gamma_2 \in \left[\lfloor x \rfloor_N - L_N^{-1}, \lfloor x \rfloor_N + 2L_N^{-1} \right]$, and so by the mean value theorem,
\begin{equation*}
\begin{aligned}
    & \left\vert U^N_{\mathcal I} (t,x) - U^N_{\mathcal I}(t, x + \gamma_2)\right\vert  \\
    & \quad \leq L_N \vert \gamma_2 \vert \max_{y \in \{\lfloor x \rfloor_N - L_N^{-1}, \lfloor x \rfloor_N, \lfloor x \rfloor_N + L_N^{-1}\}} \; \left\vert U^N_{\mathcal I} (t,y + L_N^{-1}) - U^N_{\mathcal I}(t,y)\right\vert
    \\ & \quad \leq L_N \vert \gamma_2 \vert \sum_{y \in \{\lfloor x \rfloor_N - L_N^{-1}, \lfloor x \rfloor_N, \lfloor x \rfloor_N + L_N^{-1}\}} \left\vert U^N_{\mathcal I} (t,y + L_N^{-1}) - U^N_{\mathcal I}(t,y)\right\vert.  
\end{aligned}
\end{equation*}
By~\eqref{Paper02_lem_equic_int_v}, it follows that
\begin{equation} \label{Paper02_lem_equic_int_vii}
\begin{aligned}
    \mathbb{E}_{\boldsymbol \eta^N}\Bigg[\int_{0}^{T} \left\vert U^{N}_{\mathcal{I}}\left(t,x\right) - U^{N}_{\mathcal{I}}\left(t,x + \gamma_2\right) \right\vert \, dt\Bigg] & \leq 3C^{(1)}_{T} L_N \vert \gamma_2 \vert (L_N^{-1/2} + L_N^{-1}) \\ & \leq 3C^{(1)}_T \vert \gamma_2 \vert^{1/2} (1 + L_N^{-1/2}),
\end{aligned}
\end{equation}
where the second inequality follows since $\vert \gamma_2 \vert < L_N^{-1}$.

Now we consider the case $\vert \gamma_2 \vert \geq L_N^{-1}$. By the triangle inequality,
\begin{equation} \label{Paper02_lem_equic_int_viii}
\begin{aligned}
     & \left\vert U^N_{\mathcal I} (t,x) - U^N_{\mathcal I}(t, x + \gamma_2)\right\vert  \\
     & \quad \leq  \left\vert U^N_{\mathcal I} (t,x) - U^N_{\mathcal I}(t, \lfloor x \rfloor_N)\right\vert + \left\vert U^N_{\mathcal I} (t,x + \gamma_2) - U^N_{\mathcal I}(t, \lfloor x + \gamma_2\rfloor_N)\right\vert \\ 
     & \quad \qquad + \left\vert U^N_{\mathcal I} (t, \lfloor x \rfloor_N) - U^N_{\mathcal I}(t, \lfloor x + \gamma_2\rfloor_N)\right\vert.
\end{aligned}
\end{equation}
By applying~\eqref{Paper02_lem_equic_int_viii} to the integrand and using~\eqref{Paper02_lem_equic_int_vii} to bound the first two terms and~\eqref{Paper02_lem_equic_int_v} to bound the last term with the observation that $\vert \lfloor x \rfloor_N - \lfloor x + \gamma_2\rfloor_N \vert \leq \vert \gamma_2 \vert + L_N^{-1} \; \forall x \in \mathbb R$, we obtain
\begin{equation} \label{Paper02_lem_equic_int_ix}
\begin{aligned}
    & \mathbb{E}_{\boldsymbol \eta^N}\Bigg[\int_{0}^{T} \left\vert U^{N}_{\mathcal{I}}\left(t,x\right) - U^{N}_{\mathcal{I}}\left(t,x + \gamma_2\right) \right\vert \, dt\Bigg] \\ & \quad \leq 6 C^{(1)}_{T} L_N^{-1/2}(1 + L_N^{-1/2}) + C^{(1)}_{T} \left((\vert \gamma_{2} \vert + L_N^{-1})^{1/2} + \vert \gamma_{2} \vert + L_N^{-1}\right) \\ & \quad \leq (8 + \sqrt{2})C^{(1)}_{T}(1 + L_N^{-1/2})(\vert \gamma_2 \vert^{1/2} + \vert \gamma_2 \vert),
\end{aligned}
\end{equation}
where the second inequality follows since $\vert \gamma_2 \vert \geq L_N^{-1}$. 

It remains to control the second term on the right-hand side of~\eqref{Paper02_lem_equic_int_vi}. Since~$U^N_{\mathcal I}(t, \cdot)$ and~$U^N_{\mathcal I}(t + \gamma_1, \cdot)$ are defined by linear interpolation in space, we can write
\begin{equation*}
\begin{aligned}
    & \left\vert U^N_{\mathcal I}(t, x + \gamma_2) -  U^N_{\mathcal I}(t + \gamma_1, x + \gamma_2)\cdot \mathds 1_{\{t + \gamma_1 \in [0,T]\}} \right\vert \\ & \; = \Big\vert L_N(x + \gamma_2 - \lfloor x + \gamma_2 \rfloor_N)\Big(U^N_{\mathcal I}(t, \lfloor x + \gamma_2 \rfloor_N+L_N^{-1}) \\
    & \qquad \qquad \qquad \qquad \qquad \qquad \qquad -  U^N_{\mathcal I}(t + \gamma_1, \lfloor x + \gamma_2 \rfloor_N+L_N^{-1})\cdot \mathds 1_{\{t + \gamma_1 \in [0,T]\}}\Big) \\ & \qquad + L_N(\lfloor x + \gamma_2 \rfloor_N + L_N^{-1} - (x + \gamma_2))\Big(U^N_{\mathcal I}(t, \lfloor x + \gamma_2 \rfloor_N) \\ & \qquad \qquad \qquad \qquad \qquad \qquad \qquad \qquad \qquad -  U^N_{\mathcal I}(t + \gamma_1, \lfloor x + \gamma_2 \rfloor_N)\cdot \mathds 1_{\{t + \gamma_1 \in [0,T]\}}\Big) \Big\vert \\ & \quad \leq \sum_{y \in \{\lfloor x + \gamma_2 \rfloor_N, \lfloor x + \gamma_2 \rfloor_N + L_N^{-1}\}} \left\vert U^N_{\mathcal I}(t, y) - U^N_{\mathcal I}(t + \gamma_1, y)\cdot \mathds 1_{\{t + \gamma_1 \in [0,T]\}} \right\vert.
\end{aligned}
\end{equation*}
Therefore, by~\eqref{Paper02_lem_equic_int_v},
\begin{equation} \label{Paper02_lem_equic_int_x}
\begin{aligned}
    & \mathbb{E}_{\boldsymbol \eta^N}\Bigg[\int_{0}^{T} \left\vert U^{N}_{\mathcal{I}}\left(t,x + \gamma_2\right) - U^{N}_{\mathcal{I}}\left(t + \gamma_{1},x + \gamma_2\right) \cdot \mathds 1_{\{t + \gamma_1 \in [0,T]\}} \right\vert \, dt\Bigg] \\
    & \quad \leq 2C^{(1)}_T\left\vert \gamma_{1}\right\vert^{1/4}.
\end{aligned}
\end{equation}
Recall from Assumption~\ref{Paper02_scaling_parameters_assumption} that $L_N \rightarrow \infty$ as $N \rightarrow \infty$. Therefore, by combining~\eqref{Paper02_lem_equic_int_vi} with~\eqref{Paper02_lem_equic_int_vii} and~\eqref{Paper02_lem_equic_int_x} in the case $\vert \gamma_2 \vert < L_N^{-1}$ and with~\eqref{Paper02_lem_equic_int_ix} and~\eqref{Paper02_lem_equic_int_x} in the case $\vert \gamma_2 \vert \geq L_N^{-1}$, the result follows.
\end{proof}

Since the action of the reaction part of the generator of our process involves polynomials in the local number of particles (recall~\eqref{Paper02_infinitesimal_generator}), we will need a weak equicontinuity property for moments of the local density. This will be our next result.

\begin{lemma} \label{Paper02_lemma_increments_moments_densities}
    Under the conditions of Lemma~\ref{Paper02_lemma_increments_liitle_u_hat_n_k}, for $T \geq 0$ and $r\in \mathbb{N}$ there exists $C_{r,T} > 0$ such that for any $N \in \mathbb{N}$, $k \in \mathbb{N}_{0}$ and $x_{1},x_{2} \in \mathbb{R}$,
    \begin{equation*}
    \begin{aligned}
        & \int_{0}^{T} \mathbb{E}_{\boldsymbol{\eta}^N}\left[\left\vert u_{k}^{{{N}}}(t-,x_{1})\| u^{{{N}}}(t-,x_{1}) \|_{\ell_{1}}^{r}  - u_{k}^{{{N}}}(t-,x_{2})\| u^{{{N}}}(t-,x_{2}) \|_{\ell_{1}}^{r}\right\vert\right] \, dt \\ & \qquad \leq C_{r,T} \Big(\vert x_{1} - x_{2} \vert^{1/4} + \vert x_{1} - x_{2} \vert^{1/2}\Big).
    \end{aligned}
    \end{equation*}
\end{lemma}

\begin{proof}
    Observe that for~$t \in [0,T]$, 
    \begin{equation} \label{Paper02_bound_mixed_equic_i}
    \begin{aligned}
        & \left\vert u_{k}^{{{N}}}(t-,x_{1})\| u^{{{N}}}(t-,x_{1}) \|_{\ell_{1}}^{r}  - u_{k}^{{{N}}}(t-,x_{2})\| u^{{{N}}}(t-,x_{2}) \|_{\ell_{1}}^{r} \right\vert \\ & \quad = \Big\vert \| u^{{{N}}}(t-,x_{2}) \|_{\ell_{1}}^{r} \left(u^{N}_{k}(t-,x_{1}) - u^{N}_{k}(t-,x_{2})\right) \\ & \quad \quad \, + u^{N}_{k}(t-,x_{1}) \left(\| u^{{{N}}}(t-,x_{1}) \|_{\ell_{1}}^{r}  - \| u^{{{N}}}(t-,x_{2}) \|_{\ell_{1}}^{r}\right) \Big\vert \\ & \quad \leq \| u^{{{N}}}(t-,x_{2}) \|_{\ell_{1}}^{r} \left\vert u^{N}_{k}(t-,x_{1}) - u^{N}_{k}(t-,x_{2})\right\vert \\ & \quad \quad \, + u^{N}_{k}(t-,x_{1}) r \Big( \| u^N(t-,x_1) \|_{\ell_1} + \| u^N(t-,x_2) \|_{\ell_1} \Big)^{r-1} \\
        & \hspace{4cm} \cdot\left\vert\| u^{{{N}}}(t-,x_{1}) \|_{\ell_{1}}  - \| u^{{{N}}}(t-,x_{2}) \|_{\ell_{1}}\right\vert \\ & \quad \leq r \Big(\| u^{{{N}}}(t-,x_{1}) \|_{\ell_{1}} + \| u^{{{N}}}(t-,x_{2}) \|_{\ell_{1}}\Big)^{r}\\ & \quad \qquad \, \cdot \Big( \left\vert u^{N}_{k}(t-,x_{1}) - u^{N}_{k}(t-,x_{2})\right\vert + \left\vert\| u^{{{N}}}(t-,x_{1}) \|_{\ell_{1}}  - \| u^{{{N}}}(t-,x_{2}) \|_{\ell_{1}}\right\vert \Big),
    \end{aligned}
    \end{equation}
    where the first inequality follows since for any~$a,b \geq 0$ and~$r \in [1,\infty)$, by the mean value theorem,
    \begin{equation} \label{Paper02_elementary_factorising}
        \vert a^{r} - b^{r} \vert \leq r (a \vee b)^{r-1} \vert a - b\vert \leq r (a + b)^{r-1} \vert a - b\vert,
    \end{equation}
    and the second inequality follows since $u^N_k(t-,x_1) \leq \| u^N(t-,x_1) \|_{\ell_1}$. For $\mathcal I \subseteq \mathbb{N}_0$, recall the definition of $U^N_{\mathcal I}$ in~\eqref{Paper02_local_mass_subset_indices_definition}. By~\eqref{Paper02_bound_mixed_equic_i}, the proof will be complete after establishing that for $T \geq 0$ and $r \in \mathbb N$, there exists $C_{r,T}^{(1)} > 0$ such that for any $N \in \mathbb N$, $\mathcal I \subseteq \mathbb N_0$ and $x_1,x_2 \in \mathbb R$,
    \begin{equation} \label{Paper02_bound_mixed_equic_ii}
    \begin{aligned}
         & \int_{0}^{T} \mathbb{E}_{\boldsymbol{\eta}^N}\Big[\Big(\| u^{{{N}}}(t-,x_{1}) \|_{\ell_{1}} + \| u^{{{N}}}(t-,x_{2}) \|_{\ell_{1}}\Big)^{r} \left\vert U^{N}_{\mathcal{I}}(t-,x_1) - U^{N}_{\mathcal{I}}(t-,x_2) \right\vert\Big] \, dt \\ & \quad \leq C^{(1)}_{r,T} \Big(\vert x_{1} - x_{2} \vert^{1/4} + \vert x_{1} - x_{2} \vert^{1/2}\Big).
    \end{aligned}
    \end{equation}
    To bound the integrand on the left-hand side of~\eqref{Paper02_bound_mixed_equic_ii}, we observe that by the Cauchy-Schwarz inequality, for $t \in [0,T]$,
    \begin{equation} \label{Paper02_bound_mixed_equic_iii}
    \begin{aligned}
        & \mathbb{E}_{\boldsymbol{\eta}^N}\Big[\Big(\| u^{{{N}}}(t-,x_{1}) \|_{\ell_{1}} + \| u^{{{N}}}(t-,x_{2}) \|_{\ell_{1}}\Big)^{r} \left\vert U^{N}_{\mathcal{I}}(t-,x_1) - U^{N}_{\mathcal{I}}(t-,x_2) \right\vert\Big] \\ & \quad \leq \mathbb{E}_{\boldsymbol{\eta}^N}\Big[\Big(\| u^{{{N}}}(t-,x_{1}) \|_{\ell_{1}} + \| u^{{{N}}}(t-,x_{2}) \|_{\ell_{1}}\Big)^{2r} \left\vert U^{N}_{\mathcal{I}}(t-,x_1) - U^{N}_{\mathcal{I}}(t-,x_2) \right\vert\Big]^{1/2} \\ & \qquad \quad \cdot \mathbb{E}_{\boldsymbol{\eta}^N}\Big[\left\vert U^{N}_{\mathcal{I}}(t-,x_1) - U^{N}_{\mathcal{I}}(t-,x_2) \right\vert\Big]^{1/2}.
    \end{aligned}
    \end{equation}
    Note that by the definition of $u^N_k(t, \cdot)$ by linear interpolation after~\eqref{Paper02_approx_pop_density_definition}, we have that for $x \in \mathbb R$ and $\tau \geq 0$,
    \begin{equation} \label{Paper02_bound_mixed_equic_iv}
        \| u^N(\tau,x) \|_{\ell_1} \leq \| u^N(\tau,L_N^{-1} \lfloor L_Nx \rfloor) \|_{\ell_1} + \| u^N(\tau,L_N^{-1} (\lfloor L_Nx \rfloor + 1)) \|_{\ell_1}.
    \end{equation}
    Therefore, there exists $C^{(2)}_{r,T} > 0$ such that for any $N \in \mathbb N$, $\mathcal I \subseteq \mathbb N_0$, $t \in [0,T]$ and $x_1,x_2 \in \mathbb R$,
    \begin{equation} \label{Paper02_bound_mixed_equic_v}
    \begin{split}
        & \mathbb{E}_{\boldsymbol{\eta}^N}\Big[\Big(\| u^{{{N}}}(t-,x_{1}) \|_{\ell_{1}} + \| u^{{{N}}}(t-,x_{2}) \|_{\ell_{1}}\Big)^{2r} \left\vert U^{N}_{\mathcal{I}}(t-,x_1) - U^{N}_{\mathcal{I}}(t-,x_2) \right\vert\Big] \\ & \quad \leq \mathbb{E}_{\boldsymbol{\eta}^N}\Big[\Big(\| u^{{{N}}}(t-,x_{1}) \|_{\ell_{1}} + \| u^{{{N}}}(t-,x_{2}) \|_{\ell_{1}}\Big)^{2r+1}\Big] \\ & \quad \leq 2^{2r} \Big(\mathbb{E}_{\boldsymbol{\eta}^N}\Big[\| u^{{{N}}}(t-,x_{1}) \|_{\ell_{1}}^{2r+1}\Big] + \mathbb{E}_{\boldsymbol{\eta}^N}\Big[\| u^{{{N}}}(t-,x_{2}) \|_{\ell_{1}}^{2r+1}\Big]\Big) \\ & \quad \leq 2^{4r} \sum_{y \in \{L_N^{-1}\lfloor L_N x_1 \rfloor, L_N^{-1}(\lfloor L_N x_1 \rfloor + 1), L_N^{-1}\lfloor L_N x_2 \rfloor, L_N^{-1}(\lfloor L_N x_2 \rfloor + 1) \}} \mathbb E_{\boldsymbol \eta^N} \Big[\| u^{{{N}}}(t-,y) \|_{\ell_{1}}^{2r+1}\Big] \\ & \quad \leq C^{(2)}_{r,T},
    \end{split}
    \raisetag{-2.7cm}
    \end{equation}
    where the second and third inequalities follow from the elementary inequality
    \begin{equation} \label{Paper02_power_bound_elem}
    (a + b)^{2r+1} \leq 2^{2r}(a^{2r+1} + b^{2r+1}) \quad \forall \, a,b \geq 0, \; \forall \, r \in [1, \infty),
    \end{equation}
    combined with~\eqref{Paper02_bound_mixed_equic_iv} for the third inequality, and the last inequality follows from Theorem~\ref{Paper02_bound_total_mass}. Hence, combining~\eqref{Paper02_bound_mixed_equic_iii} and~\eqref{Paper02_bound_mixed_equic_v}, for $N \in \mathbb N$, $\mathcal I \subseteq \mathbb N_0$ and~$x_1,x_2 \in \mathbb R$,
    \begin{equation} \label{Paper02_bound_mixed_equic_vi}
    \begin{aligned}
        & \int_{0}^{T} \mathbb{E}_{\boldsymbol{\eta}^N}\Big[\Big(\| u^{{{N}}}(t-,x_{1}) \|_{\ell_{1}} + \| u^{{{N}}}(t-,x_{2}) \|_{\ell_{1}}\Big)^{r} \left\vert U^{N}_{\mathcal{I}}(t-,x_1) - U^{N}_{\mathcal{I}}(t-,x_2) \right\vert\Big] \, dt \\ & \quad \leq (C^{(2)}_{r,T})^{1/2} \int_0^T \mathbb E_{\boldsymbol \eta^N} \Big[\left\vert U^{N}_{\mathcal{I}}(t-,x_1) - U^{N}_{\mathcal{I}}(t-,x_2) \right\vert\Big]^{1/2} \, dt \\ & \quad \leq (C^{(2)}_{r,T})^{1/2} T^{1/2} \Bigg(\int_0^T \mathbb E_{\boldsymbol \eta^N} \Big[\left\vert U^{N}_{\mathcal{I}}(t-,x_1) - U^{N}_{\mathcal{I}}(t-,x_2) \right\vert\Big] \, dt\Bigg)^{1/2}, 
    \end{aligned}
    \end{equation}
    where the second inequality follows from Jensen's inequality. Since $(a + b)^{1/2} \leq a^{1/2} + b^{1/2}$ for all $a,b \geq 0$, estimate~\eqref{Paper02_bound_mixed_equic_ii} follows from applying Lemma~\ref{Paper02_spatial_increment_bound_usual} to~\eqref{Paper02_bound_mixed_equic_vi}. Then, by combining~\eqref{Paper02_bound_mixed_equic_i} and~\eqref{Paper02_bound_mixed_equic_ii}, the proof is complete.
\end{proof}

To finish this section, we will apply Lemma~\ref{Paper02_lemma_little_u_k_N_hat_semimartingale} to bound the expected density of particles carrying a high number of mutations.

\begin{lemma} \label{Paper02_lemma_control_local_density_high_number_mutations}
    Suppose the conditions of Lemma~\ref{Paper02_lemma_increments_liitle_u_hat_n_k} hold. For $k \in \mathbb{N}_{0}$, let $\mathcal{I}_{k} \defeq \{j \in \mathbb{N}_{0}: \; j \geq k\}$. Then, for any $T \geq 0$,
    \begin{equation*}
        \lim_{k \rightarrow \infty} \; \sup_{N \in \mathbb{N}} \; \sup_{t \in [0,T]} \; \sup_{x \in L_{N}^{-1}\mathbb{Z}} \mathbb{E}_{\boldsymbol \eta^N}\left[U^{N}_{\mathcal{I}_{k}}(t,x)\right] = 0.
    \end{equation*}
\end{lemma}

\begin{proof}
    Let~$T \geq 0$ be fixed. By taking expectations on both sides of~\eqref{Paper02_definition_little_u_k_N_hat} and then applying~\eqref{Paper02_definition_general_formula_u_hat} from Lemma~\ref{Paper02_lemma_little_u_k_N_hat_semimartingale}, we get, for all $t \in [0,T]$, $N \in \mathbb{N}$, $k \in \mathbb{N}_{0}$ and $x \in L_{N}^{-1}\mathbb{Z}$,
    \begin{equation} \label{Paper02_expectation_capital_U_N_J}
        \mathbb{E}_{\boldsymbol \eta^N}\left[U^{N}_{\mathcal{I}_{k}}(t,x)\right] = \mathbb E_{\boldsymbol \eta^N} \left[P^{N}_{t}U^{N}_{\mathcal{I}_{k}}(0,\cdot)(x)\right] + \mathbb{E}_{\boldsymbol \eta^N}\left[A^{N,t,x}_{{\mathcal{I}_{k}}}(t)\right],
    \end{equation}
    where~$A^{N,t,x}_{{\mathcal{I}_{k}}}(t)$ is given in~\eqref{Paper02_finite_variation_process_uN}. Since $f = \left(f_{k}\right)_{k \in \mathbb{N}_{0}}: \mathbb R \rightarrow \ell_1^+$ satisfies Assumption~\ref{Paper02_assumption_initial_condition}, we have that $f \in L_{\infty}(\mathbb R; \ell_1)$ and there exists a Lebesgue null set $\mathcal{N} \subset \mathbb{R}$ such that
    \begin{equation*}
        \lim_{n \rightarrow \infty} \, \sup_{x \in \mathbb{R} \setminus \mathcal{N}} \, \sum_{k \geq n} f_{k}(x) = 0.
    \end{equation*}
   Therefore, since $\boldsymbol \eta^N$ is given by~\eqref{Paper02_initial_condition}, and since by~\eqref{Paper02_approx_pop_density_definition} and~\eqref{Paper02_local_mass_subset_indices_definition}, for $k \in \mathbb N_0$ and $x \in L_N^{-1}\mathbb Z$,
   \begin{equation*}
       U^{N}_{\mathcal{I}_{k}}(0,x) = \frac{1}{N} \sum_{j \geq k} \eta^{N}_{j}(0,x),
   \end{equation*}
   we have
   \begin{equation*}
         \lim_{k \rightarrow \infty} \; \sup_{N \in \mathbb{N}} \;  \sup_{x \in L_{N}^{-1}\mathbb{Z}} \mathbb E_{\boldsymbol \eta^N} \left[U^{N}_{\mathcal{I}_{k}}(0,x) \right] = 0.
   \end{equation*}
   Hence, by the definition of the action of the semigroup $\{P^{N}_{t}\}_{t \geq 0}$ after~\eqref{Paper02_inf_gen_rw}, we have for all $T \geq 0$,
    \begin{equation*}
        \lim_{k \rightarrow \infty} \; \sup_{N \in \mathbb{N}} \; \sup_{t \in [0,T]} \; \sup_{x \in L_{N}^{-1}\mathbb{Z}} \mathbb E_{\boldsymbol{\eta}^N} \left[P^{N}_{t}U^{N}_{\mathcal{I}_{k}}(0,\cdot)(x) \right] = 0.
    \end{equation*}
    Hence, by~\eqref{Paper02_expectation_capital_U_N_J}, it will suffice to verify that for any $T \geq 0$, there exists a sequence of positive real numbers $\left(\epsilon_{k,T}\right)_{k \in \mathbb{N}_{0}}$ such that $\epsilon_{k,T} \rightarrow 0$ as $k \rightarrow \infty$, and for all $k \in \mathbb{N}_{0}$,
    \begin{equation} \label{Paper02_equivalent_condition_control_local_density_high_number_mutations}
        \sup_{N \in \mathbb{N}} \; \sup_{t \in [0,T]} \; \sup_{x \in L_{N}^{-1}\mathbb{Z}} \mathbb{E}_{\boldsymbol \eta^N}\left[A^{N,t,x}_{\mathcal{I}_{k}}(t)\right] \leq \epsilon_{k,T}. 
    \end{equation}

    Applying the expression for $A^{N,t,x}_{\mathcal{I}_{k}}(t)$ given in~\eqref{Paper02_finite_variation_process_uN} in the statement of Lemma~\ref{Paper02_lemma_little_u_k_N_hat_semimartingale} and the definition of~$F_{k}(u)$ in~\eqref{Paper02_reaction_term_PDE}, noting that~$q_{-}$ is non-negative by Assumption~\ref{Paper02_assumption_polynomials}, and recalling that by Assumption~\ref{Paper02_assumption_fitness_sequence}(iii), the sequence $(s_{k})_{k \in \mathbb{N}_{0}}$ is monotonically non-increasing, we conclude that for $t \in [0,T]$, $N \in \mathbb{N}$ and $x \in L_{N}^{-1}\mathbb{Z}$,
    \begin{equation*}
    \begin{aligned}
        & \mathbb{E}_{\boldsymbol \eta^N}\left[A^{N,t,x}_{\mathcal{I}_{k}}(t)\right]
        \\ & \quad \leq \frac{1}{L_{N}} \sum_{y \in L_{N}^{-1}\mathbb{Z}} \int_{0}^{t} p^{N}(t-\tau,y - x) \Big(s_{k} + \mathds{1}_{\{k \geq 1\}}s_{k - 1}\Big)
        \\ & \quad \quad \quad \quad \quad \quad \quad \quad \cdot \sum_{j \geq k} \mathbb{E}_{\boldsymbol \eta^N}\Big[q_{+}\left(\| u^{N}(\tau-,y) \|_{\ell_{1}}\right)\Big(u^{N}_{j}(\tau-,y)   + \mathds{1}_{\{j \geq 1\}}u^{N}_{j-1}(\tau-,y)\Big) \Big] \, d\tau \\
        & \quad \lesssim_{q_+,T} \Big(s_{k} + \mathds{1}_{\{k \geq 1\}}s_{k - 1}\Big) \frac{1}{L_{N}} \sum_{y \in L_{N}^{-1}\mathbb{Z}} \int_{0}^{t} p^{N}(t-\tau,y - x)  \, d\tau \\
        & \quad = t\Big(s_{k} + \mathds{1}_{\{k \geq 1\}}s_{k - 1}\Big),
    \end{aligned}
    \end{equation*}
    where for the second inequality we applied Theorem~\ref{Paper02_bound_total_mass}, and for the last identity we used~\eqref{Paper02_simple_identity_total_sum_pN}. Therefore, since by Assumption~\ref{Paper02_assumption_fitness_sequence}(iv) we have $s_{k} \rightarrow 0$ as $k \rightarrow \infty$,~\eqref{Paper02_equivalent_condition_control_local_density_high_number_mutations} follows, which completes the proof.
\end{proof}

\section{Tightness} \label{Paper02_tightness}

Under Assumptions~\ref{Paper02_assumption_fitness_sequence},~\ref{Paper02_assumption_polynomials} and~\ref{Paper02_assumption_initial_condition}, for $N \in \mathbb N$, defining $\boldsymbol \eta^N$ as in~\eqref{Paper02_initial_condition}, let~$(\eta^{N}(t))_{t \geq 0}$ with $\eta^N(0) = \boldsymbol \eta^N$ denote the $\mathcal S^N$-valued càdlàg Markov process formally defined in Theorem~\ref{Paper02_thm_existence_uniqueness_IPS_spatial_muller}, and recall the stochastic process~$(u^{N}(t))_{t \geq 0}$ defined in~\eqref{Paper02_approx_pop_density_definition} and~\eqref{Paper02_definition_u_N_as_radon_measure}. Moreover, recall that as stated after~\eqref{Paper02_metric_product_topology}, $(\mathcal M(\mathbb R)^{\mathbb N_0}, d)$ is a complete and separable metric space. Therefore, the space $\mathcal D\left([0, \infty); (\mathcal M(\mathbb R)^{\mathbb N_0}, d)\right)$ is also a complete and separable metric space when equipped with the Skorokhod metric, which generates the~${J}_{1}$-topology (see e.g.~\cite[Theorem~3.5.6]{ethier2009markov}). The goal of this section is to verify the following property of~$(u^{N})_{N \in \mathbb{N}}$.

\begin{proposition} \label{Paper02_PDE_tightness_measures}
Under the assumptions of Theorem~\ref{Paper02_deterministic_scaling_foutel_rodier_etheridge}, the sequence of processes $(u^{N})_{N \in \mathbb{N}}$ is tight in $\mathcal{D}\left([0, \infty); (\mathcal{M}(\mathbb{R})^{\mathbb{N}_{0}}, d)\right)$.
\end{proposition}

To prove Proposition~\ref{Paper02_PDE_tightness_measures}, our strategy will be to apply Jakubowski's criterion \cite[Theorem~3.1]{jakubowski1986skorokhod}.  According to this criterion, tightness of~$(u^{N})_{N \in \mathbb{N}}$ will follow from Propositions~\ref{Paper02_compact_containment_condition} and~\ref{Paper02_tightness_evaluations} below.

\begin{proposition}[Compact containment condition] \label{Paper02_compact_containment_condition}
Under the assumptions of Theorem~\ref{Paper02_deterministic_scaling_foutel_rodier_etheridge}, for any $T \geq 0$ and $\varepsilon > 0$, there exists a compact set $\mathcal{K}_{\varepsilon, T} \subset \mathcal{M}(\mathbb{R})^{\mathbb{N}_{0}}$ such that
\begin{equation*}
    \inf_{N \in \mathbb N} \mathbb{P}_{\boldsymbol \eta^N}\left(u^{N}(t) \in \mathcal{K}_{\varepsilon, T} \; \forall t \in [0,T]\right) \geq 1 - \varepsilon.
\end{equation*}
\end{proposition}

\begin{proposition}[Tightness of evaluations] \label{Paper02_tightness_evaluations} Under the assumptions of Theorem~\ref{Paper02_deterministic_scaling_foutel_rodier_etheridge}, there exists $\mathbb{F} \subset \mathcal{C}(\mathcal{M}(\mathbb{R})^{\mathbb{N}_{0}}; \mathbb{R})$, i.e.~a family of real-valued continuous functions on $\mathcal{M}(\mathbb{R})^{\mathbb{N}_{0}}$, satisfying the following conditions:
\begin{enumerate}[label = (\roman*)]
    \item $\mathbb{F}$ separates points in $\mathcal{M}(\mathbb{R})^{\mathbb{N}_{0}}$, i.e.~for any $u,v \in \mathcal{M}(\mathbb{R})^{\mathbb{N}_{0}}$ with~$u \neq v$, there exists $\varphi \in \mathbb{F}$ such that $\varphi(u) \neq \varphi(v)$;
    \item $\mathbb{F}$ is closed under addition, i.e.~if $\varphi, \tilde{\varphi} \in \mathbb{F}$, then $\varphi + \tilde{\varphi} \in \mathbb{F}$;
\end{enumerate}
and such that the sequence of processes $(u^{N})_{N \in \mathbb N}$ is $\mathbb{F}$-weakly tight, i.e.~for any $\varphi \in \mathbb{F}$, the sequence $\Big((\varphi(u^{N}(t)))_{t \geq 0}\Big)_{N \in \mathbb{N}}$ is tight in $\mathcal{D}\left([0,\infty); \mathbb{R}\right)$.
\end{proposition}

We start by proving Proposition~\ref{Paper02_compact_containment_condition}. A challenge for proving this result is the fact that the number of particles per deme is not bounded \textit{a priori}. In order to average the process over compact spatial domains, we analyse the inner product of the approximate population densities with functions $\varphi \in \mathcal C^2_c(\mathbb R)$, i.e.~with twice differentiable functions with compact support. Recall that in~\eqref{Paper02_local_mass_subset_indices_definition}, we let $U^{N}_{\mathcal{I}}$ denote the approximate population density of particles whose number of mutations belong to the set of indices $\mathcal{I} \subseteq \mathbb{N}_{0}$.
Analogously to the definition of $g^{N,T,x}_{\mathcal I}(t, \cdot)$ in~\eqref{Paper02_green_function_representation_writing_as_lipschitz_map}, for~$N \in \mathbb N$,~$\varphi \in \mathcal{C}^{2}_{c}(\mathbb{R})$ and $\mathcal I \subseteq \mathbb N_0$, we define the map~$g^{N,\varphi}_{\mathcal{I}}: \mathcal{S}^{N} \rightarrow \mathbb{R}$ by letting, for~$\boldsymbol{\xi} = (\xi_k(x))_{k \in \mathbb N_0, \, x \in L_N^{-1}\mathbb Z} \in \mathcal{S}^{N}$,
\begin{equation} \label{Paper02_formulation_inner_product_with_test_function_lipschitz_continuous_i}
    g^{N,\varphi}_{\mathcal{I}}(\boldsymbol{\xi}) \defeq \frac{1}{L_{N}} \sum_{x \in L_{N}^{-1}\mathbb{Z}} \; \sum_{k \in \mathcal{I}} \frac{\xi_{k}(x)}{N} \int_{-1}^{1} (1 - \vert h \vert) \varphi(x + h L_{N}^{-1}) \, dh.
\end{equation}
To simplify notation, for $N \in \mathbb N$ and $\varphi \in \mathcal C^2_c(\mathbb R)$, define $J_{L_N}(\varphi): \mathbb R \rightarrow \mathbb R$ by letting
\begin{equation} \label{Paper02_definition_discrete_continuous_measure}
    (J_{L_N}(\varphi))(x) \defeq \int_{-1}^{1} (1 - \vert h \vert) \varphi(x + h L_{N}^{-1}) \, dh \quad \forall \, x \in \mathbb R.
\end{equation}
Recall from~\eqref{Paper02_local_mass_subset_indices_definition} and~\eqref{Paper02_approx_pop_density_definition} that ${U}^{N}_{\mathcal{I}}(T,\cdot)$ is given by the linear interpolation of $(U^N_\mathcal I(T,x))_{x \in L_N^{-1}\mathbb Z}$. Therefore, we have for $N \in \mathbb N$, $\mathcal I \subseteq \mathbb N_0$, $T \geq 0$ and~$\varphi \in \mathcal{C}^{2}_{c}(\mathbb{R})$, recalling the definition of $\langle \cdot, \cdot \rangle_N$ in~\eqref{Paper02_definition_discrete_integral_over_space},
\begin{equation} \label{Paper02_formulation_inner_product_with_test_function_lipschitz_continuous_ii}
    g^{N,\varphi}_{\mathcal{I}}(\eta^{N}(T)) = \left\langle U^{N}_{\mathcal{I}}(T,\cdot), J_{L_N}(\varphi) \right\rangle_N = \left\langle U^{N}_{\mathcal{I}}(T,\cdot), \varphi \right\rangle \defeq \int_{\mathbb{R}} U^{N}_{\mathcal{I}}(T,x) \varphi(x) \, dx.
\end{equation}
We will first verify that the map~$g^{N,\varphi}_{\mathcal{I}}$ belongs to the domain of the infinitesimal generator~$\mathcal{L}^{N}$ defined in~\eqref{Paper02_generator_foutel_etheridge_model}. For this, it will be convenient to introduce the discrete gradient $\nabla_{L_N}$ and the discrete Laplacian~$\Laplace_{L_N}$ for any $\varphi: \mathbb{R} \rightarrow \mathbb{R}$ and $x \in \mathbb{R}$ by letting
\begin{equation} \label{Paper02_discrete_gradient_discrete_laplacian}
\left\{\begin{array}{l}
     \nabla_{L_N} \varphi(x) \defeq L_{N}\left(\varphi(x+L_{N}^{-1}) - \varphi(x)\right),  \\[+1mm]
     \Laplace_{L_N} \varphi(x) \defeq L_{N}^{2}\left(\varphi(x+L_{N}^{-1}) + \varphi(x-L_{N}^{-1}) - 2\varphi(x)\right).
\end{array} \right.
\end{equation}
(Recall the analogous definition of $\nabla_{L_N}p^N$ in~\eqref{Paper02_discrete_grad}).
We highlight that the discrete gradient and the discrete Laplacian are defined in~\eqref{Paper02_discrete_gradient_discrete_laplacian} for all~$x \in \mathbb{R}$, not only for $x \in L_{N}^{-1}\mathbb{Z}$. Recall the definition of the reaction term~$F = (F_{k})_{k \in \mathbb{N}_{0}}$ given in~\eqref{Paper02_reaction_term_PDE}, and the modified reaction term~$F^+ = (F^{+}_{k})_{k \in \mathbb{N}_{0}}$ given in~\eqref{Paper02_reaction_predictable_bracket_process}, and the definition of~$\mathcal{C}_*(\mathcal{S}^{N}; \mathbb{R})$ in~\eqref{Paper02_definition_Lipschitz_function}.

\begin{lemma} \label{Paper02_well_behaved_function_inner_product_test_funciton_lemma}
    Under the assumptions of Theorem~\ref{Paper02_deterministic_scaling_foutel_rodier_etheridge}, for any~$N \in \mathbb{N}$,~$\varphi \in \mathcal{C}^{2}_{c}(\mathbb{R})$ and~$\mathcal{I} \subseteq \mathbb{N}_{0}$, the map~$g^{N,\varphi}_{\mathcal{I}}: \mathcal{S}^{N} \rightarrow \mathbb{R}$ given in~\eqref{Paper02_formulation_inner_product_with_test_function_lipschitz_continuous_i} satisfies the following conditions:
    \begin{enumerate}[label = (\roman*)]
        \item $g^{N,\varphi}_{\mathcal{I}} \in \mathcal{C}_*(\mathcal{S}^{N}; \mathbb{R})$.
        \item The map~$\mathcal S^N \ni \boldsymbol{\xi} \mapsto \Big(\mathcal{L}^{N}g^{N,\varphi}_{\mathcal{I}}\Big)(\boldsymbol{\xi})$ is in $\mathcal C(\mathcal S^N; \mathbb R)$, and is given by, for $\boldsymbol{\xi} = (\xi_k(x))_{k \in \mathbb N_0, \, x \in L_N^{-1}\mathbb Z} \in \mathcal{S}^{N}$,
        \begin{equation} \label{Paper02_action_generator_inner_product_smooth_functions}
        \begin{aligned}
            \Big(\mathcal{L}^{N}g^{N,\varphi}_{\mathcal{I}}\Big)(\boldsymbol{\xi}) & \defeq \frac{m_{N}}{2L_{N}^{3}} \sum_{x \in L_{N}^{-1}\mathbb{Z}} \left(J_{L_N}(\Laplace_{L_{N}} \varphi)\right)(x) \sum_{k \in \mathcal{I}} \frac{\xi_{k}(x)}{N} \\
            & \qquad \; + \frac{1}{L_{N}} \sum_{x \in L_{N}^{-1}\mathbb{Z}} \left(J_{L_N}(\varphi)\right)(x) \sum_{k \in \mathcal{I}} F_{k}\left(\frac{\xi(x)}{N}\right).
        \end{aligned}
        \end{equation}

    \item The map~$\mathcal S^N \ni \boldsymbol{\xi} \mapsto \Big(\mathcal{L}^{N}(g^{N,\varphi}_{\mathcal{I}})^{2}\Big)(\boldsymbol{\xi})$ is in $\mathcal{C}(\mathcal{S}^{N};\mathbb{R})$, and is given by, for~$\boldsymbol{\xi} = (\xi_k(x))_{k \in \mathbb N_0, \, x \in L_N^{-1}\mathbb Z}\in \mathcal{S}^{N}$,
    \begin{align} 
         & \Big(\mathcal{L}^{N}(g^{N,\varphi}_{\mathcal{I}})^{2}\Big)(\boldsymbol{\xi}) \notag \\ & \quad = 2 g^{N,\varphi}_{\mathcal{I}}(\boldsymbol{\xi}) \Big(\mathcal{L}^{N}g^{N,\varphi}_{\mathcal{I}}\Big)(\boldsymbol{\xi}) \notag \\ & \quad \quad \quad + \frac{m_{N}}{2NL_{N}^{4}} \sum_{x \in L_{N}^{-1}\mathbb{Z}} \left((J_{L_N}(\nabla_{L_{N}} \varphi))^2(x) + (J_{L_N}(\nabla_{L_{N}} \varphi))^2(x - L_N^{-1}) \right)\sum_{k \in \mathcal{I}} \frac{\xi_{k}(x)}{N} \label{Paper02_action_generator_square_inner_product_test_function}
         \\ 
        & \quad \quad \quad + \frac{1}{NL_{N}^{2}} \sum_{x \in L_{N}^{-1}\mathbb{Z}} (J_{L_N}(\varphi))^2(x) \sum_{k \in \mathcal{I}} F^{+}_{k}\left(\frac{\xi(x)}{N}\right). \notag
    \end{align}

    \item For any~$T \geq 0$ and $r \geq 1$, the following estimate holds: 
     \begin{align}
            & \sup_{N \in \mathbb{N}} \; \sup_{t \in [0,T]} \; \mathbb{E}_{\boldsymbol \eta^N}\Big[\vert g^{N,\varphi}_{\mathcal{I}}(\eta^{N}(t)) \vert^{r} + \Big\vert \Big(\mathcal{L}^{N}g^{N,\varphi}_{\mathcal{I}}\Big)(\eta^{N}(t))\Big\vert^{r} + \Big\vert \Big(\mathcal{L}^{N}(g^{N,\varphi}_{\mathcal{I}})^{2}\Big)(\eta^{N}(t)) \Big\vert^r\Big] \notag \\
            & \hspace{4cm} < \infty. \label{Paper02_general_estimates_moments_test_function_and_derivatives_generator}
        \end{align}

        \item The process~$(M(T))_{T \geq 0}$ given by, for~$T \geq 0$,
        \begin{equation*}
            M(T) \defeq \Big(g^{N,\varphi}_{\mathcal{I}}(\eta^{N}(T))\Big)^{2} - \Big(g^{N,\varphi}_{\mathcal{I}}(\eta^{N}(0))\Big)^{2} - \int_{0}^{T}  \Big(\mathcal{L}^{N}(g^{N,\varphi}_{\mathcal{I}})^{2}\Big)(\eta^{N}(t-)) \, dt
        \end{equation*}
        is a càdlàg martingale with respect to the filtration~$\{\mathcal{F}^{\eta^{N}}_{t+}\}_{t \geq 0}$.
    \end{enumerate}
\end{lemma}

Similarly to the proof of Lemma~\ref{Paper02_properties_greens_function_muller_ratchet}, the proof of Lemma~\ref{Paper02_well_behaved_function_inner_product_test_funciton_lemma} follows from standard arguments, which we postpone to Section~\ref{Paper02_appendix_section_proof_auxiliary_no_technical_lemmas} in the appendix. We now use Lemma~\ref{Paper02_well_behaved_function_inner_product_test_funciton_lemma} to write~$\left\langle U^{N}_{\mathcal{I}}(T,\cdot), \varphi \right\rangle$ as a semimartingale. Recall the definition of $J_{L_N}(\varphi)$ in~\eqref{Paper02_definition_discrete_continuous_measure}.

\begin{lemma} \label{Paper02_action_measure_formulation_semimartinagale}
Under the assumptions of Theorem~\ref{Paper02_deterministic_scaling_foutel_rodier_etheridge}, for any $N \in \mathbb{N}$, $\mathcal{I} \subseteq \mathbb{N}_{0}$ and $\varphi \in \mathcal{C}^{2}_{c}(\mathbb{R})$, there exist a càdlàg square integrable martingale $M^{N,\varphi}_{\mathcal{I}}$ with respect to the filtration~$\{\mathcal{F}^{\eta^{N}}_{t+}\}_{t \geq 0}$ with $M^{N,\varphi}_{\mathcal I}(0) = 0$, and a finite variation process $A^{N,\varphi}_{\mathcal{I}}$ such that for $T \geq 0$,
\begin{equation*}
    \left\langle U^{N}_{\mathcal{I}}(T,\cdot), \varphi \right\rangle = \left\langle U^{N}_{\mathcal{I}}(0,\cdot), \varphi \right\rangle + M^{N,\varphi}_{\mathcal{I}}(T) + A^{N,\varphi}_{\mathcal{I}}(T),
\end{equation*}
where for~$T \geq 0$,
\begin{equation} \label{Paper02_lemma_semimartingale_finite_variation_process}
\begin{aligned}
     A^{N,\varphi}_{\mathcal{I}}(T) & = \frac{1}{L_{N}} \sum_{x \in L_{N}^{-1}\mathbb{Z}}  \left(J_{L_N}(\Laplace_{L_{N}} \varphi)\right)(x) \int_{0}^{T} \frac{m_{N}}{2L_{N}^{2}} U^{N}_{\mathcal{I}}(t-,x) \; dt \\
     & \quad + \frac{1}{L_{N}} \sum_{x \in L_{N}^{-1}\mathbb{Z}}  \left(J_{L_N}( \varphi)\right)(x) \sum_{k \in \mathcal{I}} \int_{0}^{T} F_{k}(u^{N}(t-, x)) \; dt,
\end{aligned}
\end{equation}
and the predictable bracket process of $M^{N,\varphi}_{\mathcal{I}}$ is given by, for $T \geq 0$,
\begin{equation} \label{Paper02_lemma_semimartingale_quadratic_variation_process}
\begin{aligned}
    & \left\langle M^{N,\varphi}_{\mathcal{I}}\right\rangle (T) \\ & \quad =  \frac{m_{N}}{2NL_{N}^{4}} \sum_{x \in L_{N}^{-1}\mathbb{Z}} \left( \left(J_{L_N}(\nabla_{L_{N}} \varphi)\right)^2(x) + \left(J_{L_N}(\nabla_{L_{N}} \varphi)\right)^2(x - L_N^{-1})\right) \int_{0}^{T} U^{N}_{\mathcal{I}}(t-,x) \, dt  \\
    & \quad \quad + \frac{1}{NL_{N}^{2}} \sum_{x \in L_{N}^{-1}\mathbb{Z}} \left(J_{L_N}( \varphi)\right)^2(x)  \sum_{k \in \mathcal{I}} \int_{0}^{T} F^{+}_{k} (u^{N}(t-,x)) \, dt.
\end{aligned}
\end{equation}
\end{lemma}

\begin{proof}
The result will follow from Lemma~\ref{Paper02_well_behaved_function_inner_product_test_funciton_lemma} and an application of Lemmas~\ref{Paper02_general_integration_by_parts_formula} and~\ref{Paper02_integration_by_parts_time_predictable_bracket_process} in the appendix, and Theorem~\ref{Paper02_thm_existence_uniqueness_IPS_spatial_muller}. To see this, we first observe that combining Lemma~\ref{Paper02_well_behaved_function_inner_product_test_funciton_lemma}(i) with~Theorem~\ref{Paper02_thm_existence_uniqueness_IPS_spatial_muller}, and then using~\eqref{Paper02_formulation_inner_product_with_test_function_lipschitz_continuous_ii}, we conclude that for any~$\mathcal{I} \subseteq \mathbb{N}_{0}$,~$N \in \mathbb{N}$ and~$\varphi \in \mathcal{C}^{2}_{c}(\mathbb{R})$, the process $(M^{N,\varphi}_{\mathcal{I}}(t))_{t \geq 0}$ given by, for~$T \geq 0$
\begin{equation} \label{Paper02_initial_definition_martingale_inner_product}
\begin{aligned}
    M^{N,\varphi}_{\mathcal{I}}(T) \defeq \left\langle U^{N}_{\mathcal{I}}(T,\cdot), \varphi \right\rangle - \left\langle U^{N}_{\mathcal{I}}(0,\cdot), \varphi \right\rangle - \int_{0}^{T} \Big(\mathcal{L}^{N}g^{N,\varphi}_{\mathcal{I}}\Big)(\eta^{N}(t-)) \, dt
\end{aligned}
\end{equation}
is a càdlàg martingale with respect to the filtration~$\{\mathcal{F}^{{\eta}^{N}}_{t+}\}_{t \geq 0}$. Then, defining~$(A^{N,\varphi}_{\mathcal{I}}(t))_{t \geq 0}$ by, for $T \geq 0$,
\begin{equation*}
    A^{N,\varphi}_{\mathcal{I}}(T) \defeq \int_{0}^{T} \Big(\mathcal{L}^{N}g^{N,\varphi}_{\mathcal{I}}\Big)(\eta^{N}(t-)) \, dt,
\end{equation*}
Lemma~\ref{Paper02_well_behaved_function_inner_product_test_funciton_lemma}(ii) and~\eqref{Paper02_initial_definition_martingale_inner_product} imply~\eqref{Paper02_lemma_semimartingale_finite_variation_process}.

It remains to compute the predictable bracket process of~$M^{N,\varphi}_{\mathcal{I}}$. Note that since the map~$g^{N,\varphi}_{\mathcal{I}}: \mathcal{S}^{N} \rightarrow \mathbb{R}$ does not depend on time, we have
\begin{equation} \label{Paper02_derivative_time_inner_product_smooth_function}
    \frac{\partial}{\partial t} g^{N,\varphi}_{\mathcal{I}}(\boldsymbol{\xi}) \equiv 0 \quad \forall \, \boldsymbol{\xi} \in \mathcal{S}^{N}.
\end{equation}
We claim that~$g^{N,\varphi}_{\mathcal{I}}$ satisfies the conditions of Lemmas~\ref{Paper02_general_integration_by_parts_formula} and~\ref{Paper02_integration_by_parts_time_predictable_bracket_process} in the appendix. Indeed, checking the conditions of Lemma~\ref{Paper02_general_integration_by_parts_formula}, conditions~(i) and~(iv) follow from Lemma~\ref{Paper02_well_behaved_function_inner_product_test_funciton_lemma}(iv) and~\eqref{Paper02_derivative_time_inner_product_smooth_function}, condition~(ii) follows from~\eqref{Paper02_derivative_time_inner_product_smooth_function}, and condition~(iii) follows from Lemma~\ref{Paper02_well_behaved_function_inner_product_test_funciton_lemma}(i) and~(ii). Checking the conditions of Lemma~\ref{Paper02_integration_by_parts_time_predictable_bracket_process}, condition~(i) follows from Lemma~\ref{Paper02_well_behaved_function_inner_product_test_funciton_lemma}(iii), and condition~(ii) follows from Lemma~\ref{Paper02_well_behaved_function_inner_product_test_funciton_lemma}(iv) and~\eqref{Paper02_derivative_time_inner_product_smooth_function}. Therefore, Lemma~\ref{Paper02_integration_by_parts_time_predictable_bracket_process} implies that $M^{N,\varphi}_{\mathcal I}$ is square integrable, and for~$T \geq 0$,
\begin{equation} \label{Paper02_pre_formula_predictable_bracket_process}
    \left\langle M^{N,\varphi}_{\mathcal{I}}\right\rangle(T) = \int_{0}^{T} \left(\Big(\mathcal{L}^{N}(g^{N,\varphi}_{\mathcal{I}})^{2}\Big)(\eta^{N}(t-)) - 2g^{N,\varphi}_{\mathcal{I}}(\eta^{N}(t-))\Big(\mathcal{L}^{N}g^{N,\varphi}_{\mathcal{I}}\Big)(\eta^{N}(t-))\right) \, dt.
\end{equation}
Hence, applying~\eqref{Paper02_action_generator_square_inner_product_test_function} to~\eqref{Paper02_pre_formula_predictable_bracket_process}, we get~\eqref{Paper02_lemma_semimartingale_quadratic_variation_process}, which completes the proof.
\end{proof}

We now use Lemma~\ref{Paper02_action_measure_formulation_semimartinagale} to derive a uniform in time bound (over a compact time interval) on the total mass of particles in a compact interval.

\begin{lemma} \label{Paper02_uniform_time_control_mass_compact_intervals}
    Under the assumptions of Theorem~\ref{Paper02_deterministic_scaling_foutel_rodier_etheridge}, for any $R > 0$ and $T \geq 0$, there exists $C_{R,T} > 0$ such that
    \begin{equation*}
        \sup_{N \in \mathbb{N}} \; \mathbb{E}_{\boldsymbol \eta^N}\left[\sup_{t \in [0,T]} \, \int_{-R}^{R} \| u^{N}(t, x) \|_{\ell_{1}} \, dx\right] \leq C_{R,T}.
    \end{equation*}
\end{lemma}

\begin{proof}
    Fix $\varphi \in \mathcal{C}^{\infty}_{c}(\mathbb{R})$ such that $\varphi(x) \geq 0$ for all $x \in \mathbb{R}$ and $\varphi(x) \geq 1$ for all $x \in [-R, R]$. Then recalling from~\eqref{Paper02_local_mass_subset_indices_definition} that for $t \in [0,T]$ and $x \in \mathbb R$ we have $U^{N}_{\mathbb{N}_{0}}(t,x) = \sum_{k \in \mathbb N_0} u^{N}_{k}(t,x) = \| u^{N}(t,x) \|_{\ell_{1}}$, and recalling~\eqref{Paper02_formulation_inner_product_with_test_function_lipschitz_continuous_ii}, we have that almost surely for all $N \in \mathbb{N}$,
    \begin{equation} \label{Paper02_bound_supremum_mass_compact_interval_by_inner_product_test_function}
        \sup_{t \in [0,T]} \int_{-R}^{R} \| u^{N}(t, x) \|_{\ell_{1}} \, dx \leq \sup_{t \in [0,T]} \left\langle U^{N}_{\mathbb{N}_{0}}(t, \cdot), \varphi\right\rangle.
    \end{equation}
    By Lemma~\ref{Paper02_action_measure_formulation_semimartinagale}, we can write
    \begin{equation} \label{Paper02_intermediate_step_tightness_bound_mass_compact_intervals}
    \begin{aligned}
        \sup_{t \in [0,T]} \left\langle U^{N}_{\mathbb{N}_{0}}(t, \cdot), \varphi\right\rangle \leq \left\langle U^{N}_{\mathbb{N}_{0}}(0, \cdot), \varphi\right\rangle + \sup_{t \in [0,T]} M^{N, \varphi}_{\mathbb{N}_{0}} (t) + \sup_{t \in [0,T]} A^{N, \varphi}_{\mathbb{N}_{0}}(t).
    \end{aligned}
    \end{equation}
    We will bound each term on the right-hand side of~\eqref{Paper02_intermediate_step_tightness_bound_mass_compact_intervals} separately. For the first term, we recall that, by Assumption~\ref{Paper02_assumption_initial_condition}, the function $f = \left(f_{k}\right)_{k \in \mathbb{N}_{0}}: \mathbb{R} \rightarrow \ell_{1}^+$ that determines the initial condition~$\boldsymbol \eta^{N}$ in~\eqref{Paper02_initial_condition} satisfies~$f \in L_{\infty}(\mathbb{R}; \,\ell_{1})$, and so
    \begin{equation} \label{Paper02_bound_mass_compact_intervals_initial_condition}
    \begin{aligned}
        \sup_{N \in \mathbb{N}} \mathbb E_{\boldsymbol \eta^N} \left[\left\langle U^{N}_{\mathbb{N}_{0}}(0, \cdot), \varphi\right\rangle\right] = \sup_{N \in \mathbb{N}} \mathbb E_{\boldsymbol \eta^N} \left[\int_{\mathbb{R}} U^{N}_{\mathbb{N}_{0}}(0, x) \varphi(x) \, dx\right] \leq \| f \|_{L_{\infty}(\mathbb{R}; \ell_{1})} \| \varphi \|_{L_{1}(\mathbb{R})}.
    \end{aligned}
    \end{equation}
    To bound the finite variation term on the right-hand side of~\eqref{Paper02_intermediate_step_tightness_bound_mass_compact_intervals}, we observe that since $\varphi$ is smooth, for all~$N \in \mathbb  N$ and $x \in \mathbb{R}$,~\eqref{Paper02_discrete_gradient_discrete_laplacian} yields
    \begin{equation*}
    \begin{aligned}
        \Laplace_{L_{N}} \varphi (x) & = L_{N}^{2}\Big((\varphi(x + L_{N}^{-1}) - \varphi(x)) - (\varphi(x) - \varphi(x - L_{N}^{-1}))\Big) \\ & = L_{N}^{2}\int_{x}^{x + L_{N}^{-1}} (\varphi'(y) - \varphi'(y - L_{N}^{-1})) \, dy \\ & =  L_{N}^{2}\int_{x}^{x + L_{N}^{-1}} \int_{y - L_{N}^{-1}}^{y} \varphi''(z) \, dz \, dy.
    \end{aligned}
    \end{equation*}
    Therefore, for all~$N \in \mathbb N$ and~$x \in \mathbb{R}$,
    \begin{equation} \label{Paper02_bound_discrete_laplacian}
        \vert    \Laplace_{L_{N}} \varphi (x) \vert \leq \| \varphi'' \|_{L_{\infty}(\mathbb{R})}.
    \end{equation}
    Recall from Assumption~\ref{Paper02_scaling_parameters_assumption}(ii) that $L_N \rightarrow \infty$ as $N \rightarrow \infty$. Let $R_{\varphi} > 0$ be sufficiently large that $\supp (\varphi) \subset [-R_{\varphi} + 2L_N^{-1}, R_{\varphi} - 2L_N^{-1}]$ for every~$N \in \mathbb{N}$. Then, applying~\eqref{Paper02_bound_discrete_laplacian} to~\eqref{Paper02_lemma_semimartingale_finite_variation_process}, recalling~\eqref{Paper02_definition_discrete_continuous_measure} and using that by~\eqref{Paper02_reaction_predictable_bracket_process} and~\eqref{Paper02_reaction_term_PDE}, the reaction term  $F^+ = (F^+_k)_{k \in \mathbb N_0}$ satisfies $F^{+}_{k}(u) \geq 0 \vee F_{k}(u)$ for all $u \in \ell_{1}^{+}$ and $k \in \mathbb{N}_{0}$, we have that for $N \in \mathbb N$,
    \begin{equation*}
    \begin{aligned}
        \mathbb{E}_{\boldsymbol \eta^N}\bigg[\sup_{t \in [0,T]} A^{N, \varphi}_{\mathbb{N}_{0}}(t)\bigg] & \leq  \frac{1}{L_{N}} \|\varphi''\|_{L_\infty(\mathbb R)} \sum_{x \in L_{N}^{-1}\mathbb{Z} \cap [-R_{\varphi}, R_{\varphi}]} \int_{0}^{T} \frac{m_{N}}{2L_{N}^{2}} \mathbb{E}_{\boldsymbol{\eta}^N}\left[\| u^{N}(t-, x) \|_{\ell_{1}}\right] \, dt \\ &  + \frac{1}{L_{N}} \|\varphi\|_{L_\infty(\mathbb R)} \sum_{x \in L_{N}^{-1}\mathbb{Z} \cap [-R_{\varphi}, R_{\varphi}]}\int_{0}^{T} \mathbb{E}_{\boldsymbol \eta^N}\left[ \sum_{k \in \mathbb N_0} F^{+}_k (u^{N}(t-, x)) \right] \, dt.
    \end{aligned}
    \end{equation*}
    Hence, by Theorem~\ref{Paper02_bound_total_mass} and~\eqref{Paper02_good_behaviour_indices_modified_reaction_term} in the proof of Lemma~\ref{Paper02_lemma_increments_liitle_u_hat_n_k}, and using Assumption~\ref{Paper02_scaling_parameters_assumption}$(i)$, we conclude that there exists $C^{(1)}_{R,T,\varphi} > 0$ such that
    \begin{equation} \label{Paper02_bound_finite_variation_term}
        \sup_{N \in \mathbb N}   \, \mathbb{E}_{\boldsymbol \eta^N}\left[\sup_{t \in [0,T]} A^{N, \varphi}_{\mathbb{N}_{0}}(t)\right] \leq C^{(1)}_{R,T,\varphi}.
    \end{equation}
    Now, to bound the supremum of the martingale term on the right-hand side of~\eqref{Paper02_intermediate_step_tightness_bound_mass_compact_intervals}, we first observe that by the smoothness of~$\varphi$ and by~\eqref{Paper02_discrete_gradient_discrete_laplacian}, we have for all~$N \in \mathbb N$ and~$x \in \mathbb{R}$,
    \begin{equation*}
    \begin{aligned}
       \nabla_{L_{N}} \varphi (x) = L_{N}\Big(\varphi(x + L_{N}^{-1}) - \varphi(x)\Big) = L_{N}\int_{x}^{x + L_{N}^{-1}} \varphi'(y) \, dy,
    \end{aligned}
    \end{equation*}
    and so for all $N \in \mathbb N$ and $x \in \mathbb R$,
    \begin{equation} \label{Paper02_bound_discrete_gradient}
        \vert \nabla_{L_{N}} \varphi (x) \vert \leq \| \varphi' \|_{L_{\infty}(\mathbb{R})}.
    \end{equation}
    Taking expectations on both sides of~\eqref{Paper02_lemma_semimartingale_quadratic_variation_process} and applying~\eqref{Paper02_bound_discrete_gradient} and~\eqref{Paper02_definition_discrete_continuous_measure}, we conclude that for $N \in \mathbb N$,
    \begin{equation*}
    \begin{aligned}
        \mathbb{E}_{\boldsymbol \eta^N} & \left[ \left\langle M^{N, \varphi}_{\mathbb{N}_{0}}\right\rangle (T)\right] \\ & \leq  \frac{m_{N}}{NL_{N}^{4}} \|\varphi'\|_{L_\infty(\mathbb R)}^2\sum_{x \in L_{N}^{-1}\mathbb{Z} \cap[-R_{\varphi}, R_{\varphi}]} \int_{0}^{T} \mathbb{E}_{\boldsymbol \eta^N}\left[\| u^{N}(t-, x) \|_{\ell_{1}}\right] \, dt \\ & \quad + \frac{1}{NL_{N}^{2}} \|\varphi\|_{L_\infty(\mathbb R)}^2\sum_{x \in L_{N}^{-1}\mathbb{Z} \cap [-R_{\varphi}, R_{\varphi}]} \int_{0}^{T} \mathbb{E}_{\boldsymbol \eta^N}\left[\sum_{k \in \mathbb N_0}  F^{+}_{k}(u^{N}(t-, x) )\right] \, dt.
    \end{aligned}
    \end{equation*}
    Hence, by~Assumption~\ref{Paper02_scaling_parameters_assumption},~\eqref{Paper02_good_behaviour_indices_modified_reaction_term} and Theorem~\ref{Paper02_bound_total_mass}, and then by using Jensen's inequality and the BDG inequality as in~\eqref{Paper02_BDG_inequality}, we conclude that there exists $C^{(2)}_{R,T,\varphi} > 0$ such that
     \begin{equation} \label{Paper02_bound_martingale_term_tightness}
        \sup_{N \in \mathbb N}   \, \mathbb{E}_{\boldsymbol \eta^N} \left[\sup_{t \in [0,T]} \left\vert M^{N, \varphi}_{\mathbb{N}_{0}}(t) \right\vert \right] \leq C^{(2)}_{R,T,\varphi}.
    \end{equation}
    Therefore, substituting estimates~\eqref{Paper02_bound_mass_compact_intervals_initial_condition},~\eqref{Paper02_bound_finite_variation_term} and~\eqref{Paper02_bound_martingale_term_tightness} into~\eqref{Paper02_intermediate_step_tightness_bound_mass_compact_intervals}, the result follows from~\eqref{Paper02_bound_supremum_mass_compact_interval_by_inner_product_test_function}.
\end{proof}

Recall from the end of Section~\ref{Paper02_introduction} that we equip $\mathcal M(\mathbb R)$ with the vague topology. By combining Lemma~\ref{Paper02_uniform_time_control_mass_compact_intervals} and the characterisation of compact subsets of $\mathcal{M}(\mathbb{R})$ in the vague topology, we will now be able to prove Proposition~\ref{Paper02_compact_containment_condition}.

\begin{proof}[Proof of Proposition~\ref{Paper02_compact_containment_condition}]
Let $\left(A_{n}\right)_{n \in \mathbb{N}},\left(a_{n}\right)_{n \in \mathbb{N}} \subset (0,\infty)$ be non-decreasing with $\lim_{n \rightarrow \infty} A_{n} = \infty$. Let $\mathcal{K} \subset \mathcal{M}(\mathbb{R})$ be the subset of Radon measures given by
\begin{equation*}
    \mathcal{K} \defeq \left\{\nu \in \mathcal{M}(\mathbb{R}): \, \nu([-A_{n}, A_{n}]) \leq a_{n} \; \forall \, n \in \mathbb{N}\right\}.
\end{equation*}
Then $\mathcal{K}$ is a compact subset of $\mathcal{M}(\mathbb{R})$ in the vague topology (see~\cite[Section III.9 - Proposition 15]{bourbaki2004measures}). Furthermore, recalling~\eqref{Paper02_metric_product_topology}, since $(\mathcal{M}(\mathbb{R})^{\mathbb{N}_{0}}, d)$ has the product topology, by Tychonoff's theorem (see for instance \cite[Theorem~4.42]{folland1999real}), if $\left(\mathcal{K}_{k}\right)_{k \in \mathbb{N}_{0}}$ is a sequence of compact subsets of $\mathcal{M}(\mathbb{R})$, then the Cartesian product $\mathcal{K} \defeq \prod_{k \in \mathbb N_0} \mathcal{K}_{k}$ is a compact subset of $\mathcal{M}(\mathbb{R})^{\mathbb{N}_{0}}$. 

Take~$T \geq 0$ and for each $R \in \mathbb N$, define $C_{R,T} >0$ as in Lemma~\ref{Paper02_uniform_time_control_mass_compact_intervals}. For $\varepsilon \in (0,1)$, let $\mathcal{K}_{\varepsilon, T} = \prod_{k \in \mathbb N_0} \mathcal{K}_{k, \varepsilon, T}$, where for $k \in \mathbb{N}_{0}$, the set $\mathcal{K}_{k, \varepsilon, T}$ is given by
\begin{equation*}
    \mathcal{K}_{k, \varepsilon, T} \defeq \left\{\nu \in \mathcal{M}(\mathbb{R}): \nu([-R, R]) \leq \frac{2^{R} C_{R,T}}{\varepsilon} \, \forall \, R \in \mathbb{N}\right\}.
\end{equation*}
Then $\mathcal{K}_{\varepsilon, T}$ is a compact subset of $\mathcal{M}(\mathbb{R})^{\mathbb{N}_{0}}$. Moreover, by the definition of $\mathcal K_{\varepsilon,T}$, and then, in the third line, by a union bound, Lemma~\ref{Paper02_uniform_time_control_mass_compact_intervals} and Markov's inequality, for any $N \in \mathbb N$,
\begin{equation*}
\begin{aligned}
    \mathbb{P}_{\boldsymbol \eta^N} & \left(u^{N}(t) \in \mathcal{K}_{\varepsilon, T} \; \forall t \in [0,T]\right)  \\ & = \mathbb{P}_{\boldsymbol \eta^N}\bigg(\sup_{t \in [0,T]} \int_{-R}^{R} u^{N}_{k}(t, x) \, dx \leq \frac{2^{R}C_{R,T}}{\varepsilon} \; \forall \, k \in \mathbb{N}_{0}, \; \forall \, R \in \mathbb{N} \bigg) \\ & \geq \mathbb{P}_{\boldsymbol \eta^N}\bigg(\sup_{t \in [0,T]} \int_{-R}^{R} \| u^{N}(t, x) \|_{\ell_{1}} \, dx \leq \frac{2^{R}C_{R,T}}{\varepsilon} \; \forall \, R \in \mathbb{N} \bigg) \\ & \geq 1 - \sum_{R \in \mathbb N}\frac{\varepsilon}{2^{R}} \\ & = 1 - \varepsilon,
\end{aligned}
\end{equation*}
which completes the proof, since~$\varepsilon > 0$ was arbitrary.
\end{proof}

As explained at the start of this section, to conclude our proof of tightness for~$(u^{N})_{N \in \mathbb{N}}$, it will suffice to find a suitable set $\mathbb{F} \subset \mathcal{C}(\mathcal{M}(\mathbb{R})^{\mathbb{N}_{0}}; \mathbb{R})$ that separates points in $\mathcal{M}(\mathbb{R})^{\mathbb{N}_{0}}$ and that is closed under addition, and then prove that Proposition~\ref{Paper02_tightness_evaluations} holds with this choice of~$\mathbb{F}$. We recall that, since $\mathcal{C}^{2}_{c}(\mathbb{R})$ is a dense subset of $\mathcal{C}_{c}(\mathbb{R})$ in the topology of uniform convergence (see for instance~\cite[Exercises~3.10 and~3.11]{ethier2009markov}), it follows that $\mathcal{C}^{2}_{c}(\mathbb{R})$ separates the points of $\mathcal{M}(\mathbb{R})$ (see for instance~\cite[Proposition~2.9]{bourbaki2004measures}). Using this and the fact that $(\mathcal{M}(\mathbb{R})^{\mathbb{N}_{0}},d)$ is equipped with the product topology (by~\eqref{Paper02_metric_product_topology}), it is immediate that the set
\begin{equation} \label{Paper02_nice_set_evaluation_functions}
\begin{aligned}
    \mathbb{F} \defeq \Bigg\{\varphi: \mathcal{M}(\mathbb{R})^{\mathbb{N}_{0}} \rightarrow \mathbb{R}: \; & \exists \, n \in \mathbb N_0, \varphi_0, \varphi_1, \ldots, \varphi_n \in \mathcal C^2_c(\mathbb R) \textrm{ such that}  \\ & \quad \varphi(u) = \sum_{k = 0}^{n} \left\langle u_{k}, \varphi_{k} \right\rangle \quad \forall u = (u_{k})_{k \in \mathbb N_0} \in \mathcal M(\mathbb R)^{\mathbb N_0} \Bigg\}
\end{aligned}
\end{equation}
satisfies conditions~(i) and~(ii) of Proposition~\ref{Paper02_tightness_evaluations}. Then, applying the tightness criterion stated in \cite[Theorem~3.8.6]{ethier2009markov} and in~\cite{aldous1978stopping}, in order to prove Proposition~\ref{Paper02_tightness_evaluations}, it will be enough to establish a weak compact containment condition and Aldous' criterion. The weak compact containment condition reads as follows.

\begin{lemma} \label{Paper02_weak_compact_containment_condition_tightness}
Under the assumptions of Theorem~\ref{Paper02_deterministic_scaling_foutel_rodier_etheridge}, for $\mathbb F$ as defined in~\eqref{Paper02_nice_set_evaluation_functions}, for any $\varphi \in \mathbb{F}$, $\varepsilon > 0$ and $T \geq 0$, there is a compact set $\mathcal{K}_{\varphi, \varepsilon, T} \subset \mathbb{R}$ such that
\begin{equation*}  
\inf_{N \in \mathbb N} \; \mathbb{P}_{\boldsymbol \eta^N}\left(\varphi(u^{N}(t)) \in \mathcal{K}_{\varphi, \varepsilon, T} \; \; \forall t \in [0,T]\right) \geq 1 - \varepsilon.
\end{equation*}
\end{lemma}

\begin{proof}
Since~$\mathbb{F} \subset \mathcal{C}(\mathcal{M}(\mathbb{R})^{\mathbb{N}_{0}}; \mathbb{R})$, and by the fact that the continuous image of a compact set is compact, the result follows from Proposition~\ref{Paper02_compact_containment_condition}.
\end{proof}

We now proceed to establishing Aldous' criterion \cite{aldous1978stopping}.

\begin{lemma} [Aldous' criterion] \label{Paper02_aldous_criterion}
Under the assumptions of Theorem~\ref{Paper02_deterministic_scaling_foutel_rodier_etheridge}, for $\mathbb F$ as defined in~\eqref{Paper02_nice_set_evaluation_functions}, the following holds. For~$N \in \mathbb N$ and~$T \geq 0$, let $\mathcal{T}^N(T)$ be the family of all $\{\mathcal F^{\eta^N}_{t+}\}_{t \geq 0}$-stopping times~$\tau$ such that $\tau \leq T$ almost surely. Then, for any $\varphi \in \mathbb{F}$, $\varepsilon > 0$ and $T \geq 0$,
    \begin{equation*}
    \lim_{ \gamma \downarrow 0} \; \limsup_{N \rightarrow \infty} \; \sup_{\substack{\tau \in \mathcal{T}^N(T), \\ 0 < \theta \leq \gamma}} \mathbb{P}_{\boldsymbol \eta^N}\left(\left\vert \varphi({u}^{N}((\tau + \theta) \wedge T)) - \varphi({u}^{N}(\tau)) \right\vert > \varepsilon\right) = 0.
    \end{equation*}
\end{lemma}

\begin{proof}
By the definition of $\mathbb{F}$ in~\eqref{Paper02_nice_set_evaluation_functions}, it will be enough to verify that for any~$k \in \mathbb{N}_{0}$, $\phi \in \mathcal{C}^{2}_{c}(\mathbb{R})$, $\varepsilon > 0$ and $T \geq 0$,
\begin{equation} \label{Paper02_reformulation_aldous_criterion}
    \lim_{\gamma \downarrow 0} \; \limsup_{N \rightarrow \infty} \; \sup_{\substack{\tau \in \mathcal{T}^N(T), \\ 0 < \theta \leq \gamma}} \mathbb{P}_{\boldsymbol \eta^N}\left(\left\vert  \left\langle{u}^{N}_{k}((\tau + \theta) \wedge T,\cdot), \phi \right\rangle - \left\langle {u}^{N}_{k}(\tau, \cdot), \phi\right\rangle \right\vert > \varepsilon\right) = 0.
    \end{equation}
To prove~\eqref{Paper02_reformulation_aldous_criterion}, we will apply the semimartingale formulation of $\left\langle u^{N}_{k}(T,\cdot), \phi \right\rangle$ given by Lemma~\ref{Paper02_action_measure_formulation_semimartinagale}. Recall from~\eqref{Paper02_local_mass_subset_indices_definition} that $u^N_k = U^N_{\{k\}}$. For any $N \in \mathbb N$, $\tau \in \mathcal{T}^N(T)$, $\gamma > 0$ and $\theta \in (0, \gamma]$, we have by Lemma~\ref{Paper02_action_measure_formulation_semimartinagale} and the triangle inequality that
\begin{equation*} 
\begin{aligned}
    & \left\vert  \left\langle{u}^{N}_{k}((\tau + \theta) \wedge T,\cdot), \phi \right\rangle - \left\langle {u}^{N}_{k}(\tau, \cdot), \phi\right\rangle \right\vert \\ & \quad \leq \left\vert A^{N,\phi}_{\{k\}}((\tau + \theta) \wedge T) - A^{N,\phi}_{\{k\}}(\tau) \right\vert + \left\vert M^{N,\phi}_{\{k\}}((\tau + \theta) \wedge T) - M^{N,\phi}_{\{k\}}(\tau) \right\vert.
\end{aligned}
\end{equation*}
Therefore, by Markov's inequality, in order to prove~\eqref{Paper02_reformulation_aldous_criterion}, it will be enough to verify that
\begin{equation} \label{Paper02_finite_variation_estimate_aldous_criterion}
    \lim_{\gamma \downarrow 0} \; \limsup_{N \rightarrow \infty} \; \sup_{\substack{\tau \in \mathcal{T}^N(T), \\ 0 < \theta \leq \gamma}} \; \mathbb{E}_{\boldsymbol \eta^N}\left[\left\vert A^{N,\phi}_{\{k\}} \Big((\tau + \theta) \wedge T\Big) -   A^{N,\phi}_{\{k\}}(\tau)\right\vert \right] = 0,
\end{equation}
and that
\begin{equation} \label{Paper02_martingale_estimate_aldous_criterion}
    \lim_{\gamma \downarrow 0} \; \limsup_{N \rightarrow \infty} \; \sup_{\substack{\tau \in \mathcal{T}^N(T), \\ 0 < \theta \leq \gamma}} \; \mathbb{E}_{\boldsymbol \eta^N}\left[\left\vert M^{N,\phi}_{\{k\}} \Big((\tau + \theta) \wedge T\Big) -   M^{N,\phi}_{\{k\}}(\tau)\right\vert \right] = 0.
\end{equation}

We will start by proving~\eqref{Paper02_finite_variation_estimate_aldous_criterion}. Let $R_{\phi} > 0$ be sufficiently large that $\supp (\phi) \subset [-R_{\phi} + 2L_N^{-1}, R_{\phi} - 2L_N^{-1}]$ for every~$N \in \mathbb{N}$. Note that since by Assumption~\ref{Paper02_scaling_parameters_assumption}, $L_{N} \rightarrow \infty$ as~$N \rightarrow \infty$, and since~$\phi$ has compact support, we can choose such an~$R_{\phi} < \infty$. By the definition of $A^{N,\phi}_{\{k\}}$ in~\eqref{Paper02_lemma_semimartingale_finite_variation_process} in the statement of Lemma~\ref{Paper02_action_measure_formulation_semimartinagale}, and since~$\vert F_{k}(u) \vert \leq F^{+}_{k}(u)$ for~$u \in \ell_{1}^{+}$ by~\eqref{Paper02_reaction_term_PDE} and~\eqref{Paper02_reaction_predictable_bracket_process}, we have that for any~$N \in \mathbb N$, $\tau \in \mathcal{T}^N(T)$, $\gamma > 0$ and $\theta \in (0,\gamma]$,
\begin{equation} \label{Paper02_intermediate_step_aldous_criterion_finite_variation_process}
\begin{aligned}
    & \left\vert A^{N,\phi}_{\{k\}}\Big((\tau + \theta) \wedge T \Big) - A^{N,\phi}_{\{k\}}(\tau)\right\vert \\
    & \quad \leq \frac{1}{L_{N}} \sum_{x \in L_{N}^{-1}\mathbb{Z}} \vert \left(J_{L_N}(\Laplace_{L_{N}} \phi)\right)(x) \vert \int_{\tau}^{(\tau + \theta) \wedge T} \frac{m_{N}}{2L_{N}^{2}} u^{N}_{k}(t-,x) \; dt \\
     & \quad \quad + \frac{1}{L_{N}} \sum_{x \in L_{N}^{-1}\mathbb{Z}} \vert \left(J_{L_N}( \phi)\right)(x) \vert \int_{\tau}^{(\tau + \theta) \wedge T} F^{+}_{k}(u^{N}(t-, x)) \; dt \\
     & \quad \leq \frac{m_{N}  \|\phi ''\|_{L_{\infty}(\mathbb{R})}}{2L_{N}^{3}} \sum_{x \in L_{N}^{-1}\mathbb{Z} \cap [-R_{\phi}, R_{\phi}]} \int_{0}^{T} u^{N}_{k}(t-,x) \cdot \mathds{1}_{\{\tau \leq t \leq \tau + \theta\}} \; dt \\ & \quad \quad + \frac{\| \phi\|_{L_{\infty}(\mathbb{R})}}{L_{N}} \sum_{x \in L_{N}^{-1}\mathbb{Z} \cap [-R_{\phi}, R_{\phi}]} \int_{0}^{T} F^{+}_{k}(u^{N}(t-, x)) \cdot \mathds{1}_{\{\tau \leq t \leq \tau + \theta\}} \; dt,
\end{aligned}
\end{equation}
where the second inequality follows by our choice of~$R_{\phi}$ and by~\eqref{Paper02_bound_discrete_laplacian} and~\eqref{Paper02_definition_discrete_continuous_measure}. Since $\theta \in (0, \gamma]$, we have
\begin{equation} \label{Paper02_trivial_cauchy_schwarz}
    \left(\int_{0}^{T} \mathds{1}_{\{\tau \leq t \leq \tau + \theta\}} \; dt\right)^{1/2} =  \left(\int_{\tau}^{(\tau + \theta) \wedge T} \; dt\right)^{1/2} \leq \sqrt{\theta} \leq \sqrt{\gamma}.
\end{equation}
Hence, applying the Cauchy-Schwarz inequality and~\eqref{Paper02_trivial_cauchy_schwarz} to the integrals with respect to time on the right-hand side of~\eqref{Paper02_intermediate_step_aldous_criterion_finite_variation_process}, we obtain
\begin{equation} \label{Paper02_finite_variation_process_intermediate_step_aldous_two}
\begin{aligned}
    & \left\vert A^{N,\phi}_{\{k\}}\Big((\tau + \theta) \wedge T \Big) - A^{N,\phi}_{\{k\}}(\tau)\right\vert \\
    & \quad \leq \frac{m_{N} \| \phi ''\|_{L_{\infty}(\mathbb{R})} \sqrt{\gamma}}{2L_{N}^{3}} \sum_{x \in L_{N}^{-1}\mathbb{Z} \cap [-R_{\phi}, R_{\phi}]} \Bigg(\int_{0}^{T} u^{N}_{k}(t-,x)^{2} \; dt \Bigg)^{1/2} \\ & \qquad + \frac{\| \phi\|_{L_{\infty}(\mathbb{R})} \sqrt{\gamma}}{L_{N}} \sum_{x \in L_{N}^{-1}\mathbb{Z}\cap [-R_{\phi}, R_{\phi}]} \Bigg(\int_{0}^{T} \left(F^{+}_{k}(u^{N}(t-, x))\right)^{2} \; dt \Bigg)^{1/2}.
\end{aligned}
\end{equation}
Therefore, taking expectations on both sides of~\eqref{Paper02_finite_variation_process_intermediate_step_aldous_two}, applying Fubini's theorem, and then Jensen's inequality and again Fubini's theorem, we conclude that there exists $C_{q_+,q_-,f}(T) > 0$ such that for any $N \in \mathbb N$, $\tau \in \mathcal{T}^N(T)$, $\gamma > 0$, and $\theta \in (0, \gamma]$,
\begin{equation} \label{Paper02_intermediate_step_aldous_criterion_finite_variation_process_II}
\begin{aligned}
    & \mathbb{E}_{\boldsymbol \eta^N}\left[\left\vert A^{N,\phi}_{\{k\}}\Big((\tau + \theta) \wedge T \Big) - A^{N,\phi}_{\{k\}}(\tau)\right\vert\right] \\
     & \quad \leq \frac{m_{N} \| \phi ''\|_{L_{\infty}(\mathbb{R})} \sqrt{\gamma}}{2L_{N}^{3}} \sum_{x \in L_{N}^{-1}\mathbb{Z} \cap [-R_{\phi}, R_{\phi}]} \Bigg(\int_{0}^{T} \mathbb{E}_{\boldsymbol \eta^N}\left[u^{N}_{k}(t-,x)^{2}\right] \; dt \Bigg)^{1/2} \\ & \quad \quad \; + \frac{\| \phi\|_{L_{\infty}(\mathbb{R})} \sqrt{\gamma}}{L_{N}} \sum_{x \in L_{N}^{-1}\mathbb{Z} \cap [-R_{\phi}, R_{\phi}]} \Bigg(\int_{0}^{T} \mathbb{E}_{\boldsymbol \eta^N}\Big[\left(F^{+}_{k}(u^{N}(t, x))\right)^{2}\Big] \; dt \Bigg)^{1/2} \\
     & \quad \leq C_{q_+,q_-,f}(T)\left(\frac{m_{N} \| \phi ''\|_{L_{\infty}(\mathbb{R})} R_{\phi}}{L_{N}^{2}} + \| \phi \|_{L_{\infty}(\mathbb{R})} R_{\phi}\right)\sqrt{\gamma},
\end{aligned}
\end{equation}
where in the last inequality we used Theorem~\ref{Paper02_bound_total_mass}, the definition of~$F^+_k$ in~\eqref{Paper02_reaction_predictable_bracket_process}, the fact that $s_j \leq 1 \; \forall j \in \mathbb N_0$ by Assumption~\ref{Paper02_assumption_fitness_sequence}(i) and~(iii), and the fact that~$q_+$ and~$q_-$ are polynomials by Assumption~\ref{Paper02_assumption_polynomials}. Since by~Assumption~\ref{Paper02_scaling_parameters_assumption}$(i)$, we have $\displaystyle \frac{m_{N}}{L_{N}^{2}} \rightarrow m \in (0,\infty)$ as $N \rightarrow \infty$,~\eqref{Paper02_intermediate_step_aldous_criterion_finite_variation_process_II} implies that~\eqref{Paper02_finite_variation_estimate_aldous_criterion} holds.

It remains to establish the limit in~\eqref{Paper02_martingale_estimate_aldous_criterion}. Since, by Lemma~\ref{Paper02_action_measure_formulation_semimartinagale}, $M^{N, \phi}_{\{k\}}$ is a càdlàg martingale with respect to the filtration~$\{\mathcal{F}^{\eta^{N}}_{t+}\}_{t \geq 0}$, by the optional sampling theorem (see for instance~\cite[Theorem~2.2.13]{ethier2009markov}), for any~$\tau \in \mathcal{T}^N(T)$, $(M^{N, \phi}_{\{k\}}(\tau + t))_{t \geq 0}$ is a càdlàg martingale with respect to the filtration~$\{\mathcal{F}^{\eta^{N}}_{(\tau + t)+}\}_{t \geq 0}$. Then, recalling the expression for the predictable bracket process of~$M^{N,\phi}_{\{k\}}$ in~\eqref{Paper02_lemma_semimartingale_quadratic_variation_process} in the statement of Lemma~\ref{Paper02_action_measure_formulation_semimartinagale}, by~\eqref{Paper02_bound_discrete_gradient},~\eqref{Paper02_definition_discrete_continuous_measure} and the fact that we chose $R_\phi$ such that $\supp(\phi) \subset [-R_\phi + 2L_N^{-1},R_\phi -2L_N^{-1}] \; \forall \, N \in \mathbb N $, for any $N \in \mathbb N$, $\tau \in \mathcal{T}^N(T)$, $\gamma > 0$ and $\theta \in (0, \gamma]$,
\begin{equation} \label{Paper02_quadratic_variation_Aldous_first_part}
\begin{aligned}
    & \left\langle M^{N,\phi}_{\{k\}}\right\rangle \Big((\tau + \theta) \wedge T\Big) -  \left\langle M^{N,\phi}_{\{k\}}\right\rangle (\tau) \\
     & \quad \leq \frac{m_{N} \| \phi'\|_{L_{\infty}(\mathbb{R})}^{2}}{NL_{N}^{4}}\sum_{x \in L_{N}^{-1}\mathbb{Z} \cap [-R_{\phi}, R_{\phi}]} \int_{0}^{T} u^{N}_{k}(t-,x) \cdot \mathds{1}_{\{\tau \leq t \leq \tau + \theta\}} \; dt\\
     &  \qquad \; + \frac{\| \phi\|_{L_{\infty}(\mathbb{R})}^{2}}{NL_{N}^{2}} \sum_{x \in L_{N}^{-1}\mathbb{Z} \cap [-R_{\phi}, R_{\phi}]} \int_{0}^{T} F^{+}_{k}(u^{N}(t-, x)) \cdot \mathds{1}_{\{\tau \leq t \leq \tau + \theta\}} \; dt.
\end{aligned}
\end{equation}
By an argument analogous to the one we used to derive estimate~\eqref{Paper02_finite_variation_process_intermediate_step_aldous_two}, applying~\eqref{Paper02_trivial_cauchy_schwarz} and the Cauchy-Schwarz inequality to~\eqref{Paper02_quadratic_variation_Aldous_first_part}, it follows that
\begin{equation} \label{Paper02_quadratic_variation_Aldous_first_part_II}
\begin{aligned}
    & \left\langle M^{N,\phi}_{\{k\}}\right\rangle \Big((\tau + \theta) \wedge T\Big) -  \left\langle M^{N,\phi}_{\{k\}}\right\rangle (\tau) \\ & \quad \leq \frac{m_{N} \| \phi'\|_{L_{\infty}(\mathbb{R})}^{2} \sqrt{\gamma}}{NL_{N}^{4}}\sum_{x \in L_{N}^{-1}\mathbb{Z} \cap [-R_{\phi}, R_{\phi}]} \Bigg(\int_{0}^{T} u^{N}_{k}(t-,x)^{2} \; dt \Bigg)^{1/2} \\
     &  \qquad \; + \frac{\| \phi\|_{L_{\infty}(\mathbb{R})}^{2} \sqrt{\gamma}}{NL_{N}^{2}} \sum_{x \in L_{N}^{-1}\mathbb{Z} \cap [-R_{\phi}, R_{\phi}]} \Bigg(\int_{0}^{T} \left(F^{+}_{k}(u^{N}(t-,x))\right)^{2} \, dt \Bigg)^{1/2}.
\end{aligned}
\end{equation}
Then, taking expectations on both sides of~\eqref{Paper02_quadratic_variation_Aldous_first_part_II} and applying Fubini's theorem, and then Jensen's inequality and Fubini's theorem, we conclude that there exists $C'_{q_+,q_-,f}(T) > 0$ such that for any $N \in \mathbb N$, $\tau \in \mathcal T^N(T)$, $\gamma > 0$ and $\theta \in (0,\gamma]$,
\begin{equation} \label{Paper02_quadratic_variation_Aldous_first_part_III}
\begin{aligned}
    & \mathbb{E}_{\boldsymbol \eta^N}\left[\left\langle M^{N,\phi}_{\{k\}}\right\rangle \Big((\tau + \theta) \wedge T\Big) -  \left\langle M^{N,\phi}_{\{k\}}\right\rangle (\tau)\right] \\
    \\ & \quad \leq \frac{m_{N} \| \phi'\|_{L_{\infty}(\mathbb{R})}^{2} \sqrt{\gamma}}{NL_{N}^{4}}\sum_{x \in L_{N}^{-1}\mathbb{Z} \cap [-R_{\phi}, R_{\phi}]} \Bigg(\int_{0}^{T} \mathbb{E}_{\boldsymbol \eta^N}\left[u^{N}_{k}(t-,x)^{2}\right] \; dt \Bigg)^{1/2} \\
     &  \quad \quad + \frac{\| \phi\|_{L_{\infty}(\mathbb{R})}^{2} \sqrt{\gamma}}{NL_{N}^{2}} \sum_{x \in L_{N}^{-1}\mathbb{Z} \cap [-R_{\phi}, R_{\phi}]} \Bigg(\int_{0}^{T} \mathbb{E}_{\boldsymbol \eta^N}\Big[ \left(F^{+}_{k}(u^{N}(t-,x))\right)^{2}\Big]\; dt \Bigg)^{1/2} \\
     & \quad \leq C'_{q_+,q_-,f}(T) \left( \frac{m_{N} \| \phi'\|_{L_{\infty}(\mathbb{R})}^{2} R_{\phi}}{NL_{N}^{3}} + \frac{\| \phi\|_{L_{\infty}(\mathbb{R})}^{2} R_{\phi}}{NL_{N}}\right)\sqrt{\gamma},
\end{aligned}
\end{equation}
where in the last inequality we used Theorem~\ref{Paper02_bound_total_mass}, the definition of~$F^+_k$ in~\eqref{Paper02_reaction_predictable_bracket_process}, the fact that $s_j \leq 1 \; \forall j \in \mathbb N_0$ by Assumption~\ref{Paper02_assumption_fitness_sequence}(i) and~(iii), and the fact that~$q_+$ and~$q_-$ are polynomials by Assumption~\ref{Paper02_assumption_polynomials}, in the same way as for~\eqref{Paper02_intermediate_step_aldous_criterion_finite_variation_process_II}. Finally, we conclude from Assumption~\ref{Paper02_scaling_parameters_assumption} and~\eqref{Paper02_quadratic_variation_Aldous_first_part_III} that
\begin{equation*}
    \lim_{\gamma \downarrow 0} \; \limsup_{N \rightarrow \infty} \; \sup_{\substack{\tau \in \mathcal{T}^N(T), \\ 0 < \theta \leq \gamma}} \; \mathbb{E}_{\boldsymbol \eta^N}\left[\left\langle M^{N,\phi}_{\{k\}}\right\rangle \Big((\tau + \theta) \wedge T\Big) -  \left\langle M^{N,\phi}_{\{k\}}\right\rangle (\tau)\right] = 0,
\end{equation*}
which implies~\eqref{Paper02_martingale_estimate_aldous_criterion} by an application of Jensen's inequality and the BDG inequality as in~\eqref{Paper02_BDG_inequality}. Hence, both~\eqref{Paper02_finite_variation_estimate_aldous_criterion} and~\eqref{Paper02_martingale_estimate_aldous_criterion} hold, and therefore, as explained after~\eqref{Paper02_reformulation_aldous_criterion}, this establishes~\eqref{Paper02_reformulation_aldous_criterion} and completes the proof.
\end{proof}

We are now ready to prove Proposition~\ref{Paper02_tightness_evaluations}.

\begin{proof}[Proof of Proposition~\ref{Paper02_tightness_evaluations}]
We checked before~\eqref{Paper02_nice_set_evaluation_functions} that~$\mathbb{F}$ satisfies conditions~(i) and~(ii). By the tightness criterion stated in \cite[Theorem~3.8.6]{ethier2009markov} and in \cite{aldous1978stopping}, the result follows directly from Lemmas~\ref{Paper02_weak_compact_containment_condition_tightness} and~\ref{Paper02_aldous_criterion}.
\end{proof}

We can finally prove Proposition~\ref{Paper02_PDE_tightness_measures}.

\begin{proof}[Proof of Proposition~\ref{Paper02_PDE_tightness_measures}]
By Jakubowski's tightness criterion~\cite[Theorem~3.1]{jakubowski1986skorokhod}, since the metric space $(\mathcal M(\mathbb R)^{\mathbb N_0}, d)$ is complete and separable (see e.g.~the comment before Proposition~3.4.6 in~\cite{ethier2009markov} for a discussion on the product topology), the result follows directly from Propositions~\ref{Paper02_compact_containment_condition} and~\ref{Paper02_tightness_evaluations}.
\end{proof}

\section{Characterisation of the limiting process} \label{Paper02_section_limiting_process_L1_construction}

Proposition~\ref{Paper02_PDE_tightness_measures} implies the weak convergence of subsequences of $(u^N)_{N \in \mathbb N}$ to a limiting process~$u = (u_{k})_{k \in \mathbb{N}_{0}}$ in the Polish space $\mathcal{D}\left([0,\infty), (\mathcal{M}(\mathbb{R})^{\mathbb{N}_{0}},d)\right)$. 
To complete the proof of Theorem~\ref{Paper02_deterministic_scaling_foutel_rodier_etheridge}, it remains to establish that any subsequential limit is a solution to the system of PDEs~\eqref{Paper02_PDE_scaling_limit}, and then to prove the uniqueness of the limit. A major challenge to overcome is that we must verify that there exists a sequence of measurable functions~$v = (v_{k})_{k \in \mathbb{N}_{0}}: [0, \infty) \times \mathbb{R} \rightarrow \ell_{1}$ such that for each~$k \in \mathbb{N}_{0}$, $v_{k}$ is the density of the measure~$u_{k}$ with respect to the Lebesgue measure on~$[0, \infty) \times \mathbb{R}$, and then that the sequence of densities~$v = (v_{k})_{k \in \mathbb{N}_{0}}$ satisfies the system of PDEs~\eqref{Paper02_PDE_scaling_limit}.

To achieve this goal, it will be useful to establish tightness of~$(u^{N})_{N \in \mathbb{N}}$ in a function space, so that the sequence of densities~$v = (v_{k})_{k \in \mathbb{N}_{0}}$ coincides with the subsequential limit of~$(u^{N})_{N \in \mathbb{N}}$ in this function space. Bearing this aim in mind, recall from Section~\ref{Paper02_introduction} that we let~$\lambda$ denote the Lebesgue measure on~$\mathbb{R}^{2}$, and now, for any $T \geq 0$, let $\hat{\lambda}$ denote the measure on $[0,T] \times \mathbb{R}$ given by
\begin{equation} \label{Paper02_measure_time_space_box_i}
    \hat{\lambda}(dt \; dx) \defeq \frac{\mathds{1}_{\{t \in [0,T]\}}}{1 + \vert x \vert^{2}} \lambda (dt \; dx).
\end{equation}
For~$T >0$ and~$r \in [1,\infty)$, we introduce the function space
\begin{equation} \label{Paper02_functional_space_L1_l1}
\begin{split}
    & {L}_{r}([0,T] \times \mathbb{R}, \hat{\lambda}; \ell_{1}) \\ & \quad \defeq \Bigg\{v: [0,T] \times \mathbb{R} \rightarrow \ell_{1} \; \textrm{s.t.} \; \| v \|_{{L}_{r}([0,T] \times \mathbb{R}, \hat{\lambda}; \ell_{1})} \defeq \Bigg(\int_{0}^{T} \int_{\mathbb{R}} \frac{\| v(t,x) \|_{\ell_{1}}^{r}}{1 + \vert x \vert^{2}} \, dx \, dt \Bigg)^{1/r} < \infty\Bigg\}.
    \raisetag{-1.2cm}
\end{split}
\end{equation}

\begin{remark}
    Observe that, by~\eqref{Paper02_functional_space_L1_l1}, the space~$L_{r}([0,T] \times \mathbb{R}, \hat{\lambda}; \ell_{1})$ differs from the space
    \begin{equation*}
    \begin{aligned}
    & {L}_{r}([0,T] \times \mathbb{R} \times \mathbb{N}_{0}, \hat{\lambda}; \mathbb{R}) \\
    & \quad \defeq \Bigg\{v: [0,T] \times \mathbb{R} \times \mathbb{N}_{0} \rightarrow \mathbb{R} \; \textrm{s.t.} \;  \sum_{k \in \mathbb N_0} \, \int_{0}^{T} \int_{\mathbb{R}} \frac{\vert v_{k}(t,x) \vert^{r}}{1 + \vert x \vert^{2}} \, dx \, dt < \infty\Bigg\}.
    \end{aligned}
    \end{equation*}
    In particular, the advantage of embedding the sequence of density processes in~$L_r([0,T] \times \mathbb{R}, \hat{\lambda}; \ell_1)$ rather than in~${L}_{r}([0,T] \times \mathbb{R} \times \mathbb{N}_{0}, \hat{\lambda}; \mathbb{R})$ is because of the structure of the reaction term~$F = (F_{k})_{k \in \mathbb{N}_{0}}$ defined in~\eqref{Paper02_reaction_term_PDE}, since for $u \in \ell_1^+$ and $k \in \mathbb{N}_{0}$, the expression defining $F_k(u)$ contains polynomials of~$\| u \|_{\ell_{1}}$.
\end{remark}

Since~$\ell_{1}$ is a complete and separable Banach space, for every $r \in [1, \infty)$, the space $L_{r}([0,T] \times \mathbb{R}, \hat{\lambda}; \ell_{1})$ is also a complete and separable Banach space when equipped with the norm~$\| \cdot \|_{L_{r}([0,T] \times \mathbb{R}, \hat{\lambda}; \ell_{1})}$ defined in~\eqref{Paper02_functional_space_L1_l1} (see~\cite[Proposition~1.2.29]{hytonen2016analysis} for a proof of this fact). Recall that under the assumptions of Theorem~\ref{Paper02_deterministic_scaling_foutel_rodier_etheridge},~$q_{-}$ is a polynomial satisfying Assumption~\ref{Paper02_assumption_polynomials}, and, by Theorem~\ref{Paper02_bound_total_mass} and since $u^N_k(t,x)$ is defined by linear interpolation for $x \not\in L_N^{-1}\mathbb Z$, for every~$N \in \mathbb{N}$ and $T \geq 0$, we have
\begin{equation*}
    \int_{0}^{T} \int_{\mathbb{R}} \frac{\mathbb{E}_{\boldsymbol \eta^N}\Big[\| u^{N}(t,x) \|_{\ell_{1}}^{4\deg q_{-}}\Big]}{1 + \vert x \vert^{2}} \, dx \, dt < \infty,
\end{equation*}
which implies that~$(u^{N}(t,x))_{t \in [0,T], \, x \in \mathbb{R}}$ is almost surely an element of~$L_{4\deg q_{-}}([0,T] \times \mathbb{R}, \hat{\lambda}; \ell_{1})$, for every $N \in \mathbb{N}$. We can then study the tightness of the sequence
\[
((u^{N}(t,x))_{t \in [0,T], \, x \in \mathbb{R}})_{N \in \mathbb{N}}
\]
in~$L_{4\deg q_{-}}([0,T] \times \mathbb{R}, \hat{\lambda}; \ell_{1})$, and use it to characterise any limiting~$\mathcal{M}(\mathbb{R})^{\mathbb{N}_{0}}$-valued process~$u$. Recall the definition of $\mathcal{C}_c(\mathbb R)$ from Section~\ref{Paper02_introduction}. 
For $\rho \in \mathcal{M}(\mathbb R)$ and $\varphi \in \mathcal{C}_c(\mathbb R)$, let $\langle \rho, \varphi \rangle$ be defined as in Section~\ref{Paper02_introduction}. Recall that the measure defined in~\eqref{Paper02_definition_u_N_as_radon_measure} induced by the random function~$u^{N}_{k}(t,\cdot)$  is denoted by~$u^{N}_{k}(t)$, for every~$N \in \mathbb{N}$,~$k \in \mathbb{N}_{0}$ and~$t \geq 0$. Recall the definition of a mild solution to the system of PDEs~\eqref{Paper02_PDE_scaling_limit} in Definition~\ref{Paper02_definition_mild_solution}. Recall from Section~\ref{Paper02_introduction} that
for $H \in \mathbb N$, we let $\llbracket H \rrbracket \defeq \{1,2,\ldots, H\}$. We are now ready to state the main result of this section.

\begin{proposition} \label{Paper02_absolutely_continuity_limiting_process}
Under the assumptions of Theorem~\ref{Paper02_deterministic_scaling_foutel_rodier_etheridge}, for any~$T > 0$, the sequence
$$\Big((u^{N}(t))_{t \in [0,T]}, (u^{N}(t,x))_{t \in [0,T], \, x \in \mathbb{R}}\Big)_{N \in \mathbb{N}}$$
is tight in~$\mathcal{D}\left([0,T], (\mathcal{M}(\mathbb{R})^{\mathbb{N}_{0}},d)\right) \times L_{4\deg q_-}([0,T] \times \mathbb{R}, \hat{\lambda}; \ell_{1})$. Moreover, the following conditions are satisfied by any subsequential limit~$\Big((u(t))_{t \in [0,T]}, \, (v(t,x))_{t \in [0,T], \, x \in \mathbb{R}}\Big)$:
\begin{enumerate}[label = (\roman*)]
    \item For every~$H \in \mathbb{N}$,~$(t_{h})_{h \in \llbracket H \rrbracket} \in [0,T)^{H}$,~$(\varphi_{h})_{h \in \llbracket H \rrbracket} \in \mathcal{C}_{c}(\mathbb{R})^{H}$ and~$(k_{h})_{h \in \llbracket H \rrbracket} \in (\mathbb{N}_{0})^{H}$,
    \begin{equation} \label{Paper02_distribution_characterisation_limiting_process}
       (\langle u_{k_{h}}(t_{h}), \, \varphi_{h} \rangle)_{h \in \llbracket H \rrbracket} \overset{d}{{=}}  \bigg(\lim_{t' \downarrow t_{h}} \frac{1}{t' - t_{h}}\int_{t_{h}}^{t'} \int_{\mathbb R} v_{k_{h}}(\tau, x) \varphi_{h}(x) \, dx \, d\tau\bigg)_{h \in \llbracket H \rrbracket}.
    \end{equation}
     \item $(v(t,x))_{t \in [0,T], \, x \in \mathbb{R}}$ is a mild solution to the system of PDEs~\eqref{Paper02_PDE_scaling_limit}.
\end{enumerate}
\end{proposition}

For a subsequential limit $(u,v) \in \mathcal{D}\left([0,T], (\mathcal{M}(\mathbb{R})^{\mathbb{N}_{0}},d)\right) \times L_{4\deg q_-}([0,T] \times \mathbb{R}, \hat{\lambda}; \ell_{1})$ as in Proposition~\ref{Paper02_absolutely_continuity_limiting_process}, we will refer to the random function~$v = (v_{k})_{k \in \mathbb{N}_{0}}: [0,T] \times \mathbb{R} \rightarrow \ell_{1}$ as the limiting density process. In the remainder of this section, we will prove Proposition~\ref{Paper02_absolutely_continuity_limiting_process}. For didactic reasons, we split the proof into two subsections. We will establish the desired tightness in Section~\ref{Paper02_tightness_L_P_l_1spaces}, and then characterise the limiting processes $\Big((u(t))_{t \in [0,T]}, \, (v(t,x))_{t \in [0,T], \, x \in \mathbb{R}}\Big)$ in Section~\ref{Paper02_Characterisation_limiting_process}.

\subsection{Tightness in $L_{4\deg q_-}([0,T] \times \mathbb{R}, \hat{\lambda}; \ell_{1})$} \label{Paper02_tightness_L_P_l_1spaces}

Since Proposition~\ref{Paper02_PDE_tightness_measures} provides the tightness of $((u^{N}(t))_{t \in [0,T]})_{N \in \mathbb{N}}$ in~$\mathcal{D}\left([0,\infty), (\mathcal{M}(\mathbb{R})^{\mathbb{N}_{0}},d)\right)$, it will suffice to establish that
\[
((u^{N}(t,x))_{t \in [0,T], \, x \in \mathbb{R}})_{N \in \mathbb{N}}
\]
is tight in~$L_{4\deg q_-}([0,T] \times \mathbb{R}, \hat{\lambda}; \ell_{1})$. Our first step in this direction will be to establish conditions that guarantee the desired tightness. These conditions rely on the characterisation of compact subsets of~${L}_{r}([0,T] \times \mathbb{R}, \hat{\lambda}; \ell_{1})$ given by D\'{i}az and Mayoral's compactness theorem~\cite[Theorem~3.2]{diaz1999compactness}. It will be convenient to introduce some more notation.
For~$T > 0$ and~$r \in [1,\infty)$, define the function space
\begin{equation} \label{Paper02_functional_space_real_functions_measure_time_space_box}
\begin{aligned}
    & L_{r}([0,T] \times \mathbb{R}, \hat{\lambda}; \mathbb{R}) \\ & \quad \defeq \bigg\{v^{*}: [0,T] \times \mathbb{R} \rightarrow \mathbb{R} \; \textrm{such that} \\[-4mm]
    & \hspace{3cm} \| v^{*} \|_{L_{r}([0,T] \times \mathbb{R}, \hat{\lambda}; \mathbb{R})} \defeq \bigg(\int_{0}^{T} \int_{\mathbb{R}} \frac{\vert v^{*}(t,x) \vert^{r}}{1 + \vert x \vert^{2}} \, dx \, dt\bigg)^{1/r} < \infty\bigg\},
\end{aligned}
\end{equation}
and observe that~$\Big(L_{r}([0,T] \times \mathbb{R}, \hat{\lambda}; \mathbb{R}), \| \cdot \|_{L_{r}([0,T] \times \mathbb{R}, \hat{\lambda}; \mathbb{R})} \Big)$ is also a complete and separable Banach space by the usual theory of~$L_p$ spaces (see e.g.~\cite[Section~6.1]{folland1999real}). We also define, for any $\gamma \in \mathbb R$, the following shift maps: for any function~$g: [0,T] \times \mathbb{R} \rightarrow \mathbb{R}$ and any $(t,x) \in [0,T] \times \mathbb{R}$, we let
\begin{equation} \label{Paper02_shift_maps}
    (\theta^{(1)}_{\gamma}g)(t,x) \defeq g(t + \gamma,x) \mathds 1_{\{t + \gamma \in [0,T]\}} \quad \textrm{and} \quad  (\theta^{(2)}_{\gamma}g)(t,x) \defeq g(t,x + \gamma).
\end{equation}

We now state a criterion for relative compactness in~$L_{r}([0,T] \times \mathbb{R}, \hat{\lambda}; \ell_{1})$, which will be used for a tightness criterion.

\begin{lemma} \label{Paper02_alternative_criterion_compactness_diaz_mayoral}
    Let~$T > 0$ and~$r \in [1, \infty)$ be fixed. Let~$C_{1},C_{2} > 0$, $\beta > 1$, $J \in \mathbb{N}$, let~$(k_{h})_{h \in \mathbb{N}} \subseteq \mathbb N_0$ be strictly increasing, and let~$(D_{h})_{h \in \mathbb{N}} \subset [0, \infty)$ be such that $D_{h} \rightarrow 0$ as $h \rightarrow \infty$. Define the following subsets of~$L \defeq L_{r}([0,T] \times \mathbb{R}, \hat{\lambda}; \ell_{1})$:
    \begin{equation*}
    \begin{aligned}
  & \mathcal{K}_{(1)}(C_{1}, T) \defeq \Big\{v \in L: \; \| v \|_{{L}_{2r}([0,T] \times \mathbb{R}, \hat{\lambda}; \ell_{1})} \leq C_{1}\Big\}, \\
  & \mathcal{K}_{(2)}((k_{h})_{h \in \mathbb{N}}, (D_{h})_{h \in \mathbb{N}}, T) \\
& \qquad\qquad\qquad \defeq \bigg\{v = (v_k)_{k \in \mathbb N_0} \in L: \; \sum_{k = k_{h}}^\infty \| v_{k} \|_{{L}_{1}([0,T] \times \mathbb{R}, \hat{\lambda}; \mathbb{R})} < D_{h} \; \forall \, h \in \mathbb{N}\bigg\},
\\ & \mathcal{K}_{(3)}(C_{2},\beta, J, T) \defeq \Big\{v = (v_k)_{k \in \mathbb N_0} \in L: \; \| \theta^{(i)}_{(-1)^{i'} 2^{-(k+j)}}v_{k} - v_{k} \|_{{L}_{1}([0,T] \times \mathbb{R}, \hat{\lambda}; \mathbb{R})} \leq C_{2}\beta^{-(k+j)} \\ & \qquad \qquad \qquad \qquad \qquad \qquad \qquad \qquad \qquad  \; \forall \, k \in \mathbb{N}_{0}, \, j \in \mathbb{N} \cap [J,\infty), \, i,i' \in \{1,2\}\Big\},
    \end{aligned}
    \end{equation*}
    and let $\mathcal{K} = \mathcal{K}(C_{1},C_{2}, (k_{h})_{h \in \mathbb{N}},(D_{h})_{h \in \mathbb{N}},\beta,J,T) \subset L_r([0,T] \times \mathbb{R}, \hat{\lambda}; \ell_{1})$ denote the set
    \begin{equation*}
        \mathcal{K} \defeq \mathcal{K}_{(1)}(C_{1}, T) \cap \mathcal{K}_{(2)}((k_{h})_{h \in \mathbb{N}},(D_{h})_{h \in \mathbb{N}}, T) \cap \mathcal{K}_{(3)}(C_{2},\beta, J, T).
    \end{equation*}
    Then $\mathcal{K}$ is a relatively compact subset of~$L_{r}([0,T] \times \mathbb{R}, \hat{\lambda}; \ell_{1})$.
\end{lemma}

The proof of Lemma~\ref{Paper02_alternative_criterion_compactness_diaz_mayoral} relies on D\'{i}az and Mayoral's compactness theorem (see Theorem~\ref{Paper02_diaz_mayoral_compactness_theorem} in the appendix) and on the Kolmogorov-Riesz-Fr\'{e}chet compactness theorem for the usual $L_p$ spaces (see Theorem~\ref{Paper02_kolmogorov_riesz_frechet} in the appendix), and it is postponed until Section~\ref{Paper02_appendix_proof_lemma_compact_sets_of_L_p} in the appendix. The advantage of using Lemma~\ref{Paper02_alternative_criterion_compactness_diaz_mayoral} to prove tightness of a sequence of density processes is that it allows us to define relatively compact subsets of~$L_{r}([0,T] \times \mathbb{R}, \hat{\lambda}; \ell_{1})$ as countable intersections of certain sets. Our next result uses this strategy to give a tightness criterion.

\begin{lemma} \label{Paper02_easy_tightness_criterion_L_p_spaces}
    Let~$T > 0$ and~$r \in [1, \infty)$ be fixed, and let~$((v^{N}_{k}(t,x))_{k \in \mathbb{N}_{0}, \, t \in [0,T], \, x \in \mathbb{R}})_{N \in \mathbb{N}}$ be a sequence of~$L_{r}([0,T] \times \mathbb{R}, \hat{\lambda}; \ell_{1})$-valued random variables such that the following conditions hold:
    \begin{enumerate}[label = (\roman*)]
        \item The~$2r$-moments of~$\Big((v^{N}(t,x))_{t \in [0,T], \, x \in \mathbb{R}}\Big)_{N \in \mathbb{N}}$ are uniformly bounded, i.e.
        \begin{equation*}
            \sup_{N \in \mathbb{N}} \; \sup_{t \in [0,T]} \; \sup_{x \in \mathbb{R}} \; \mathbb{E}\Big[\| v^{N}(t,x) \|_{\ell_{1}}^{2r}\Big] < \infty.
        \end{equation*}
        \item The following limit holds:
        \begin{equation*}
            \lim_{k \rightarrow \infty} \; \sup_{N \in \mathbb{N}} \; \sup_{t \in [0,T]} \; \sup_{x \in \mathbb{R}} \; \mathbb{E}\Bigg[\sum_{j = k}^\infty \, \vert v^{N}_{j}(t,x) \vert\Bigg] = 0.
        \end{equation*}
        \item There exist~$l_{1},l_{2},C_{T} > 0$ such that for every~$\gamma \in (-1,1)$ and~$i \in \{1,2\}$,
        \begin{equation*}
            \sup_{N \in \mathbb{N}} \; \sup_{k \in \mathbb{N}_{0}} \mathbb{E}\Big[\| \theta^{(i)}_{\gamma}v^{N}_{k} - v^{N}_{k} \|_{{L}_{1}([0,T] \times \mathbb{R}, \hat{\lambda}; \mathbb{R})}\Big] \leq C_{T} \vert \gamma \vert^{l_{i}}.
        \end{equation*}
    \end{enumerate}
    Then~$((v^{N}_{k}(t,x))_{k \in \mathbb{N}_{0}, \, t \in [0,T], \, x \in \mathbb{R}})_{N \in \mathbb{N}}$ is tight in $L_{r}([0,T] \times \mathbb{R}, \hat{\lambda}; \ell_{1})$.
\end{lemma}

\begin{remark} \label{Paper02_Observations_equivalent_supression_mutations_positive_process}
    We collect here some observations regarding Lemma~\ref{Paper02_easy_tightness_criterion_L_p_spaces}.
    \begin{enumerate}
    \item[(a)] Although condition~(iii) is stated for arbitrary~$l_{1}, l_{2} > 0$, in applications we typically consider~$l_{1}, l_{2} \in (0,1)$, since the case~$l_{1}, l_{2} > 1$ allows one to apply the Kolmogorov continuity criterion and derive a stronger result than that stated in Lemma~\ref{Paper02_easy_tightness_criterion_L_p_spaces}.

    \item[(b)] Lemma~\ref{Paper02_easy_tightness_criterion_L_p_spaces} applies directly to finite-dimensional stochastic processes, since we can identify~$\mathbb{R}^d$ with a subspace of~$\ell_1$ via the canonical embedding~$\mathbb R^d \ni x \mapsto (x_1, \dots, x_d, 0, 0, \dots)$.
    \end{enumerate}
\end{remark}

\begin{proof}[Proof of Lemma~\ref{Paper02_easy_tightness_criterion_L_p_spaces}]
    We must establish that for every~$\varepsilon > 0$, there exists a compact set~$\mathcal{K}_{\varepsilon,T} \subset L_{r}([0,T] \times \mathbb{R}, \hat{\lambda}; \ell_{1})$ such that
    \begin{equation} \label{Paper02_tightness_prohorov_condition}
        \inf_{N \in \mathbb{N}} \; \mathbb{P}\Big((v_{k}^{N}(t,x))_{k \in \mathbb{N}_{0}, \, t \in [0,T], \, x \in \mathbb{R}} \in \mathcal{K}_{\varepsilon,T} \Big) \geq 1 - \varepsilon.
    \end{equation}
    It will suffice to construct~$\mathcal{K}_{\varepsilon,T}$ satisfying both~\eqref{Paper02_tightness_prohorov_condition} and the conditions of Lemma~\ref{Paper02_alternative_criterion_compactness_diaz_mayoral}. First, we observe that by~\eqref{Paper02_functional_space_L1_l1} and Jensen's inequality, and then by Fubini's theorem and condition (i), there exists~$C_{1} = C_1(T) >0$ such that
    \begin{equation*}
        \sup_{N \in \mathbb{N}} \; \mathbb{E}\Big[\| v^{N}\|_{L_{2r}([0,T] \times \mathbb{R}, \hat{\lambda}; \ell_{1})}\Big] \leq \sup_{N \in \mathbb{N}} \; \mathbb{E}\bigg[\int_{0}^{T} \int_{\mathbb{R}} \frac{\| v^{N}(t,x) \|_{\ell_{1}}^{2r}}{1 + \vert x \vert^{2}} \, dx \, dt\bigg]^{1/(2r)} \leq C_{1}.
    \end{equation*}
    Then, using the notation introduced in the statement of Lemma~\ref{Paper02_alternative_criterion_compactness_diaz_mayoral}, we have by Markov's inequality,
    \begin{equation} \label{Paper02_first_condition_tightness_L_p_higher_moments}
        \inf_{N \in \mathbb{N}} \; \mathbb{P}\Big(v^{N} \in \mathcal{K}_{(1)}(4C_{1}\varepsilon^{-1}, T)\Big) \geq 1 - \frac{\varepsilon}{4}.
    \end{equation}
    Moreover, by using condition~(ii), Fubini's theorem,~\eqref{Paper02_functional_space_real_functions_measure_time_space_box} and the fact that the measure~$\hat{\lambda}$ on~$[0,T] \times \mathbb{R}$ defined in~\eqref{Paper02_measure_time_space_box_i} is finite, we conclude that there exists a strictly increasing sequence~$(k_{h})_{h \in \mathbb{N}} = (k_{h}(\varepsilon))_{h \in \mathbb{N}} \subseteq \mathbb N_0$ such that for every~$h \in \mathbb{N}$,
    \begin{equation*}
        \sup_{N \in \mathbb{N}} \; \mathbb{E}\bigg[\sum_{k = k_{h}}^\infty \| v^{N}_{k} \|_{L_{1}([0,T] \times \mathbb{R}, \hat{\lambda}; \mathbb{R})}\bigg] \leq \frac{\varepsilon}{2^{2h + 2}}.
    \end{equation*}
    Recall the definition of~$\mathcal{K}_{(2)}((k_{h})_{h \in \mathbb{N}}, (2^{-h})_{h \in \mathbb{N}}, T)$ from the statement of Lemma~\ref{Paper02_alternative_criterion_compactness_diaz_mayoral}. By Markov's inequality and a union bound, we have
    \begin{equation} \label{Paper02_second_condition_tightness_uniform_tightness}
          \inf_{N \in \mathbb{N}} \; \mathbb{P}\Big(v^{N} \in \mathcal{K}_{(2)}((k_{h})_{h \in \mathbb{N}}, (2^{-h})_{h \in \mathbb{N}}, T)\Big) \geq 1 - \frac{\varepsilon}{4}.
    \end{equation}
    In the language of  Lemma~\ref{Paper02_alternative_criterion_compactness_diaz_mayoral}, it remains to determine~$C_{2}(\varepsilon,T) > 0$, $\beta(\varepsilon,T) > 1$ and $J(\varepsilon,T) \in \mathbb{N}$ such that
    \begin{equation} \label{Paper02_third_condition_tightness_equicontinuity}
          \inf_{N \in \mathbb{N}} \; \mathbb{P}\Big(v^{N} \in \mathcal{K}_{(3)}(C_{2}, \beta, J, T)\Big) \geq 1 - \frac{\varepsilon}{2}.
    \end{equation}
    Let $l = l_{1} \wedge l_{2} > 0$, and take~$\beta \in (1, 2^{l})$. For~$i,i' \in \{1,2\}$,~$N,j \in \mathbb{N}$ and~$k \in \mathbb{N}_{0}$, define the event
    \begin{equation*}
        \mathcal{A}^{i',N}_{i,k,j,\beta} \defeq \Big\{\| \theta^{(i)}_{(-1)^{i'}2^{-(k+j)}}v_{k}^{N} - v_{k}^{N} \|_{L_{1}([0,T] \times \mathbb{R}, \hat{\lambda}; \mathbb{R})} \geq C_{T}\beta^{-(k+j)}\Big\},
    \end{equation*}
    where~$C_{T}> 0$ is the constant introduced in condition~(iii). Then, by condition~(iii) and Markov's inequality, we have
    \begin{equation*}
        \sup_{N \in \mathbb{N}} \, \mathbb{P}\Big(\mathcal{A}^{i',N}_{i,k,j,\beta}\Big) \leq \left(\frac{\beta}{2^{l}}\right)^{k+j}.
    \end{equation*}
    Hence, since by construction we have~$\beta < 2^{l}$, by a union bound, we have for every~$J \in \mathbb{N}$,
    \begin{equation} \label{Paper02_union_bound_equicontinuity_integral_form}
         \sup_{N \in \mathbb{N}} \, \mathbb{P}\bigg(\bigcup_{i,i' \in \{1,2\}} \, \bigcup_{k \in \mathbb{N}_{0}} \, \bigcup_{j \geq J}\mathcal{A}^{i',N}_{i,k,j,\beta}\bigg) \leq 4\left(\frac{2^{l}}{2^{l} - \beta}\right)^{2}  \left(\frac{\beta}{2^{l}}\right)^{J}.
    \end{equation}
    Since by construction we have~$\beta \in (1,2^{l})$, the right-hand side of~\eqref{Paper02_union_bound_equicontinuity_integral_form} vanishes as $J \rightarrow \infty$. Therefore, by taking~$C_{2} = C_{T}$, $\beta \in (1,2^{l})$ and then~$J = J(\varepsilon,T) \in \mathbb{N}$ sufficiently large, we obtain~\eqref{Paper02_third_condition_tightness_equicontinuity}. Finally, by combining~\eqref{Paper02_first_condition_tightness_L_p_higher_moments},~\eqref{Paper02_second_condition_tightness_uniform_tightness} and~\eqref{Paper02_third_condition_tightness_equicontinuity}, we conclude that for every~$\varepsilon > 0$, there exist $C_{1},C_{2} > 0$, $\beta > 1$, $J \in \mathbb{N}$ and a strictly increasing sequence~$(k_{h})_{h \in \mathbb{N}} \subseteq \mathbb N_0$ such that letting
     \begin{equation*}
     \begin{aligned}
        & \mathcal{K}(C_{1}, C_{2}, (k_{h})_{h \in \mathbb{N}}, (2^{-h})_{h \in \mathbb N},\beta,J,T) \\ & \quad \defeq \mathcal{K}_{(1)}(C_{1}, T) \cap \mathcal{K}_{(2)}((k_{h})_{h \in \mathbb{N}}, (2^{-h})_{h \in \mathbb N}, T) \cap \mathcal{K}_{(3)}(C_{2},\beta, J, T),
    \end{aligned}
    \end{equation*}
    we have
    \begin{equation*}
        \inf_{N \in \mathbb{N}} \; \mathbb{P}\Big(v^{N} \in  \mathcal{K}(C_{1}, C_{2}, (k_{h})_{h \in \mathbb{N}}, (2^{-h})_{h \in \mathbb N},\beta,J,T) \Big) \geq 1 - \varepsilon.
    \end{equation*}
    Since Lemma~\ref{Paper02_alternative_criterion_compactness_diaz_mayoral} implies that~$\mathcal{K}(C_{1}, C_{2}, (k_{h})_{h \in \mathbb{N}}, (2^{-h})_{h \in \mathbb N},\beta,J,T)$ is relatively compact,~\eqref{Paper02_tightness_prohorov_condition} holds, which completes the proof.
\end{proof}

We are now ready to prove tightness of $(u^N)_{N \in \mathbb N}$ in the space $L_{4\deg q_-}([0,T] \times \mathbb{R}, \hat{\lambda}; \ell_{1})$.

\begin{lemma} \label{Paper02_tightness_muller_ratchet_L_p}
    Under the assumptions of Theorem~\ref{Paper02_deterministic_scaling_foutel_rodier_etheridge}, for any~$T > 0$, the sequence $$\Big((u^{N}(t))_{t \in [0,T]}, (u^{N}_{k}(t,x))_{k \in \mathbb{N}_{0}, \, t \in [0,T], \, x \in \mathbb{R}}\Big)_{N \in \mathbb{N}}$$ is tight in~$\mathcal{D}\left([0,T], (\mathcal{M}(\mathbb{R})^{\mathbb{N}_{0}}, d)\right) \times L_{4\deg q_-}([0,T] \times \mathbb{R}, \hat{\lambda}; \ell_{1})$.
\end{lemma}

\begin{proof}
    As mentioned at the beginning of Section~\ref{Paper02_tightness_L_P_l_1spaces}, by Proposition~\ref{Paper02_PDE_tightness_measures}, it will suffice to establish tightness in~$L_{4\deg q_-}([0,T] \times \mathbb{R}, \hat{\lambda}; \ell_{1})$. This follows if the sequence
    \[
    (u^{N})_{N \in \mathbb{N}} = \Big((u^{N}_{k}(t,x))_{k \in \mathbb{N}_{0},t \in [0,T],x \in \mathbb{R}}\Big)_{N \in \mathbb{N}}
    \]
    satisfies the conditions of Lemma~\ref{Paper02_easy_tightness_criterion_L_p_spaces}. Observe that~$(u^{N})_{N \in \mathbb{N}}$ satisfies condition~(i) by Theorem~\ref{Paper02_bound_total_mass} (and since $u^N(t,x)$ for $x \not\in L_N^{-1}\mathbb Z$ is defined by linear interpolation), condition~(ii) by~\eqref{Paper02_local_mass_subset_indices_definition} and Lemma~\ref{Paper02_lemma_control_local_density_high_number_mutations}, and condition~(iii) by~\eqref{Paper02_local_mass_subset_indices_definition},~\eqref{Paper02_shift_maps},~\eqref{Paper02_functional_space_real_functions_measure_time_space_box} and Lemma~\ref{Paper02_spatial_increment_bound_usual}. Therefore, Lemma~\ref{Paper02_easy_tightness_criterion_L_p_spaces} implies the desired tightness, which completes the proof.
\end{proof}

\subsection{Characterisation of the limiting density process} \label{Paper02_Characterisation_limiting_process}

We now proceed to characterise the limiting density process. Recall the definition of $\left(\mathcal M(\mathbb R)^{\mathbb N_0},d\right)$ introduced in~\eqref{Paper02_metric_product_topology}, and the definition of $L_{r}([0,T] \times \mathbb R, \hat \lambda; \ell_1)$ in~\eqref{Paper02_functional_space_L1_l1} for $T > 0$ and $r \in [1, \infty)$. To simplify notation, we will, throughout this subsection, assume that the sequence
\begin{equation*}
    \Big((u^{N}_{k}(t))_{k \in \mathbb{N}_{0}, \, t \in [0,T]}, (u^{N}_{k}(t,x))_{k \in \mathbb{N}_{0}, \, t \in [0,T], \, x \in \mathbb{R}}\Big)_{N \in \mathbb{N}}
\end{equation*}
is a subsequence that converges weakly in
\begin{equation*}
    \mathcal{D}\left([0,T], (\mathcal{M}(\mathbb{R})^{\mathbb{N}_{0}},d)\right) \times L_{4\deg q_-}([0,T] \times \mathbb{R}, \hat{\lambda}; \ell_{1})
\end{equation*}
to the random maps
\begin{equation} \label{Paper02_lim_meas_func_subseq}
    \Big((u_k(t))_{k \in \mathbb{N}_{0}, \, t \in [0,T]}, \, (v_{k}(t,x))_{k \in \mathbb{N}_{0}, \, t \in [0,T], \, x \in \mathbb{R}}\Big).
\end{equation}

Our next result, which is a direct consequence of Skorokhod's representation theorem, the properties of the Skorokhod $J_1$-topology and the fundamental theorem of calculus, establishes a useful relation between $u$ as an~$\mathcal{M}(\mathbb{R})^{\mathbb{N}_{0}}$-valued process, and~$v$ as an element of~$L_{4\deg q_-}([0,T] \times \mathbb{R}, \hat{\lambda}; \ell_{1})$. Recall the definition of $\mathcal{C}_c(\mathbb R)$ from Section~\ref{Paper02_introduction}. 
For $\rho \in \mathcal{M}(\mathbb R)$ and $\varphi \in \mathcal{C}_c(\mathbb R)$, let $\langle \rho, \varphi \rangle$ be defined as in Section~\ref{Paper02_introduction}.

\begin{lemma} \label{Paper02_relation_measure_valued_L_P_time_space_boxes_solutions}
    Under the assumptions of Theorem~\ref{Paper02_deterministic_scaling_foutel_rodier_etheridge}, for $(u,v)$ as defined in~\eqref{Paper02_lim_meas_func_subseq}, for every~$H \in \mathbb{N}$,~$(t_{h})_{h \in \llbracket H \rrbracket} \in [0,T)^{H}$,~$(\varphi_{h})_{h \in \llbracket H \rrbracket} \in \mathcal{C}_{c}(\mathbb{R})^{H}$ and~$(k_{h})_{h \in \llbracket H \rrbracket} \in (\mathbb{N}_{0})^{H}$, the following equality holds in distribution:
    \begin{equation*}
       (\langle u_{k_{h}}(t_{h}), \, \varphi_{h} \rangle)_{h \in \llbracket H \rrbracket} \overset{d}{=} \bigg(\lim_{t' \downarrow t_{h}} \frac{1}{t' - t_{h}}\int_{t_{h}}^{t'} \int_{\mathbb R} v_{k_{h}}(\tau, x)\varphi_{h}(x) \, dx \, d\tau\bigg)_{h \in \llbracket H \rrbracket}.
    \end{equation*}
\end{lemma}

\begin{proof}
From now on in the proof, whenever $\rho$ is a
$\mathcal D([0,T],\mathcal{M}(\mathbb{R}))$-valued random variable on the
probability space $(\Omega,\mathcal{F},\mathbb{P})$, where
$\mathcal{M}(\mathbb{R})$ is endowed with the vague topology, for $\omega \in \Omega$, we write
$\rho(t,\omega)$ for the value at $t\in[0,T]$ of the sample path
$\rho(\cdot,\omega)$. Likewise, if $\nu$ is an
$L_{4\deg q_-}([0,T]\times\mathbb{R},\hat\lambda;\ell_1)$-valued random
variable, we fix for $\omega \in \Omega$ a measurable representative of $\nu(\omega)$ and write
$\nu(t,x,\omega) =(\nu_k(t,x,\omega))_{k \in \mathbb N_0}$ for its value at $(t,x)\in[0,T]\times\mathbb{R}$.

By Lemma~\ref{Paper02_tightness_muller_ratchet_L_p} and Skorokhod's representation theorem, it is possible to construct the random maps
\begin{equation*} 
\left\{\begin{array}{l}
    \Big((u^{N}_{k}(\tau,\omega))_{k \in \mathbb{N}_{0}, \, \tau \in [0,T], \, \omega \in \Omega}, (u^{N}_{k}(\tau,x,\omega))_{k \in \mathbb{N}_{0}, \, \tau \in [0,T], \, x \in \mathbb{R}, \omega \in \Omega}\Big)_{N \in \mathbb{N}},  \\ \Big((u_k(\tau,\omega))_{k \in \mathbb{N}_{0}, \, \tau \in [0,T],\, \omega \in \Omega}, \, (v_{k}(\tau,x,\omega))_{k \in \mathbb{N}_{0}, \, \tau \in [0,T], \, x \in \mathbb{R}, \, \omega \in \Omega}\Big)
\end{array}\right.
\end{equation*}
on the same probability space $(\Omega, \mathcal F, \mathbb P)$ in such a way that for $\mathbb P$-almost every $\omega \in \Omega$,
\begin{equation} \label{Paper02_construction_Skorokhod}
\begin{aligned}
    & \lim_{N \rightarrow \infty} \Big((u^{N}_{k}(\tau,\omega))_{k \in \mathbb{N}_{0}, \, \tau \in [0,T]}, (u^{N}_{k}(\tau,x, \omega))_{k \in \mathbb{N}_{0}, \, \tau \in [0,T], \, x \in \mathbb{R}}\Big) \\ & \quad = \Big((u_k(\tau, \omega))_{k \in \mathbb{N}_{0}, \, \tau \in [0,T]}, \, (v_{k}(\tau,x, \omega))_{k \in \mathbb{N}_{0}, \, \tau \in [0,T], \, x \in \mathbb{R}}\Big)
\end{aligned}
\end{equation}
in $\mathcal{D}\left([0,T], (\mathcal{M}(\mathbb R)^{\mathbb N_0}, d)\right) \times L_{4\deg q_-}([0,T] \times \mathbb R, \hat \lambda; \ell_1)$. Recall that $d_{\textrm{vague}}$ introduced after~\eqref{Paper02_definition_u_N_as_radon_measure} is a metric that induces the vague topology on $\mathcal{M}(\mathbb R)$, and that the metric $d$ in $\mathcal{M}(\mathbb{R})^{\mathbb N_0}$ defined in~\eqref{Paper02_metric_product_topology} induces the product topology on $\mathcal{M}(\mathbb{R})^{\mathbb N_0}$. In particular, for every $k \in \mathbb N_0$ and every $\varphi \in \mathcal{C}_c(\mathbb R)$, the map $\mathcal{M}(\mathbb R)^{\mathbb N_0} \ni \rho = (\rho_j)_{j \in \mathbb N_0} \mapsto \langle \rho_k, \varphi \rangle$
is continuous with respect to the topology induced by $d$. Hence, by standard results on the Skorokhod $J_1$-topology (see e.g.~\cite[Exercise~3.13]{ethier2009markov}),~\eqref{Paper02_construction_Skorokhod} implies that for $\mathbb P$-almost every $\omega \in \Omega$ the following holds: for any $k \in \mathbb N_0$ and~$\varphi \in \mathcal C_c(\mathbb R)$, the sequence of real-valued càdlàg processes $\Big((\langle u^{N}_{k}(\tau,\omega), \varphi \rangle)_{\tau \in [0,T]}\Big)_{N \in \mathbb{N}}$ converges to $(\langle u_{k}(\tau,\omega), \varphi \rangle)_{\tau \in [0,T]}$ in $\mathcal{D}([0,T], \mathbb R)$ as $N \rightarrow \infty$.

Recall from Section~\ref{Paper02_introduction} that we let $\lambda$ denote the Lebesgue measure on $\mathbb R^d \; \forall \, d \in \mathbb N$. Since convergence in the Skorokhod $J_1$-topology implies pointwise convergence at any continuity point of the limit (see e.g.~\cite[Proposition~3.5.2]{ethier2009markov}), and since the set of discontinuities for any càdlàg path is countable (see e.g.~\cite[Lemma~3.5.1]{ethier2009markov}), we conclude that for $\mathbb P$-almost every $\omega \in \Omega$ the following holds: for any $k \in \mathbb N_0$ and~$\varphi \in \mathcal C_c(\mathbb R)$,
\begin{equation} \label{Paper02_relation_u_v_elementary_step_proof_v}
    \lim_{N \rightarrow \infty} \langle u^N_k(\tau,\omega), \varphi \rangle = \langle u_k(\tau,\omega), \varphi \rangle \quad \textrm{ for } \lambda\textrm{-almost every } \tau \in [0,T].
\end{equation}
Moreover, the set
\[
\mathcal K_{k,\varphi}(\omega)
\defeq
\Big\{ \big(\langle u_k(\tau,\omega),\varphi\rangle\big)_{\tau\in[0,T]}, \big(\langle u_k^N(\tau,\omega),\varphi\rangle\big)_{\tau\in[0,T]} : N\in\mathbb N
\Big\}
\subset \mathcal D([0,T],\mathbb R)
\]
is compact with respect to the Skorokhod $J_1$-topology, which implies (see e.g.~\cite[Exercise~3.16]{ethier2009markov}) that there exists $C_{k,\varphi,T}(\omega) > 0$ such that
\begin{equation} \label{Paper02_relation_u_v_elementary_step_proof_extra}
\sup_{h \in \mathcal K_{k,\varphi}(\omega)} \; \sup_{\tau \in [0,T]} \; \vert h(\tau) \vert \leq C_{k,\varphi,T}(\omega).
\end{equation}
Therefore, using~\eqref{Paper02_relation_u_v_elementary_step_proof_v},~\eqref{Paper02_relation_u_v_elementary_step_proof_extra} and dominated convergence, we conclude that for $\mathbb P$-almost every $\omega \in \Omega$, the following holds: for any $\varphi \in \mathcal C_c(\mathbb R)$, $k \in \mathbb N_0$ and $0 \leq t \leq t' \leq T$,
\begin{equation} \label{Paper02_relation_u_v_elementary_step_proof_vi}
\begin{aligned}
    \int_t^{t'} \langle u_k(\tau,\omega), \varphi \rangle \, d\tau & = \lim_{N \rightarrow \infty} \int_t^{t'} \langle u^N_k(\tau,\omega), \varphi \rangle \, d\tau \\ & = \lim_{N \rightarrow \infty} \int_t^{t'} \int_{\mathbb R} u^N_k(\tau,x,\omega) \varphi(x) \, dx \, d\tau \\ & = \int_t^{t'} \int_{\mathbb R} v_k(\tau,x,\omega) \varphi(x) \, dx \, d\tau,
\end{aligned}
\end{equation}
where for the third equality we used the fact that $(u^{N}_{k}(\tau,x, \omega))_{k \in \mathbb{N}_{0}, \, \tau \in [0,T], \, x \in \mathbb{R}}$ converges in $L_{4\deg q_-}([0,T] \times \mathbb R, \hat \lambda; \ell_1)$ to $(v_{k}(\tau,x, \omega))_{k \in \mathbb{N}_{0}, \, \tau \in [0,T], \, x \in \mathbb{R}}$ by~\eqref{Paper02_construction_Skorokhod}, and the fact that $\varphi$ has compact support.

Since the map~$[0,T] \ni \tau \mapsto \langle u_{k}(\tau, \omega), \, \varphi \rangle$ is right-continuous, by applying the fundamental theorem of calculus to~\eqref{Paper02_relation_u_v_elementary_step_proof_vi}, we conclude that for $\mathbb P$-almost every $\omega \in \Omega$, the following holds: for any $\varphi \in \mathcal C_c(\mathbb R)$, $k \in \mathbb N_0$ and $t \in [0,T)$,
\begin{equation} \label{Paper02_relation_u_v_elementary_step_proof_ii}
\begin{aligned}
    \langle u_{k}(t, \omega), \, \varphi \rangle & = \lim_{t' \downarrow t} \, \frac{1}{t' - t} \int_{t}^{t'} \langle u_{k}(\tau,\omega), \, \varphi \rangle \, d\tau \\
    & = \lim_{t' \downarrow t} \, \frac{1}{t' - t} \int_{t}^{t'} \int_{\mathbb R} v_{k}(\tau, x, \omega) \varphi(x) \, dx \, d\tau.
\end{aligned}
\end{equation}
By observing that~\eqref{Paper02_relation_u_v_elementary_step_proof_ii} holds for any choice of~$t \in [0,T)$,~$\varphi \in \mathcal{C}_{c}(\mathbb{R})$ and~$k \in \mathbb{N}_{0}$, the proof is complete.
\end{proof}

By Lemmas~\ref{Paper02_tightness_muller_ratchet_L_p} and~\ref{Paper02_relation_measure_valued_L_P_time_space_boxes_solutions}, to complete the proof of Proposition~\ref{Paper02_absolutely_continuity_limiting_process}, it remains to establish that~$v$ is a mild solution to the system of PDEs~\eqref{Paper02_PDE_scaling_limit}. This will follow from our next result. Recall that under the assumptions of Theorem~\ref{Paper02_deterministic_scaling_foutel_rodier_etheridge}, the function~$f = (f_{k})_{k \in \mathbb{N}_{0}} \in L_{\infty}(\mathbb{R}; \ell_{1})$ determines the initial condition~$\boldsymbol \eta^{N}$ in~\eqref{Paper02_initial_condition}. Recall the reaction term~$F = (F_{k})_{k \in \mathbb{N}_{0}}$ defined in~\eqref{Paper02_reaction_term_PDE}, and recall from before~\eqref{Paper02_semigroup_BM_action_ell_1_functions} that we let~$\{P_{t}\}_{t \geq 0}$ be the semigroup corresponding to a Brownian motion on $\mathbb R$ run at speed~$m \in (0, \infty)$. We also recall the Green's function representation of $u^N$ given in Lemma~\ref{Paper02_lemma_little_u_k_N_hat_semimartingale}. Moreover, recall from after~\eqref{Paper02_inf_gen_rw} that we let $\{P^{N}_{t}\}_{t \geq 0}$ be the semigroup associated with a simple symmetric random walk~$(X^{N}(t))_{t \geq 0}$ on~$L_{N}^{-1}\mathbb{Z}$ with total jump rate~$m_{N}$, and recall the definition of the normalised transition density $p^{N}$ of~$X^{N}$ in~\eqref{Paper02_definition_p_N_test_function_green}. Also, for $t > 0$ and $x \in \mathbb R$, recall the definition of the Gaussian kernel $p(t,x)$ in~\eqref{Paper02_Gaussian_kernel}.

\begin{lemma} \label{Paper02_weak_convergence_mild_solution}
   Under the assumptions of Theorem~\ref{Paper02_deterministic_scaling_foutel_rodier_etheridge}, the $L_{4\deg q_{-}}([0,T] \times \mathbb{R}, \hat{\lambda}; \ell_{1})$-valued random variable $v$ defined in~\eqref{Paper02_lim_meas_func_subseq} satisfies the following identity:
   \begin{equation*}
       \mathbb{E}\Bigg[\int_{0}^{T} \int_{\mathbb{R}} \, \sum_{k = 0}^{\infty} \frac{1}{1 + \vert x \vert^{2}} \left\vert v_{k}(t,x) - (P_{t}f_{k})(x) - \int_{0}^{t} \Big(P_{t-\tau}F_{k}(v(\tau, \cdot))\Big)(x) \,d\tau\right\vert \, dx \, dt\Bigg] = 0.
   \end{equation*}
\end{lemma}

\begin{proof}
   By Fubini's theorem, the lemma will be proved after establishing that for every~$k \in \mathbb{N}_{0}$,
    \begin{equation} \label{Paper02_equivalent_formulation_easier_mild_solution_coordinatewise}
    \begin{aligned}
       & \mathbb{E}\Bigg[\int_{0}^{T} \int_{\mathbb{R}} \, \frac{1}{1 + \vert x \vert^{2}} \left\vert v_{k}(t,x) - (P_{t}f_{k})(x) - \int_{0}^{t} \Big(P_{t-\tau}F_{k}(v(\tau, \cdot))\Big)(x) \,d\tau\right\vert \, dx \, dt\Bigg] \\
       & \quad = 0.
    \end{aligned}
   \end{equation}
   By Skorokhod's representation theorem, it is possible to construct the sequence of density processes~$\Big((u^{N}_{k}(t,x))_{k \in \mathbb{N}_{0}, \, t \in [0,T], \, x \in \mathbb{R}}\Big)_{N \in \mathbb{N}}$ and~$(v_{k}(t,x))_{k \in \mathbb{N}_{0}, \, t \in [0,T], \, x \in \mathbb{R}}$ on the same probability space in such a way that
   \begin{equation} \label{Paper02_skorkohod_representation_thm_a.s._convergence_L_p}
       (u^{N}_{k}(t,x))_{k \in \mathbb{N}_{0}, \, t \in [0,T], \, x \in \mathbb{R}} \rightarrow (v_{k}(t,x))_{k \in \mathbb{N}_{0}, \, t \in [0,T], \, x \in \mathbb{R}}
   \end{equation}
   almost surely in  $L_{4\deg q_{-}}([0,T] \times \mathbb{R}, \hat{\lambda}; \ell_{1})$ as $N \rightarrow \infty$. Since the measure~$\hat{\lambda}$ on~$[0,T] \times \mathbb{R}$, defined in~\eqref{Paper02_measure_time_space_box_i}, is finite, by~\eqref{Paper02_functional_space_L1_l1} and Jensen's inequality the limit in~\eqref{Paper02_skorkohod_representation_thm_a.s._convergence_L_p} also holds almost surely in~$L_{r}([0,T] \times \mathbb{R}, \hat{\lambda}; \ell_{1})$, for every~$r \in [1,4\deg q_{-}]$. In particular, since convergence of a sequence of scalar-valued functions in~${L}_{r}([0,T] \times \mathbb{R}, \hat{\lambda}; \mathbb{R})$ implies the existence of a subsequence which converges $\lambda$-almost everywhere~in $[0,T] \times \mathbb{R}$ (see e.g.~\cite[Corollary~2.32]{folland1999real}),~\eqref{Paper02_skorkohod_representation_thm_a.s._convergence_L_p} implies that for every~$k \in \mathbb{N}_{0}$, almost surely there exists a subsequence~$(N^{(k)}_{j})_{j \in \mathbb{N}}$ such that for $\lambda$-almost every~$(t,x) \in [0,T] \times \mathbb{R}$,
   \begin{equation} \label{Paper02_limit_lambda_a.e.}
       \lim_{j \rightarrow \infty} u^{N^{(k)}_{j}}_{k}(t,x) = v_{k}(t,x).
   \end{equation}
   Moreover, by the definition of the reaction term~$F = (F_{k})_{k \in \mathbb{N}_{0}}$ in~\eqref{Paper02_reaction_term_PDE}, by the fact that by Assumption~\ref{Paper02_assumption_polynomials}, $q_+,q_-: [0, \infty) \rightarrow [0,\infty)$ are polynomials with $0 \leq \deg q_+ < \deg q_-$, and by \eqref{Paper02_aux_L_p_dom_c_ii} in Lemma~\ref{Paper02_standard_L_p_gaussian_lem} in the appendix, for all $k \in \mathbb N_0$ and $(t,x) \in [0,T] \times \mathbb R$, the following limit holds almost surely
   \begin{equation} \label{Paper02_limit_heat_kernel_immediate_assumption_a.s._convergence}
       \int_{0}^{t}  \Big(P_{t-\tau}F_{k}(v(\tau, \cdot))\Big)(x) \,d\tau = \lim_{N \rightarrow \infty} \int_{0}^{t}  \Big(P_{t-\tau}F_{k}(u^{N}(\tau, \cdot))\Big)(x) \, d\tau.
   \end{equation}
   Combining~\eqref{Paper02_limit_heat_kernel_immediate_assumption_a.s._convergence} and~\eqref{Paper02_limit_lambda_a.e.}, and then using Fatou's lemma (up to a subsequence if necessary), we conclude that for any $k \in \mathbb N_0$, the term on the left-hand side of~\eqref{Paper02_equivalent_formulation_easier_mild_solution_coordinatewise} is bounded by
   \begin{equation} \label{Paper02_inequality_fatous_lemma_mild_solution}
    \begin{aligned}
        & \mathbb{E}\bigg[\int_{0}^{T} \int_{\mathbb{R}} \, \frac{1}{1 + \vert x \vert^{2}} \left\vert v_{k}(t,x) - (P_{t}f_{k})(x) - \int_{0}^{t} \Big(P_{t-\tau}F_{k}(v(\tau, \cdot))\Big)(x) \,d\tau\right\vert \, dx \, dt\bigg] \\ & \quad \leq \limsup_{N \rightarrow \infty} \mathbb{E}_{\boldsymbol{\eta}^N}\bigg[\int_{0}^{T} \int_{\mathbb{R}} \, \frac{1}{1 + \vert x \vert^{2}} \bigg\vert u^{N}_{k}(t,x) - (P_{t}f_{k})(x) \\
        & \hspace{6cm} - \int_{0}^{t} \Big(P_{t-\tau}F_{k}(u^{N}(\tau, \cdot))\Big)(x) \,d\tau\bigg\vert \, dx \, dt\bigg].
    \end{aligned}
   \end{equation}
   Therefore,~\eqref{Paper02_equivalent_formulation_easier_mild_solution_coordinatewise} will be proved after establishing that the right-hand side of~\eqref{Paper02_inequality_fatous_lemma_mild_solution} vanishes.
   
   Our strategy will be to use the Green's function representation given in Lemma~\ref{Paper02_lemma_little_u_k_N_hat_semimartingale} to prove that the limit on the right-hand side of~\eqref{Paper02_inequality_fatous_lemma_mild_solution} is equal to~$0$. Observe that for every~$N \in \mathbb{N}$ and any non-negative real-valued function~$g \in {L}_{1}([0,T] \times \mathbb{R}, \hat{\lambda}; \mathbb{R})$,
   \begin{equation} \label{Paper02_simple_transformation_integral_space_sum_grid}
   \begin{aligned}
       \int_{0}^{T} \int_{\mathbb{R}} \frac{g(t,x)}{1 + \vert x \vert^{2}} \, dx \, dt & = \int_{0}^{T}  \sum_{x \in L_{N}^{-1}\mathbb{Z}} \frac{1}{L_{N}} \int_{0}^{1} \frac{g(t, x + hL_{N}^{-1})}{1 + \vert x + hL_{N}^{-1}\vert^{2}} \, dh \, dt \\ & \lesssim \int_{0}^{T} \sum_{x \in L_{N}^{-1}\mathbb{Z}} \frac{1}{L_{N} (1 + \vert x \vert^{2})} \int_{0}^{1} g(t, x + hL_{N}^{-1}) \, dh \, dt,
    \end{aligned}
   \end{equation}
   where for the last inequality we used the fact that by Assumption~\ref{Paper02_scaling_parameters_assumption}, $L_{N} \rightarrow \infty$ as $N \rightarrow \infty$. Moreover, recall that the function~$u^{N}_{k}(t,\cdot)$ was defined after~\eqref{Paper02_approx_pop_density_definition} via a linear interpolation in space between points in $L_N^{-1}\mathbb Z$, and so for $x \in L_N^{-1}\mathbb Z$ and $h \in [0,1]$ we have
   \[
   \vert u^N_k(t,x) - u^N_k(t, x+hL_N^{-1}) \vert \leq  \vert u^N_k(t,x) - u^N_k(t, x+L_N^{-1}) \vert.
   \]
   Hence, by~\eqref{Paper02_simple_transformation_integral_space_sum_grid} and the triangle inequality, for $N \in \mathbb N$, we can bound the expectation on the right-hand side of~\eqref{Paper02_inequality_fatous_lemma_mild_solution} by
   \begin{equation*}
    \begin{aligned}
    & \mathbb{E}_{\boldsymbol{\eta}^N}\bigg[\int_{0}^{T} \int_{\mathbb{R}} \, \frac{1}{1 + \vert x \vert^{2}} \left\vert u^{N}_{k}(t,x) - (P_{t}f_{k})(x) - \int_{0}^{t} \Big(P_{t-\tau}F_{k}(u^{N}(\tau, \cdot))\Big)(x) \,d\tau\right\vert \, dx \, dt\bigg] \\
    & \quad \lesssim \int_0^T \sum_{x \in L_N^{-1}\mathbb Z} \frac{1}{L_N(1 + \vert x \vert^2)} \int_0^1  \bigg(\mathbb E_{\boldsymbol \eta^N} \left[\vert u^N_k(t,x) - u^N_k(t, x+ L_N^{-1}) \vert\right]
    \\ & \qquad \qquad \; + \mathbb E_{\boldsymbol \eta^N} \bigg[\bigg\vert u^N_k(t,x) - (P_tf_k)(x + hL_N^{-1}) \\
    & \qquad \qquad \qquad \qquad \qquad - \int_0^t \left(P_{t-\tau} F_k(u^N(\tau, \cdot))\right)(x+hL_N^{-1}) \, d\tau \bigg\vert\bigg]\bigg) \, dh \, dt.
    \end{aligned}
   \end{equation*}
   Therefore, by Lemma~\ref{Paper02_lemma_little_u_k_N_hat_semimartingale},~\eqref{Paper02_definition_little_u_k_N_hat} and the triangle inequality,
   \begin{align} 
        & \mathbb{E}_{\boldsymbol \eta^N}\bigg[\int_{0}^{T} \int_{\mathbb{R}} \, \frac{1}{1 + \vert x \vert^{2}} \left\vert u^{N}_{k}(t,x) - (P_{t}f_{k})(x) - \int_{0}^{t} \Big(P_{t-\tau}F_{k}(u^{N}(\tau, \cdot))\Big)(x) \,d\tau\right\vert \, dx \, dt\bigg] \notag\\ & \; \lesssim \int_{0}^{T} \sum_{x \in L_{N}^{-1}\mathbb{Z}} \frac{1}{L_{N}(1 + \vert x \vert^{2})} \mathbb{E}_{\boldsymbol \eta^N}\Big[\Big\vert M^{N,t,x}_{\{k\}}(t) \Big\vert\Big] \, dt \notag \\ & \quad \; + \int_{0}^{T} \sum_{x \in L_{N}^{-1}\mathbb{Z}} \frac{1}{L_{N}(1 + \vert x \vert^{2})} \mathbb{E}_{\boldsymbol{\eta}^N}\Big[\Big\vert u^{N}_{k}(t,x) - u^{N}_{k}(t,x + L_{N}^{-1}) \Big\vert\Big] \, dt \notag \\ & \quad \; + \int_{0}^{T} \sum_{x \in L_{N}^{-1}\mathbb{Z}} \frac{1}{L_{N}(1 + \vert x \vert^{2})} \label{Paper02_proof_mild_solution_step_i} \\
        & \hspace{2.5cm} \cdot \int_{0}^{1} \mathbb{E}_{\boldsymbol \eta^N}\left[\left\vert A^{N,t,x}_{\{k\}}(t) - \int_{0}^{t} \Big(P_{t-\tau}F_{k}(u^{N}(\tau, \cdot))\Big)(x + hL_{N}^{-1}) \,d\tau \right\vert\right]dh\, dt \notag
        \\ & \quad \; + \int_{0}^{T} \sum_{x \in L_{N}^{-1}\mathbb{Z}} \frac{1}{L_{N}(1 + \vert x \vert^{2})} \int_{0}^{1} \mathbb E_{\boldsymbol \eta^N} \left[\Big\vert (P_{t}f_{k})(x + hL_{N}^{-1}) - (P^{N}_{t}u^{N}_{k}(0, \cdot))(x) \Big\vert \right] \, dh \, dt, \notag
    \end{align}
   where for every~$N \in \mathbb{N}$,~$k \in \mathbb{N}_{0}$,~$t \in [0,T]$ and~$x \in L_{N}^{-1}\mathbb{Z}$, $\big(M^{N,t,x}_{\{k\}}(\tau)\big)_{\tau \in [0,t]}$ is a càdlàg square integrable martingale with $M^{N,t,x}_{\{k\}}(0) = 0$ whose predictable bracket process is given in~\eqref{Paper02_predictable_bracket_process_uN_k_hat}, and $\big(A^{N,t,x}_{\{k\}}(\tau)\big)_{\tau \in [0,t]}$ is a finite variation process defined in~\eqref{Paper02_finite_variation_process_uN}.
   We now bound each of the terms on the right-hand side of~\eqref{Paper02_proof_mild_solution_step_i} separately.

   \medskip

\myemph{Step (1): Bound on the norm of the martingale term}
   For the first term on the right-hand side of~\eqref{Paper02_proof_mild_solution_step_i}, by using Jensen's inequality and the BDG inequality as in~\eqref{Paper02_BDG_inequality} and then applying~\eqref{Paper02_predictable_bracket_process_uN_k_hat}, we have that for every~$N \in \mathbb{N}$,~$k \in \mathbb{N}_{0}$,~$t \in [0,T]$ and~$x \in L_{N}^{-1}\mathbb{Z}$,
   \begin{equation*}
    \begin{aligned}
        & \mathbb{E}_{\boldsymbol \eta^N}\Big[\Big\vert M^{N,t,x}_{\{k\}}(t) \Big\vert\Big]^{2} \\ & \quad \lesssim  \frac{1}{NL_{N}^{2}} \, \sum_{y \in L_{N}^{-1}\mathbb{Z}} \, \int_{0}^{t} p^{N}(t-\tau,y - x)^{2} \mathbb{E}_{\boldsymbol \eta^N}\Big[ F_{k}^{+}(u^{N}(\tau-, y))\Big] \, d\tau \\ & \quad \quad + \frac{m_{N}}{2NL_{N}^{4}} \, \sum_{y \in L_{N}^{-1}\mathbb{Z}} \, \int_{0}^{t} \Big( \nabla_{L_N} p^{N}(t-\tau,y - L_N^{-1}- x)^2 \\ & \qquad \qquad \qquad \qquad \qquad \qquad \qquad + \nabla_{L_N} p^{N}(t-\tau,y - x)^{2} \Big) \mathbb{E}_{\boldsymbol \eta^N}\Big[u^{N}_{k}(\tau-, y)\Big]  \, d\tau \\ 
        & \quad \lesssim_{T} \frac{1}{NL_{N}^{2}} \, \sum_{y \in L_{N}^{-1}\mathbb{Z}} \, \int_{0}^{T} p^{N}(\tau,y - x)^{2} \, d\tau + \frac{m_{N}}{NL_{N}^{4}} \, \sum_{y \in L_{N}^{-1}\mathbb{Z}} \, \int_{0}^{T}  \nabla_{L_N} p^{N}(\tau,y - x)^{2} d\tau \\ & \quad \lesssim_{T} N^{-1},
    \end{aligned}
   \end{equation*}
   where for the second inequality we used estimate~\eqref{Paper02_good_behaviour_indices_modified_reaction_term} and Theorem~\ref{Paper02_bound_total_mass}, and for the last inequality, for the first term we used the fact that~$p^{N}(\tau, y-x) \leq L_{N}$, (by~\eqref{Paper02_definition_p_N_test_function_green}) and identity~\eqref{Paper02_simple_identity_total_sum_pN}, and for the second term we used~\eqref{Paper02_discrete_grad}, the standard random walk estimate~\eqref{Paper02_random_walk_estimates_v} from Lemma~\ref{Paper02_random_walks_lemma} in the appendix and the fact that, by Assumption~\ref{Paper02_scaling_parameters_assumption}, $m_N/L_{N}^2 \rightarrow m \in (0,\infty)$ as $N \rightarrow \infty$. We then conclude that
   \begin{equation} \label{Paper02_mild_solution_step_iii}
       \lim_{N \rightarrow \infty} \int_{0}^{T} \sum_{x \in L_{N}^{-1}\mathbb{Z}} \frac{1}{L_{N}(1 + \vert x \vert^{2})} \mathbb{E}_{\boldsymbol \eta^N}\Big[\Big\vert M^{N,t,x}_{\{k\}}(t) \Big\vert\Big] \, dt = 0.
   \end{equation}

   \medskip
   
\myemph{Step (2): Bound on the error term arising from the linear interpolation}
   For the second term on the right-hand side of~\eqref{Paper02_proof_mild_solution_step_i}, applying Fubini’s theorem and then~\eqref{Paper02_local_mass_subset_indices_definition} and Lemma~\ref{Paper02_spatial_increment_bound_usual} yields
   \begin{equation*}
    \begin{aligned}
    & \int_{0}^{T} \sum_{x \in L_{N}^{-1}\mathbb{Z}} \frac{1}{L_{N}(1 + \vert x \vert^{2})} \mathbb{E}_{\boldsymbol \eta^N}\Big[\Big\vert u^{N}_{k}(t,x) - u^{N}_{k}(t,x + L_{N}^{-1}) \Big\vert\Big] \, dt \\
    & \quad \lesssim_{T} \sum_{x \in L_{N}^{-1}\mathbb{Z}} \frac{1}{L_{N}^{3/2}(1 + \vert x \vert^{2})} \\
    & \quad \lesssim L_{N}^{-1/2},
    \end{aligned}
   \end{equation*}
   and so, since $L_N \rightarrow \infty$ as $N \rightarrow \infty$ by Assumption~\ref{Paper02_scaling_parameters_assumption},
   \begin{equation} \label{Paper02_mild_solution_step_iv}
       \lim_{N \rightarrow \infty} \int_{0}^{T} \sum_{x \in L_{N}^{-1}\mathbb{Z}} \frac{1}{L_{N}(1 + \vert x \vert^{2})} \mathbb{E}_{\boldsymbol \eta^N}\Big[\Big\vert u^{N}_{k}(t,x) - u^{N}_{k}(t,x + L_{N}^{-1}) \Big\vert\Big] \, dt = 0.
   \end{equation}

   \medskip

\myemph{Step (3): Bound on the error arising from approximating $\{P_t\}_{t \geq 0}$ by $\{P^N_t\}_{t \geq 0}$}
   For the third term on the right-hand side of~\eqref{Paper02_proof_mild_solution_step_i}, we note that by the definition of the Brownian semigroup~$\{P_{t}\}_{t \geq 0}$ before~\eqref{Paper02_semigroup_BM_action_ell_1_functions} and the definition of $p(t,x)$ in~\eqref{Paper02_Gaussian_kernel}, for any~$N \in \mathbb N$, $t \in [0,T]$, $x \in L_{N}^{-1}\mathbb{Z}$ and~$h \in [0,1]$,
   \begin{equation} \label{Paper02_rewriting_action_brownian_semigroup}
    \begin{split}
        & \int_{0}^{t} \Big(P_{t - \tau} F_{k}(u^{N}(\tau, \cdot))\Big)(x + hL_{N}^{-1}) \, d\tau \\ & \, = \sum_{y \in L_{N}^{-1}\mathbb{Z}} \, \int_{0}^{t} \frac{1}{L_{N}} \int_{0}^{1} p\Big(t - \tau, y-x + (h'-h)L_N^{-1}\Big) F_{k}(u^{N}(\tau, y + h'L_{N}^{-1})) \, dh' \, d\tau.
    \end{split}
    \raisetag{-1.5cm}
   \end{equation}
   Combining~\eqref{Paper02_finite_variation_process_uN} and~\eqref{Paper02_rewriting_action_brownian_semigroup}, and then applying the triangle inequality and Fubini's theorem, we conclude that for any~$N \in \mathbb N$, $t \in [0,T]$, $x \in L_{N}^{-1}\mathbb{Z}$ and $h \in [0,1]$,
    \begin{align}
        & \mathbb{E}_{\boldsymbol \eta^N}\bigg[\bigg\vert A^{N,t,x}_{\{k\}}(t) - \int_{0}^{t} \Big(P_{t-\tau}F_{k}(u^{N}(\tau, \cdot))\Big)(x + hL_{N}^{-1}) \,d\tau \bigg\vert\bigg] \notag
        \\ & \quad \leq \frac{1}{L_{N}} \sum_{y \in L_{N}^{-1}\mathbb{Z}} \int_{0}^{t} \left\vert p^{N}(t-\tau, y-x) - p(t-\tau,y-x)\right\vert \mathbb{E}_{\boldsymbol \eta^N}\Big[\Big\vert F_{k}(u^{N}(\tau-, y))\Big\vert\Big] \, d\tau \notag
        \\ & \quad \quad + \frac{1}{L_{N}}\sum_{y \in L_{N}^{-1}\mathbb{Z}}\int_{0}^{t} p(t-\tau,y-x) \notag \\[-6mm]
        & \hspace{4cm} \cdot \int_{0}^{1} \mathbb{E}_{\boldsymbol \eta^N} \Big[\Big\vert F_{k}(u^{N}(\tau-,y)) - F_{k}(u^{N}(\tau, y + h'L_{N}^{-1}))\Big\vert\Big] \, dh' d\tau \label{Paper02_mild_solution_step_v}
        \\ & \quad \quad + \frac{1}{L_{N}} \sum_{y \in L_{N}^{-1}\mathbb{Z}} \int_{0}^{t} \int_{0}^{1} \mathbb{E}_{\boldsymbol \eta^N}\Big[\Big\vert F_{k}(u^{N}(\tau, y + h'L_{N}^{-1})) \Big\vert\Big] \notag
        \\ & \qquad \qquad \qquad \qquad \qquad \cdot \left\vert p(t-\tau, y-x) - p(t-\tau, y- x + (h'-h)L_N^{-1})\right\vert \, dh'  d\tau. \notag
    \end{align}
   
   We now bound each of the terms on the right-hand side of~\eqref{Paper02_mild_solution_step_v} separately. For the first term, recall from~\eqref{Paper02_reaction_term_PDE} and~\eqref{Paper02_reaction_predictable_bracket_process} that $\vert F_k(u) \vert \leq F^+_k(u) \; \forall \, u \in \ell_1^+$. Hence, by~\eqref{Paper02_good_behaviour_indices_modified_reaction_term}, there exists $C^{(1)}_{q_+,q_-,f}(T)>0$ such that for any $N \in \mathbb N$, $t \in [0,T]$ and $x \in L_N^{-1}\mathbb Z$,
   \begin{equation} \label{Paper02_mild_solution_step_vi_modified}
    \begin{aligned}
        & \frac{1}{L_{N}} \sum_{y \in L_{N}^{-1}\mathbb{Z}} \int_{0}^{t} \left\vert p^{N}(t-\tau, y-x) - p(t-\tau,y-x)\right\vert \mathbb{E}_{\boldsymbol \eta^N}\Big[\Big\vert F_{k}(u^{N}(\tau-, y))\Big\vert\Big] \, d\tau \\ & \quad \leq C^{(1)}_{q_+,q_-,f}(T) \frac{1}{L_N} \sum_{y \in L_{N}^{-1}\mathbb{Z}} \int_{0}^{t} \left\vert p^{N}(t-\tau, y-x) - p(t-\tau,y-x)\right\vert \, d\tau \\ & \quad \leq C^{(1)}_{q_+,q_-,f}(T) \frac{1}{L_N} \sum_{y \in L_N^{-1}\mathbb Z} \int_{0}^{T} \left\vert p^{N}(\tau, y) - p(\tau,y)\right\vert \, d\tau.
    \end{aligned}
   \end{equation}
   Fix $\varepsilon \in (0,T)$, and let $(W(t))_{t \geq 0}$ denote a Brownian motion run at rate $m$. Also, recall from before~\eqref{Paper02_inf_gen_rw} that for $N \in \mathbb N$, we let $(X^{N}(t))_{t \geq 0}$ denote a simple symmetric random walk on $L_{N}^{-1}\mathbb{Z}$ with total jump rate $m_{N}$. Suppose $N$ is sufficiently large that $\varepsilon^{-1} > L_N^{-1}$. By symmetry, and then by the triangle inequality, identity~\eqref{Paper02_simple_identity_total_sum_pN} and since $p(\tau, \cdot)$ is decreasing on $[0, \infty)$ for $\tau > 0$ with $\int_0^\infty p(\tau,z) \, dz = 1/2$, we can write
   \begin{equation} \label{Paper02_mild_solution_step_modified_ii}
    \begin{split}
        & \frac{1}{L_N} \sum_{y \in L_{N}^{-1}\mathbb{Z}} \int_{0}^{T} \left\vert p^{N}(\tau, y) - p(\tau,y)\right\vert \, d\tau
        \\ & \quad \leq 2 \int_0^{\infty} \int_0^T \left\vert p^{N}(\tau, L_N^{-1} \lfloor yL_N \rfloor) - p(\tau, L_N^{-1} \lfloor yL_N \rfloor)\right\vert \, d\tau \, dy
        \\ & \quad \leq 2 \int_0^{\varepsilon^{-1}} \int_\varepsilon^T \left\vert p^{N}(\tau, L_N^{-1} \lfloor yL_N \rfloor) - p(\tau, L_N^{-1} \lfloor yL_N \rfloor)\right\vert \, d\tau \, dy \\
        & \qquad + 2 \int_0^\varepsilon \left(\frac 32 + L_N^{-1}p(\tau,0)\right) \, d\tau
        \\ & \quad \quad + 2\int_0^T \Big(\mathbb P_0(X^N(\tau) \geq L_N^{-1} \lfloor \varepsilon^{-1}L_N \rfloor) + \mathbb P_0(W(\tau) \geq L_N^{-1} \lfloor \varepsilon^{-1}L_N \rfloor - L_N^{-1})\Big)  \, d\tau.
    \end{split}
    \raisetag{-10.5em}
   \end{equation}
   By Lemma~\ref{Paper2_lem_LCLT} in the appendix, which is a local central limit theorem for continuous-time random walks, and using dominated convergence and that $m_N/L_N^2 \rightarrow m \in (0, \infty)$ as $N \rightarrow \infty$ by Assumption~\ref{Paper02_scaling_parameters_assumption}, the first term on the right-hand side of~\eqref{Paper02_mild_solution_step_modified_ii} converges to $0$ as $N \rightarrow \infty$.

   For the third term on the right-hand side of~\eqref{Paper02_mild_solution_step_modified_ii}, note that for $t \geq 0$ and $N \in \mathbb N$, letting $(\mathcal N^+(\tau))_{\tau \geq 0}$ and $(\mathcal N^-(\tau))_{\tau \geq 0}$ denote independent Poisson processes with rate 1, we have
   \begin{equation} \label{Paper02_second_moment_random_walk}
    \begin{aligned}
        \mathbb E_0 \left[X^N(t)^2\right] & = \mathbb E\left[L_N^{-2}(\mathcal N^+(m_Nt/2) - \mathcal N^-(m_Nt/2))^2 \right] \\
        & = L_N^{-2}\left(2\left(\tfrac 14 m_N^2t^2 + \tfrac 12 m_Nt\right) - 2 \cdot \tfrac 14 m_N^2t^2 \right) \\ & = m_NL_N^{-2}t.
    \end{aligned}
   \end{equation}
   Therefore, using Markov's inequality and that $m_N/L_N^2 \rightarrow m \in (0, \infty)$ and $L_N\to \infty$ as $N \rightarrow \infty$ by Assumption~\ref{Paper02_scaling_parameters_assumption},
   \begin{equation*}
    \begin{aligned}
        & \limsup_{N \rightarrow \infty} \int_0^T \Big(\mathbb P_0(X^N(\tau) \geq L_N^{-1} \lfloor \varepsilon^{-1}L_N \rfloor) + \mathbb P_0(W(\tau) \geq L_N^{-1} \lfloor \varepsilon^{-1}L_N \rfloor - L_N^{-1})\Big)  \, d\tau \\ & \quad \leq \limsup_{N \rightarrow \infty} \int_0^T \left(m_NL_N^{-2}\tau  \cdot (\varepsilon^{-1}/2)^{-2} + m\tau (\varepsilon^{-1}/2)^{-2}\right) \, d\tau  \\ & \quad = 4 \varepsilon^2 mT^2.
    \end{aligned}
   \end{equation*}

   For the second term on the right-hand side of~\eqref{Paper02_mild_solution_step_modified_ii}, we have
   \begin{equation*}
    \begin{aligned}
        \limsup_{N \rightarrow \infty} \int_0^\varepsilon \left(\frac 32  + L_N^{-1}p(\tau,0)\right) \, d\tau & = \limsup_{N \rightarrow \infty} \left(\frac{3\varepsilon}2 + L_N^{-1}\cdot\frac{2}{\sqrt{2\pi m}}\varepsilon^{1/2}\right) \\ & = 3\varepsilon/2.
    \end{aligned}
   \end{equation*}
   Since $\varepsilon \in (0,T)$ can be taken arbitrarily small, it now follows from~\eqref{Paper02_mild_solution_step_modified_ii} that
   \begin{equation} \label{Paper02_mild_solution_step_modified_iii}
       \lim_{N \rightarrow \infty} \frac{1}{L_N} \sum_{y \in L_{N}^{-1}\mathbb{Z}} \int_{0}^{T} \left\vert p^{N}(\tau, y) - p(\tau,y)\right\vert \, d\tau = 0.
   \end{equation}
   
   For the second term on the right-hand side of~\eqref{Paper02_mild_solution_step_v}, fix $r \in (1,2)$; then by using H\"{o}lder's inequality with exponents~$r$ and~$r/(r-1)$, and Jensen's inequality, and then using~\eqref{Paper02_good_behaviour_indices_modified_reaction_term} and the fact that $\vert F_k(u) \vert \leq F^+_k(u) \; \forall \, u \in \ell_1^+$ by~\eqref{Paper02_reaction_term_PDE} and~\eqref{Paper02_reaction_predictable_bracket_process}, there exists $C^{(2)}_{q_+,q_-,f,r}(T) > 0$ such that for any $N \in \mathbb N$, $t \in [0,T]$ and $x \in L_N^{-1}\mathbb Z$,
   \begin{equation} \label{Paper02_mild_solution_step_modified_iv}
    \begin{split}
        & \frac{1}{L_{N}}\sum_{y \in L_{N}^{-1}\mathbb{Z}}\int_{0}^{t} p(t-\tau, y-x) \int_{0}^{1} \mathbb{E}_{\boldsymbol \eta^N} \Big[\Big\vert F_{k}(u^{N}(\tau-,y)) - F_{k}(u^{N}(\tau, y + h'L_{N}^{-1}))\Big\vert\Big] \, dh' \, d\tau 
        \\ & \quad \leq \frac{1}{L_{N}} \sum_{y \in L_{N}^{-1}\mathbb{Z}} e^{- \frac{(y - x)^{2}}{2mT}} \bigg(\int_{0}^{T} (2\pi m \tau)^{-r/2} \, d\tau\bigg)^{1/r} 
        \\ & \hspace{1cm} \cdot \bigg(\int_{0}^{T} \int_{0}^{1} \mathbb{E}_{\boldsymbol \eta^N} \Big[\Big\vert F_{k}(u^{N}(\tau-,y)) - F_{k}(u^{N}(\tau, y + h'L_{N}^{-1}))\Big\vert\Big] ^{r/(r-1)} \,dh' \, d\tau\bigg)^{(r-1)/r} 
        \\ & \quad \leq C^{(2)}_{q_+,q_-,f,r}(T) \frac{1}{L_{N}} \sum_{y \in L_{N}^{-1}\mathbb{Z}} e^{- \frac{(y - x)^{2}}{2mT}} 
        \\ & \hspace{1.1cm} \cdot \bigg(\int_{0}^{1} \int_{0}^{T} \mathbb{E}_{\boldsymbol \eta^N} \Big[\Big\vert F_{k}(u^{N}(\tau-,y)) - F_{k}(u^{N}(\tau, y + h'L_{N}^{-1}))\Big\vert\Big] \,d\tau \, dh'\bigg)^{(r-1)/r}.
    \end{split}
    \raisetag{-4cm}
   \end{equation}
   By the definition of $F_k$ in~\eqref{Paper02_reaction_term_PDE} and by Lemma~\ref{Paper02_lemma_increments_moments_densities}, there exists $C^{(3)}_{q_+,q_-,f}(T) > 0$ such that for $N \in \mathbb N$, $y \in L_N^{-1}\mathbb Z$ and $h' \in [0,1]$,
   \[
   \int_{0}^{T} \mathbb{E}_{\boldsymbol \eta^N} \Big[\Big\vert F_{k}(u^{N}(\tau-,y)) - F_{k}(u^{N}(\tau, y + h'L_{N}^{-1}))\Big\vert\Big] \,d\tau  \leq C^{(3)}_{q_+,q_-,f}(T) L_N^{-1/4}.
   \]
   Therefore, substituting into~\eqref{Paper02_mild_solution_step_modified_iv}, there exists $C^{(4)}_{q_+,q_-,f,r}(T) > 0$ such that for any $N\in \mathbb N$, $t\in [0,T]$ and $x\in L_N^{-1}\mathbb Z$, we have
   \begin{equation} \label{Paper02_mild_solution_step_vii}
    \begin{aligned}
         & \frac{1}{L_{N}}\sum_{y \in L_{N}^{-1}\mathbb{Z}}\int_{0}^{t} p(t-\tau, y-x) \\
         & \hspace{2.5cm} \cdot \int_{0}^{1} \mathbb{E}_{\boldsymbol \eta^N} \Big[\Big\vert F_{k}(u^{N}(\tau-,y)) - F_{k}(u^{N}(\tau, y + h'L_{N}^{-1}))\Big\vert\Big] \, dh' \, d\tau 
         \\ & \quad \leq C^{(4)}_{q_+,q_-,f,r}(T) L_N^{-(r-1)/(4r)},
    \end{aligned}
   \end{equation}
   and since we chose $r > 1$, the right-hand side converges to $0$ as $N \rightarrow \infty$ by Assumption~\ref{Paper02_scaling_parameters_assumption}.
   
   For the third term on the right-hand side of~\eqref{Paper02_mild_solution_step_v}, by~\eqref{Paper02_good_behaviour_indices_modified_reaction_term} and the fact that $\vert F_k(u) \vert \leq F^+_k(u) \; \forall \, u \in \ell_1^+$ by~\eqref{Paper02_reaction_term_PDE} and~\eqref{Paper02_reaction_predictable_bracket_process}, there exists $C^{(5)}_{q_+,q_-,f}(T) > 0$ such that for any $N \in \mathbb N$, $t \in [0,T]$, $x \in L_N^{-1}\mathbb Z$ and $h \in [0,1]$,
   \begin{equation} \label{Paper02_mild_solution_step_modified_v}
    \begin{split}
        & \frac{1}{L_{N}} \sum_{y \in L_{N}^{-1}\mathbb{Z}} \int_{0}^{t} \int_{0}^{1} \mathbb{E}_{\boldsymbol \eta^N}\Big[\Big\vert F_{k}(u^{N}(\tau, y + h'L_{N}^{-1})) \Big\vert\Big] \\ & \qquad \qquad \qquad \qquad \qquad \qquad \cdot \left\vert p(t-\tau, y-x) - p(t-\tau, y- x + (h'-h)L_N^{-1})\right\vert \, dh'  d\tau \\ & \quad \leq C^{(5)}_{q_+,q_-,f}(T) \\
        & \hspace{1.5cm} \cdot \int_0^T \int_{\mathbb R} \int_0^1 \left\vert p(\tau, L_N^{-1}\lfloor L_N y\rfloor) - p(\tau, L_N^{-1}\lfloor L_N (y + (h'-h)L_N^{-1}) \rfloor)\right\vert \, dh'  \, dy \, d\tau.
    \end{split}
    \raisetag{-2cm}
   \end{equation}
   Now fix $\varepsilon \in (0,T)$, then since $p(\tau,z)$ is decreasing in $\vert z \vert$ and $\int_{\mathbb R} p(\tau,z) \, dz = 1$ for $\tau > 0$, we can write
   \begin{equation} \label{Paper02_mild_solution_step_modified_vi}
   \begin{split}
       & \int_0^T \int_{\mathbb R} \int_0^1 \left\vert p(\tau, L_N^{-1}\lfloor L_N y\rfloor) - p(\tau, L_N^{-1}\lfloor L_N (y + (h'-h)L_N^{-1}) \rfloor)\right\vert \, dh'  \, dy \, d\tau \\ & \quad \leq \int_0^{\varepsilon} \left(2 + 2L_N^{-1}p(\tau,0)\right) \, d\tau \\ & \qquad \qquad + \int_\varepsilon^T \int_{\mathbb R} \int_0^1 \left\vert p(\tau, L_N^{-1}\lfloor L_N y\rfloor) - p(\tau, L_N^{-1}\lfloor L_N (y + (h'-h)L_N^{-1}) \rfloor)\right\vert \, dh'  \, dy \, d\tau.
   \end{split}
   \raisetag{-2.8cm}
   \end{equation}
   By dominated convergence, the second term on the right-hand side of~\eqref{Paper02_mild_solution_step_modified_vi} converges to $0$ uniformly in $h \in [0,1]$ as $N \rightarrow \infty$. Therefore, since $\varepsilon \in (0,T)$ can be chosen arbitrarily small, we have
   \begin{equation} \label{Paper02_mild_solution_step_modified_vii}
   \begin{split}
       & \lim_{N \rightarrow \infty} \,
       \sup_{h \in [0,1]} \int_0^T \int_\mathbb R \int_0^1 \left\vert p(\tau, L_N^{-1}\lfloor L_N y\rfloor)  - p(\tau, L_N^{-1}\lfloor L_N (y + (h'-h)L_N^{-1}) \rfloor)\right\vert \, dh'  \, dy \, d\tau  
       \\[+3mm]
       & \qquad\qquad = 0.
    \end{split}
    \raisetag{-1.6cm}
   \end{equation}
   Applying~\eqref{Paper02_mild_solution_step_vi_modified},~\eqref{Paper02_mild_solution_step_modified_iii},~\eqref{Paper02_mild_solution_step_vii},~\eqref{Paper02_mild_solution_step_modified_v} and~\eqref{Paper02_mild_solution_step_modified_vii} to~\eqref{Paper02_mild_solution_step_v}, we conclude that 
   \begin{equation} \label{Paper02_mild_solution_step_ix}
   \begin{split}
    & \int_{0}^{T} \sum_{x \in L_{N}^{-1}\mathbb{Z}} \frac{1}{L_{N}(1 + \vert x \vert^{2})} \\
    & \hspace{1.7cm} \cdot \int_{0}^{1} \mathbb{E}_{\boldsymbol \eta^N}\bigg[\bigg\vert A^{N,t,x}_{\{k\}}(t) - \int_{0}^{t} \Big(P_{t-\tau}F_{k}(u^{N}(\tau, \cdot))\Big)(x + hL_{N}^{-1}) \,d\tau \bigg\vert\bigg] \, dh \, dt \rightarrow 0
    \end{split}
    \raisetag{-2cm}
   \end{equation}
   as~$N \rightarrow \infty$.

   \medskip

\myemph{Step (4): Bound on the error term arising from approximating the initial condition}   
   Finally, for the fourth term on the right-hand side of~\eqref{Paper02_proof_mild_solution_step_i}, note that for $N \in \mathbb N$, $x \in L_N^{-1}\mathbb Z$, $h \in [0,1]$ and $t \in (0,T]$, we can write
   \begin{equation*}
    \begin{aligned}
        & \mathbb E_{\boldsymbol \eta^N} \left[\Big\vert (P_{t}f_{k})(x + hL_{N}^{-1}) - (P^{N}_{t}u^{N}_{k}(0, \cdot))(x) \Big\vert \right] \\ & \quad = \left\vert \int_{\mathbb R} \left( p(t,y) f_k(x + hL_N^{-1}-y) - p^N(t, L_N^{-1}\lfloor L_Ny \rfloor) N^{-1}\eta^N_k(x - L_N^{-1} \lfloor L_Ny \rfloor)\right) \, dy  \right\vert.
    \end{aligned}
   \end{equation*}
   Now fix $\varepsilon \in (0,T)$; then for $N \in \mathbb N$, since $\sup_{y \in L_N^{-1} \mathbb Z} \eta^N_k(y) \leq N \| f \|_{L_{\infty}(\mathbb R; \ell_1)}$ by~\eqref{Paper02_initial_condition}, by the triangle inequality we have
   \begin{equation} \label{Paper02_mild_solution_step_xiii}
    \begin{split}
        & \int_{0}^{T} \sum_{x \in L_{N}^{-1}\mathbb{Z}} \frac{1}{L_{N}(1 + \vert x \vert^{2})} \int_{0}^{1} \mathbb E_{\boldsymbol \eta^N} \left[\Big\vert (P_{t}f_{k})(x + hL_{N}^{-1}) - (P^{N}_{t}u^{N}_{k}(0, \cdot))(x) \Big\vert \right] \, dh \, dt \\ & \quad \leq \int_0^\varepsilon \sum_{x \in L_N^{-1}\mathbb Z} \frac{\| f \|_{L_{\infty}(\mathbb R; \ell_1)}}{L_N(1 + \vert x \vert^2)} \, dt \\
        & \qquad + \int_0^T \sum_{x \in L_N^{-1} \mathbb Z} \frac{\| f \|_{L_{\infty}(\mathbb R; \ell_1)}}{L_N(1 + \vert x \vert^2)} \int_{\mathbb R \setminus [-\varepsilon^{-1}, \varepsilon^{-1}]} \left(p(t,y) + p^N(t,L_N^{-1} \lfloor L_Ny \rfloor)\right) \, dy \, dt \\ & \qquad + \int_0^T \sum_{\{x \in L_N^{-1}\mathbb Z: \, \vert x \vert \geq \varepsilon^{-1}\}} \frac{\| f \|_{L_{\infty}(\mathbb R; \ell_1)}}{L_N(1 + \vert x \vert^2)} \, dt \\ & \qquad + \int_{\varepsilon}^T \sum_{x \in L_N^{-1}\mathbb Z \cap [-\varepsilon^{-1}, \varepsilon^{-1}]} \, \frac{1}{L_N(1 + \vert x \vert^2)}  \int_0^1 \bigg\vert \int_{-\varepsilon^{-1}}^{\varepsilon^{-1}} \Big(p(t,y) f_k(x + hL_N^{-1} - y) \\ & \qquad \qquad \qquad \qquad \qquad \qquad  - p^N(t,L_N^{-1} \lfloor L_N y \rfloor) N^{-1} \eta^N_k(x - L_N^{-1} \lfloor L_Ny \rfloor) \Big) \, dy\bigg\vert \, dh \, dt.
    \end{split}
    \raisetag{-4.6cm}
   \end{equation}
   By the definition of $\eta^N_k$ in~\eqref{Paper02_initial_condition}, and by Lemma~\ref{Paper2_lem_LCLT} in the appendix, and since $\| f \|_{L_\infty(\mathbb R; \ell_1)} < \infty$ and $f$ is continuous almost everywhere by Assumption~\ref{Paper02_assumption_initial_condition}, and $m_N / L_N^2 \rightarrow m \in (0, \infty)$ as $N \rightarrow \infty$ by Assumption~\ref{Paper02_scaling_parameters_assumption}, we can apply dominated convergence to see that the last term on the right-hand side of~\eqref{Paper02_mild_solution_step_xiii} converges to $0$ as $N \rightarrow \infty$. For the second term on the right-hand side of~\eqref{Paper02_mild_solution_step_xiii}, letting $(W(t))_{t \geq 0}$ denote a Brownian motion run at speed $m$ and $(X^N(t))_{t \geq 0}$ a simple symmetric random walk on $L_N^{-1}\mathbb Z$ with total jump rate $m_N$, we can write
   \begin{equation*}
    \begin{aligned}
        & \limsup_{N \rightarrow \infty} \int_0^T \int_{\mathbb R \setminus [- \varepsilon^{-1}, \varepsilon^{-1}]} \left(p(t,y) + p^N(t,L_N^{-1} \lfloor L_Ny \rfloor)\right) \, dy \, dt \\ & \quad \leq \limsup_{N \rightarrow \infty} \; 2 \int_0^T \left(\mathbb P_0(W(t) \geq \varepsilon^{-1}) + \mathbb P_0(X^N(t) \geq \varepsilon^{-1}/2)\right) \, dt \\ & \quad \leq \limsup_{N \rightarrow \infty} \; 2 \int_0^T \left(\varepsilon^2 mt + 4 \varepsilon^2 m_N L_N^{-2} t\right) \, dt \\ & \quad = 5 \varepsilon^2 mT^2,
    \end{aligned}
   \end{equation*}
   where in the second inequality we used~\eqref{Paper02_second_moment_random_walk} and Markov's inequality. Therefore, taking $\varepsilon \in (0,T)$ arbitrarily small, we conclude that
   \begin{equation} \label{Paper02_mild_solution_step_x}
   \begin{split}
      & \lim_{N \rightarrow \infty} \, \int_{0}^{T} \sum_{x \in L_{N}^{-1}\mathbb{Z}} \frac{1}{L_{N}(1 + \vert x \vert^{2})} \int_{0}^{1} \mathbb E_{\boldsymbol \eta^N} \Big[\Big\vert (P_{t}f_{k})(x + hL_{N}^{-1}) - (P^{N}_{t}u^{N}_{k}(0, \cdot))(x) \Big\vert\Big] \, dh \, dt \\
      & \qquad\qquad = 0.
    \end{split}
    \raisetag{-1.7cm}
   \end{equation}
   Therefore, by applying~\eqref{Paper02_proof_mild_solution_step_i} to~\eqref{Paper02_inequality_fatous_lemma_mild_solution}, and then using~\eqref{Paper02_mild_solution_step_iii},~\eqref{Paper02_mild_solution_step_iv},~\eqref{Paper02_mild_solution_step_ix} and~\eqref{Paper02_mild_solution_step_x}, the proof is complete.
\end{proof}
We are now ready to prove Proposition~\ref{Paper02_absolutely_continuity_limiting_process}.

\begin{proof}[Proof of Proposition~\ref{Paper02_absolutely_continuity_limiting_process}]
    The desired tightness follows from Lemma~\ref{Paper02_tightness_muller_ratchet_L_p}, while condition~(i) follows from Lemma~\ref{Paper02_relation_measure_valued_L_P_time_space_boxes_solutions}. To establish that~$v$ is a mild solution to the system of PDEs~\eqref{Paper02_PDE_scaling_limit}, we first observe that since~$v \in L_{4\deg q_{-}}([0,T] \times \mathbb{R}, \hat{\lambda}; \ell_{1})$, since the heat kernel decays exponentially fast in space and by the definition of the reaction term~$F = (F_{k})_{k \in \mathbb{N}_{0}}$ in~\eqref{Paper02_reaction_term_PDE}, for any $(t,x) \in [0,T] \times \mathbb{R}$, $$\left(P_{t-\tau}\| F(v(\tau,\cdot)) \|_{\ell_{1}}\right) (x) < \infty,$$
    for $\lambda$-almost every $\tau \in [0,t]$. Moreover, by Lemma~\ref{Paper02_weak_convergence_mild_solution}, for $\lambda$-almost every $(t,x) \in [0,T] \times \mathbb{R}$,
    \begin{equation*}
        v(t,x) = (P_{t}f)(x) + \int_{0}^{t} \Big(P_{t-\tau} F(v(\tau,\cdot))\Big)(x) \, d\tau.
    \end{equation*}
    Then, it follows from Definition~\ref{Paper02_definition_mild_solution} that~$v$ is a mild solution to the system of PDEs~\eqref{Paper02_PDE_scaling_limit}, which completes the proof.
\end{proof}

\section{Proof of Theorems~\ref{Paper02_deterministic_scaling_foutel_rodier_etheridge} and~\ref{Paper02_long_time_behaviour_IPS}} \label{Paper02_uniqueness_weak_solutions_section}

In this section, we complete the proofs of the main results of this article by combining the regularity theory for mild solutions to~\eqref{Paper02_PDE_scaling_limit}, established in~\cite{madeira2025PDE_surfing}, with Propositions~\ref{Paper02_PDE_tightness_measures} and~\ref{Paper02_absolutely_continuity_limiting_process}. We begin with the proof of Theorem~\ref{Paper02_deterministic_scaling_foutel_rodier_etheridge}, which establishes a functional law of large numbers for the spatial Muller's ratchet.

\begin{proof}[Proof of Theorem~\ref{Paper02_deterministic_scaling_foutel_rodier_etheridge}]
    By Proposition~\ref{Paper02_PDE_tightness_measures}, the sequence $((u^{N}(t))_{t \geq 0})_{N \in \mathbb N}$ is tight in the $J_{1}$-topology on $\mathcal{D}\left([0,\infty), (\mathcal{M}(\mathbb{R})^{\mathbb{N}_{0}}, d)\right)$. Moreover, since Proposition~\ref{Paper02_absolutely_continuity_limiting_process}(i) implies that the finite-dimensional distributions of any subsequential limiting process $(u(t))_{t \geq 0}$ satisfy~\eqref{Paper02_distribution_characterisation_limiting_process}, it follows from Proposition~\ref{Paper02_absolutely_continuity_limiting_process}(ii) and Proposition~\ref{Paper02_thm_uniqueness_weak_solutions} that $(u(t))_{t \geq 0}$ is unique.

    We now characterise the limiting process $(u(t))_{t \geq 0} = ((u_k(t))_{k \in \mathbb N_0})_{t \geq 0}$. We start by observing that by~\eqref{Paper02_approx_pop_density_definition} and~\eqref{Paper02_initial_condition}, and since by Assumption~\ref{Paper02_assumption_initial_condition}(i) $f = (f_k) : \mathbb R \rightarrow \ell_1^+$ is continuous $\lambda$-almost everywhere, the following limit holds almost surely for $\lambda$-almost every $x \in \mathbb R$:
    \[
    \lim_{N \rightarrow \infty} u^N_k(0,x) = f_k(x).
    \]
    Therefore, by~\eqref{Paper02_initial_condition},~\eqref{Paper02_approx_pop_density_definition}, Assumption~\ref{Paper02_assumption_initial_condition}(ii) and dominated convergence, for all $k \in \mathbb N_0$ and $\varphi \in \mathcal C_c(\mathbb R)$, 
    \begin{equation} \label{Paper02_limit_measure_time_0_trivial}
        \lim_{N \rightarrow \infty} \langle u^N_k(0), \varphi \rangle = \lim_{N \rightarrow \infty} \int_{\mathbb R} u^N_k(0,x) \varphi(x) \, dx = \int_{\mathbb R} f_k(x) \varphi(x) \, dx.
    \end{equation}
    Since $(u(t))_{t \geq 0}$ is an element of $\mathcal D\left([0, \infty); (\mathcal M(\mathbb R)^{\mathbb N_0}, d)\right)$, $u$ is continuous at~$t = 0$, and therefore, by~\eqref{Paper02_limit_measure_time_0_trivial} and standard properties of the Skorokhod topology (see e.g.~\cite[Proposition~3.5.2]{ethier2009markov}), for all $k \in \mathbb N_0$, the measure $u_k(0) \in \mathcal M(\mathbb R)$ is absolutely continuous with respect to the Lebesgue measure and has density given by $f_k$. We now prove that for all $t > 0$ and $k \in \mathbb N_0$, $u_k(t)$ is absolutely continuous with respect to the Lebesgue measure, and characterise the density. Let $v = (v_k)_{k \in \mathbb N_0}: [0, \infty) \times \mathbb R \rightarrow \ell_1^+$ be the unique map satisfying the conditions (i)-(v) from Proposition~\ref{Paper02_smoothness_mild_solution}. By Proposition~\ref{Paper02_absolutely_continuity_limiting_process} and Proposition~\ref{Paper02_thm_uniqueness_weak_solutions}, we have that for all $t > 0$, $k \in \mathbb N_0$ and $\varphi \in \mathcal C_c(\mathbb R)$,
    \[
    \langle u_k(t), \varphi \rangle = \lim_{t' \downarrow t} \frac{1}{t'-t}\int_t^{t'} \int_{\mathbb R} v_k(\tau,x) \varphi(x) \, dx \, d\tau = \int_{\mathbb R} v_k(t,x) \varphi(x) \, dx,
    \]
    where the second equality holds since by Proposition~\ref{Paper02_smoothness_mild_solution}(iv), $v_k \in \mathcal C^{1,2}((0,\infty) \times \mathbb R; \mathbb R)$. Hence, for all $t > 0$ and $k \in \mathbb N_0$, $u_k(t)$ is absolutely continuous with respect to the Lebesgue measure, with density given by $v_k(t,\cdot)$. In particular, the limiting process $(u(t))_{t \geq 0} = ((u_k(t))_{k \in \mathbb N_0})_{t \geq 0}$ satisfies condition~(i) from Theorem~\ref{Paper02_deterministic_scaling_foutel_rodier_etheridge}. Moreover, by conditions~(i), (iii) and (iv) from Proposition~\ref{Paper02_smoothness_mild_solution}, $(u(t))_{t \geq 0} = ((u_k(t))_{k \in \mathbb N_0})_{t \geq 0}$ also satisfies conditions~(ii) and (iii) from Theorem~\ref{Paper02_deterministic_scaling_foutel_rodier_etheridge}. The fact that $(u(t))_{t \geq 0} = ((u_k(t))_{k \in \mathbb N_0})_{t \geq 0}$ is the unique element in $\mathcal D([0, \infty), \mathcal M(\mathbb R)^{\mathbb N_0})$ satisfying Theorem~\ref{Paper02_deterministic_scaling_foutel_rodier_etheridge}(i)-(iii) follows from Proposition~\ref{Paper02_thm_uniqueness_weak_solutions}. 
    Since for every $k \in \mathbb N_0$, $u_k(0)$ is a Radon measure with density $f_k$ with respect to the Lebesgue measure, the associated density process $(v_k(t,x))_{k \in \mathbb N_0, \, t \in [0, \infty), \, x \in \mathbb R}$ satisfies Lemma~\ref{Paper02_lem_equiv_mild_weak_sol}(i). Moreover, since $u_k(t)$ has density given by $v_k(t, \cdot)$, we conclude by Proposition~\ref{Paper02_smoothness_mild_solution}(i) and~(ii) that $(v_k(t,x))_{k \in \mathbb N_0, \, t \in [0, \infty), \, x \in \mathbb R}$ satisfies both conditions (ii) and (iii) from Lemma~\ref{Paper02_lem_equiv_mild_weak_sol}. Therefore, it follows from Lemma~\ref{Paper02_lem_equiv_mild_weak_sol} that $(v_k(t,x))_{k \in \mathbb N_0, \, t \in [0, \infty), \, x \in \mathbb R}$ is a weak solution to the system of PDEs~\eqref{Paper02_PDE_scaling_limit} in the sense of Definition~\ref{Paper02_weak_solution_distributional_sense}.
    By observing that any weak solution satisfying conditions~(i) and~(ii) of Theorem~\ref{Paper02_deterministic_scaling_foutel_rodier_etheridge} is a mild solution to the system of PDEs~\eqref{Paper02_PDE_scaling_limit}, uniqueness follows from condition~(ii) and Proposition~\ref{Paper02_smoothness_mild_solution}, which completes the proof.
\end{proof}

Recall that by Theorem~\ref{Paper02_prop_control_proportions} (proved in our companion article~\cite{madeira2025PDE_surfing}), we can control  the density of particles with $k$ mutations 
in the PDE limit~\eqref{Paper02_PDE_scaling_limit} 
when the reaction term $F = (F_k)_{k \in \mathbb N_0}$ defined in~\eqref{Paper02_reaction_term_PDE} is monostable in the sense of Definition~\ref{Paper02_assumption_monostable_condition}. We now combine this control with the law of large numbers given by Theorem~\ref{Paper02_deterministic_scaling_foutel_rodier_etheridge} to prove Theorem~\ref{Paper02_long_time_behaviour_IPS}.

\begin{proof}[Proof of Theorem~\ref{Paper02_long_time_behaviour_IPS}]
By Theorem~\ref{Paper02_prop_control_proportions}, for any $\delta > 0$ and~$K \in \mathbb{N}$, there exists $T_{\delta, K} > 0$ such that for any $T \geq T_{\delta, K}$, $x \in \mathbb{R}$ and $k \leq K$,
\begin{equation*}
    \bigg(\alpha_{k}\left(\frac{\mathfrak{Q}_{\min}}{\mathfrak{Q}_{\max}}\right)^{k}-\frac{\delta}{2}\bigg) u_{0}(T, x)\leq u_{k}(T,x) \leq  \bigg(\alpha_{k}\left(\frac{\mathfrak{Q}_{\max}}{\mathfrak{Q}_{\min}}\right)^{k} + \frac{\delta}{2}\bigg) u_{0}(T,x),
\end{equation*}
where~$u = (u_{k})_{k \in \mathbb{N}_{0}}: [0,\infty) \times \mathbb{R} \rightarrow \ell_{1}^{+}$ is the continuous mild solution to the system of PDEs~\eqref{Paper02_PDE_scaling_limit} given in Proposition~\ref{Paper02_smoothness_mild_solution}. Fix $T \geq T_{\delta, K}$, and let $\mathcal{I}$ be a compact interval in $\mathbb{R}$ which is not a singleton. Since the measure induced by~$u_{k}$ on~$\mathbb{R}$ is absolutely continuous with respect to the Lebesgue measure for every~$k \in \mathbb{N}_{0}$, by Theorem~\ref{Paper02_deterministic_scaling_foutel_rodier_etheridge}, Skorokhod's representation theorem and standard results on the vague topology in the space of Radon measures (see, for instance,~\cite[Lemma~4.1(iv)]{kallenberg2017random}), it is possible to construct the sequence of stochastic processes $(u^{N})_{N \in \mathbb{N}}$ on the same probability space in such a way that for every~$k \in \mathbb{N}_{0}$, almost surely
\begin{equation*}
    u^{N}_{k}(T)(\mathcal{I}) \xrightarrow{N \rightarrow \infty} \int_{\mathcal{I}} u_{k}(T,x) \, dx.
\end{equation*}
Moreover, by Theorem~\ref{Paper02_prop_control_proportions},
\[
\int_{\mathcal I} u_0(T,x)\,dx>0.
\]
Since almost sure convergence implies convergence in probability, we then conclude that for every~$\varepsilon > 0$, we can choose $N_{\delta, \varepsilon, \mathcal{I}, K, T} \in \mathbb{N}$ such that for every $N \geq N_{\delta, \varepsilon, \mathcal{I}, K, T}$ and~$k \in \llbracket K \rrbracket$,
\begin{equation*}
\begin{aligned}
& \mathbb{P}\bigg(\bigg(\alpha_{k}\left(\frac{\mathfrak{Q}_{\min}}{\mathfrak{Q}_{\max}}\right)^{k}-\delta\bigg) u^{N}_{0}(T)(\mathcal{I})\leq u^{N}_{k}(T)(\mathcal{I}) \leq  \bigg(\alpha_{k}\left(\frac{\mathfrak{Q}_{\max}}{\mathfrak{Q}_{\min}}\right)^{k} +\delta\bigg) u^{N}_{0}(T)(\mathcal{I}) \bigg) \\
& \quad \geq 1 - \varepsilon,
\end{aligned}
\end{equation*}
which completes the proof.
\end{proof}


\appendix

\section*{Appendix}

\section{Stochastic chain rule} \label{Paper02_appendix_section_technical_lemmas}

In this section, we will state a general stochastic chain rule formula. In what follows, let~$(\mathcal{S}, d_{\mathcal{S}})$ be some complete and separable metric space, let $A \subseteq \mathcal{C}(\mathcal{S}; \mathbb{R})$ be some vector subspace of real-valued continuous functions defined on~$\mathcal{S}$, let $\mathcal L\Big\vert_{A}: A \rightarrow \mathcal{C}(\mathcal{S}; \mathbb{R})$ be a (possibly unbounded) operator and let $(\eta(t))_{t \geq 0}$ be an~$\mathcal{S}$-valued càdlàg process with initial distribution~$\nu$, which is strongly Markovian with respect to its right-continuous natural filtration~$\{\mathcal{F}^{\eta}_{t+}\}_{t \geq 0}$. We will assume that~$\mathcal L$ is the infinitesimal generator of~$(\eta(t))_{t \geq 0}$, and that~$(\eta(t))_{t \geq 0}$ is a solution of the martingale problem for~$(\mathcal L, A, \nu)$, i.e.~for  any~$\phi \in A$, the process~$(M^{\phi}(t))_{t \geq 0}$ given by, for all $T \geq 0$,
    \begin{equation*}
        M^{\phi}(T) \defeq \phi(\eta(T)) - \phi(\eta(0)) - \int_{0}^{T} (\mathcal L\phi)(\eta(t-)) \, dt
    \end{equation*}
    is a càdlàg martingale with respect to the filtration~$\{\mathcal{F}^{\eta}_{t+}\}_{t \geq 0}$.

We begin by stating a general chain rule formula for Markov processes whose transition rates are not almost surely bounded. This result is classical for Markov processes with uniformly bounded transition rates. We prove the version that we state here in the companion article~\cite[Lemma~A.8]{madeira2025existence}.

\begin{lemma} \label{Paper02_general_integration_by_parts_formula}
Let~$T \geq 0$ be fixed, and let~$g \in \mathcal{C}([0,T] \times \mathcal{S}; \mathbb{R})$. Suppose the following conditions hold:
    \begin{enumerate}[label = (\roman*)]
        \item $
            \sup_{t_{1},t_{2} \in [0,T]}  \mathbb{E}_{\nu}\left[\vert g(t_{1}, \eta(t_{2})) \vert \right] < \infty$.
        \item The map~$[0,T] \times \mathcal{S}\ni (t,\boldsymbol{\eta})  \mapsto \left(\frac{\partial}{\partial t} g(\cdot, \boldsymbol{\eta})\right)(t)$ is in $ \mathcal{C}([0,T] \times \mathcal{S}; \mathbb{R})$.
        \item For any $t \in [0,T]$, the map~$\mathcal{S} \ni \boldsymbol{\eta}  \mapsto g(t, \boldsymbol{\eta}) $ is in $ A$, and the map~$[0,T] \times \mathcal{S} \ni (t,\boldsymbol{\eta})  \mapsto (\mathcal Lg)(t,\cdot)(\boldsymbol{\eta}) $ is in $ \mathcal{C}([0,T] \times \mathcal{S}; \mathbb{R})$.
        \item 
        $
        \!\begin{aligned}[t]
\sup_{t_{1},t_{2} \in [0,T]} \; \mathbb{E}_{\nu}\bigg[\left(\frac{\partial}{\partial t} g(\cdot, \eta(t_{2}))\right)^{2}(t_{1})  \bigg] + \sup_{t_{1},t_{2} \in [0,T]} \; \mathbb{E}_{\nu}\bigg[ \Big(\mathcal Lg(t_{1}, \cdot )\Big)^{2} (\eta(t_{2})) \bigg] < \infty.
\end{aligned}
    $
    \end{enumerate}
    Then the process~$\Big(M^{g}(t \wedge T)\Big)_{t \geq 0}$ given by, for all~$t \geq 0$,
    \begin{equation*}
    \begin{aligned}
        & M^{g}(t \wedge T) \\ & \quad \defeq g\Big( t \wedge T, \eta(t \wedge T)\Big) - g(0, \eta(0)) \\
        & \qquad \quad - \int_{0}^{t \wedge T} \left(\left(\frac{\partial}{\partial t'} g(\cdot, \eta(t'-))\right)(t') + \Big(\mathcal Lg(t', \cdot )\Big) (\eta(t'-))\right) \, dt'
    \end{aligned}
    \end{equation*}
    is a càdlàg martingale with respect to the filtration~$\{\mathcal{F}^{\eta}_{t+}\}_{t \geq 0}$.
\end{lemma}

We also need an expression for the predictable bracket process of the martingale~$M^{g}$ introduced in Lemma~\ref{Paper02_general_integration_by_parts_formula}. This will be our next result. A similar claim is stated without proof in~\cite[Theorem~2.6.3]{demasi1991mathematical}; we provide a proof here for completeness.

\begin{lemma} \label{Paper02_integration_by_parts_time_predictable_bracket_process}
    Let~$T \geq 0$ be fixed, and let~$g \in \mathcal{C}([0,T] \times \mathcal{S}, \mathbb{R})$. Suppose that, in addition to the conditions (i)-(iv) of Lemma~\ref{Paper02_general_integration_by_parts_formula}, the following conditions hold:
    \begin{enumerate}[label = (\roman*)]
        \item For any $t \in [0,T]$, the map~$\mathcal S \ni \boldsymbol{\eta} \mapsto g^{2}(t, \boldsymbol{\eta})$ is in~$A$, and the map~$[0,T] \times \mathcal{S} \ni (t,\boldsymbol{\eta}) \mapsto (\mathcal Lg^{2})(t,\cdot)(\boldsymbol{\eta})$ is in $ \mathcal{C}([0,T] \times \mathcal{S}, \mathbb{R})$.
        \item
         $
        \!\begin{aligned}[t]
        \sup_{t_{1},t_{2} \in [0,T]} \; \mathbb{E}_{\nu}\bigg[& g^{2}(t_{1},\eta(t_{2})) + g^{2}(t_{1}, \eta(t_{2}))\left(\frac{\partial}{\partial t} g(\cdot, \eta(t_{2}))\right)^{2}(t_{1}) \\ & \hspace{5cm} + \Big(\mathcal Lg^{2}(t_{1}, \cdot )\Big)^{2} (\eta(t_{2}))\bigg] < \infty.
        \end{aligned}
        $
    \end{enumerate}
    Then the process~$M^{g}$ introduced in Lemma~\ref{Paper02_general_integration_by_parts_formula} is a square-integrable martingale, and its predictable bracket process~$\Big(\langle M^{g} \rangle (t \wedge T)\Big)_{t \geq 0}$ is given by, for all $t \geq 0$,
    \begin{equation} \label{Paper02_gen_formula_predictable_bracket}
    \begin{aligned}
        & \langle M^{g} \rangle(t \wedge T) \\
        & \quad = \int_{0}^{t\wedge T} \bigg(\Big(\mathcal Lg^{2}(t', \cdot) \Big) (\eta(t'-)) -2 g(t',\eta(t'-)) \Big(\mathcal Lg(t', \cdot) \Big) (\eta(t'-))\bigg) \, dt'.
    \end{aligned}
    \end{equation}
\end{lemma}

\begin{proof}
    Since the proof is a straightforward generalisation of the classical argument for Markov processes with uniformly bounded transition rates, we will omit the details and only point out how to adapt the usual proof to our scenario. Let $(A^g(t))_{t \geq 0}$ denote the process given by the right-hand side of~\eqref{Paper02_gen_formula_predictable_bracket}. Condition~(iv) of Lemma~\ref{Paper02_general_integration_by_parts_formula} and condition~(ii) of this lemma imply that the process~$M^{g}$ is a square-integrable martingale. Hence, to prove our claim, it will suffice to verify that the process~$\left(Q^{g}(t)\right)_{t \geq 0}$ given by, for all~$t \geq 0$,
    \begin{equation*}
        Q^{g}(t) \defeq \Big(M^{g}(t \wedge T)\Big)^{2} -   A^{g}(t \wedge T)
    \end{equation*}
    is a martingale with respect to the filtration~$\{\mathcal{F}^{\eta}_{t+}\}_{t \geq 0}$. By observing that for all $T \geq 0$ and càdlàg functions $a,b:[0,T] \rightarrow \mathbb R$ we have
    \[
    \left(\int_0^T a(t-)b(t-) \, dt \right)^2 \leq \left(\int_0^T a(t-)^2 \, dt \right)\left(\int_0^T b(t-)^2 \, dt \right) \leq T^2 \sup_{t \leq T} a^2(t) \cdot \sup_{t \leq T} b^2(t),
    \]
    we conclude from condition~(iv) of Lemma~\ref{Paper02_general_integration_by_parts_formula} and condition~(ii) of this lemma that~$\left(Q^{g}(t)\right)_{t \geq 0}$ is integrable. Also, observe that by the definition of the process~$M^{g}$ given by Lemma~\ref{Paper02_general_integration_by_parts_formula}, we have, after rearranging terms, for all~$t \geq 0$,
    \begin{equation} \label{Paper02_decomposing_square_martingale_started_random_time}
        \Big(M^{g}(t \wedge T)\Big)^{2} =  A^{g,(1)}(t) + A^{g,(2)}(t),
    \end{equation}
    where
    \begin{equation*} 
    \begin{aligned}
        & A^{g,(1)}(t) \\
        & \quad \defeq g^{2}\Big(t \wedge T, \eta(t \wedge T)\Big) - g^2(0,\eta(0)) \\ & \qquad - 2g\Big(t \wedge T, \eta(t \wedge T)\Big)\int_{0}^{t \wedge T} \left(\left(\frac{\partial}{\partial t'} g(\cdot, \eta(t'-))\right)(t') + \Big(\mathcal Lg(t', \cdot )\Big) (\eta(t'-))\right) \, dt' \\ & \qquad  + \bigg(\int_{0}^{t \wedge T} \left(\left(\frac{\partial}{\partial t'} g(\cdot, \eta(t'-))\right)(t') + \Big(\mathcal Lg(t', \cdot )\Big) (\eta(t'-))\right) \, dt'\bigg)^{2},
    \end{aligned}
    \end{equation*}
    and
    \begin{equation*} 
    \begin{aligned}
        A^{g,(2)}(t) & \defeq -2g(0, \eta(0))g\Big(t \wedge T, \eta(t \wedge T)\Big)  + 2g^{2}(0, \eta(0)) \\ & \quad \quad + 2g(0, \eta(0))\int_{0}^{t \wedge T} \left(\left(\frac{\partial}{\partial t'} g(\cdot, \eta(t'-))\right)(t') + \Big(\mathcal Lg(t', \cdot )\Big) (\eta(t'-))\right) \, dt'.
    \end{aligned}
    \end{equation*}
    By Lemma~\ref{Paper02_general_integration_by_parts_formula}, we conclude that~$(A^{g,(2)}(t))_{t \geq 0}$ is a càdlàg martingale. It then remains to analyse~$(A^{g,(1)}(t))_{t \geq 0}$. For $t \in [0,T]$, let $A^{g}(t)$ denote the term on the right-hand side of~\eqref{Paper02_gen_formula_predictable_bracket}. Then $(A^{g}(t))_{t \geq 0}$ is an increasing finite variation process. The fact that the process $\Big((A^{g,(1)} -  A^{g}) (t)\Big)_{t \geq 0}$ is a martingale can be proved by an application of Itô's formula as in the classical case, with the observation that one should replace the assumption in the classical case of uniform boundedness of~$g$ and of~$\mathcal Lg^{2}$ by Lemma~\ref{Paper02_general_integration_by_parts_formula} and condition~(ii) of this lemma to establish that the process
    \begin{equation*}\begin{aligned}
        \bigg(g^{2}\Big(t \wedge T, & \eta(t \wedge T)\Big)  \\ & - \int_{0}^{t \wedge T} \Big(2g(t', \eta(t'-))\big(\frac{\partial}{\partial t'} g(\cdot, \eta(t'-))\Big)(t') + \big(\mathcal Lg^{2}(t', \cdot )\big) (\eta(t'-))\big) \, dt' \bigg)_{t \geq 0}
    \end{aligned}
    \end{equation*}
    is a martingale with respect to the filtration~$\{\mathcal{F}^{\eta}_{t+}\}_{t \geq 0}$. Since the remainder of the proof follows exactly the same arguments as in the classical case, we refer the interested reader to, for instance,~\cite[Lemma~A.1.5.1]{kipnis1998scaling}. We then conclude that $ \big(\big(M^{g}(t \wedge T)\big)^{2} - A^{g}(t)\big)_{t \geq 0}$ is a martingale with respect to the filtration~$\{\mathcal{F}^{\eta}_{t+}\}_{t \geq 0}$, which completes the proof.
\end{proof}

\ignore{To conclude this section, we state a simplification of the identity defining the predictable bracket process introduced in Lemma~\ref{Paper02_integration_by_parts_time_predictable_bracket_process}.

\begin{lemma} \label{Paper02_lemma_explicit_easy_formula_compute_predictable_bracket_process}
Suppose that, in addition to the conditions of Lemmas~\ref{Paper02_general_integration_by_parts_formula} and~\ref{Paper02_integration_by_parts_time_predictable_bracket_process}, there exists a map~$\Lambda: \mathcal{S} \rightarrow \mathcal{P}(\mathbb{N})$ such that for each 
    
\end{lemma}

This simplification is not new, and can be found, for instance, in~\cite{etheridge2019genealogical}. Nonetheless, since Lemmas~\ref{Paper02_general_integration_by_parts_formula} and~\ref{Paper02_integration_by_parts_time_predictable_bracket_process} are stated under general assumptions, we} 

\section{Random walk estimates} \label{Paper02_subsection_rw_estimates}

In this section, we will state some random walk estimates which are consequences of the local central limit theorem (see~\cite[Chapter~2]{lawler2010random} for a review of the local central limit theorem). These estimates are used to control the spatial and time increments of $u^N$ in Section~\ref{Paper02_section_greens_function}. For $t \geq 0$ and $x \in L_N^{-1}\mathbb{Z}$, recall the definition of $p^N(t,x)$ from~\eqref{Paper02_definition_p_N_test_function_green}. We will need the following inequalities, most of which can be found in~\cite{durrett2016genealogies}.

\begin{lemma}[Random walk estimates] \label{Paper02_random_walks_lemma}
    Suppose Assumption~\ref{Paper02_scaling_parameters_assumption} holds. Then there exists $C > 0$ such that for all $\tau \geq 0$, $N \in \mathbb{N}$, $x_{1},x_{2} \in L_{N}^{-1}\mathbb{Z}$ and $T \geq 0$,
\begin{align}
    & \int_{0}^{T} p^{N}(t, 0) \, dt \leq C\sqrt{T}, \label{Paper02_random_walk_estimate_special} \\ 
    0 \leq & \int_{0}^{T} \left(p^{N}(t,0) - p^{N}(t + \tau,0)\right) \, dt \leq \int_{0}^{\infty} \left(p^{N}(t,0) - p^{N}(t + \tau,0)\right) \, dt \leq C\sqrt{\tau}, \label{Paper02_random_walk_estimate_special_ii} \\
    0 \leq & \int_0^T \left(p^N(t,0) - p^N(t,x_1)\right) \, dt \leq \int_0^\infty \left(p^N(t,0) - p^N(t,x_1)\right) \, dt \leq C \vert x_1 \vert, \label{Paper02_random_walk_estimate_special_extra}
    \\
    & \frac{1}{L_{N}} \sum_{y \in L_{N}^{-1}\mathbb{Z}} \int_{0}^{T} \left\vert p^{N}(t,y) - p^{N}(t + \tau,y) \right\vert dt \leq C\sqrt{T\tau}, \label{Paper02_random_walk_estimates_i} \\ & \frac{1}{L_{N}} \sum_{y \in L_{N}^{-1}\mathbb{Z}} \int_{0}^{T} \left\vert p^{N}\left(t,y - x_{1}\right) - p^{N}\left(t,y - x_{2}\right) \right\vert dt \leq C\sqrt{T}\vert x_{1} - x_{2} \vert, \label{Paper02_random_walk_estimates_ii}
    \\ 
    & \frac{1}{L_{N}} \sum_{y \in L_{N}^{-1}\mathbb{Z}} \int_{0}^{T} p^{N}(T - t, y)^{2} \, dt \leq C\sqrt{T}, \label{Paper02_random_walk_estimates_iv} \\ & \frac{1}{L_{N}} \sum_{y \in L_{N}^{-1}\mathbb{Z}} \int_{0}^{T} \Big(p^{N}(T + \tau - t, y - x_{1}) - p^{N}(T - t, y - x_{2})\Big)^{2} \, dt \notag \\
    & \hspace{7cm} \leq C(\sqrt{\tau} + \vert x_{1} - x_{2} \vert). \label{Paper02_random_walk_estimates_v}
\end{align}
\end{lemma}

\begin{proof}
Estimates~\eqref{Paper02_random_walk_estimate_special},~\eqref{Paper02_random_walk_estimate_special_ii},~\eqref{Paper02_random_walk_estimate_special_extra},~\eqref{Paper02_random_walk_estimates_i} and~\eqref{Paper02_random_walk_estimates_ii} are stated explicitly in~\cite{durrett2016genealogies} (see inequalities~(46),~(47),~(48),~(49) and~(50) in \cite{durrett2016genealogies}, which follow from Proposition~2.4.1 and Theorem~2.5.6 in~\cite{lawler2010random}). It remains to verify estimates~\eqref{Paper02_random_walk_estimates_iv} and~\eqref{Paper02_random_walk_estimates_v}. As noted in \cite[Section~7]{durrett2016genealogies}, by symmetry and then by the Chapman-Kolmogorov equation, for all $t \geq 0$, we have
\begin{equation} \label{Paper02_intermediate_step_RW_estimates}
    \frac{1}{L_{N}} \sum_{y \in L_{N}^{-1}\mathbb{Z}} p^{N}(t,y)^{2} =  \frac{1}{L_{N}} \sum_{y \in L_{N}^{-1}\mathbb{Z}} p^{N}(t,y)p^N(t,-y) = p^{N}(2t,0).
\end{equation}
Hence, by Fubini's theorem and then by estimate~\eqref{Paper02_random_walk_estimate_special}, for all $T \geq 0$ we have
\begin{equation*}
\begin{aligned}
    \frac{1}{L_{N}} \sum_{y \in L_{N}^{-1}\mathbb{Z}} \int_{0}^{T} p^{N}(T - t, y)^{2} \, dt = \int_{0}^{T} p^{N}(2(T - t), 0) \, dt \leq \tfrac{1}{2}C \sqrt{2T},
\end{aligned}
\end{equation*}
Hence, estimate~\eqref{Paper02_random_walk_estimates_iv} is proved. In order to prove~\eqref{Paper02_random_walk_estimates_v}, we write
\begin{equation*}
\begin{aligned}
& \frac{1}{L_{N}} \sum_{y \in L_{N}^{-1}\mathbb{Z}} \int_{0}^{T} \left(p^{N}\left(T + \tau - t,y - x_{1}\right) - p^{N}\left(T - t,y-x_{2}\right)\right)^{2} dt \\
& \quad =  \frac{1}{L_{N}} \sum_{y \in L_{N}^{-1}\mathbb{Z}} \int_{0}^{T} \left(p^{N}\left(T + \tau - t,y-x_{1}\right)^{2} + p^{N}\left(T- t,y-x_{2}\right)^{2}\right) dt \\
& \quad \quad - \frac{1}{L_{N}} \sum_{y \in L_{N}^{-1}\mathbb{Z}} \int_{0}^{T} 2p^{N}\left(T + \tau - t, y -x_{1}\right)p^{N}\left(T - t, y -x_{2}\right) dt.
\end{aligned}
\end{equation*}
Applying the Chapman-Kolmogorov equation to the previous identity as in~\eqref{Paper02_intermediate_step_RW_estimates}, we conclude that
\begin{equation*}
\begin{aligned}
    & \frac{1}{L_{N}} \sum_{y \in L_{N}^{-1}\mathbb{Z}} \int_{0}^{T} \left(p^{N}\left(T + \tau - t,y - x_{1}\right) - p^{N}\left(T - t,y-x_{2}\right)\right)^{2} dt \\
    & \quad = \int_{0}^{T} \left(p^{N}\left(2(T + \tau - t),0\right) + p^{N}\left(2(T-t),0\right) - 2p^{N}\left(2T + \tau - 2t,x_{2} - x_{1}\right)\right) dt \\
    & \leq C\tfrac{1 + \sqrt{2}}{2}\sqrt{\tau} + C\vert x_{1} - x_{2} \vert,
\end{aligned}
\end{equation*}
where the last inequality is derived after adding and subtracting $2p^{N}\left(2T + \tau - 2t,0\right)$ inside the integral and applying estimates~\eqref{Paper02_random_walk_estimate_special_ii} and~\eqref{Paper02_random_walk_estimate_special_extra}. This completes the proof.
\end{proof}

We also use the following version of the local central limit theorem for continuous-time random walks that can be found in~\cite[Theorem~2.5.6]{lawler2010random}.

\begin{lemma}[Local central limit theorem] \label{Paper2_lem_LCLT}
    For $t > 0$ and $x \in L_N^{-1}\mathbb Z$ with $\vert x \vert \leq m_Nt/(2L_N)$,
    \[
    p^N(t,x) = \frac{L_N}{\sqrt
    {2\pi m_N t}} \exp\left(-\frac{L_N^2}{m_N} \frac{x^2}{2t} + O\left(\frac{1}{\sqrt{m_N t}} + \frac{L_N^3 \vert x \vert^3}{m_N^2t^2}\right)\right).
    \]
\end{lemma}

\section{Topological properties of~$L_{r}([0,T] \times \mathbb{R}, \hat{\lambda}; \ell_{1})$} \label{Paper02_appendix_proof_lemma_compact_sets_of_L_p}

In this section, we will prove Lemma~\ref{Paper02_alternative_criterion_compactness_diaz_mayoral} and some properties regarding the Banach space $L_{r}([0,T] \times \mathbb{R}, \hat{\lambda}; \ell_{1})$ introduced in~\eqref{Paper02_functional_space_L1_l1}. For this purpose, we need to state a general criterion for compact subsets of~$L_{r}([0,T] \times \mathbb{R}, \hat{\lambda}; \ell_{1})$.  Recall the definition of the measure~$\hat{\lambda}$ in~\eqref{Paper02_measure_time_space_box_i}, and the definition of the function space~$L_{r}([0,T] \times \mathbb{R}, \hat \lambda; \mathbb{R})$ in~\eqref{Paper02_functional_space_real_functions_measure_time_space_box}. Let~$\ell_{\infty}$ denote the space of bounded real-valued sequences, and recall that $\ell_\infty$ can be identified with the dual space of $\ell_1$. Define for every~$w \in \ell_{\infty}$, the map
\begin{equation*} 
\begin{aligned}
    \Psi_{w}: L_{r}([0,T] \times \mathbb{R}, \hat{\lambda}; \ell_{1}) & \rightarrow L_{r}([0,T] \times \mathbb{R}, \hat{\lambda}; \mathbb{R}) \\ v & \mapsto \Psi_{w}(v),
\end{aligned}
\end{equation*}
where, for any~$(t,x) \in [0,T] \times \mathbb{R}$,
\begin{equation} \label{Paper02_dual_map_L_p_ell_1}
   \Psi_{w}(v)(t,x) \defeq \sum_{k \in \mathbb{N}_{0}} w_{k} v_{k}(t,x).
\end{equation}
Observe that for every~$w \in \ell_{\infty}$ and~$v \in L_{r}([0,T] \times \mathbb{R}, \hat{\lambda}; \ell_1)$,
\begin{equation} \label{Paper02_continuity_duality_map}
    \| \Psi_{w}(v) \|_{L_{r}([0,T] \times \mathbb{R}, \hat{\lambda}; \mathbb{R})} \leq \| w \|_{\ell_{\infty}}  \| v \|_{L_{r}([0,T] \times \mathbb{R}, \hat{\lambda}; \ell_{1})},
\end{equation}
and so~$\Psi_{w}$ is well defined. We are now ready to state a version of D\'{i}az and Mayoral's compactness theorem from~\cite{diaz1999compactness} which is adapted to our specific setting.

\begin{theorem} \label{Paper02_diaz_mayoral_compactness_theorem}
\emph{(Special case of~\cite[Theorem~3.2]{diaz1999compactness})} 
Let~$T > 0$ and~$r \in [1,\infty)$ be fixed. A subset $\mathcal{K} \subset  L_{r}([0,T] \times \mathbb{R}, \hat{\lambda}; \ell_{1})$ is relatively compact if and only if the following conditions are satisfied:
\begin{enumerate}[label = (\roman*)]
    \item $\mathcal{K}$ is~$\| \cdot \|_{L_{r}([0,T] \times \mathbb{R}, \hat{\lambda}; \ell_{1})}$-bounded, i.e.
    \begin{equation*}
        \sup_{v \in \mathcal{K}} \, \| v \|_{L_{r}([0,T] \times \mathbb{R}, \hat{\lambda}; \ell_{1})} < \infty.
    \end{equation*}
    \item $\mathcal{K}$ is $r$-uniformly integrable, i.e.
    \begin{equation*}
        \lim_{\gamma \rightarrow \infty} \; \sup_{v \in \mathcal{K}} \Big\| v \cdot \mathds{1}_{\{\| v \|_{\ell_{1}} > \gamma\}} \Big\|_{L_{r}([0,T] \times \mathbb{R}, \hat{\lambda}; \ell_{1})}\; = 0.
    \end{equation*}
    \item $\mathcal{K}$ is uniformly tight, i.e.~for every~$\delta > 0$, there exists a compact subset $\mathcal{A}_{\delta} \subset \ell_{1}$ such that
    \begin{equation*}
        \sup_{v \in \mathcal{K}} \; \hat{\lambda}(\{(t,x) \in [0,T] \times \mathbb{R}: \; v(t,x) \not\in \mathcal{A}_{\delta}\}) \leq \delta.
    \end{equation*}
    \item $\mathcal{K}$ is scalarly relatively compact, i.e.~for every~$w = (w_{k})_{k \in \mathbb{N}_{0}} \in \ell_{\infty}$, the set
    \begin{equation*}
    \begin{aligned}
        \mathcal{K}^{*}_{w} \defeq \Big\{\Psi_{w}(v): v \in \mathcal{K}\Big\}
    \end{aligned}
    \end{equation*}
    is a relatively compact subset of~$L_{r}([0,T] \times \mathbb{R}, \hat \lambda; \mathbb{R})$.
\end{enumerate}
\end{theorem}

In order to apply Theorem~\ref{Paper02_diaz_mayoral_compactness_theorem}, we recall a characterisation of relatively compact subsets of
$L_{r}([0,T]\times\mathbb R,\hat\lambda;\mathbb R)$,
which follows from a straightforward adaptation of the
Kolmogorov--Riesz--Fréchet compactness theorem to the weighted measure $\hat \lambda$ defined in~\eqref{Paper02_measure_time_space_box_i} (see e.g.~\cite[Theorem~4.26]{brezis2011functional}). Recall, for~$i \in \{1,2\}$ and $\gamma \in \mathbb{R}$, the definition of the shift maps~$\theta^{(i)}_{\gamma}$ in~\eqref{Paper02_shift_maps}.

\begin{theorem}[Special case of Kolmogorov-Riesz-Fréchet's compactness theorem] \label{Paper02_kolmogorov_riesz_frechet}
    Let $T > 0$ and $r \in [1,\infty)$ be fixed. A bounded 
    subset $\mathcal{K} \subset L_{r}([0,T] \times \mathbb{R}, \hat \lambda; \mathbb{R})$ is relatively compact in $L_r([0,T] \times \mathbb R, \hat \lambda; \mathbb R)$ if and only if the following limits hold:
    \begin{align}
        \lim_{\vert \gamma \vert \downarrow 0} \; & \sup_{v^{*} \in \mathcal{K}} \; \left(\| \theta^{(1)}_{\gamma}v^{*} - v^{*} \|_{L_{r}([0,T] \times \mathbb{R}, \hat \lambda; \mathbb{R})} + \| \theta^{(2)}_{\gamma}v^{*} - v^{*} \|_{L_{r}([0,T] \times \mathbb{R}, \hat \lambda; \mathbb{R})}\right) = 0, \label{Paper02_integral_equicontinuity_property} \\ \textrm{and } \lim_{R \rightarrow \infty} \; & \sup_{v^{*} \in \mathcal{K}} \; \int_{0}^{T} \int_{\mathbb{R} \setminus [-R,R]} \frac{\vert v^{*}(t,x) \vert^{r}}{1 + \vert x \vert^{2}} dx dt = 0. \label{Paper02_uniform_integrability_real_space_L_p}
    \end{align}
\end{theorem}

\begin{remark}\label{Paper02_corollary_Kolmogorov_Riesz}
For each $i\in\{1,2\}$, the translation operators $\theta^{(i)}_\gamma$
are strongly continuous on
$L_1([0,T]\times\mathbb R,\hat\lambda;\mathbb R)$, that is,
for all $g\in L_1([0,T]\times\mathbb R,\hat\lambda;\mathbb R)$,
\[
\lim_{|\gamma|\downarrow 0}
\|\theta^{(i)}_\gamma g - g\|_{L_1([0,T]\times\mathbb R,\hat\lambda;\mathbb R)} = 0.
\]
This property follows, for instance, by applying
Theorem~\ref{Paper02_kolmogorov_riesz_frechet} to the singleton set $\{g\}$,
or directly from the fact that $\mathcal C_c^{\infty}([0,T]\times\mathbb R)$ is a dense subset of
$L_1([0,T]\times\mathbb R,\hat\lambda;\mathbb R)$.
\end{remark}

We are now ready to prove Lemma~\ref{Paper02_alternative_criterion_compactness_diaz_mayoral}.

\begin{proof}[Proof of Lemma~\ref{Paper02_alternative_criterion_compactness_diaz_mayoral}]
    It will suffice to establish that
    \[
    \mathcal{K} = \mathcal{K}(C_{1},C_{2}, (k_{h})_{h \in \mathbb{N}},(D_{h})_{h \in \mathbb{N}},\beta,J,T)
    \]
    satisfies the conditions of Theorem~\ref{Paper02_diaz_mayoral_compactness_theorem}. We will divide the proof into steps corresponding to each of the conditions~(i)-(iv) of Theorem~\ref{Paper02_diaz_mayoral_compactness_theorem}.

    \medskip
    

    \myemph{Step~$(1)$: Boundedness}
   Since~$\hat{\lambda}$, defined in~\eqref{Paper02_measure_time_space_box_i}, is a finite measure on~$[0,T] \times \mathbb{R}$, we may apply H\"{o}lder's inequality with exponents~$2r/q$ and~$2r/(2r - q)$ to conclude that for any~$q \in [1, 2r)$,
    \begin{equation} \label{Paper02_boundededness_K_L_P}
        \sup_{v \in \mathcal{K}} \; \| v \|_{{L}_{q}([0,T] \times \mathbb{R}, \hat \lambda; \ell_{1})} \leq \hat{\lambda}([0,T] \times \mathbb{R})^{(2r-q)/2rq}\sup_{v \in \mathcal{K}} \; \| v \|_{{L}_{2r}([0,T] \times \mathbb{R}, \hat \lambda; \ell_{1})} \lesssim_{T,r,q} C_{1},
    \end{equation}
    where the last inequality comes from the fact that~$\mathcal{K} \subset \mathcal{K}_{(1)}(C_{1},T)$, and that, by definition, the set~$\mathcal{K}_{(1)}(C_{1},T)$ is~$\| \cdot \|_{{L}_{2r}([0,T] \times \mathbb{R}, \hat \lambda; \ell_{1})}$-bounded by~$C_{1}$. Hence, applying~\eqref{Paper02_boundededness_K_L_P} with~$q = r$ we conclude that~$\mathcal{K}$ is bounded.

    \medskip

    \myemph{Step~$(2)$: $r$-Uniform integrability}
   We need to show that
    \begin{equation} \label{Paper02_p-uniform_integrablity}
        \lim_{\gamma \rightarrow \infty} \; \sup_{v \in \mathcal{K}} \int_{0}^{T} \int_{\mathbb{R}} \frac{\| v(t,x) \|_{\ell_{1}}^{r}  \cdot \mathds{1}_{\{\| v (t,x) \|_{\ell_{1}} > \gamma\}}}{1 + \vert x \vert^{2}} \, dx \, dt = 0.
    \end{equation}
    Observe that by the Cauchy-Schwarz inequality, for every~$\gamma > 0$ and every $v \in \mathcal{K}$, we can bound the left-hand side of~\eqref{Paper02_p-uniform_integrablity} by
    \begin{equation}
    \begin{aligned}
        & \int_{0}^{T} \int_{\mathbb{R}} \frac{\| v(t,x) \|_{\ell_{1}}^{r}  \cdot \mathds{1}_{\{\| v(t,x) \|_{\ell_{1}} > \gamma\}}}{1 + \vert x \vert^{2}} \, dx \, dt \\ & \quad \leq \bigg(\int_{0}^{T} \int_{\mathbb{R}} \frac{\| v(t,x) \|_{\ell_{1}}^{2r}}{1 + \vert x\vert^{2}} \, dx \, dt\bigg)^{1/2} \hat{\lambda}\Big(\left\{(t,x) \in [0,T] \times \mathbb{R}: \; \| v(t,x) \|_{\ell_{1}} > \gamma\right\}\Big)^{1/2} \\ & \quad \leq\Big(\sup_{v' \in \mathcal{K}} \; \| v' \|_{L_{2r}([0,T] \times \mathbb{R}, \hat \lambda; \ell_{1})}\Big)^{r} \cdot \bigg(\frac{\sup_{v' \in \mathcal{K}} \; \| v' \|_{L_{1}([0,T] \times \mathbb{R}, \hat \lambda; \ell_{1})}}{\gamma}\bigg)^{1/2},
    \end{aligned}
    \end{equation}
    where for the second inequality we applied Markov's inequality. Applying~\eqref{Paper02_boundededness_K_L_P} with $q = 1$ and using the fact that $\mathcal K \subset \mathcal K_{(1)}(C_1,T)$, and hence $\mathcal K$ is~$\| \cdot \|_{{L}_{2r}([0,T] \times \mathbb{R}, \hat \lambda; \ell_{1})}$-bounded, and then taking the limit as $\gamma \rightarrow \infty$, we conclude that~\eqref{Paper02_p-uniform_integrablity} holds, which completes the proof of step~(2).

    \medskip

    \myemph{Step~$(3)$: Uniform tightness}
    We must establish that for every~$\delta > 0$, there exists a compact subset $\mathcal{A}_{\delta} \subset \ell_{1}$ such that
    \begin{equation} \label{Paper02_uniform_tightness_sequence_space}
        \sup_{v \in \mathcal{K}} \; \hat{\lambda}(\{(t,x) \in [0,T] \times \mathbb{R}: \; v(t,x) \not\in \mathcal{A}_{\delta}\}) \leq \delta.
    \end{equation}
    We first recall (see, for instance, \cite[Exercise~I.6]{diestel2012sequences}) that~$\mathcal{A} \subset \ell_{1}$ is relatively compact if and only if the following conditions hold:
    \begin{enumerate}
        \item[(a)] $\mathcal{A}$ is~$\| \cdot \|_{\ell_{1}}$-bounded;
        \item[(b)] $\displaystyle \lim_{k \rightarrow \infty} \; \sup_{z = (z_i)_{i \in \mathbb N_0} \in \mathcal{A}} \; \sum_{j = k}^{\infty} \vert z_{j} \vert = 0$.
    \end{enumerate}
    Recall the sequences~$(k_{h})_{h \in \mathbb{N}}$ and~$(D_{h})_{h \in \mathbb{N}}$ introduced in the statement of Lemma~\ref{Paper02_alternative_criterion_compactness_diaz_mayoral}. Observe that since $D_{h} \rightarrow 0$ as~$h \rightarrow \infty$ and~$(k_{h})_{h \in \mathbb{N}} \subseteq \mathbb N_0$ is strictly increasing, by passing to a subsequence of~$(k_{h})_{h \in \mathbb{N}}$, we can assume without loss of generality that $D_{h} \leq 2^{-h}$ for every $h \in \mathbb{N}$. For~$\delta > 0$, let~$\mathcal{A}_{\delta} \subset \ell_{1}$ be given by
    \begin{equation} \label{Paper02_definition_compact_subset_ell_1}
    \begin{aligned}
    & \mathcal{A}_{\delta} \defeq \bigg\{z = (z_j)_{j \in \mathbb N_0} \in \ell_{1}: \; \| z \|_{\ell_{1}} \leq \frac{2}{\delta}\sup_{v \in \mathcal{K}} \; \| v \|_{{L}_{1}([0,T] \times \mathbb{R}, \hat \lambda;\ell_{1})} \\[-2mm]
    & \hspace{5cm} \textrm{and } \; \sum_{j = k_{2h}}^{\infty} \vert z_{j} \vert \leq \frac{1}{\delta 2^{h-1}} \; \forall \, h \in \mathbb{N}\bigg\}.
    \end{aligned}
    \end{equation}
    Since, using~\eqref{Paper02_boundededness_K_L_P} with $q = 1$, $\mathcal A_\delta$ satisfies conditions~(a) and~(b) above, we have that~$\mathcal{A}_{\delta}$ is a compact subset of~$\ell_{1}$, for every~$\delta > 0$. Therefore, it will suffice to establish that for any~$\delta > 0$,~\eqref{Paper02_uniform_tightness_sequence_space} holds with the choice of~$\mathcal{A}_{\delta}$ given in~\eqref{Paper02_definition_compact_subset_ell_1}. 
    
    Observe that by Markov's inequality, for every~$v \in \mathcal{K}$,
    \begin{equation} \label{Paper02_first_step_tightness_analysts_definition}
        \hat{\lambda}\bigg(\bigg\{(t,x) \in [0,T] \times \mathbb{R}: \; \| v(t,x) \|_{\ell_{1}} >  \frac{2}{\delta}\sup_{v' \in \mathcal{K}} \; \| v' \|_{{L}_{1}([0,T] \times \mathbb{R}, \hat \lambda; \ell_{1})} \bigg\}\bigg) \leq \frac{\delta}{2}.
    \end{equation}
Recall the definition of $\mathcal{K}_{(2)}((k_{h})_{h \in \mathbb{N}},(D_{h})_{h \in \mathbb{N}},T)$ in the statement of this lemma, and recall our assumption that~$D_{h} \leq 2^{-h} \; \forall \, h \in \mathbb N$, which is possible as explained before~\eqref{Paper02_definition_compact_subset_ell_1}. Hence, using Markov's inequality, we conclude that for every~$v = (v_j)_{j \in \mathbb N_0} \in \mathcal K$ and~$h \in \mathbb{N}$,
\begin{equation} \label{Paper02_second_step_tightness_analysts_definition}
\begin{aligned}
& \hat{\lambda}\bigg(\bigg\{(t,x) \in [0,T] \times \mathbb{R}: \; \sum_{j = k_{2h}}^{\infty} \; \vert v_{j}(t,x) \vert  > \frac{1}{\delta 2^{h-1}}  \bigg\}\bigg) \\
& \quad \leq \delta 2^{h-1} \sum_{j = k_{2h}}^{\infty} \|v_j\|_{L_1([0,T] \times \mathbb R, \hat \lambda; \mathbb R)} 
\\
& \quad \leq \frac{\delta}{2^{h+1}}.
\end{aligned}
\end{equation}
    Combining~\eqref{Paper02_first_step_tightness_analysts_definition} and~\eqref{Paper02_second_step_tightness_analysts_definition}, and then by the definition of~$\mathcal{A}_{\delta}$ in~\eqref{Paper02_definition_compact_subset_ell_1} and a union bound, we have that~\eqref{Paper02_uniform_tightness_sequence_space} holds, which completes the proof of step~(3).

    \medskip

\myemph{Step~$(4)$: Scalarly relative compactness}
    We must establish that for every~$w \in \ell_{\infty}$, the set~$\mathcal{K}^{*}_{w} \defeq \Big\{\Psi_{w}(v): \; v \in \mathcal{K}\Big\}$ is relatively compact in~$L_{r}([0,T] \times \mathbb{R}, \hat{\lambda}; \mathbb{R})$. Fix $w = (w_k)_{k \in \mathbb N_0} \in \ell_{\infty}$, and note that by~\eqref{Paper02_continuity_duality_map} and~\eqref{Paper02_boundededness_K_L_P}, $\mathcal{K}^{*}_{w}$ is $\| \cdot \|_{L_{r}([0,T] \times \mathbb{R}, \hat{\lambda}; \mathbb{R})}$-bounded. Therefore, by Theorem~\ref{Paper02_kolmogorov_riesz_frechet}, it will suffice to prove that the limits~\eqref{Paper02_integral_equicontinuity_property} and~\eqref{Paper02_uniform_integrability_real_space_L_p} hold with~$\mathcal K$ replaced by~$\mathcal K^*_w$. Starting with~\eqref{Paper02_uniform_integrability_real_space_L_p}, observe that since~$w \in \ell_{\infty}$, and then by the Cauchy-Schwarz inequality, we have for every~$R > 0$,
    \begin{equation*}
    \begin{aligned}
        & \sup_{v = (v_k)_{k \in \mathbb N_0} \in \mathcal{K}} \; \int_{0}^{T} \int_{\mathbb{R} \setminus [-R,R]} \frac{\Big\vert \sum_{k \in \mathbb N_0} w_{k}v_{k}(t,x) \Big\vert^{r}} {1 + \vert x \vert^{2}} \, dx \, dt \\ & \quad \leq \| w \|^r_{\ell_{\infty}} \sup_{v \in \mathcal{K}} \int_{0}^{T} \int_{\mathbb{R}} \frac{\| v(t,x) \|^{r}_{\ell_{1}} \cdot \mathds{1}_{\{\vert x \vert > R\}}}{1 + \vert x \vert^{2}} \, dx \, dt \\ & \quad \leq \| w \|_{\ell_{\infty}}^r \sup_{v \in \mathcal{K}} \| v \|_{L_{2r}([0,T] \times \mathbb{R}, \hat{\lambda}; \ell_{1})}^{r} \hat{\lambda}([0,T] \times (\mathbb{R} \setminus [-R,R]))^{1/2}.
    \end{aligned}
    \end{equation*}
    Since~$\mathcal K_{(1)}(C_1,T)$, and hence~$\mathcal{K}$, is bounded in $L_{2r}([0,T] \times \mathbb{R}, \hat{\lambda}; \ell_{1})$, by taking the limit as $R \rightarrow \infty$ and then applying the definition of~$\hat{\lambda}$ in~\eqref{Paper02_measure_time_space_box_i}, by the definition of $\Psi_w(v)$ in~\eqref{Paper02_dual_map_L_p_ell_1} we obtain~\eqref{Paper02_uniform_integrability_real_space_L_p} with~$\mathcal K$ replaced by~$\mathcal K^*_w$.

    We now proceed to the proof of~\eqref{Paper02_integral_equicontinuity_property} with $\mathcal K$ replaced by $\mathcal K^*_w$. It will suffice to establish that there exists $C_{r,T,\mathcal{K},w} > 0$ such that for any~$\varepsilon > 0$, there exists~$\delta = \delta(r,T,\varepsilon, \mathcal K, w) > 0$ such that for~$\vert \gamma \vert < \delta$ and $i \in \{1,2\}$,
    \begin{equation} \label{Paper02_equivalent_formulation_limit_integral_equicontinuity}
        \sup_{v \in \mathcal{K}} \; \| \Psi_{w}(v) - \theta^{(i)}_{\gamma}\Psi_{w}(v) \|^{r}_{L_{r}([0,T] \times \mathbb{R}, \hat \lambda; \mathbb{R})} \leq C_{r,T, \mathcal{K},w} \varepsilon.
    \end{equation}
    For any~$\gamma \in (-1,1)$ and each~$i \in \{1,2\}$, for any $H \in \mathbb N$, by applying the elementary inequality
    \begin{equation} \label{Paper02_elem_ineq_binomial_power_p}
    \vert a + b \vert^r \leq 2^{r-1}(\vert a \vert^r + \vert b\vert^r) \quad \forall \, a,b \in \mathbb R,
    \end{equation}
    we have
    \begin{align}
        & \| \Psi_{w}(v) - \theta^{(i)}_{\gamma}\Psi_{w}(v) \|^{r}_{L_{r}([0,T] \times \mathbb{R}, \hat \lambda; \mathbb{R})} \nonumber
        \\ & \quad \leq 2^{r-1} \int_{0}^{T} \int_{\mathbb{R}} \frac{1}{1 + \vert x \vert^{2}} \bigg(\sum_{k = 0}^{H} \vert w_{k} v_{k}(t,x) - w_{k}(\theta^{(i)}_{\gamma}v_{k})(t,x)\vert\bigg)^{r} \, dx \, dt \label{Paper02_first_integral_equicontinuity}
        \\ & \quad \quad \quad + 2^{r-1} \int_{0}^{T} \int_{\mathbb{R}} \frac{1}{1 + \vert x \vert^{2}} \bigg(\sum_{k = H + 1}^\infty \vert w_{k} v_{k}(t,x) - w_{k}(\theta^{(i)}_{\gamma}v_{k})(t,x)\vert\bigg)^{r} \, dx \, dt \label{Paper02_second_integral_equicontinuity}.
    \end{align}
   We will bound the integrals~\eqref{Paper02_first_integral_equicontinuity} and~\eqref{Paper02_second_integral_equicontinuity} separately. Starting with~\eqref{Paper02_second_integral_equicontinuity}, observing that
   \begin{equation*}
   \begin{aligned}
       & \bigg(\sum_{k = H + 1}^\infty \vert w_{k} v_{k} - w_{k}(\theta^{(i)}_{\gamma}v_{k})\vert\bigg)^{r} \\
       & \quad = \bigg(\sum_{k = H + 1}^\infty \vert w_{k} v_{k} - w_{k}(\theta^{(i)}_{\gamma}v_{k})\vert\bigg)^{1/r} \bigg(\sum_{k = H + 1}^\infty \vert w_{k} v_{k} - w_{k}(\theta^{(i)}_{\gamma}v_{k})\vert\bigg)^{(r^{2}-1)/r},
    \end{aligned}
   \end{equation*}
   and then by applying H\"{o}lder's inequality with exponents~$r$ and $r/(r-1)$ in the case~$r > 1$, we obtain
   \begin{equation} \label{Paper02_intermediate_step_integral_equicontinuity_i}
   \begin{split}
     & \int_{0}^{T} \int_{\mathbb{R}} \frac{1}{1 + \vert x \vert^{2}} \bigg(\sum_{k = H + 1}^\infty \vert w_{k} v_{k}(t,x) - w_{k}(\theta^{(i)}_{\gamma}v_{k})(t,x)\vert\bigg)^{r} dx \, dt \\ & \quad \leq  \bigg(\int_{0}^{T} \int_{\mathbb{R}} \frac{1}{1 + \vert x \vert^{2}} \sum_{k = H + 1}^\infty \vert w_{k} v_{k}(t,x) - w_{k}(\theta^{(i)}_{\gamma}v_{k})(t,x) \vert \, dx \, dt \bigg)^{1/r} \\ & \quad \quad \quad \cdot \bigg(\int_{0}^{T} \int_{\mathbb{R}} \frac{1}{1 + \vert x \vert^{2}} \bigg(\sum_{k = H + 1}^\infty \vert w_{k} v_{k}(t,x) - w_{k}(\theta^{(i)}_{\gamma}v_{k})(t,x) \vert \bigg)^{r+1} dx \, dt \bigg)^{(r-1)/r}.
     \end{split}
     \raisetag{-3.3cm}
   \end{equation}
   We will bound each factor on the right-hand side of~\eqref{Paper02_intermediate_step_integral_equicontinuity_i} separately. Starting with the second factor, recall the definition of~$\theta^{(i)}_\gamma$ in~\eqref{Paper02_shift_maps}, and observe that since $\vert x + \gamma \vert^2 \leq 2 \vert x \vert^2 + 2\vert \gamma \vert^2$, we have
   \begin{equation} \label{Paper02_trivial_ratio_bound_1_plus_square}
       \frac{1 + \vert x + \gamma \vert^{2}}{1 + \vert x \vert^{2}} \leq 3 \quad \forall \, \gamma \in (-1,1) \textrm{ and } \forall \, x \in \mathbb{R}. 
   \end{equation}
   Hence, by applying the elementary inequality~\eqref{Paper02_elem_ineq_binomial_power_p}, and then a suitable substitution, we conclude that for all~$\gamma \in (-1,1)$,~$q \geq 1$,~$g \in L_{q}([0,T] \times \mathbb{R}, \hat \lambda; \mathbb{R})$ and~$i \in \{1,2\}$,
   \begin{equation} \label{Paper02_bounding_difference_shift_function_in_terms_original_function}
       \| \theta^{(i)}_{\gamma}g - g \|_{ L_{q}([0,T] \times \mathbb{R}, \hat \lambda; \mathbb{R})}^{q} \leq 2^{q} \cdot 3 \cdot \| g \|_{L_{q}([0,T] \times \mathbb{R}, \hat \lambda; \mathbb{R})}^q.
   \end{equation}
   Using~\eqref{Paper02_bounding_difference_shift_function_in_terms_original_function} with $g = \sum_{k = H + 1}^\infty \vert w_{k} v_{k}\vert$ and~$q = r+1$, we conclude that the second factor on the right-hand side of~\eqref{Paper02_intermediate_step_integral_equicontinuity_i} is bounded by
   \begin{equation*}
    \begin{aligned}
        & \bigg(\int_{0}^{T} \int_{\mathbb{R}} \frac{1}{1 + \vert x \vert^{2}} \bigg(\sum_{k = H + 1}^\infty \vert w_{k} v_{k}(t,x) - w_{k}(\theta^{(i)}_{\gamma}v_{k})(t,x)\vert \bigg)^{r+1} dx \, dt \bigg)^{(r-1)/r}
        \\ & \quad \leq (2^{r+1}3)^{(r-1)/r} \bigg(\int_{0}^{T} \int_{\mathbb{R}} \frac{1}{1 + \vert x \vert^{2}} \bigg(\sum_{k = H+1}^\infty \vert w_{k} v_{k}(t,x) \vert\bigg)^{r+1} dx \, dt\bigg)^{(r-1)/r}
        \\ & \quad \leq (2^{r+1}3)^{(r-1)/r} \| w \|_{\ell_{\infty}}^{(r+1)(r-1)/r} \| v \|_{{L}_{r+1}([0,T] \times \mathbb{R}, \hat \lambda;\ell_{1})}^{(r+1)(r-1)/r}.
    \end{aligned}
   \end{equation*}
   Since~$r \geq 1$, we have $r+1 \leq 2r$, and therefore, by~\eqref{Paper02_boundededness_K_L_P}, we conclude that
   \begin{equation} \label{Paper02_intermediate_step_integral_equicontinuity_i_second_factor}
   \begin{split}
       & \sup_{v = (v_k)_{k \in \mathbb N_0} \in \mathcal{K}} \bigg(\int_{0}^{T} \int_{\mathbb{R}} \frac{1}{1 + \vert x \vert^{2}} \bigg(\sum_{k = H + 1}^\infty \vert w_{k} v_{k}(t,x) - w_{k}(\theta^{(i)}_{\gamma}v_{k})(t,x)\vert \bigg)^{r+1} dx \, dt \bigg)^{(r-1)/r} \\[+3mm]
       & \hspace{2cm} \lesssim_{r,w,T} C^{(r+1)(r-1)/r}_1.
    \end{split}
    \raisetag{-1.9cm}
   \end{equation}
   We now proceed to bound the first factor on the right-hand side of~\eqref{Paper02_intermediate_step_integral_equicontinuity_i}. Observe that by using~\eqref{Paper02_bounding_difference_shift_function_in_terms_original_function} with~$q = 1$, we have
   \begin{equation} \label{Paper02_intermediate_step_integral_equicontinuity_ii}
   \begin{aligned}
       & \sup_{v = (v_k)_{k \in \mathbb N_0} \in \mathcal K} \bigg(\int_{0}^{T} \int_{\mathbb{R}} \frac{1}{1 + \vert x \vert^{2}} \sum_{k = H + 1}^\infty \vert w_{k} v_{k}(t,x) - w_{k}(\theta^{(i)}_{\gamma}v_{k})(t,x)\vert \, dx \, dt \bigg)^{1/r} \\ & \quad \leq 6^{1/r}  \sup_{v = (v_k)_{k \in \mathbb N_0} \in \mathcal K} \bigg(\int_{0}^{T} \int_{\mathbb{R}} \frac{1}{1 + \vert x \vert^{2}} \sum_{k = H + 1}^\infty \vert w_{k} v_{k}(t,x) \vert \, dx \, dt \bigg)^{1/r}
       \\ & \quad \leq 6^{1/r} \| w \|^{1/r}_{\ell_{\infty}} \sup_{v = (v_k)_{k \in \mathbb N_0} \in \mathcal K} \bigg( \sum_{k = H+1}^\infty \| v_{k} \|_{L_{1}([0,T] \times \mathbb{R}, \hat{\lambda}; \mathbb{R})}\bigg)^{1/r}.
   \end{aligned}
   \end{equation}
   By the definition of~$\mathcal{K}_{(2)}((k_{h})_{h \in \mathbb{N}}, (D_{h})_{h \in \mathbb{N}}, T)$ in the statement of the lemma, the right-hand side of~\eqref{Paper02_intermediate_step_integral_equicontinuity_ii} vanishes as~$H \rightarrow \infty$. Therefore, for every~$\varepsilon > 0$, we can take~$H \in \mathbb{N}$ sufficiently large that
   \begin{equation} \label{Paper02_intermediate_step_integral_equicontinuity_iii}
   \begin{split}
       & \sup_{v = (v_k)_{k \in \mathbb N_0} \in \mathcal{K}} \bigg(\int_{0}^{T} \int_{\mathbb{R}} \frac{1}{1 + \vert x \vert^{2}} \sum_{k = H + 1}^\infty \vert w_{k} v_{k}(t,x) - w_{k}(\theta^{(i)}_{\gamma}v_{k})(t,x)\vert \, dx \, dt \bigg)^{1/r} \\
       & \quad \lesssim_{r,w} \varepsilon.
    \end{split}
   \end{equation}
   Applying~\eqref{Paper02_intermediate_step_integral_equicontinuity_i_second_factor} and~\eqref{Paper02_intermediate_step_integral_equicontinuity_iii} to~\eqref{Paper02_intermediate_step_integral_equicontinuity_i}, we conclude that for~$H \in \mathbb{N}$ sufficiently large, for any $\gamma \in (-1,1)$,
   \begin{equation} \label{Paper02_intermediate_step_integral_equicontinuity_iv}
   \begin{split}
       & \sup_{v = (v_k)_{k \in \mathbb N_0} \in \mathcal{K}} \int_{0}^{T} \int_{\mathbb{R}} \frac{1}{1 + \vert x \vert^{2}} \bigg(\sum_{k = H + 1}^\infty \vert w_{k} v_{k}(t,x) - w_{k}(\theta^{(i)}_{\gamma}v_{k})(t,x)\vert\bigg)^{r} dx \, dt \\
       & \quad \lesssim_{r,w,T} C_1^{(r+1)(r-1)/r}\varepsilon.
    \end{split}
   \end{equation}

   We will now bound the integral~\eqref{Paper02_first_integral_equicontinuity}. As in the derivation of~\eqref{Paper02_intermediate_step_integral_equicontinuity_i}, by applying H\"{o}lder's inequality with exponents~$r$ and~$r/(r-1)$ in the case~$r > 1$, we obtain
   \begin{equation} \label{Paper02_intermediate_step_integral_equicontinuity_v}
    \begin{split}
        & \int_{0}^{T} \int_{\mathbb{R}} \frac{1}{1 + \vert x \vert^{2}} \bigg(\sum_{k = 0}^{H} \vert w_{k} v_{k}(t,x) - w_{k}(\theta^{(i)}_{\gamma}v_{k})(t,x)\vert\bigg)^{r} \, dx \, dt \\ & \quad \leq \bigg(\int_{0}^{T} \int_{\mathbb{R}} \frac{1}{1 + \vert x \vert^{2}} \sum_{k = 0}^{H} \vert w_{k} v_{k}(t,x) - w_{k}(\theta^{(i)}_{\gamma}v_{k})(t,x)\vert \, dx \, dt \bigg)^{1/r} \\ & \quad \quad \quad \cdot \bigg(\int_{0}^{T} \int_{\mathbb{R}} \frac{1}{1 + \vert x \vert^{2}} \bigg(\sum_{k = 0}^{H} \vert w_{k} v_{k}(t,x) - w_{k}(\theta^{(i)}_{\gamma}v_{k})(t,x)\vert\bigg)^{r+1} dx \, dt \bigg)^{(r-1)/r}.
    \end{split}
    \raisetag{-3.4cm}
   \end{equation}
   We will bound each factor on the right-hand side of~\eqref{Paper02_intermediate_step_integral_equicontinuity_v} separately. Starting with the second factor, as in the derivation of~\eqref{Paper02_intermediate_step_integral_equicontinuity_i_second_factor}, by using~\eqref{Paper02_bounding_difference_shift_function_in_terms_original_function}, we have
   \begin{equation*}
    \begin{aligned}
        & \bigg(\int_{0}^{T} \int_{\mathbb{R}} \frac{1}{1 + \vert x \vert^{2}} \bigg(\sum_{k = 0}^{H} \vert w_{k} v_{k}(t,x) - w_{k}(\theta^{(i)}_{\gamma}v_{k})(t,x)\vert\bigg)^{r+1} dx \, dt \bigg)^{(r-1)/r} 
        \\ & \quad \leq (2^{r+1}3)^{(r-1)/r} \bigg(\int_{0}^{T} \int_{\mathbb{R}} \frac{1}{1 + \vert x \vert^{2}} \bigg(\sum_{k = 0}^{H} \vert w_{k} v_{k}(t,x) \vert\bigg)^{r+1} dx \, dt\bigg)^{(r-1)/r}
        \\ & \quad \leq (2^{r+1}3)^{(r-1)/r} \| w \|_{\ell_{\infty}}^{(r+1)(r-1)/r} \| v \|_{{L}_{r+1}([0,T] \times \mathbb{R}, \hat \lambda; \ell_{1})}^{(r+1)(r-1)/r},
    \end{aligned}
   \end{equation*}
   and therefore, by~\eqref{Paper02_boundededness_K_L_P}, 
   \begin{equation} \label{Paper02_intermediate_step_integral_equicontinuity_v_second_factor}
   \begin{split}
       & \sup_{v = (v_k)_{k \in \mathbb N_0} \in \mathcal{K}} \bigg(\int_{0}^{T} \int_{\mathbb{R}} \frac{1}{1 + \vert x \vert^{2}} \bigg(\sum_{k =0}^{H} \vert w_{k} v_{k}(t,x) - w_{k}(\theta^{(i)}_{\gamma}v_{k})(t,x)\vert \bigg)^{r+1} dx \, dt \bigg)^{(r-1)/r} \\[+3mm]
       & \hspace{2cm} \lesssim_{r,w,T} C_1^{(r+1)(r-1)/r}.
    \end{split}
    \raisetag{-1.9cm}
   \end{equation}
   It remains to bound the first factor on the right-hand side of~\eqref{Paper02_intermediate_step_integral_equicontinuity_v}. Without loss of generality, assume~$\gamma > 0$, as the case~$\gamma < 0$ can be treated analogously. We start by observing that the triangle inequality, and then a suitable substitution and~\eqref{Paper02_trivial_ratio_bound_1_plus_square} imply that for every~$(\gamma_{j})_{j \in \mathbb N_0} \subset [0,1]^{\mathbb N_0} \cap \ell_1$ with $\gamma_{0} = 0$ and $\sum_{j = 1}^\infty \gamma_j \leq 1$, $i \in \{1,2\}$, $A \in \mathbb N$ and~$g \in {L}_{1}([0,T] \times \mathbb{R}, \hat \lambda; \mathbb{R})$,
   \begin{equation} \label{Paper02_generalisation_triangle_inequality_shift_function_modified}
   \begin{aligned}
       \Big\vert \Big\vert \theta^{(i)}_{\sum_{j = 1}^{A} \gamma_{j}}g - g \Big\vert \Big\vert_{L_{1}([0,T] \times \mathbb{R}, \hat{\lambda}; \mathbb{R})} & \leq \sum_{j = 1}^{A}  \Big\vert \Big\vert \theta^{(i)}_{\sum_{l = 1}^{j} \gamma_{l}}g - \theta^{(i)}_{\sum_{l = 0}^{j-1} \gamma_{l}} g \Big\vert \Big\vert_{L_{1}([0,T] \times \mathbb{R}, \hat{\lambda}; \mathbb{R})} \\ & \leq 3 \sum_{j = 1}^{A}  \| \theta^{(i)}_{\gamma_{j}}g - g \|_{L_{1}([0,T] \times \mathbb{R}, \hat{\lambda}; \mathbb{R})}.
    \end{aligned}
   \end{equation}
   By taking the limit as $A \rightarrow \infty$ on both sides of~\eqref{Paper02_generalisation_triangle_inequality_shift_function_modified} and using Remark~\ref{Paper02_corollary_Kolmogorov_Riesz}, we conclude that for any~$(\gamma_{j})_{j \in \mathbb N_0} \subset [0,1]^{\mathbb N_0} \cap \ell_1$ with $\gamma_{0} = 0$ and $\sum_{j = 1}^\infty \gamma_j \leq 1$, $i \in \{1,2\}$ and~$g \in {L}_{1}([0,T] \times \mathbb{R}, \hat \lambda; \mathbb{R})$,
   \begin{equation} \label{Paper02_generalisation_triangle_inequality_shift_function}
       \Big\vert \Big\vert \theta^{(i)}_{\sum_{j = 1}^{\infty} \gamma_{j}}g - g \Big\vert \Big\vert_{L_{1}([0,T] \times \mathbb{R}, \hat{\lambda}; \mathbb{R})} \leq 3 \sum_{j = 1}^{\infty}  \| \theta^{(i)}_{\gamma_{j}}g - g \|_{L_{1}([0,T] \times \mathbb{R}, \hat{\lambda}; \mathbb{R})}. 
   \end{equation}
   Recall the role of the parameter~$J \in \mathbb{N}$ in the definition of the set~$\mathcal{K}_{(3)}(C_{2}, \beta, J, T)$ in the statement of Lemma~\ref{Paper02_alternative_criterion_compactness_diaz_mayoral}. Take~$J^{*} \in \mathbb{N}$ satisfying~$J^{*} \geq J$, whose exact value will be specified later. Recalling that $\beta > 1$, suppose $H \in \mathbb N$ is sufficiently large that $H+1 < \beta^{H+1}$. By taking~$\gamma \in (0, 2^{-(H + J^{*})})$, we conclude from the dyadic expansion of~$\gamma$ that there exists~$\mathcal{I}_\gamma \subseteq \mathbb{N} \cap [H + J^{*} + 1, \infty)$ such that~$\gamma = \sum_{j \in \mathcal{I}_\gamma}2^{-j}$. By applying this representation and~\eqref{Paper02_generalisation_triangle_inequality_shift_function}, and then using the definition of~$\mathcal{K}_{(3)}(C_{2}, \beta, J, T)$, we conclude that the first factor on the right-hand side of~\eqref{Paper02_intermediate_step_integral_equicontinuity_v} is bounded, for all~$v = (v_k)_{k \in \mathbb N_0} \in \mathcal{K}$, by
   \begin{equation*}
    \begin{aligned}
        & \bigg(\int_{0}^{T} \int_{\mathbb{R}} \frac{1}{1 + \vert x \vert^{2}} \sum_{k = 0}^{H} \vert w_{k} v_{k}(t,x) - w_{k}(\theta^{(i)}_{\gamma}v_{k})(t,x)\vert \, dx \, dt \bigg)^{1/r}
        \\ & \quad \leq 3^{1/r} \| w \|_{\ell_{\infty}}^{1/r} \bigg(\sum_{k = 0}^{H} \, \sum_{j \in \mathcal{I}_\gamma} \| \theta^{(i)}_{2^{-j}}v_{k} - v_{k} \|_{L_{1}([0,T] \times \mathbb{R}, \hat{\lambda}; \mathbb{R})}\bigg)^{1/r}
        \\ & \quad \leq 3^{1/r} \| w \|_{\ell_{\infty}}^{1/r} \bigg(\sum_{k = 0}^{H} \, \sum_{j = H + J^{*} + 1}^\infty \| \theta^{(i)}_{2^{-j}}v_{k} - v_{k} \|_{L_{1}([0,T] \times \mathbb{R}, \hat{\lambda}; \mathbb{R})}\bigg)^{1/r}
        \\ & \quad \leq 3^{1/r} \| w \|_{\ell_{\infty}}^{1/r}C_{2}^{1/r} \bigg(\sum_{k = 0}^{H} \, \sum_{j = H + J^{*} + 1}^\infty \beta^{- j}\bigg)^{1/r}
        \\ & \quad = \left(\frac{3 \| w \|_{\ell_{\infty}}\beta C_{2} (H+1)}{\beta^{H + J^{*} + 1} (\beta - 1)}\right)^{1/r}
        \\ & \quad \leq \left(\frac{3 \| w \|_{\ell_{\infty}}\beta C_{2}}{\beta^{J^{*}} (\beta - 1)}\right)^{1/r},
    \end{aligned}
   \end{equation*}
   where for the fourth line we used the fact that~$\beta > 1$, and the last inequality follows from our choice of $H$. Hence, using the same argument for the case $\gamma < 0$, for any~$\varepsilon > 0$, we can choose~$H,J^{*} \in \mathbb N$ sufficiently large that for any~$\gamma \in (-2^{-(H + J^{*})},2^{-(H + J^{*})})$,
   \begin{equation} \label{Paper02_intermediate_step_integral_equicontinuity_vi}
   \begin{aligned}
       & \sup_{v = (v_k)_{k \in \mathbb N_0} \in \mathcal{K}} \bigg(\int_{0}^{T} \int_{\mathbb{R}} \frac{1}{1 + \vert x \vert^{2}} \sum_{k = 0}^{H} \vert w_{k} v_{k}(t,x) - w_{k}(\theta^{(i)}_{\gamma}v_{k})(t,x)\vert \, dx \, dt \bigg)^{1/r} \\
       & \quad \lesssim_{r,w,\beta,C_2} \varepsilon.
    \end{aligned}
   \end{equation}
   By applying~\eqref{Paper02_intermediate_step_integral_equicontinuity_v_second_factor} and~\eqref{Paper02_intermediate_step_integral_equicontinuity_vi} to~\eqref{Paper02_intermediate_step_integral_equicontinuity_v}, we conclude that for $\varepsilon > 0$, by choosing~$H\in \mathbb N$ and~$J^{*} \geq J$ sufficiently large, for any $\gamma \in (-2^{-(H+J^*)},2^{-(H+J^*)})$,
   \begin{equation} \label{Paper02_intermediate_step_integral_equicontinuity_vii}
   \begin{aligned}
       & \sup_{v = (v_k)_{k \in \mathbb N_0} \in \mathcal K} \int_{0}^{T} \int_{\mathbb{R}} \frac{1}{1 + \vert x \vert^{2}} \bigg(\sum_{k = 0}^{H} \vert w_{k} v_{k}(t,x) - w_{k}(\theta^{(i)}_{\gamma}v_{k})(t,x)\vert\bigg)^{r} \, dx \, dt \\
       & \quad \lesssim_{r,w,T,C_1,\beta,C_2} \varepsilon.
    \end{aligned}
   \end{equation}
   Applying~\eqref{Paper02_intermediate_step_integral_equicontinuity_vii} and~\eqref{Paper02_intermediate_step_integral_equicontinuity_iv} to~\eqref{Paper02_first_integral_equicontinuity} and~\eqref{Paper02_second_integral_equicontinuity}, we conclude that by choosing~$H \in \mathbb{N}$ and~$J^{*} \geq J$ sufficiently large, and then letting~$\gamma \in (-2^{- (H + J^{*})}, 2^{- (H + J^{*})})$, we obtain~\eqref{Paper02_equivalent_formulation_limit_integral_equicontinuity}. As explained before~\eqref{Paper02_equivalent_formulation_limit_integral_equicontinuity}, it then follows by Theorem~\ref{Paper02_kolmogorov_riesz_frechet} that for every~$w \in \ell_{\infty}$, the set~$\mathcal{K}_{w}^{*}$ is a relatively compact subset of~$L_{r}([0,T] \times \mathbb{R}, \hat \lambda; \mathbb{R})$, i.e.~that~$\mathcal{K}$ is scalarly relatively compact.

   \medskip

   Finally, combining steps~$(1)$-$(4)$, we conclude the set~$\mathcal{K}$ satisfies the conditions of Theorem~\ref{Paper02_diaz_mayoral_compactness_theorem}, and therefore that~$\mathcal{K}$ is relatively compact in~$L_{r}([0,T] \times \mathbb{R}, \hat{\lambda}; \ell_{1})$.
\end{proof}

Next, we establish a result on $L_r([0,T] \times \mathbb R, \hat \lambda; \ell_1)$ which is used in the proof of Lemma~\ref{Paper02_weak_convergence_mild_solution} in Section~\ref{Paper02_Characterisation_limiting_process}.

\begin{lemma} \label{Paper02_standard_L_p_gaussian_lem}
      Suppose $m > 0$, and let $p:(0, \infty) \times \mathbb R \rightarrow (0, \infty)$ be defined as in~\eqref{Paper02_Gaussian_kernel}. Let $T > 0$ and $r \in [4, \infty)$ be fixed. Let \[ (\nu^N)_{N \in \mathbb N} = \Big((\nu^N_k)_{k \in \mathbb N_0}\Big)_{N \in \mathbb N} \subset L_{r}([0,T] \times \mathbb R, \hat \lambda; \ell_1)\] denote a sequence that converges to $\nu = (\nu_k)_{k \in \mathbb N_0}$ in~$L_{r}([0,T] \times \mathbb R, \hat \lambda; \ell_1)$ as $N \rightarrow \infty$. Then, for any $l \in \{0\} \cup [1,r/4]$, $k \in \mathbb N_0$, and $(t,x) \in [0,T] \times \mathbb R$,
\begin{equation} \label{Paper02_aux_L_p_dom_c_ii}
\begin{aligned}
    & \lim_{N \rightarrow \infty} \int_0^t \int_\mathbb R p(t-\tau,x - y) \nu^{N}_k(\tau,y) \| \nu^{N}(\tau,y) \|_{\ell_1}^l \, dy\, d\tau \\ & \quad = \int_0^t \int_\mathbb R p(t - \tau,x-y) \nu_k(\tau,y) \| \nu(\tau,y) \|_{\ell_1}^l \, dy \, d\tau.
\end{aligned}
\end{equation}
\end{lemma}

\begin{proof}
   It will suffice to show that for all~$l \in \{0\} \cup [1,r/4]$,~$k \in \mathbb N_0$ and~$(t,x) \in [0,T] \times \mathbb  R$,
    \begin{equation} \label{Paper02_simple_Gaussian_convergence_l_0_i}
    \lim_{N \rightarrow \infty} \int_0^t \int_\mathbb R p(t-\tau,x-y) \Big\vert \nu^{N}_k(\tau,y) \| \nu^{N}(\tau,y) \|_{\ell_1}^l - \nu_k(\tau,y) \| \nu(\tau,y) \|_{\ell_1}^l \, \Big\vert \, dy \, d\tau = 0.
    \end{equation}
    Observe that for all~$N \in \mathbb N$,~$l \in \{0\} \cup [1,r/4]$, $k \in  \mathbb N_0$ and~$(t,y) \in [0,T] \times \mathbb R$, by the triangle inequality,
    \begin{equation} \label{Paper02_aux_L_p_dom_c_iii}
    \begin{aligned}
        &  \Big\vert \nu^{N}_k(\tau,y) \| \nu^{N}(\tau,y) \|_{\ell_1}^l - \nu_k(\tau,y) \| \nu(\tau,y) \|_{\ell_1}^l \, \Big\vert
        \\ & \quad \leq \vert \nu^N_k(\tau,y) - \nu_k(\tau,y) \vert \cdot \| \nu(\tau,y) \|_{\ell_1}^l + \vert \nu^N_k(\tau,y) \vert \cdot \Big\vert \| \nu(\tau,y) \|_{\ell_1}^l -  \| \nu^N(\tau,y) \|_{\ell_1}^l \Big\vert
        \\ & \quad \leq \| \nu^N(\tau,y) - \nu(\tau,y) \|_{\ell_1} \cdot \| \nu(\tau,y) \|_{\ell_1}^l 
        \\ & \qquad + l \| \nu^N(\tau,y) \|_{\ell_1} \Big(\| \nu^N(\tau,y) \|_{\ell_1}^{l-1} + \| \nu(\tau,y) \|_{\ell_1}^{l-1} \Big) \Big\vert \| \nu(\tau,y) \|_{\ell_1} -  \| \nu^N(\tau,y) \|_{\ell_1} \Big\vert
        \\ & \quad \leq (l+1) \| \nu^N(\tau,y) - \nu(\tau,y) \|_{\ell_1} \Big(\| \nu^N(\tau,y) \|_{\ell_1} + \| \nu(\tau,y) \|_{\ell_1}\Big)^l,
    \end{aligned}
    \end{equation}
    where the second inequality follows from the elementary inequality~\eqref{Paper02_elementary_factorising} and since $\vert \nu^N_k(\tau,y) \vert \leq \| \nu^N(\tau,y) \|_{\ell_1}$, and the third inequality follows from the triangle inequality. We use~\eqref{Paper02_aux_L_p_dom_c_iii} to bound the argument of the limit on the left-hand side of~\eqref{Paper02_simple_Gaussian_convergence_l_0_i}, obtaining
    \begin{equation} \label{Paper02_simple_Gaussian_convergence_l_0_ii}
    \begin{split}
        & \int_0^t \int_\mathbb R p(t-\tau, x-y) \Big\vert \nu^{N}_k(\tau,y) \| \nu^{N}(\tau,y) \|_{\ell_1}^l - \nu_k(\tau,y) \| \nu(\tau,y) \|_{\ell_1}^l \, \Big\vert \, dy \, d\tau \\
        & \quad \leq (l+1)\int_0^t \int_\mathbb R (1 + \vert y \vert^2)^{1/2} p(t-\tau, x-y) \\
        & \qquad \quad \cdot \left((1 + \vert y \vert^2)^{-1/2}\| \nu^N(\tau,y) - \nu(\tau,y) \|_{\ell_1}\Big(\| \nu^N(\tau,y) \|_{\ell_1} + \| \nu(\tau,y) \|_{\ell_1}\Big)^l\right) \, dy \, d\tau \\
        & \quad \leq (l+1)\left(\int_0^t \int_\mathbb R (1 + \vert y \vert^2)p^2(t-\tau, x - y) \, dy \, d\tau\right)^{1/2} \\
        & \qquad \quad \cdot \left(\int_0^t \int_\mathbb R \frac{1}{1 + \vert y \vert^2} \| \nu^N(\tau,y) - \nu(\tau,y) \|_{\ell_1}^2 \Big(\| \nu^N(\tau,y) \|_{\ell_1} + \| \nu(\tau,y) \|_{\ell_1}\Big)^{2l} \, dy \, d\tau \right)^{1/2},
    \end{split}
    \raisetag{-3.15cm}
    \end{equation}
    where for the second inequality we used the Cauchy-Schwarz inequality. By the definition of the Gaussian kernel $p$ in~\eqref{Paper02_Gaussian_kernel}, we have for all $(t,x) \in [0,T] \times \mathbb R$,
    \begin{equation} \label{Paper02_simple_Gaussian_convergence_l_0_iii}
    \begin{aligned}
        & \int_0^t \int_\mathbb R (1 + \vert y \vert^2)p^2(t-\tau, x - y) \, dy \, d\tau \\
        & \quad = \int_0^t \frac{1}{2\sqrt{\pi m (t-\tau)}} \int_\mathbb R (1 + \vert y \vert^2) \frac{1}{\sqrt{\pi m (t - \tau)}} \exp \left(- \frac{\vert  x- y \vert^2}{m(t-\tau)}\right) \, dy \, d\tau \\
        & \quad = \int_0^t \frac{1}{2\sqrt{\pi m (t-\tau)}} \left((1 + x^2) + \tfrac{m}{2}(t-\tau)\right) \, d\tau \\
        & \quad = (1 + x^2)\sqrt{\tfrac{t}{\pi m}} + \frac{1}{6}\sqrt{\tfrac{m}{\pi}} t^{3/2}.
    \end{aligned}
    \end{equation}
    Hence, by~\eqref{Paper02_simple_Gaussian_convergence_l_0_ii} and~\eqref{Paper02_simple_Gaussian_convergence_l_0_iii}, the limit in~\eqref{Paper02_simple_Gaussian_convergence_l_0_i} will be proved after establishing that the second factor on the right-hand side of~\eqref{Paper02_simple_Gaussian_convergence_l_0_ii} vanishes as $N \rightarrow \infty$. By applying the Cauchy-Schwarz inequality to the second factor on the right-hand side of~\eqref{Paper02_simple_Gaussian_convergence_l_0_ii} and the definition of $L_4([0,T] \times \mathbb R, \hat \lambda; \ell_1)$ in~\eqref{Paper02_functional_space_real_functions_measure_time_space_box},
    \begin{equation*}
    \begin{aligned}
        & \int_0^t \int_\mathbb R \frac{1}{1 + \vert y \vert^2} \| \nu^N(\tau,y) - \nu(\tau,y) \|_{\ell_1}^2 \Big(\| \nu^N(\tau,y) \|_{\ell_1} + \| \nu(\tau,y) \|_{\ell_1}\Big)^{2l} \, dy \, d\tau \\
        & \quad \leq \|\nu^N - \nu\|_{L_4([0,T] \times \mathbb R, \hat \lambda; \ell_1)}^2 \left(\int_0^t \int_\mathbb R \frac{1}{1 + \vert y \vert^2} \Big(\| \nu^N(\tau,y) \|_{\ell_1} + \| \nu(\tau,y) \|_{\ell_1}\Big)^{4l} \, dy \, d\tau \right)^{1/2}.
    \end{aligned}
    \end{equation*}
    By applying the elementary inequality~\eqref{Paper02_power_bound_elem}, the definition of $L_{4l}([0,T] \times \mathbb R, \hat \lambda; \ell_1)$ in~\eqref{Paper02_functional_space_real_functions_measure_time_space_box} for $l \in [1,r/4]$, and the elementary inequality $\sqrt{a + b} \leq \sqrt{a} + \sqrt{b} \; \forall \, a,b \geq 0$, we conclude that
    \begin{equation} \label{Paper02_simple_Gaussian_convergence_l_0_iv}
    \begin{aligned}
        & \int_0^t \int_\mathbb R \frac{1}{1 + \vert y \vert^2} \| \nu^N(\tau,y) - \nu(\tau,y) \|_{\ell_1}^2 \Big(\| \nu^N(\tau,y) \|_{\ell_1} + \| \nu(\tau,y) \|_{\ell_1}\Big)^{2l} \, dy \, d\tau \\
        & \, \leq \|\nu^N - \nu\|_{L_4([0,T] \times \mathbb R, \hat \lambda; \ell_1)}^2 \\
        & \quad \; \cdot \Big(\mathds 1_{\{l = 0\}} \hat \lambda ([0, T] \times \mathbb R)^{1/2} \\
        & \qquad \quad + \mathds 1_{\{l \in [1,r/4]\}} 2^{(4l-1)/2} \left(\|\nu^N\|_{L_{4l}([0,T] \times \mathbb R, \hat \lambda; \ell_1)}^{2l}  + \|\nu\|_{L_{4l}([0,T] \times \mathbb R, \hat \lambda; \ell_1)}^{2l}\right)\Big).
    \end{aligned}
    \end{equation}
    Observe that by H\"older's inequality and the fact that the measure $\hat \lambda$ defined in~\eqref{Paper02_measure_time_space_box_i} is a finite measure on $[0,T] \times \mathbb R$, for all $q \in [1,r)$ and $\nu^* \in L_r([0,T] \times \mathbb R, \hat \lambda; \ell_1)$,
    \[
    \| \nu^* \|_{L_q([0,T] \times \mathbb R, \hat \lambda; \ell_1)} \leq \hat \lambda ([0,T] \times \mathbb R)^{(r-q)/(rq)}    \| \nu^* \|_{L_r([0,T] \times \mathbb R, \hat \lambda; \ell_1)}.
    \]
    Therefore, by taking the limit as $N \rightarrow \infty$ on both sides of~\eqref{Paper02_simple_Gaussian_convergence_l_0_iv}, and using that $r \in [4, \infty)$, $l \in \{0\} \cup [1, r/4]$ and $\lim_{N \rightarrow \infty} \|\nu^N - \nu\|_{L_r([0,T] \times \mathbb R, \hat \lambda; \ell_1)} = 0$, and then by using~\eqref{Paper02_simple_Gaussian_convergence_l_0_ii} and~\eqref{Paper02_simple_Gaussian_convergence_l_0_iii}, we conclude that the limit in~\eqref{Paper02_simple_Gaussian_convergence_l_0_i} holds, which completes the proof.
    \end{proof}

\section{Proof of Lemmas~\ref{Paper02_properties_greens_function_muller_ratchet} and~\ref{Paper02_well_behaved_function_inner_product_test_funciton_lemma}} \label{Paper02_appendix_section_proof_auxiliary_no_technical_lemmas}

In this section, we will prove two lemmas used in Sections~\ref{Paper02_section_greens_function} and~\ref{Paper02_tightness} whose proofs follow from standard arguments. We start by proving Lemma~\ref{Paper02_properties_greens_function_muller_ratchet}.

\begin{proof}[Proof of Lemma~\ref{Paper02_properties_greens_function_muller_ratchet}]
     We will divide the proof into steps corresponding to each of the assertions (i) - (vi).

     \medskip

\myemph{Proof of assertion~(i)}
     We start by establishing that $g^{N,T,x}_{\mathcal I} \in \mathcal C([0,T] \times \mathcal S^N; \mathbb R)$. Note that by~\eqref{Paper02_trivial_large_deviation_estimatre_CTRW}, there exists $C_{x,T}^N > 0$ such that for all $y \in L_N^{-1}\mathbb Z$ and $t \in [0,T]$,
     \begin{equation} \label{Paper02_trivial_exp_poly_bound_RW}
         \phi^{N,T,x}(t,y) = p^N(T-t,x-y) \leq C_{x,T}^N \frac{1}{(1 + \vert y \vert)^2}.
     \end{equation}
     Therefore, by~\eqref{Paper02_green_function_representation_writing_as_lipschitz_map} and~\eqref{Paper02_green_function_test_function}, and then by using~\eqref{Paper02_trivial_exp_poly_bound_RW} and Proposition~\ref{Paper02_topological_properties_state_space}(i) and~(ii), we can use dominated convergence and the fact that $p^N(\cdot, x) \in \mathcal C([0,T]; \mathbb R)$ for all $x \in L_N^{-1}\mathbb Z$, to conclude that for all $(t,\boldsymbol \xi) \in [0,T] \times \mathcal S^N$ and $x \in L_{N}^{-1}\mathbb Z$, the following limit holds:
     \begin{equation*}
     \begin{aligned}
        \lim_{(t^*,\boldsymbol{\xi}^*) \rightarrow (t, \boldsymbol{\xi})} g^{N,T,x}_{\mathcal I} (t^*,\boldsymbol{\xi}^*) & = \lim_{(t^*,\boldsymbol{\xi}^*) \rightarrow (t, \boldsymbol{\xi})}  \sum_{y \in L_N^{-1}\mathbb Z} \; \sum_{k \in \mathcal{I}} \frac{\xi^*_{k}(y)}{N} p^N(T-t^*,x-y) \\
        & = \sum_{y \in L_N^{-1}\mathbb Z} \; \sum_{k \in \mathcal{I}} \lim_{(t^*,\boldsymbol{\xi}^*) \rightarrow (t, \boldsymbol{\xi})} \frac{\xi^*_{k}(y)}{N} p^N(T-t^*,x-y) \\
        & = \sum_{y \in L_N^{-1}\mathbb Z} \; \sum_{k \in \mathcal{I}} \frac{\xi_{k}(y)}{N} p^N(T-t,x-y)
        \\
        & = g^{N,T,x}_{\mathcal I}(t,\boldsymbol{\xi}).
     \end{aligned}
     \end{equation*}
     Hence, $g^{N,T,x}_{\mathcal I} \in \mathcal C([0,T] \times \mathcal S^N; \mathbb R)$. In order to characterise $\frac{\partial}{\partial t} g^{N,T,x}_{\mathcal I}$, we observe that~\eqref{Paper02_green_function_test_function} implies that for all~$t \in (0,T)$ and~$y \in L_{N}^{-1}\mathbb{Z}$,
     \begin{equation} \label{Paper02_derivative_time_green_funciton_first_part}
     \begin{split}
         & \left(\frac{\partial}{\partial t} \phi^{N,T,x}(\cdot, y)\right)(t) \\[+3mm]
         & \quad = \frac{m_{N}}{2} \Big(2p^{N}(T-t,y-x) - p^{N}(T-t,y - L_{N}^{-1}-x) - p^{N}(T-t,y + L_{N}^{-1}-x) \Big).
        \end{split}
    \raisetag{-1.6cm}
     \end{equation}
     By using~\eqref{Paper02_green_function_representation_writing_as_lipschitz_map},~\eqref{Paper02_derivative_time_green_funciton_first_part},~\eqref{Paper02_trivial_exp_poly_bound_RW}, and the definition of $\mathcal S^N$ in~\eqref{Paper02_definition_state_space_formal}, we conclude by dominated convergence that for all $t \in (0,T)$ and $\boldsymbol \xi = (\xi_k(y))_{k \in \mathbb N_0, y \in L_N^{-1}\mathbb Z} \in \mathcal S^N$,
     \begin{equation*}
     \begin{aligned}
         \left(\frac{\partial}{\partial t} g^{N,T,x}_{\mathcal{I}}(\cdot, \boldsymbol{\xi})\right)(t) & =  \bigg\langle \sum_{k \in \mathcal{I}} \, \frac{\xi_{k}(\cdot)}{N}, \,   \frac{\partial}{\partial t} \phi^{N,T,x}(t, \cdot)\bigg\rangle_{N} \\ & = \frac{m_{N}}{2}\Big(2g^{N,T,x}_{\mathcal{I}} - g^{N,T,x - L_{N}^{-1}}_{\mathcal{I}} - g^{N,T,x + L_{N}^{-1}}_{\mathcal{I}}\Big)(t,\boldsymbol{\xi}),
     \end{aligned}
     \end{equation*}
     where the last equality follows from~\eqref{Paper02_derivative_time_green_funciton_first_part} and the fact that by~\eqref{Paper02_trivial_exp_poly_bound_RW} and~\eqref{Paper02_definition_state_space_formal}, the series defining the map~$g^{N,T,x}_{\mathcal{I}}$ is absolutely convergent. This gives us identity~\eqref{Paper02_formula_derivative_time_greens_function}. The desired continuity of the map~$(0,T) \times \mathcal{S}^{N} \ni (t,\boldsymbol{\xi}) \mapsto \left(\frac{\partial}{\partial t} g^{N,T,x}_{\mathcal{I}}(\cdot, \boldsymbol{\xi})\right)(t)$ then follows from the continuity of~$g^{N,T,x}_{\mathcal{I}}$. This concludes the proof of assertion~(i).

     \medskip

\myemph{Proof of assertion~(ii)}
For any~$t \in [0,T]$ and~$\boldsymbol{\xi} = (\xi(y))_{y \in L_N^{-1}\mathbb Z}, \boldsymbol{\zeta} = (\zeta(y))_{y \in L_N^{-1}\mathbb Z} \in \mathcal{S}^{N}$, we have, from the definition of $g^{N,T,x}_{\mathcal{I}}$ in~\eqref{Paper02_green_function_representation_writing_as_lipschitz_map},
     \begin{equation*}
     \begin{aligned}
         \Big\vert g^{N,T,x}_{\mathcal{I}}(t,\boldsymbol{\xi}) - g^{N,T,x}_{\mathcal{I}}(t,\boldsymbol{\zeta}) \Big\vert & \leq \frac{1}{NL_N}\sum_{y \in L_{N}^{-1}\mathbb{Z}} \| \xi(y) - \zeta(y) \|_{\ell_{1}} p^{N}(T-t,y-x) \\ & \lesssim_{x,N} \sum_{y \in L_{N}^{-1}\mathbb{Z}} \frac{\| \xi(y) - \zeta(y) \|_{\ell_{1}}}{(1 + \vert y \vert)^{2(1 + \deg q_-)}},
     \end{aligned}
     \end{equation*}
     where the last estimate follows from~\eqref{Paper02_trivial_exp_poly_bound_RW}. Therefore, by the definition of $\mathcal C_*(\mathcal S^N; \mathbb R)$ in~\eqref{Paper02_definition_Lipschitz_function} for any~$t \in [0,T]$, the map~$\mathcal{S}^N \ni \boldsymbol{\xi}\mapsto g^{N,T,x}_{\mathcal{I}}(t, \boldsymbol{\xi})$ is in $\mathcal{C}_*(\mathcal{S}^{N};\mathbb{R})$, as desired.
     
     We now establish~\eqref{Paper02_definition_action_generator_greens_function}. Recall from before~\eqref{Paper02_definition_Lipschitz_function} that for~$k \in \mathbb{N}_{0}$ and~$y \in L_{N}^{-1}\mathbb{Z}$, we let~$\boldsymbol{e}^{(y)}_{k}$ denote the configuration in~$\mathcal{S}^{N}$ consisting of exactly one particle carrying~$k$ mutations at deme~$y$. By the definition of $g^{N,T,x}_{\mathcal{I}}$ in~\eqref{Paper02_green_function_representation_writing_as_lipschitz_map}, we have the following identities, for any $y \in L_{N}^{-1}\mathbb{Z}$, $k \in \mathcal{I}$, $\boldsymbol{\xi} = (\xi_j(z))_{j \in \mathbb N_0, z \in L_N^{-1}\mathbb Z} \in \mathcal{S}^{N}$ and $t \in [0,T]$:
\begin{equation} \label{Paper02_pre_application_generator_green_function}
\begin{split}
    & \mathds 1_{\{\xi_k(y) > 0\}}\Big(g^{N,T,x}_{\mathcal{I}}\Big(t, \boldsymbol{\xi} + \boldsymbol{e}^{(y+L_{N}^{-1})}_{k} - \boldsymbol{e}^{(y)}_{k}\Big) - g^{N,T,x}_{\mathcal{I}}\left(t, \boldsymbol{\xi}\right)\Big) \\ & \quad \quad  \quad \quad \quad \quad \quad \quad \quad = \frac{1}{NL_{N}}\mathds 1_{\{\xi_k(y) > 0\}}\Big( p^{N}(T-t,y + L_{N}^{-1} - x) - p^{N}(T-t,y - x) \Big),\\[1ex]
    & \mathds 1_{\{\xi_k(y) > 0\}}\Big(g^{N,T,x}_{\mathcal{I}}\Big(t, \boldsymbol{\xi} + \boldsymbol{e}^{(y-L_{N}^{-1})}_{k} - \boldsymbol{e}^{(y)}_{k}\Big) - g^{N,T,x}_{\mathcal{I}}\left(t, \boldsymbol{\xi}\right)\Big) \\ & \quad \quad  \quad \quad \quad \quad \quad \quad \quad = \frac{1}{NL_{N}}\mathds 1_{\{\xi_k(y) > 0\}}\Big( p^{N}(T-t,y - L_{N}^{-1} - x) - p^{N}(T-t,y - x) \Big), \\[1ex]
    & g^{N,T,x}_{\mathcal{I}}\Big(t, \boldsymbol{\xi} + \boldsymbol{e}^{(y)}_{k}\Big) - g^{N,T,x}_{\mathcal{I}}\left(t, \boldsymbol{\xi}\right)  = \frac{1}{NL_{N}}p^{N}(T-t,y - x), \\[1ex]
      & \mathds 1_{\{\xi_k(y) > 0\}}\Big(g^{N,T,x}_{\mathcal{I}}\left(t, \boldsymbol{\xi} - \boldsymbol{e}^{(y)}_{k}\right) - g^{N,T,x}_{\mathcal{I}}\left(t, \boldsymbol{\xi}\right)\Big) = - \frac{1}{NL_{N}}\mathds 1_{\{\xi_k(y) > 0\}}p^{N}(T-t,y - x).
\end{split}
\raisetag{-4cm}
\end{equation}
Combining the definition of the infinitesimal generator~$\mathcal{L}^{N}$ given in~\eqref{Paper02_generator_foutel_etheridge_model} and in~\eqref{Paper02_infinitesimal_generator}, with~\eqref{Paper02_pre_application_generator_green_function}, we conclude that for all~$t \in [0,T]$ and all~$\boldsymbol{\xi} = (\xi_k(y))_{k \in \mathbb N_0,y \in L_N^{-1}\mathbb Z} \in \mathcal{S}^{N}$,
\begin{equation} \label{Paper02_ugly_formulation_generator_green_function}
\begin{aligned}
    & \Big(\mathcal{L}^{N}g^{N,T,x}_{\mathcal{I}}(t, \cdot)\Big)(\boldsymbol{\xi}) \\ & \, = \frac{m_{N}}{2}\sum_{y \in L_{N}^{-1}\mathbb Z} \, \sum_{k \in \mathcal{I}} \, \frac{\xi_{k}(y)}{NL_{N}} \Big(p^{N}(T-t,y+L_{N}^{-1} - x) + p^{N}(T-t,y-L_{N}^{-1} - x) \\
    & \qquad \qquad \qquad \qquad \qquad \qquad - 2p^{N}(T-t,y - x)\Big) \\ & \quad + \sum_{y \in L_{N}^{-1}\mathbb Z} \, \sum_{k \in \mathcal{I}} \, \frac{1}{L_{N}} F_{k}\left(\frac{\xi(y)}{N}\right) p^{N}(T-t,y-x),
\end{aligned}
\end{equation}
where~$F = (F_{k})_{k \in \mathbb{N}_{0}}$ is the reaction term defined in~\eqref{Paper02_reaction_term_PDE}. Since~\eqref{Paper02_definition_state_space_formal} implies that, for all~$\boldsymbol{\xi} = (\xi(y))_{y \in L_N^{-1}\mathbb Z} \in \mathcal{S}^{N}$, $\| \xi(y) \|_{\ell_{1}}$ grows at most polynomially in $\vert y \vert$ as~$\vert y - x \vert \rightarrow \infty$, it follows from~\eqref{Paper02_trivial_large_deviation_estimatre_CTRW} that the series on the right-hand side of~\eqref{Paper02_ugly_formulation_generator_green_function} are absolutely convergent  uniformly in $t \in [0,T]$. This ensures the desired continuity and justifies rearranging the terms of the series to obtain~\eqref{Paper02_definition_action_generator_greens_function}. This concludes the proof of assertion~(ii).

\medskip

\myemph{Proof of assertion~(iii)}
We must first verify that the map~$[0,T] \times \mathcal{S}^{N} \ni (t,\boldsymbol{\xi})  \mapsto \mathcal{L}^{N}\Big((g^{N,T,x}_{\mathcal{I}})^{2}(t, \cdot)\Big)(\boldsymbol{\xi})$ is well defined and continuous. By the first identity in~\eqref{Paper02_pre_application_generator_green_function}, and using the definition of $\nabla_{L_N} p^N$ in~\eqref{Paper02_discrete_grad}, note that, for any $y \in L_{N}^{-1}\mathbb{Z}$, $k \in \mathcal{I}$, $\boldsymbol{\xi} = (\xi_j(z))_{j \in \mathbb N_0, z \in L_N^{-1}\mathbb Z} \in \mathcal{S}^{N}$ and~$t \in [0,T]$ with $\xi_k(y) > 0$,
\begin{equation} \label{Paper02_first_identity_migration_generator_square_greens_function}
\begin{aligned}
    & (g^{N,T,x}_{\mathcal{I}})^{2}\Big(t, \boldsymbol{\xi} + \boldsymbol{e}^{(y+L_{N}^{-1})}_{k} - \boldsymbol{e}^{(y)}_{k}\Big) - (g^{N,T,x}_{\mathcal{I}})^{2}\left(t, \boldsymbol{\xi}\right) \\ & \quad = \frac{1}{NL_{N}^2}\Big(g^{N,T,x}_{\mathcal{I}}\Big(t, \boldsymbol{\xi} + \boldsymbol{e}^{(y+L_{N}^{-1})}_{k} - \boldsymbol{e}^{(y)}_{k}\Big) + g^{N,T,x}_{\mathcal{I}}\left(t, \boldsymbol{\xi}\right)\Big) \nabla_{L_N}p^N(T-t,y-x) \\ 
    & \quad = \frac{1}{N L_{N}^2} \left(2g^{N,T,x}_{\mathcal{I}}\left(t, \boldsymbol{\xi}\right) + \frac{1}{NL_{N}^2} \nabla_{L_N} p^N(T-t,y-x)\right) \nabla_{L_N} p^N(T-t,y-x)\\ & \quad = \frac{2g^{N,T,x}_{\mathcal{I}}(t,\boldsymbol{\xi})}{NL_{N}^2}\nabla_{L_N}p^N(T-t,y-x) + \frac{1}{N^{2}L_{N}^{4}} \left(\nabla_{L_N}p^N(T-t,y-x)\right)^2.
\end{aligned}
\end{equation}
By a similar argument, the following identities hold  for any $y \in L_{N}^{-1}\mathbb{Z}$, $k \in \mathcal{I}$, $\boldsymbol{\xi} = (\xi_j(z))_{j \in \mathbb N_0, z \in L_N^{-1}\mathbb Z} \in \mathcal{S}^{N}$ and $t \in [0,T]$:
\begin{equation} \label{Paper02_pre_application_generator_green_function_square}
\begin{aligned}
   & \mathds 1_{\{\xi_k(y) > 0\}} \Big((g^{N,T,x}_{\mathcal{I}})^{2}\Big(t, \boldsymbol{\xi} + \boldsymbol{e}^{(y-L_{N}^{-1})}_{k} - \boldsymbol{e}^{(y)}_{k}\Big) - (g^{N,T,x}_{\mathcal{I}})^{2}\left(t, \boldsymbol{\xi}\right)\Big) \\
   & \quad = \mathds 1_{\{\xi_k(y) > 0\}}\bigg(-\frac{2g^{N,T,x}_{\mathcal{I}}(t,\boldsymbol{\xi})}{N L_{N}^2} \nabla_{L_N}p^N(T-t,y-L_N^{-1} - x) \\
   & \qquad \qquad \qquad \qquad + \frac{1}{N^{2}L_{N}^{4}} \left(\nabla_{L_N} p^N(T-t, y-L_N^{-1} - x)\right)^{2}\bigg), \\[1ex]
    & (g^{N,T,x}_{\mathcal{I}})^{2}\Big(t, \boldsymbol{\xi} + \boldsymbol{e}^{(y)}_{k}\Big) - (g^{N,T,x}_{\mathcal{I}})^{2}\left(t, \boldsymbol{\xi}\right)\\ & \quad = \frac{2g^{N,T,x}_{\mathcal{I}}\left(t, \boldsymbol{\xi}\right)}{NL_{N}}p^{N}(T-t,y - x) + \frac{1}{N^{2}L_{N}^{2}}p^{N}(T-t,y - x)^{2}, \\[1ex]
      & \mathds 1_{\{\xi_k(y) > 0\}}\Big((g^{N,T,x}_{\mathcal{I}})^{2}\Big(t, \boldsymbol{\xi} - \boldsymbol{e}^{(y)}_{k}\Big) - (g^{N,T,x}_{\mathcal{I}})^{2}\left(t, \boldsymbol{\xi}\right)\Big) \\ & \quad = \mathds 1_{\{\xi_k(y) > 0\}} \left(- \frac{2g^{N,T,x}_{\mathcal{I}}\left(t, \boldsymbol{\xi}\right)}{NL_{N}}p^{N}(T-t,y - x) + \frac{1}{N^{2}L_{N}^{2}}p^{N}(T-t,y - x)^{2}\right).
\end{aligned}
\end{equation}
Similarly to the proof of assertion~(ii), since~\eqref{Paper02_definition_state_space_formal} implies that, for all~$\boldsymbol{\xi} = (\xi(y))_{y \in L_N^{-1}\mathbb Z} \in \mathcal{S}^{N}$, $\| \xi(y) \|_{\ell_{1}}$ grows at most polynomially in $\vert y \vert$ as~$\vert y - x \vert \rightarrow \infty$, it follows from~\eqref{Paper02_infinitesimal_generator},~\eqref{Paper02_trivial_large_deviation_estimatre_CTRW}, the fact that by assertion~(i), $g^{N,T,x}_{\mathcal I} \in \mathcal C([0,T] \times \mathcal S^N; \mathbb R)$, and from identities~\eqref{Paper02_first_identity_migration_generator_square_greens_function} and~\eqref{Paper02_pre_application_generator_green_function_square}, that the series defining the action of~$\mathcal{L}^{N}$ on~$(g^{N,T,x}_{\mathcal{I}})^{2}$ is absolutely convergent uniformly over $t \in [0,T]$. This ensures that the map~$[0,T] \times \mathcal S^N \ni (t,\boldsymbol{\xi}) \mapsto \mathcal{L}^{N}\Big((g^{N,T,x}_{\mathcal{I}})^{2}(t, \cdot)\Big)(\boldsymbol{\xi})$ is well defined and continuous. Moreover, rearranging the terms of the series, combining~\eqref{Paper02_definition_action_generator_greens_function} with~\eqref{Paper02_first_identity_migration_generator_square_greens_function} and~\eqref{Paper02_pre_application_generator_green_function_square}, and recalling~\eqref{Paper02_reaction_predictable_bracket_process}, we obtain~\eqref{Paper02_definition_action_generator_on_square_greens_function}, as desired.

\medskip

\myemph{Proof of assertion~(iv)}
We will bound the supremum, over~$t_{1},t_{2} \in [0,T]$, of the expectation of each term on the left-hand side of~\eqref{Paper02_general_estimates_moments_green_function_and_derivatives_time_generator} separately. Starting with the first term, by the definition of~$g^{N,T,x}_{\mathcal{I}}$ in~\eqref{Paper02_green_function_representation_writing_as_lipschitz_map} and by Lemma~\ref{Paper02_simple_estimate_greens_function_representation_moments}, we conclude that for any~$r \geq 1$,
\begin{equation} \label{Paper02_trivial_bound_moments_green_function}
    \sup_{y \in L_{N}^{-1}\mathbb{Z}} \; \sup_{t_{1},t_{2} \in [0,T]} \mathbb{E}_{\boldsymbol \eta^N}\Big[(g^{N,T,y}_{\mathcal{I}})^{r}(t_{1}, \eta^{N}(t_{2}))\Big] < \infty.
\end{equation}
In particular, by~\eqref{Paper02_formula_derivative_time_greens_function}, estimate~\eqref{Paper02_trivial_bound_moments_green_function} implies that for any~$r \geq 1$,
\begin{equation} \label{Paper02_trivial_bound_moments_derivative_time_greens}
     \sup_{t_{1},t_{2} \in [0,T]} \mathbb{E}_{\boldsymbol \eta^N}\bigg[\bigg(\frac{\partial}{\partial t} g^{N,T,x}_{\mathcal{I}}(\cdot, \eta^{N}(t_{2}))\bigg)^{r}(t_{1})\bigg] < \infty.
\end{equation}
To bound the third term on the left-hand side of~\eqref{Paper02_general_estimates_moments_green_function_and_derivatives_time_generator}, we apply the Cauchy-Schwarz inequality, and then estimates~\eqref{Paper02_trivial_bound_moments_green_function} and~\eqref{Paper02_trivial_bound_moments_derivative_time_greens}, obtaining
\begin{equation} \label{Paper02_estimate_third_term_moments_generator_derivative_green}
\begin{split}
     & \sup_{t_{1},t_{2} \in [0,T]} \mathbb{E}_{\boldsymbol \eta^N}\bigg[(g^{N,T,x}_{\mathcal{I}})^{2}(t_{1}, \eta^{N}(t_{2}))\bigg(\frac{\partial}{\partial t} g^{N,T,x}_{\mathcal{I}}(\cdot, \eta^{N}(t_{2}))\bigg)^{2}(t_{1})\bigg] \\ & \quad \leq \sup_{t_{1},t_{2} \in [0,T]} \mathbb{E}_{\boldsymbol \eta^N}\Big[(g^{N,T,x}_{\mathcal{I}})^{4}(t_{1}, \eta^{N}(t_{2}))\Big]^{1/2}\mathbb{E}_{\boldsymbol \eta^N}\bigg[\bigg(\frac{\partial}{\partial t} g^{N,T,x}_{\mathcal{I}}(\cdot, \eta^{N}(t_{2}))\bigg)^{4}(t_{1})\bigg]^{1/2} \\ & \quad < \infty.
\end{split}
\raisetag{-2cm}
\end{equation}
For the fourth term on the left-hand side of~\eqref{Paper02_general_estimates_moments_green_function_and_derivatives_time_generator}, by the definition of~$F=(F_{k})_{k \in \mathbb{N}_{0}}$ in~\eqref{Paper02_reaction_term_PDE} and $F^+ = (F_{k}^+)_{k \in \mathbb{N}_{0}}$ in~\eqref{Paper02_reaction_predictable_bracket_process}, and since~$q_{+}$ and~$q_{-}$ are non-negative polynomials by Assumption~\ref{Paper02_assumption_polynomials}, and then in the last line since $s_k \leq 1 \; \forall \, k \in \mathbb N_0$ by Assumption~\ref{Paper02_assumption_fitness_sequence}(i) and~(iii), we have that for any~$\boldsymbol{\xi} = (\xi_k(y))_{k \in \mathbb N_0, y \in L_N^{-1}\mathbb Z} \in \mathcal{S}^{N}$ and any~$y \in L_{N}^{-1}\mathbb{Z}$,
\begin{equation} \label{Paper02_trivial_bound_sum_reaction_terms}
\begin{aligned}
     & \bigg\vert \sum_{k \in \mathcal{I}} F_{k}\left(\frac{\xi(y)}{N}\right) \bigg\vert
     \\ & \quad \leq \sum_{k \in \mathcal{I}} F^{+}_{k}\left(\frac{\xi(y)}{N}\right)
     \\ & \quad =  \sum_{k \in \mathcal{I}} \bigg( q_{+}\left(\frac{\| \xi(y) \|_{\ell_{1}}}{N}\right)\left(s_{k}(1-\mu)\frac{\xi_{k}(y)}{N} + \mathds{1}_{\{k \geq 1\}}s_{k-1}\mu \frac{\xi_{k-1}(y)}{N}\right) \\
     & \hspace{8cm}+ q_{-}\left(\frac{\| \xi(y) \|_{\ell_{1}}}{N}\right)\frac{\xi_{k}(y)}{N}\bigg)
     \\ & \quad \leq \frac{\| \xi(y) \|_{\ell_{1}}}{N} \left(q_{+}\left(\frac{\| \xi(y) \|_{\ell_{1}}}{N}\right) + q_{-}\left(\frac{\| \xi(y) \|_{\ell_{1}}}{N}\right)\right).
\end{aligned}
\end{equation}
Using again the fact that~$q_{+}$ and~$q_{-}$ are polynomials, by Lemma~\ref{Paper02_simple_estimate_greens_function_representation_moments} we conclude that for any~$r \geq 1$,
\begin{equation} \label{Paper02_bounding_moments_action_generator_green_function_i}
     \sup_{t_{1},t_{2} \in [0,T]} \mathbb{E}_{\boldsymbol \eta^N}\bigg[\bigg\langle \bigg\vert \sum_{k \in \mathcal{I}} F_{k}\left(u^{N}(t_{2}, \cdot)\right) \bigg\vert, \phi^{N,T,x}(t_{1}, \cdot) \bigg\rangle^{r}_{N}\bigg] < \infty.
\end{equation}
Combining identity~\eqref{Paper02_definition_action_generator_greens_function} with estimates~\eqref{Paper02_trivial_bound_moments_green_function} and~\eqref{Paper02_bounding_moments_action_generator_green_function_i}, we conclude that for any~$r \geq 1$,
\begin{equation} \label{Paper02_bounding_moments_action_generator_green_function_ii}
     \sup_{t_{1},t_{2} \in [0,T]} \mathbb{E}_{\boldsymbol \eta^N}\Big[\Big\vert \mathcal{L}^{N}\Big(g^{N,T,x}_{\mathcal{I}}(t_{1}, \cdot)\Big)(\eta^{N}(t_{2})) \Big\vert^{r}\Big] < \infty.
\end{equation}
We now proceed to bound the fifth term on the left-hand side of~\eqref{Paper02_general_estimates_moments_green_function_and_derivatives_time_generator}. As in the derivation of~\eqref{Paper02_estimate_third_term_moments_generator_derivative_green}, by combining the Cauchy-Schwarz inequality with estimates~\eqref{Paper02_trivial_bound_moments_green_function} and~\eqref{Paper02_bounding_moments_action_generator_green_function_ii}, we conclude that
\begin{equation} \label{Paper02_bounding_moments_action_generator_green_function_iii}
     \sup_{t_{1},t_{2} \in [0,T]} \mathbb{E}_{\boldsymbol \eta^N}\left[\Big(g^{N,T,x}_{\mathcal{I}}(t_{1}, \eta^{N}(t_{2}))\Big)^{2}\Big(\mathcal{L}^{N}\Big(g^{N,T,x}_{\mathcal{I}}(t_{1}, \cdot)\Big)(\eta^{N}(t_{2}))\Big)^{2} \right] < \infty.
\end{equation}
Moreover, by~\eqref{Paper02_definition_p_N_test_function_green}, we have that for any~$t \in [0,T]$ and any~$y \in L_{N}^{-1}\mathbb{Z}$,
\begin{equation} \label{Paper02_very_elementary_bound_p_N}
    p^{N}(T-t,y-x)^{2} \leq L_{N} p^{N}(T-t,y-x).
\end{equation}
Therefore, by~\eqref{Paper02_very_elementary_bound_p_N},~\eqref{Paper02_trivial_bound_sum_reaction_terms} and the fact that $0 \leq \deg q_+ < \deg q_-$ by Assumption~\ref{Paper02_assumption_polynomials}, we conclude that for all~$t \in [0,T]$ and~$\boldsymbol{\xi} = (\xi_k(y))_{k \in \mathbb N_0, y \in L_N^{-1}\mathbb Z} \in \mathcal{S}^{N}$,
\begin{equation} \label{Paper02_bounding_moments_action_generator_green_function_iv}
\begin{aligned}
     & \frac{1}{N L_{N}^{2}} \sum_{y \in L_{N}^{-1}\mathbb{Z}} \; \sum_{k \in \mathcal{I}} F^{+}_{k}\left(\frac{\xi(y)}{N}\right) p^{N}(T-t,y-x)^{2}
     \\ & \; + \frac{m_{N}}{2NL^{4}_{N}} \sum_{y \in L_{N}^{-1}\mathbb{Z}} \; \sum_{k \in \mathcal{I}} \frac{\xi_{k}(y)}{N} \Big( \left(\nabla_{L_N} p^{N}(T-t,y - L_N^{-1} - x)\right)^2 \\[-4mm]
     & \hspace{6cm} + \left(\nabla_{L_N} p^{N}(T-t,y - x)\right)^2\Big)
     \\ & \; \lesssim_{N,q_{+},q_{-}} \left\langle \left(\frac{\| \xi(\cdot) \|_{\ell_{1}}}{N}\right)^{1 + \deg q_{-}}, \, \phi^{N,T,x}(t, \cdot) \right\rangle_{N} \\
     & \hspace{2cm} + \sum_{x' \in \{x,x-L_N^{-1}, x + L_N^{-1}\}} \left\langle \frac{\| \xi(\cdot) \|_{\ell_{1}}}{N}, \, \phi^{N,T,x'}(t, \cdot) \right\rangle_{N}.
\end{aligned}
\end{equation}
Combining~\eqref{Paper02_definition_action_generator_on_square_greens_function},~\eqref{Paper02_bounding_moments_action_generator_green_function_iii}, Lemma~\ref{Paper02_simple_estimate_greens_function_representation_moments} and estimate~\eqref{Paper02_bounding_moments_action_generator_green_function_iv}, we conclude that
\begin{equation} \label{Paper02_bounding_moments_action_generator_green_function_v}
     \sup_{t_{1},t_{2} \in [0,T]} \mathbb{E}_{\boldsymbol \eta^N}\left[\Big(\mathcal{L}^{N}(g^{N,T,x}_{\mathcal{I}})^{2}(t_{1}, \cdot)\Big)^{2}(\eta^{N}(t_{2})) \right] < \infty.
\end{equation}
Hence, by~\eqref{Paper02_trivial_bound_moments_green_function},~\eqref{Paper02_trivial_bound_moments_derivative_time_greens},~\eqref{Paper02_estimate_third_term_moments_generator_derivative_green},~\eqref{Paper02_bounding_moments_action_generator_green_function_ii} and~\eqref{Paper02_bounding_moments_action_generator_green_function_v}, estimate~\eqref{Paper02_general_estimates_moments_green_function_and_derivatives_time_generator} holds, as desired.

\medskip

\myemph{Proof of assertion~(v)}
It follows from assertion~(iv) that $(M(t))_{t \geq 0}$ is an integrable process. It remains to verify that~$M$ satisfies the martingale property with respect to the filtration~$\{\mathcal{F}^{\eta^{N}}_{t+}\}_{t \geq 0}$. For each~$n \in \mathbb{N}$, define~$h^{N,T,x,n}_{\mathcal{I}}: [0,T] \times \mathcal{S}^{N} \rightarrow \mathbb{R}$ by
\begin{equation} \label{Paper02_truncation_square_green_function}
    h^{N,T,x,n}_{\mathcal{I}}(t,\boldsymbol{\xi}) \defeq \Big(g^{N,T,x}_{\mathcal{I}}\Big)^{2}(t,\boldsymbol{\xi}) \wedge n^{2} \quad \forall (t, \boldsymbol{\xi}) \in [0,T] \times \mathcal{S}^{N}.
\end{equation}
Recall that~$t_{1} \in [0,T]$ is fixed. By assertion~(ii) of this lemma, $g^{N,T,x}_{\mathcal{I}}(t_{1},\cdot) \in \mathcal{C}_*(\mathcal{S}^{N}; \mathbb{R})$. By~\eqref{Paper02_truncation_square_green_function}, for all~$\boldsymbol{\xi}, \boldsymbol{\zeta} \in \mathcal{S}^{N}$ and every~$n \in \mathbb{N}$, we have
\begin{equation*}
\begin{aligned}
    \Big\vert h^{N,T,x,n}_{\mathcal{I}}(t_{1}, \boldsymbol{\xi}) - h^{N,T,x,n}_{\mathcal{I}}(t_{1}, \boldsymbol{\zeta}) \Big\vert & \leq 2n \Big\vert g^{N,T,x}_{\mathcal{I}}(t_{1}, \boldsymbol{\xi}) - g^{N,T,x}_{\mathcal{I}}(t_{1}, \boldsymbol{\zeta}) \Big\vert.
\end{aligned}
\end{equation*}
Hence, in particular, $h^{N,T,x,n}_{\mathcal{I}}(t_{1}, \cdot) \in \mathcal{C}_*(\mathcal{S}^{N}; \mathbb{R})$ for every~$n \in \mathbb{N}$. Therefore, by Theorem~\ref{Paper02_thm_existence_uniqueness_IPS_spatial_muller}, for every~$n \in \mathbb{N}$, the process~$(M^{(n)}(t))_{t \geq 0}$ given by, for all~$t \geq 0$,
\begin{equation} \label{Paper02_associated_truncated_martingale}
\begin{split}
     M^{(n)}(t) & 
 \defeq h^{N,T,x,n}_{\mathcal{I}}(t_{1}, \eta^{N}(t)) - h^{N,T,x,n}_{\mathcal{I}}(t_{1}, \eta^{N}(0)) \\ & \qquad\qquad - \int_{0}^{t} \mathcal{L}^{N}\Big(h^{N,T,x,n}_{\mathcal{I}}(t_{1}, \cdot)\Big)(\eta^{N}(\tau-)) \, d\tau
\end{split}
\raisetag{-1.1cm}
\end{equation}
is a càdlàg martingale with respect to the filtration~$\{\mathcal{F}^{\eta^{N}}_{t+}\}_{t \geq 0}$.

We now claim that for all~$t \geq 0$,
\begin{equation} \label{Paper02_constructing_sequence_UI_martingales_square_green}
    \lim_{n \rightarrow \infty} M^{(n)}(t) = M(t) \quad \textrm{almost surely}.
\end{equation}
Indeed, by the definition of~$h^{N,T,x,n}_{\mathcal{I}}$ in~\eqref{Paper02_truncation_square_green_function}, it is immediate that for any $\boldsymbol{\xi} \in \mathcal{S}^{N}$,
\begin{equation} \label{Paper02_limit_truncation_square_greens_function}
    \lim_{n \rightarrow \infty} \, h^{N,T,x,n}_{\mathcal{I}}(t_{1}, \boldsymbol{\xi}) = \Big(g^{N,T,x}_{\mathcal{I}}\Big)^{2}(t_{1}, \boldsymbol{\xi}).
\end{equation}
Hence,~\eqref{Paper02_constructing_sequence_UI_martingales_square_green} will be proved after verifying that for all~$t \geq 0$, the following limit holds almost surely:
\begin{equation} \label{Paper02_desired_limit_almost_sure_square_green_function}
\begin{aligned}
    & \lim_{n \rightarrow \infty} \, \int_{0}^{t} \mathcal{L}^{N}\Big(h^{N,T,x,n}_{\mathcal{I}}(t_{1}, \cdot)\Big)(\eta^{N}(\tau-)) \, d\tau \\
    & \quad = \int_{0}^{t}  \mathcal{L}^{N}\Big((g^{N,T,x}_{\mathcal{I}})^{2}(t_{1}, \cdot)\Big)(\eta^{N}(\tau-)) \, d\tau.
\end{aligned}
\end{equation}
Since, by Theorem~\ref{Paper02_thm_existence_uniqueness_IPS_spatial_muller},~$(\eta^{N}(\tau))_{\tau \geq 0}$ is an~$\mathcal{S}^{N}$-valued càdlàg process, and since $\mathcal S^N$ is a complete and separable metric space by Proposition~\ref{Paper02_topological_properties_state_space}, by standard results in stochastic analysis (see e.g.~\cite[Remark~3.6.4]{ethier2009markov}), for any~$t \geq 0$ and almost every realisation~$(\eta^{N}_{\omega}(\tau))_{\tau \geq 0}$ of $(\eta^N(\tau))_{\tau \geq 0}$, there exists a compact set~$\mathcal{K}_{t,\omega} \subset \mathcal{S}^{N}$ such that $\eta^{N}_{\omega}(\tau) \in \mathcal{K}_{t,\omega} \; \forall \, \tau \in [0,t]$.
By the characterisation of compact subsets of~$\mathcal{S}^{N}$ given by Proposition~\ref{Paper02_topological_properties_state_space}, we conclude that for almost every realisation~$(\eta^{N}_{\omega}(\tau))_{\tau \geq 0}$ of $(\eta^N(\tau))_{\tau \geq 0}$,
$\sup_{\tau \leq t} \; \vert \vert \vert \eta^{N}_{\omega}(\tau) \vert \vert \vert_{\mathcal{S}^{N}} < \infty$.
In particular, by~\eqref{Paper02_definition_state_space_formal}, for almost every realisation~$(\eta^{N}_{\omega}(\tau))_{\tau \geq 0}$ of $(\eta^N(\tau))_{\tau \geq 0}$, there exists $A_{\omega,t} > 0$ such that
\begin{equation} \label{Paper02_corollary_cadlag_norm_property}
    \| \eta^N_\omega (\tau,y)  \|_{\ell_1} \leq A_{\omega,t}(1 + \vert y \vert^2) \quad \forall (\tau,y) \in [0,t] \times L_N^{-1}\mathbb Z. 
\end{equation}
We also observe that the following elementary inequality holds for any~$a,b \in \mathbb{R}$ and~$n \in \mathbb{N}$:
\begin{equation} \label{Paper02_elementary_inequality_difference_truncated_squares}
\begin{aligned}
    \vert (a^{2} \wedge n^{2}) - (b^{2} \wedge n^{2}) \vert \leq \vert a^{2} -b^{2} \vert.
\end{aligned}
\end{equation}
Hence, combining the definition of~$h^{N,T,x,n}_{\mathcal{I}}$ in~\eqref{Paper02_truncation_square_green_function} with the identities given in~\eqref{Paper02_first_identity_migration_generator_square_greens_function} and~\eqref{Paper02_pre_application_generator_green_function_square}, we conclude that for any $y \in L_{N}^{-1}\mathbb{Z}$, $k \in \mathcal{I}$, $n \in \mathbb{N}$, $t \in [0,T]$, and $\boldsymbol{\xi} = (\xi_j(z))_{j \in \mathbb N_0, z \in L_N^{-1}\mathbb Z} \in \mathcal{S}^{N}$, the following estimates hold:
\begin{equation} \label{Paper02_simple_estimates_generator_on_truncation_square_green_function}
\begin{aligned}
    & \mathds 1_{\{\xi_k(y) > 0\}}\Big\vert h^{N,T,x,n}_{\mathcal{I}}\Big(t, \boldsymbol{\xi} + \boldsymbol{e}^{(y+L_{N}^{-1})}_{k} - \boldsymbol{e}^{(y)}_{k}\Big) - h^{N,T,x,n}_{\mathcal{I}}\left(t, \boldsymbol{\xi}\right) \Big\vert \\ & \quad \leq \frac{2g^{N,T,x}_{\mathcal{I}}(t,\boldsymbol{\xi})}{NL_{N}}\Big(p^{N}(T-t,y + L_{N}^{-1} - x) + p^{N}(T-t,y - x) \Big) \\ & \qquad \qquad + \frac{1}{N^{2}L_{N}^{4}} \left(\nabla_{L_N}p^{N}(T-t,y - x) \right)^{2}, \\
 & \mathds 1_{\{\xi_k(y) > 0\}} \Big\vert h^{N,T,x,n}_{\mathcal{I}}\Big(t, \boldsymbol{\xi} + \boldsymbol{e}^{(y-L_{N}^{-1})}_{k} - \boldsymbol{e}^{(y)}_{k}\Big) - h^{N,T,x,n}_{\mathcal{I}}\left(t, \boldsymbol{\xi}\right) \Big\vert \\ & \quad \leq \frac{2g^{N,T,x}_{\mathcal{I}}(t,\boldsymbol{\xi})}{N L_{N}} \Big( p^{N}(T-t,y - L_{N}^{-1} - x) + p^{N}(T-t,y - x) \Big)\\ & \qquad  \qquad + \frac{1}{N^{2}L_{N}^{4}} \left(\nabla_{L_N}p^{N}(T-t,y - L_{N}^{-1} - x)\right)^{2}, \\[1ex]
\end{aligned}
\end{equation}
as well as
\begin{equation} \label{Paper02_simple_estimates_generator_on_truncation_square_green_function_2}
\begin{aligned}
& \Big\vert h^{N,T,x,n}_{\mathcal{I}}\Big(t, \boldsymbol{\xi} + \boldsymbol{e}^{(y)}_{k}\Big) - h^{N,T,x,n}_{\mathcal{I}}\left(t, \boldsymbol{\xi}\right) \Big\vert
    \\ & \quad \leq \frac{2g^{N,T,x}_{\mathcal{I}}\left(t, \boldsymbol{\xi}\right)}{NL_{N}}p^{N}(T-t,y - x) + \frac{1}{N^{2}L_{N}^{2}}p^{N}(T-t,y - x)^{2}, \\[1ex]
      & \mathds 1_{\{\xi_k(y) > 0\}} \Big\vert h^{N,T,x,n}_{\mathcal{I}}\Big(t, \boldsymbol{\xi} - \boldsymbol{e}^{(y)}_{k}\Big) - h^{N,T,x,n}_{\mathcal{I}}\left(t, \boldsymbol{\xi}\right) \Big\vert \\
      & \quad \leq \frac{2g^{N,T,x}_{\mathcal{I}}\left(t, \boldsymbol{\xi}\right)}{NL_{N}}p^{N}(T-t,y - x) + \frac{1}{N^{2}L_{N}^{2}}p^{N}(T-t,y - x)^{2}.
\end{aligned}
\end{equation}
Using the definition of the action of the generator~$\mathcal{L}^{N}$ given in~\eqref{Paper02_generator_foutel_etheridge_model} and~\eqref{Paper02_infinitesimal_generator}, together with estimates~\eqref{Paper02_trivial_large_deviation_estimatre_CTRW},~\eqref{Paper02_corollary_cadlag_norm_property}, \eqref{Paper02_simple_estimates_generator_on_truncation_square_green_function} and~\eqref{Paper02_simple_estimates_generator_on_truncation_square_green_function_2}, the limit in~\eqref{Paper02_limit_truncation_square_greens_function} and the dominated convergence theorem, we conclude that for any $t \geq 0$, the limit in~\eqref{Paper02_desired_limit_almost_sure_square_green_function} holds almost surely. This completes the proof of the claim~\eqref{Paper02_constructing_sequence_UI_martingales_square_green}. 

Moreover, the definition of~$h^{N,T,x,n}_{\mathcal{I}}$ in~\eqref{Paper02_truncation_square_green_function} and of~$M^{(n)}$ in~\eqref{Paper02_associated_truncated_martingale} imply that for every~$n \in \mathbb{N}$ and all~$t \geq 0$,
\begin{equation} \label{Paper02_expectation_square_truncation_martingale}
\begin{aligned}
    \mathbb{E}_{\boldsymbol \eta^N}\Big[(M^{(n)}(t))^{2}\Big] & \lesssim \mathbb{E}_{\boldsymbol \eta^N}\Big[(g^{N,T,x}_{\mathcal{I}})^{4}(t_{1}, \eta^{N}(t))\Big] + (g^{N,T,x}_{\mathcal{I}})^{4}(t_{1}, \boldsymbol{\eta}^{N}) \\
    & \qquad + \mathbb{E}_{\boldsymbol \eta^N}\bigg[\bigg(\int_{0}^{t} \Big(\mathcal{L}^{N}h^{N,T,x,n}_{\mathcal{I}}(t_{1}, \cdot)\Big)(\eta^{N}(\tau-)) \, d\tau\bigg)^{2}\bigg].
\end{aligned}
\end{equation}
Observe that by Jensen's inequality and Fubini's theorem, we have
\begin{equation} \label{Paper02_easy_trick_cauchy_schwarz}
\begin{aligned}
    & \mathbb{E}_{\boldsymbol \eta^N}\bigg[\bigg(\int_{0}^{t} \Big(\mathcal{L}^{N}h^{N,T,x,n}_{\mathcal{I}}(t_{1}, \cdot)\Big)(\eta^{N}(\tau-)) \, d\tau\bigg)^{2}\bigg] \\
    & \quad \leq t \int_{0}^{t} \mathbb{E}_{\boldsymbol \eta^N}\Big[\Big(\mathcal{L}^{N}h^{N,T,x,n}_{\mathcal{I}}(t_{1}, \cdot)\Big)^{2}(\eta^{N}(\tau-))\Big] \, d\tau.
\end{aligned}
\end{equation}
As in the proof of~\eqref{Paper02_bounding_moments_action_generator_green_function_iv}, to bound the right-hand side of~\eqref{Paper02_easy_trick_cauchy_schwarz}, we first observe that by the definition of $\mathcal L^N$ in~\eqref{Paper02_generator_foutel_etheridge_model} and~\eqref{Paper02_infinitesimal_generator}, estimates~\eqref{Paper02_very_elementary_bound_p_N}, \eqref{Paper02_simple_estimates_generator_on_truncation_square_green_function} and~\eqref{Paper02_simple_estimates_generator_on_truncation_square_green_function_2}, the definition of $F^+ = (F^+_k)_{k \in \mathbb N_0}$ in~\eqref{Paper02_reaction_predictable_bracket_process}, and the fact that $0 \leq \deg q_+ < \deg q_-$ by Assumption~\ref{Paper02_assumption_polynomials}, for all~$t_1 \in [0,T]$,~$\tau \in [0,T]$, and $n \in \mathbb N$, we have almost surely
\begin{equation} \label{Paper02_expectation_square_truncation_martingale_green_i}
\begin{aligned}
    & \left\vert \mathcal L^N \Big(h^{N,T,x,n}_\mathcal I(t_1,\cdot)\Big)(\eta^N(\tau-))\right\vert
    \\ & \quad \lesssim_{N,q_+,q_-} \left(g^{N,T,x}_{\mathcal I}(t_1,\eta^N(\tau-)) + 1\right) \\
    & \hspace{2.5cm} \cdot \sum_{x' \in \{x,x-L_N^{-1}, x + L_N^{-1}\}} \bigg\langle \frac{\| \eta^N(\tau-,\cdot) \|_{\ell_{1}}}{N}, \, \phi^{N,T,x'}(t_1, \cdot) \bigg\rangle_{N}
    \\ & \qquad \qquad \qquad  + \left(g^{N,T,x}_{\mathcal I}(t_1,\eta^N(\tau-)) + 1\right) \\
    & \hspace{3cm} \cdot \bigg\langle \left(\frac{\| \eta^N(\tau-,\cdot) \|_{\ell_{1}}}{N}\right)^{1 + \deg q_{-}}, \, \phi^{N,T,x}(t_1, \cdot) \bigg\rangle_{N} \\
    & \quad \lesssim_N \left(g^{N,T,x}_{\mathcal I}(t_1,\eta^N(\tau-)) + 1\right) \\
    & \hspace{1.5cm} \cdot \sum_{x' \in \{x,x-L_N^{-1}, x + L_N^{-1}\}} \bigg\langle \frac{\| \eta^N(\tau-,\cdot) \|_{\ell_{1}}^{1 + \deg q_-}}{N}, \, \phi^{N,T,x'}(t_1, \cdot) \bigg\rangle_{N}.
\end{aligned}
\end{equation}
By squaring both sides of~\eqref{Paper02_expectation_square_truncation_martingale_green_i} and applying the elementary inequality
\[
\left(\sum_{i = 1}^J a_i\right)^2 \leq J \sum_{i=1}^J a_i^2 \quad \forall \; J \in \mathbb N, \, (a_i)_{i = 1}^J \subset \mathbb R, 
\]
and then by taking expectations, we conclude that for all~$t_1 \in [0,T]$,~$\tau \in [0,T]$ and $n \in \mathbb N$,
\begin{equation} \label{Paper02_expectation_square_truncation_martingale_green_ii}
\begin{split}
    & \mathbb E_{\boldsymbol \eta^N} \left[\Big(\mathcal{L}^{N}h^{N,T,x,n}_{\mathcal{I}}(t_{1}, \cdot)\Big)^{2}(\eta^{N}(\tau-))\right] \\
    & \quad \lesssim_{N,q_+,q_-} \sum_{z \in \{0,-L_N^{-1},L_N^{-1}\}} \bigg(\mathbb E_{\boldsymbol \eta^N} \bigg[\bigg\langle \frac{\| \eta^N(\tau-,\cdot) \|^{1 + \deg q_-}_{\ell_{1}}}{N}, \, \phi^{N,T,x + z}(t_1, \cdot) \bigg\rangle_{N}^2\bigg] \\
    & \qquad \qquad + \mathbb E_{\boldsymbol \eta^N} \bigg[\left(g^{N,T,x}_{\mathcal I}(t_1,\eta^N(\tau-))\right)^2\bigg\langle \frac{\| \eta^N(\tau-,\cdot) \|^{1 + \deg q_-}_{\ell_{1}}}{N}, \, \phi^{N,T,x + z}(t_1, \cdot) \bigg\rangle_{N}^2\bigg] \bigg) \\
    & \quad \leq \sum_{z \in \{0,-L_N^{-1},L_N^{-1}\}} \bigg(\mathbb E_{\boldsymbol \eta^N} \bigg[\bigg\langle \frac{\| \eta^N(\tau-,\cdot) \|^{1 + \deg q_-}_{\ell_{1}}}{N}, \, \phi^{N,T,x + z}(t_1, \cdot) \bigg\rangle_{N}^2\bigg] \\
    & \hspace{2cm} + \mathbb E_{\boldsymbol \eta^N} \left[\left(g^{N,T,x}_{\mathcal I}(t_1,\eta^N(\tau-))\right)^4\right]^{1/2} \\
    & \hspace{3cm} \cdot \mathbb E_{\boldsymbol \eta^N}\bigg[\bigg\langle \frac{\| \eta^N(\tau-,\cdot) \|^{1 + \deg q_-}_{\ell_{1}}}{N}, \, \phi^{N,T,x + z}(t_1, \cdot) \bigg\rangle_{N}^4\bigg]^{1/2}\bigg), 
\end{split}
\raisetag{-4.3cm}
\end{equation}
where for the second inequality we used the Cauchy-Schwarz inequality. Applying~\eqref{Paper02_easy_trick_cauchy_schwarz} and~\eqref{Paper02_expectation_square_truncation_martingale_green_ii} to~\eqref{Paper02_expectation_square_truncation_martingale}, together with estimate~\eqref{Paper02_trivial_bound_moments_green_function} and Lemma~\ref{Paper02_simple_estimate_greens_function_representation_moments}, we conclude that for all~$t \geq 0$,
\begin{equation} \label{Paper02_uniformly_integrable_martingale_truncation}
    \sup_{n \in \mathbb{N}} \; \mathbb{E}_{\boldsymbol \eta^N}\Big[(M^{(n)}(t))^{2}\Big] < \infty.
\end{equation}
Hence, for all~$t \geq 0$, the collection of random variables~$\{M^{(n)}(\tau): \, 0 \leq \tau \leq t \textrm{ and } n \in \mathbb{N}\}$ is uniformly integrable. By~\eqref{Paper02_constructing_sequence_UI_martingales_square_green}, we then conclude that~$M$ satisfies the martingale property with respect to the filtration~$\{\mathcal{F}^{\eta^{N}}_{t+}\}_{t \geq 0}$, and therefore assertion~(v) holds.

\medskip

\myemph{Proof of assertion~(vi)}
We start by observing that for all $\boldsymbol \xi, \boldsymbol \zeta \in \mathcal S^N$,
\begin{equation} \label{Paper02_preparing_generator_difference_martingale_i}
\begin{aligned}
    & (g^{N,T_{1},x_{1}}_{\mathcal{I}} - g^{N,T_{2},x_{2}}_{\mathcal{I}})^{2}(t, \boldsymbol \xi) - (g^{N,T_{1},x_{1}}_{\mathcal{I}} - g^{N,T_{2},x_{2}}_{\mathcal{I}})^{2}(t, \boldsymbol \zeta) \\
    & \, = \left((g^{N,T_1,x_1}_\mathcal I)^2(t,\boldsymbol \xi) - (g^{N,T_1,x_1}_\mathcal I)^2(t,\boldsymbol \zeta)\right) + \left((g^{N,T_2,x_2}_\mathcal I)^2(t,\boldsymbol \xi) - (g^{N,T_2,x_2}_\mathcal I)^2(t,\boldsymbol \zeta)\right) \\
    & \qquad - 2\left(g^{N,T_1,x_1}_{\mathcal I}(t,\boldsymbol \xi) - g^{N,T_1,x_1}_{\mathcal I}(t,\boldsymbol \zeta)\right)g^{N,T_2,x_2}_\mathcal I(t,\boldsymbol \zeta) \\
    & \qquad - 2g^{N,T_1,x_1}_{\mathcal I}(t,\boldsymbol \xi)\left(g^{N,T_2,x_2}_{\mathcal I}(t,\boldsymbol \xi) - g^{N,T_2,x_2}_{\mathcal I}(t,\boldsymbol \zeta)\right).
\end{aligned}
\end{equation}
Identity~\eqref{Paper02_identity_martingale_difference_increments_generator} then follows from combining the definition of $\mathcal L^N$ in~\eqref{Paper02_generator_foutel_etheridge_model} and~\eqref{Paper02_infinitesimal_generator} with the identities in~\eqref{Paper02_preparing_generator_difference_martingale_i},~\eqref{Paper02_pre_application_generator_green_function},~\eqref{Paper02_first_identity_migration_generator_square_greens_function} and~\eqref{Paper02_pre_application_generator_green_function_square}, and with the observation that for all $\boldsymbol \xi \in \mathcal S^N$, the series defining~$\mathcal{L}^{N}\Big((g^{N,T_{1},x_{1}}_{\mathcal{I}} - g^{N,T_{2},x_{2}}_{\mathcal{I}})^{2}(t, \cdot)\Big)(\boldsymbol{\xi})$ is absolutely convergent by~\eqref{Paper02_trivial_large_deviation_estimatre_CTRW} and~\eqref{Paper02_definition_state_space_formal}. Since this argument is similar to the one used in the proof of~\eqref{Paper02_definition_action_generator_greens_function} and~\eqref{Paper02_definition_action_generator_on_square_greens_function}, we omit the details. The fact that the map $[0,T_{1} \wedge T_{2}] \times \mathcal{S}^{N} \ni (t, \boldsymbol{\xi})  \, \mapsto \mathcal{L}^{N}\Big((g^{N,T_{1},x_{1}}_{\mathcal{I}} - g^{N,T_{2},x_{2}}_{\mathcal{I}})^{2}(t, \cdot)\Big)(\boldsymbol{\xi})$ is in $\mathcal{C}([0,T_{1} \wedge T_{2}] \times \mathcal{S}^{N}; \mathbb{R})$ also follows from the fact that the series defining~$\mathcal{L}^{N}\Big((g^{N,T_{1},x_{1}}_{\mathcal{I}} - g^{N,T_{2},x_{2}}_{\mathcal{I}})^{2}(t, \cdot)\Big)(\boldsymbol{\xi})$ is absolutely convergent uniformly in $t \in [0,T_1 \wedge T_2]$ and from the continuity of $p^N$. Finally, estimate~\eqref{Paper02_square_gen_diff_square_green} is proved by using~\eqref{Paper02_identity_martingale_difference_increments_generator},~\eqref{Paper02_general_estimates_moments_green_function_and_derivatives_time_generator} and Lemma~\ref{Paper02_simple_estimate_greens_function_representation_moments}. Since the proof of~\eqref{Paper02_square_gen_diff_square_green} is similar to the proof of~\eqref{Paper02_uniformly_integrable_martingale_truncation}, we omit the details.
\end{proof}

Finally, we prove Lemma~\ref{Paper02_well_behaved_function_inner_product_test_funciton_lemma}.

\begin{proof}[Proof of Lemma~\ref{Paper02_well_behaved_function_inner_product_test_funciton_lemma}]
     We will divide the proof into steps corresponding to each of the assertions (i) - (v). We highlight that, since the proofs of these assertions are very similar to the proof of Lemma~\ref{Paper02_properties_greens_function_muller_ratchet}, we will omit some details. 

     \medskip

\myemph{Proof of assertion~(i)}
     Since~$\supp(\varphi) \subset \mathbb{R}$ is compact, and using~\eqref{Paper02_formulation_inner_product_with_test_function_lipschitz_continuous_i}, there exists~$R > 0$ such that for~$\boldsymbol{\xi} = (\xi_k(y))_{k \in \mathbb N_0, y \in L_N^{-1}\mathbb Z}\in \mathcal{S}^{N}$,
     \begin{equation} \label{Paper02_using_compactness_support_smooth_test_function}
           g^{N,\varphi}_{\mathcal{I}}(\boldsymbol{\xi}) = \frac{1}{L_{N}} \sum_{\{x \in L_{N}^{-1}\mathbb{Z}: \vert x \vert < R\}} \; \sum_{k \in \mathcal{I}} \frac{\xi_{k}(x)}{N} \int_{-1}^{1} (1 - \vert h \vert) \varphi(x + hL_{N}^{-1}) \, dh .
     \end{equation}
     Hence, by the triangle inequality and the fact that~$\varphi \in \mathcal{C}^{2}_{c}(\mathbb{R})$, and then by~\eqref{Paper02_definition_semi_metric_state_space},  we have, for~$\boldsymbol{\xi} = (\xi(y))_{y \in L_N^{-1}\mathbb Z},\boldsymbol{\zeta} = (\zeta(y))_{y \in L_N^{-1}\mathbb Z} \in \mathcal{S}^{N}$,
     \begin{equation*}
    \begin{aligned}
        \left\vert g^{N,\varphi}_{\mathcal{I}}(\boldsymbol{\xi}) - g^{N,\varphi}_{\mathcal{I}}(\boldsymbol{\zeta})\right\vert & \leq \frac{\| \varphi \|_{L_{\infty}(\mathbb{R})}}{L_{N}} \sum_{\{x \in L_{N}^{-1}\mathbb{Z}: \vert x \vert < R\}} \frac{\| \xi(x) - \zeta(x) \|_{\ell_{1}}}{N} \lesssim_{\varphi,N} d_{\mathcal{S}^{N}}(\boldsymbol{\xi},\boldsymbol{\zeta}).
    \end{aligned}
    \end{equation*}
    Therefore, by~\eqref{Paper02_definition_Lipschitz_function} we have~$g^{N,\varphi}_{\mathcal{I}} \in \mathcal{C}_*(\mathcal{S}^{N}; \mathbb{R})$.

    \medskip

\myemph{Proof of assertion~(ii)}
    Recall from before~\eqref{Paper02_definition_Lipschitz_function} that for $k \in \mathbb N_0$ and $x \in L_N^{-1}\mathbb Z$, we let $\boldsymbol{e}^{(x)}_{k} \in (\mathbb{N}_{0}^{\mathbb{N}_{0}})^{L_{N}^{-1}\mathbb{Z}}$ denote the configuration with only one particle carrying exactly $k$ mutations in deme $x$. Then, by the definition of~$g^{N,\varphi}_{\mathcal{I}}$ in~\eqref{Paper02_formulation_inner_product_with_test_function_lipschitz_continuous_i}, the following identities hold for~$\boldsymbol{\xi} = (\xi_j(y))_{j \in \mathbb N_0, y \in L_N^{-1}\mathbb Z} \in \mathcal{S}^{N}$,~$k \in \mathcal{I}$ and~$x \in L_{N}^{-1}\mathbb{Z}$:
    \begin{equation} \label{Paper02_elementary_identities_action_generator_inner_product_smooth_function}
    \begin{aligned}
        &\mathds 1_{\{\xi_k(x) > 0\}}\Big(g^{N,\varphi}_{\mathcal{I}}\Big(\boldsymbol{\xi} + \boldsymbol{e}^{(x \pm L_{N}^{-1})}_{k} - \boldsymbol{e}^{(x)}_{k} \Big) - g^{N,\varphi}_{\mathcal{I}}(\boldsymbol{\xi})\Big)
        \\ & \qquad = \frac{1}{NL_{N}} \mathds 1_{\{\xi_k(x) > 0\}} \int_{-1}^{1} (1 - \vert h \vert) \Big(\varphi(x \pm L_{N}^{-1} + hL_{N}^{-1}) - \varphi(x + hL_{N}^{-1}) \Big) \, dh, 
        \\ & g^{N,\varphi}_{\mathcal{I}}\Big(\boldsymbol{\xi} + \boldsymbol{e}^{(x)}_{k} \Big) - g^{N,\varphi}_{\mathcal{I}}(\boldsymbol{\xi}) = \frac{1}{NL_{N}} \int_{-1}^{1} (1 - \vert h \vert) \varphi(x + hL_{N}^{-1}) \, dh,
        \\ & \mathds 1_{\{\xi_k(x) > 0\}}\Big(g^{N,\varphi}_{\mathcal{I}}\Big(\boldsymbol{\xi} - \boldsymbol{e}^{(x)}_{k} \Big) - g^{N,\varphi}_{\mathcal{I}}(\boldsymbol{\xi})\Big) \\
        & \qquad = - \frac{1}{NL_{N}} \mathds 1_{\{\xi_k(x) > 0\}} \int_{-1}^{1} (1 - \vert h \vert) \varphi(x + hL_{N}^{-1}) \, dh.
    \end{aligned}
    \end{equation} Combining~\eqref{Paper02_elementary_identities_action_generator_inner_product_smooth_function}, the fact that~$\supp(\varphi) \subset \mathbb{R}$ is compact, and that~$\| \varphi \|_{L_{\infty}(\mathbb{R})} < \infty$ and then the triangle inequality, we conclude that the series in~\eqref{Paper02_generator_foutel_etheridge_model} defining the action of~$\mathcal{L}^{N}$ on~$g^{N,\varphi}_{\mathcal{I}}$ is absolutely convergent, and that the map~$\mathcal S^N \ni \boldsymbol{\xi} \mapsto \Big(\mathcal{L}^{N}g^{N,\varphi}_{\mathcal{I}}\Big)(\boldsymbol{\xi})$ is continuous. Moreover, by using~\eqref{Paper02_elementary_identities_action_generator_inner_product_smooth_function}, rearranging terms and then using~\eqref{Paper02_definition_discrete_continuous_measure},~\eqref{Paper02_discrete_gradient_discrete_laplacian} and~\eqref{Paper02_reaction_term_PDE}, we get~\eqref{Paper02_action_generator_inner_product_smooth_functions}, which completes the proof of assertion~(ii).

     \medskip

\myemph{Proof of assertion~(iii)}
     By~\eqref{Paper02_formulation_inner_product_with_test_function_lipschitz_continuous_i} and the same argument as we used to derive~\eqref{Paper02_first_identity_migration_generator_square_greens_function} and~\eqref{Paper02_pre_application_generator_green_function_square}, the following identities hold for~$\boldsymbol{\xi} = (\xi_j(y))_{j \in \mathbb N_0,y \in L_N^{-1}\mathbb Z} \in \mathcal{S}^{N}$,~$k \in \mathcal{I}$ and~$x \in L_{N}^{-1}\mathbb{Z}$:
     \begin{equation} \label{Paper02_elementary_identities_action_generator_inner_product_smooth_function_square}
     \begin{split}
          & \mathds 1_{\{\xi_k(x) > 0\}}\Big((g^{N,\varphi}_{\mathcal{I}})^{2}\Big(\boldsymbol{\xi} + \boldsymbol{e}^{(x \pm L_{N}^{-1})}_{k} - \boldsymbol{e}^{(x)}_{k} \Big) - (g^{N,\varphi}_{\mathcal{I}})^{2}(\boldsymbol{\xi})\Big) \\ & \quad = \mathds 1_{\{\xi_k(x) > 0\}}\bigg(\frac{2 g^{N,\varphi}_{\mathcal{I}}(\boldsymbol{\xi})}{NL_{N}}\int_{-1}^{1} (1 - \vert h \vert) \Big(\varphi(x \pm L_{N}^{-1} + hL_{N}^{-1}) - \varphi(x + hL_{N}^{-1}) \Big) \, dh \\ & \qquad \qquad \quad + \frac{1}{N^{2}L_{N}^{2}} \bigg(\int_{-1}^{1} (1 - \vert h \vert) \Big(\varphi(x \pm L_{N}^{-1} + hL_{N}^{-1}) - \varphi(x + hL_{N}^{-1}) \Big) \, dh\bigg)^{2}\bigg), \\ 
          & (g^{N,\varphi}_{\mathcal{I}})^{2}\Big(\boldsymbol{\xi} + \boldsymbol{e}^{(x)}_{k} \Big) - (g^{N,\varphi}_{\mathcal{I}})^{2}(\boldsymbol{\xi}) = \frac{2g^{N,\varphi}_{\mathcal{I}}(\boldsymbol{\xi})}{NL_{N}} \int_{-1}^{1} (1 - \vert h \vert) \varphi(x + hL_{N}^{-1}) \, dh \\ & \quad \quad 
 \quad  \quad  \quad  \quad  \quad  \quad  \quad 
 \quad \quad \quad 
 \quad  \quad  \quad  \quad  \quad  + \frac{1}{N^{2}L_{N}^{2}} \bigg(\int_{-1}^{1} (1 - \vert h \vert) \varphi(x + hL_{N}^{-1}) \, dh\bigg)^{2}, \\ & \mathds 1_{\{\xi_k(x) > 0\}} \Big((g^{N,\varphi}_{\mathcal{I}})^{2}\Big(\boldsymbol{\xi} - \boldsymbol{e}^{(x)}_{k} \Big) - (g^{N,\varphi}_{\mathcal{I}})^{2}(\boldsymbol{\xi})\Big) \\
 & \quad = \mathds 1_{\{\xi_k(x) > 0\}}\bigg(- \frac{2g^{N,\varphi}_{\mathcal{I}}(\boldsymbol{\xi})}{NL_{N}} \int_{-1}^{1} (1 - \vert h \vert) \varphi(x + hL_{N}^{-1}) \, dh \\
 & \qquad \qquad \qquad + \frac{1}{N^{2}L_{N}^{2}} \bigg(\int_{-1}^{1} (1 - \vert h \vert) \varphi(x + hL_{N}^{-1}) \, dh\bigg)^{2}\bigg),
     \end{split}
     \raisetag{-5.5cm}
     \end{equation}
where~$\boldsymbol{e}^{(x)}_{k} \in (\mathbb{N}_{0}^{\mathbb{N}_{0}})^{L_{N}^{-1}\mathbb{Z}}$ denotes the configuration with only one particle carrying exactly $k$ mutations in deme $x$. Then, since~$\supp(\varphi) \subset\mathbb{R}$ is compact and~$\| \varphi \|_{L_{\infty}(\mathbb{R})} < \infty$, by Proposition~\ref{Paper02_topological_properties_state_space}(i) and~(ii), and using assertion~(ii) of this lemma, the identities in~\eqref{Paper02_elementary_identities_action_generator_inner_product_smooth_function_square} imply that the series in~\eqref{Paper02_infinitesimal_generator} defining the action of~$\mathcal{L}^{N}$ on~$(g^{N,\varphi}_{\mathcal{I}})^{2}$ is absolutely convergent uniformly on compact subsets of $\mathcal S^N$. Therefore, the map~$\mathcal S^N \ni \boldsymbol{\xi} \mapsto \Big(\mathcal{L}^{N}(g^{N,\varphi}_{\mathcal{I}})^2\Big)(\boldsymbol{\xi})$ is continuous. Moreover, by using~\eqref{Paper02_elementary_identities_action_generator_inner_product_smooth_function_square}, rearranging terms and then using~\eqref{Paper02_discrete_gradient_discrete_laplacian},~\eqref{Paper02_definition_discrete_continuous_measure},~\eqref{Paper02_action_generator_inner_product_smooth_functions} and~\eqref{Paper02_reaction_predictable_bracket_process}, we get~\eqref{Paper02_action_generator_square_inner_product_test_function}, which completes the proof of assertion~(iii).

\myemph{Proof of assertion~$(iv)$}
We will bound each term within the expectation on the left-hand side of~\eqref{Paper02_general_estimates_moments_test_function_and_derivatives_generator} individually. For the first term, observe that~\eqref{Paper02_using_compactness_support_smooth_test_function}, the triangle inequality, and the fact that~$\| \varphi \|_{L_{\infty}(\mathbb{R})} < \infty$ imply the existence of~$R > 0$ such that for~$N \in \mathbb N$ and~$\boldsymbol{\xi} = (\xi(y))_{y \in L_N^{-1}\mathbb Z} \in \mathcal{S}^{N}$,
\begin{equation} \label{Paper02_intermediate_step_moments_inner_product_test_function_i}
    \vert g^{N,\varphi}_{\mathcal{I}} (\boldsymbol{\xi}) \vert \leq \frac{1}{L_{N}} \sum_{\{x \in L_{N}^{-1}\mathbb{Z}: \vert x \vert < R\}} \| \varphi \|_{L_\infty(\mathbb R)}\frac{\| \xi(x) \|_{\ell_{1}}}{N}.
\end{equation}
We recall that the following elementary inequality holds for~$r \geq 1$,~$n \in \mathbb{N}$ and~$(a_{i})_{i = 1}^n \subset [0,\infty)$:
\begin{equation} \label{Paper02_elementary_inequality_power_arbitrary_finite_sum_positive_numbers}
    \Bigg(\sum_{i = 1}^{n} a_{i} \Bigg)^{r} \leq n^{r-1} \sum_{i = 1}^{n} a_{i}^{r}.
\end{equation}
For $r \geq 1$, taking both sides of~\eqref{Paper02_intermediate_step_moments_inner_product_test_function_i} to the power of~$r$ and then using~\eqref{Paper02_elementary_inequality_power_arbitrary_finite_sum_positive_numbers} and that $\vert \{x \in L_N^{-1}\mathbb Z: \, \vert x \vert < R\}\vert \leq 2RL_N$, we get, for~$\boldsymbol{\xi} = (\xi(x))_{x \in L_N^{-1}\mathbb Z} \in \mathcal{S}^{N}$,
\begin{equation} \label{Paper02_intermediate_step_moments_inner_product_test_function_ii}
     \vert g^{N,\varphi}_{\mathcal{I}} (\boldsymbol{\xi}) \vert^{r} \leq (2R)^{r-1}\| \varphi \|_{L_{\infty}(\mathbb{R})}^r \frac{1}{L_{N}} \sum_{\{x \in L_{N}^{-1}\mathbb{Z}: \vert x \vert < R\}} \frac{\| \xi(x) \|_{\ell_{1}}^{r}}{N^{r}}.
\end{equation}
Taking expectations on both sides of~\eqref{Paper02_intermediate_step_moments_inner_product_test_function_ii}, and then applying Theorem~\ref{Paper02_bound_total_mass}, we conclude that for any~$T \geq 0$ and~$r \geq 1$,
\begin{equation} \label{Paper02_intermediate_step_moments_inner_product_test_function_iii}
    \sup_{N \in \mathbb{N}} \; \sup_{t \in [0,T]} \; \mathbb{E}_{\boldsymbol \eta^N}\Big[\vert g^{N,\varphi}_{\mathcal{I}} (\eta^{N}(t)) \vert^{r}\Big] < \infty.
\end{equation}
For the second term within the expectation on the left-hand side of~\eqref{Paper02_general_estimates_moments_test_function_and_derivatives_generator}, observe that by~\eqref{Paper02_action_generator_inner_product_smooth_functions}, the triangle inequality, the fact that~$\varphi$ has compact support,~\eqref{Paper02_bound_discrete_laplacian} and the fact that~$\vert F_{k}(u) \vert \leq F^{+}_{k}(u)$ for~$u \in \ell_{1}^{+}$ and~$k \in \mathbb{N}_{0}$ by~\eqref{Paper02_reaction_term_PDE} and~\eqref{Paper02_reaction_predictable_bracket_process}, we conclude that there exists~$R > 0$ such that for~$N \in \mathbb N$ and~$\boldsymbol{\xi} = (\xi_k(x))_{k \in \mathbb N_0, x \in L_N^{-1}\mathbb Z} \in \mathcal{S}^{N}$,
\begin{equation} \label{Paper02_intermediate_step_moments_inner_product_test_function_iv}
\begin{aligned}
    \Big\vert \mathcal{L}^{N}g^{N,\varphi}_{\mathcal{I}} (\boldsymbol{\xi}) \Big\vert & \leq 2\frac{m_{N}}{L_{N}^{3}} \| \varphi'' \|_{L_{\infty}(\mathbb{R})} \sum_{\{x \in L_{N}^{-1}\mathbb{Z}: \vert x \vert < R\}} \; \sum_{k \in \mathcal{I}} \frac{\xi_{k}(x)}{N} \\
    & \quad + \frac{1}{L_{N}} \| \varphi \|_{L_{\infty}(\mathbb{R})} \sum_{\{x \in L_{N}^{-1}\mathbb{Z}: \vert x \vert < R\}} \; \sum_{k \in \mathcal{I}} F^{+}_{k}\left(\frac{\xi(x)}{N}\right).
\end{aligned}
\end{equation}
For $r \geq 1$, by taking both sides of~\eqref{Paper02_intermediate_step_moments_inner_product_test_function_iv} to the power of~$r$, and using~\eqref{Paper02_elementary_inequality_power_arbitrary_finite_sum_positive_numbers} and~\eqref{Paper02_trivial_bound_sum_reaction_terms}, and then by Theorem~\ref{Paper02_bound_total_mass} and the fact that by Assumption~\ref{Paper02_scaling_parameters_assumption} we have~$\displaystyle \frac{m_{N}}{L_{N}^{2}} \rightarrow m \in (0,\infty)$ as~$N \rightarrow \infty$, the following bound holds for~$T \geq 0$:
\begin{equation} \label{Paper02_intermediate_step_moments_inner_product_test_function_v}
     \sup_{N \in \mathbb{N}} \; \sup_{t \in [0,T]} \; \mathbb{E}_{\boldsymbol \eta^N}\Big[\Big\vert \Big(\mathcal{L}^{N}g^{N,\varphi}_{\mathcal{I}}\Big)(\eta^{N}(t)) \Big\vert^{r} \Big] < \infty.
\end{equation}
It remains to bound the third term within the expectation on the left-hand side of~\eqref{Paper02_general_estimates_moments_test_function_and_derivatives_generator}. By~\eqref{Paper02_action_generator_square_inner_product_test_function}, the triangle inequality,~\eqref{Paper02_bound_discrete_gradient} and the fact that~$\varphi$ has compact support, there exists~$R > 0$ such that for $N \in \mathbb N$ and~$\boldsymbol{\xi} = (\xi_k(x))_{k \in \mathbb N_0, x \in L_N^{-1}\mathbb Z} \in \mathcal{S}^{N}$
\begin{equation} \label{Paper02_intermediate_step_moments_inner_product_test_function_vi}
\begin{aligned}
    \Big\vert \Big(\mathcal{L}^{N}(g^{N,\varphi}_{\mathcal{I}})^{2}\Big)(\boldsymbol{\xi}) \Big\vert & \leq 2 \vert g^{N,\varphi}_{\mathcal{I}}(\boldsymbol{\xi})\vert \Big\vert \Big(\mathcal{L}^{N}g^{N,\varphi}_{\mathcal{I}}\Big)(\boldsymbol{\xi})\Big\vert + E(N,\varphi,\mathcal{I},\boldsymbol{\xi}),
\end{aligned}
\end{equation}
where
\begin{equation} \label{Paper02_intermediate_step_moments_inner_product_test_function_vii}
\begin{aligned}
E(N,\varphi,\mathcal{I},\boldsymbol{\xi}) & \defeq  \frac{m_{N}}{NL_{N}^{4}} \| \varphi'\|_{L_{\infty}(\mathbb{R})}^{2}\sum_{\{x \in L_{N}^{-1}\mathbb{Z}: \vert x \vert < R\}} \; \sum_{k \in \mathcal{I}} \frac{\xi_{k}(x)}{N} \\ & \qquad + \frac{1}{NL_{N}^{2}} \| \varphi \|^{2}_{L_{\infty}(\mathbb{R})} \sum_{\{x \in L_{N}^{-1}\mathbb{Z}: \vert x \vert < R\}} \; \sum_{k \in \mathcal{I}} F^{+}_{k}\left(\frac{\xi(x)}{N}\right).
\end{aligned}
\end{equation}
We now bound the moments of the terms on the right-hand side of~\eqref{Paper02_intermediate_step_moments_inner_product_test_function_vi} with $\boldsymbol \xi = \eta^N(t)$. For the first term, by the Cauchy-Schwarz inequality, and then by estimates~\eqref{Paper02_intermediate_step_moments_inner_product_test_function_iii} and~\eqref{Paper02_intermediate_step_moments_inner_product_test_function_v}, we have for all~$r \geq 1$ and~$T \geq 0$,
\begin{equation} \label{Paper02_intermediate_step_moments_inner_product_test_function_viii}
\begin{aligned}
    & \sup_{N \in \mathbb{N}} \; \sup_{t \in [0,T]} \mathbb{E}_{\boldsymbol \eta^N}\Big[\vert g^{N,\varphi}_{\mathcal{I}}(\eta^{N}(t))\vert^{r} \Big\vert \Big(\mathcal{L}^{N}g^{N,\varphi}_{\mathcal{I}}\Big)(\eta^{N}(t))\Big\vert^{r}\Big] \\ & \quad \leq \sup_{N \in \mathbb{N}} \; \sup_{t \in [0,T]} \mathbb{E}_{\boldsymbol \eta^N}\Big[\vert g^{N,\varphi}_{\mathcal{I}}(\eta^{N}(t))\vert^{2r}\Big]^{1/2} \mathbb{E}_{\boldsymbol \eta^N}\Big[\Big\vert \Big(\mathcal{L}^{N}g^{N,\varphi}_{\mathcal{I}}\Big)(\eta^{N}(t))\Big\vert^{2r}\Big]^{1/2} \\ & \quad < \infty.
\end{aligned}
\end{equation}
To bound the moments of the second term on the right-hand side of~\eqref{Paper02_intermediate_step_moments_inner_product_test_function_vi}, for $r \geq 1$, we take both sides of~\eqref{Paper02_intermediate_step_moments_inner_product_test_function_vii} to the power of~$r$, use~\eqref{Paper02_elementary_inequality_power_arbitrary_finite_sum_positive_numbers} and~\eqref{Paper02_trivial_bound_sum_reaction_terms}, and then Theorem~\ref{Paper02_bound_total_mass} and the fact that by Assumption~\ref{Paper02_scaling_parameters_assumption}, $\displaystyle \frac{m_{N}}{L_{N}^{2}} \rightarrow m \in (0,\infty)$ and~$L_{N} \rightarrow \infty$ as~$N \rightarrow \infty$ to conclude that for all~$r \geq 1$ and~$T \geq 0$,
\begin{equation} \label{Paper02_intermediate_step_moments_inner_product_test_function_ix}
    \lim_{N \rightarrow \infty} \; \sup_{t \in [0,T]} \; \mathbb{E}_{\boldsymbol \eta^N}[ E(N,\varphi,\mathcal{I},\eta^{N}(t))^{r}] = 0.
\end{equation}
By~\eqref{Paper02_intermediate_step_moments_inner_product_test_function_vi},~\eqref{Paper02_intermediate_step_moments_inner_product_test_function_viii} and~\eqref{Paper02_intermediate_step_moments_inner_product_test_function_ix}, we have for~$T \geq 0$ and~$r \geq 1$,
\begin{equation} \label{Paper02_intermediate_step_moments_inner_product_test_function_x}
    \sup_{N \in \mathbb{N}} \; \sup_{t \in [0,T]} \; \mathbb{E}_{\boldsymbol \eta^N}\Big[\Big\vert \Big(\mathcal{L}^{N}(g^{N,\varphi}_{\mathcal{I}})^{2}\Big)(\eta^{N}(t)) \Big\vert^{r} \Big] < \infty.
\end{equation}
Assertion~(iv) then follows from combining~\eqref{Paper02_intermediate_step_moments_inner_product_test_function_iii},~\eqref{Paper02_intermediate_step_moments_inner_product_test_function_v} and~\eqref{Paper02_intermediate_step_moments_inner_product_test_function_x}.

\medskip

\myemph{Proof of assertion~(v)}
Since the proof that the process~$M$ is a càdlàg martingale is similar to the proof of Lemma~\ref{Paper02_properties_greens_function_muller_ratchet}(v), we will omit some details. It follows from assertion~(iv) of this lemma and the triangle inequality that $M$ is an integrable càdlàg process. It remains to verify that~$M$ satisfies the martingale property with respect to the filtration~$\{\mathcal{F}^{\eta^{N}}_{t+}\}_{t \geq 0}$. For every~$n \in \mathbb{N}$, define~$h^{N,\varphi,n}_{\mathcal{I}}: \mathcal{S}^{N} \rightarrow \mathbb{R}$ by, for all~$\boldsymbol{\xi} \in \mathcal{S}^{N}$,
\begin{equation} \label{Paper02_truncation_square_inner_product_function}
    h^{N,\varphi,n}_{\mathcal{I}}(\boldsymbol{\xi}) \defeq (g^{N,\varphi}_{\mathcal{I}})^{2}(\boldsymbol{\xi}) \wedge n^{2}.
\end{equation}
By assertion~(i) of this lemma, $g^{N,\varphi}_{\mathcal{I}} \in \mathcal{C}_*(\mathcal{S}^{N}; \mathbb{R})$. By~\eqref{Paper02_truncation_square_inner_product_function}, for all~$\boldsymbol{\xi}, \boldsymbol{\zeta} \in \mathcal{S}^{N}$ and every~$n \in \mathbb{N}$, we have
\begin{equation*}
\begin{aligned}
    \Big\vert h^{N,\varphi,n}_{\mathcal{I}}(\boldsymbol{\xi}) - h^{N,\varphi,n}_{\mathcal{I}}( \boldsymbol{\zeta}) \Big\vert & \leq 2n \Big\vert g^{N,\varphi}_{\mathcal{I}}(\boldsymbol{\xi}) - g^{N,\varphi}_{\mathcal{I}}(\boldsymbol{\zeta}) \Big\vert.
\end{aligned}
\end{equation*}
Therefore, $h^{N,\varphi,n}_{\mathcal{I}} \in \mathcal{C}_*(\mathcal{S}^{N}; \mathbb{R})$ for every~$n \in \mathbb{N}$. Hence, by Theorem~\ref{Paper02_thm_existence_uniqueness_IPS_spatial_muller}, for every~$n \in \mathbb{N}$, the process~$(M^{(n)}(T))_{T \geq 0}$ given by, for all~$T \geq 0$,
\begin{equation} \label{Paper02_associated_truncated_martingale_ii}
    M^{(n)}(T) \defeq h^{N,\varphi,n}_{\mathcal{I}}(\eta^{N}(T)) - h^{N,\varphi,n}_{\mathcal{I}}(\eta^{N}(0)) - \int_{0}^{T} \Big(\mathcal{L}^{N}h^{N,\varphi,n}_{\mathcal{I}}\Big)(\eta^{N}(t-)) \, dt
\end{equation}
is a càdlàg martingale with respect to the filtration~$\{\mathcal{F}^{\eta^{N}}_{t+}\}_{t \geq 0}$. By the definitions of~$g^{N,\varphi}_{\mathcal{I}}$ in~\eqref{Paper02_formulation_inner_product_with_test_function_lipschitz_continuous_i} and of~$h^{N,\varphi,n}_{\mathcal{I}}$ in~\eqref{Paper02_truncation_square_inner_product_function}, and by the fact that~$\varphi$ is bounded and has compact support, and by~\eqref{Paper02_definition_state_space_formal}, by the same argument as we used to derive~\eqref{Paper02_constructing_sequence_UI_martingales_square_green}, the following limit holds almost surely for $T \geq 0$:
\begin{equation} \label{Paper02_almost_sure_convergence_trunctaed_martingales_inner_product}
    \lim_{n \rightarrow \infty} \; M^{(n)}(T) = M(T).
\end{equation}
As in the proof of Lemma~\ref{Paper02_properties_greens_function_muller_ratchet}(v), by~\eqref{Paper02_almost_sure_convergence_trunctaed_martingales_inner_product}, the martingale property of~$M$ will be proved after establishing that for all~$T \geq 0$, the sequence of random variables~$(M^{(n)}(T))_{n \in \mathbb{N}}$ is uniformly integrable. By~\eqref{Paper02_intermediate_step_moments_inner_product_test_function_iii} and the same argument as used to derive~\eqref{Paper02_easy_trick_cauchy_schwarz}, it will suffice to prove that for every~$N \in \mathbb{N}$ and all~$T \geq 0$,
\begin{equation} \label{Paper02_second_moment_estimate_application_generator_truncated_function}
    \sup_{n \in \mathbb{N}} \; \sup_{t \in [0,T]} \; \mathbb{E}_{\boldsymbol \eta^N}\Big[\Big(\mathcal{L}^{N}h^{N,\varphi,n}_{\mathcal{I}}\Big)^{2}(\eta^{N}(t))\Big] < \infty.
\end{equation}
By applying~\eqref{Paper02_elementary_inequality_difference_truncated_squares} and~\eqref{Paper02_elementary_identities_action_generator_inner_product_smooth_function_square}, and then the triangle inequality and~\eqref{Paper02_bound_discrete_gradient}, we conclude that for any $k \in \mathcal{I}$,~$n \in \mathbb{N}$,~$x \in L_{N}^{-1}\mathbb{Z}$ and $\boldsymbol{\xi} = (\xi_j(y))_{j \in \mathbb N_0, \, y \in L_N^{-1}\mathbb Z} \in \mathcal{S}^{N}$, the following estimates hold:
\begin{equation} \label{Paper02_bound_application_individual_generator_action_test_function}
\begin{aligned}
    & \mathds 1_{\{\xi_k(x)> 0\}}\Big\vert h^{N,\varphi,n}_{\mathcal{I}}\Big(\boldsymbol{\xi} + \boldsymbol{e}^{(x+L_{N}^{-1})}_{k} - \boldsymbol{e}^{(x)}_{k}\Big) - h^{N,\varphi,n}_{\mathcal{I}}\left(\boldsymbol{\xi}\right) \Big\vert \\
    & \quad \leq \frac{2\vert g^{N,\varphi}_{\mathcal{I}}(\boldsymbol{\xi})\vert}{NL_{N}}\| \varphi \|_{L_{\infty}(\mathbb{R})} + \frac{1}{N^{2}L_{N}^{2}} \| \varphi' \|_{L_{\infty}(\mathbb{R})}^{2},
    \\  & \mathds 1_{\{\xi_k(x)> 0\}}\Big\vert h^{N,\varphi,n}_{\mathcal{I}}\Big(\boldsymbol{\xi} + \boldsymbol{e}^{(x-L_{N}^{-1})}_{k} - \boldsymbol{e}^{(x)}_{k}\Big) - h^{N,\varphi,n}_{\mathcal{I}}\left(\boldsymbol{\xi}\right) \Big\vert \\ & \quad \leq \frac{2\vert g^{N,\varphi}_{\mathcal{I}}(\boldsymbol{\xi})\vert}{NL_{N}}\| \varphi \|_{L_{\infty}(\mathbb{R})} + \frac{1}{N^{2}L_{N}^{2}} \| \varphi' \|_{L_{\infty}(\mathbb{R})}^{2}, \\
    & \Big\vert h^{N,\varphi,n}_{\mathcal{I}}\Big(\boldsymbol{\xi} + \boldsymbol{e}^{(x)}_{k}\Big) - h^{N,\varphi,n}_{\mathcal{I}}\left(\boldsymbol{\xi}\right) \Big\vert \leq \frac{2\vert g^{N,\varphi}_{\mathcal{I}}(\boldsymbol{\xi})\vert}{NL_{N}}\| \varphi \|_{L_{\infty}(\mathbb{R})} + \frac{1}{N^{2}L_{N}^{2}} \| \varphi \|_{L_{\infty}(\mathbb{R})}^{2},
    \\  & \mathds 1_{\{\xi_k(x)> 0\}}\Big\vert h^{N,\varphi,n}_{\mathcal{I}}\Big(\boldsymbol{\xi} - \boldsymbol{e}^{(x)}_{k}\Big) - h^{N,\varphi,n}_{\mathcal{I}}\left(\boldsymbol{\xi}\right) \Big\vert \\
    & \quad \leq \frac{2\vert g^{N,\varphi}_{\mathcal{I}}(\boldsymbol{\xi})\vert}{NL_{N}}\| \varphi \|_{L_{\infty}(\mathbb{R})} + \frac{1}{N^{2}L_{N}^{2}} \| \varphi \|_{L_{\infty}(\mathbb{R})}^{2}.
\end{aligned}
\end{equation}
Combining~\eqref{Paper02_generator_foutel_etheridge_model}, the triangle inequality, the fact that~$\varphi$ has compact support,~\eqref{Paper02_bound_application_individual_generator_action_test_function} and then Fubini's theorem, the definition of~$F^{+} = (F^{+}_{k})_{k \in \mathbb{N}_{0}}$ in~\eqref{Paper02_reaction_predictable_bracket_process}, and~\eqref{Paper02_elementary_inequality_power_arbitrary_finite_sum_positive_numbers}, we conclude that there exists~$R > 0$ such that for every~$n \in \mathbb{N}$ and~$t \geq 0$,
\begin{equation*}
\begin{aligned}
    & \mathbb{E}_{\boldsymbol{\eta}^N}\Big[\Big(\mathcal{L}^{N}h^{N,\varphi,n}_{\mathcal{I}}\Big)^{2}(\eta^{N}(t))\Big] \\
    & \quad \lesssim_{N,\varphi} \mathbb{E}_{\boldsymbol \eta^N}\bigg[ (g^{N,\varphi}_{\mathcal{I}})^{2}(\eta^{N}(t))\bigg(\sum_{\{x \in L_{N}^{-1}\mathbb{Z}: \vert x \vert < R\}} \sum_{k \in \mathcal{I}} u^{N}_{k}(t,x)\bigg)^{2}\bigg] \\
    & \qquad \quad + \mathbb{E}_{\boldsymbol \eta^N}\bigg[\bigg(\sum_{\{x \in L_{N}^{-1}\mathbb{Z}: \vert x \vert < R\}} \sum_{k \in \mathcal{I}} u^{N}_{k}(t,x)\bigg)^{2}\bigg]
    \\
    & \qquad \quad + \mathbb{E}_{\boldsymbol \eta^N}\bigg[ (g^{N,\varphi}_{\mathcal{I}})^{2}(\eta^{N}(t))\bigg(\sum_{\{x \in L_{N}^{-1}\mathbb{Z}: \vert x \vert < R\}} \sum_{k \in \mathcal{I}} F^{+}_{k}(u^{N}(t,x))\bigg)^{2}\bigg] \\
    & \qquad \quad + \mathbb{E}_{\boldsymbol \eta^N}\bigg[\bigg(\sum_{\{x \in L_{N}^{-1}\mathbb{Z}: \vert x \vert < R\}} \sum_{k \in \mathcal{I}} F^{+}_{k}(u^{N}(t,x))\bigg)^{2}\bigg].
\end{aligned}
\end{equation*}
By applying the Cauchy-Schwarz inequality, we obtain
\begin{equation} \label{Paper02_bound_application_individual_generator_action_test_function_i}
\begin{aligned}
    & \mathbb{E}_{\boldsymbol{\eta}^N}\Big[\Big(\mathcal{L}^{N}h^{N,\varphi,n}_{\mathcal{I}}\Big)^{2}(\eta^{N}(t))\Big] \\
    & \quad \lesssim_{N,\varphi} \mathbb{E}_{\boldsymbol \eta^N}\Big[(g^{N,\varphi}_{\mathcal{I}})^{4}(\eta^{N}(t))\Big]^{1/2}\mathbb{E}_{\boldsymbol \eta^N}\bigg[\bigg(\sum_{\{x \in L_{N}^{-1}\mathbb{Z}: \vert x \vert < R\}} \sum_{k \in \mathcal{I}} u^{N}_{k}(t,x)\bigg)^{4}\bigg]^{1/2} \\
    & \qquad \quad + \mathbb{E}_{\boldsymbol \eta^N}\bigg[\bigg(\sum_{\{
    x \in L_{N}^{-1}\mathbb{Z}: \vert x \vert < R\}} \sum_{k \in \mathcal{I}} u^{N}_{k}(t,x)\bigg)^{2}\bigg] \\
    & \qquad \quad + \mathbb{E}_{\boldsymbol \eta^N}\Big[(g^{N,\varphi}_{\mathcal{I}})^{4}(\eta^{N}(t))\Big]^{1/2}\mathbb{E}_{\boldsymbol \eta^N}\bigg[\bigg(\sum_{\{x \in L_{N}^{-1}\mathbb{Z}: \vert x \vert < R\}} \sum_{k \in \mathcal{I}} F^{+}_{k}(u^{N}(t,x))\bigg)^{4}\bigg]^{1/2} \\
    & \qquad \quad + \mathbb{E}_{\boldsymbol \eta^N}\bigg[\bigg(\sum_{\{x \in L_{N}^{-1}\mathbb{Z}: \vert x \vert < R\}} \sum_{k \in \mathcal{I}} F^{+}_{k}(u^{N}(t,x))\bigg)^{2}\bigg].
\end{aligned}
\end{equation}
By substituting~\eqref{Paper02_intermediate_step_moments_inner_product_test_function_iii} and~\eqref{Paper02_elementary_inequality_power_arbitrary_finite_sum_positive_numbers} into~\eqref{Paper02_bound_application_individual_generator_action_test_function_i}, and then applying Fubini's theorem and then Theorem~\ref{Paper02_bound_total_mass}, we get~\eqref{Paper02_second_moment_estimate_application_generator_truncated_function}. Then, by~\eqref{Paper02_associated_truncated_martingale_ii},~\eqref{Paper02_intermediate_step_moments_inner_product_test_function_iii},~\eqref{Paper02_second_moment_estimate_application_generator_truncated_function} and the same argument as for~\eqref{Paper02_easy_trick_cauchy_schwarz}, we have for any~$T \geq 0$ and~$N \in \mathbb{N}$,
\begin{equation*}
    \sup_{n \in \mathbb{N}} \, \mathbb{E}_{\boldsymbol \eta^N}\Big[(M^{(n)}(T))^{2}\Big] < \infty,
\end{equation*}
which implies the sequence of random variables~$(M^{(n)}(T))_{n \in \mathbb{N}}$ is uniformly integrable. As explained after~\eqref{Paper02_almost_sure_convergence_trunctaed_martingales_inner_product}, combining~\eqref{Paper02_almost_sure_convergence_trunctaed_martingales_inner_product} and the uniform integrability of~$(M^{(n)}(T))_{n \in \mathbb{N}}$ implies that~$M$ is a martingale, which completes the proof.
\end{proof}

\textbf{Acknowledgments.}
The authors are grateful to Matthias Birkner, Alison Etheridge, Félix Foutel-Rodier, Karsten Matthies and Matt Roberts for helpful comments and suggestions. 
While this work was being carried out, JLdOM was supported by a scholarship from the EPSRC Centre for Doctoral Training in Statistical Applied Mathematics at Bath (SAMBa), under the project EP/S022945/1.
MO is partially supported by EPSRC grant EP/X040089/1. 
SP is supported by a Royal Society University Research Fellowship. While part of this work was being carried out, SP was visiting SLMath as a Research Member of the Probability and Statistics of Discrete Structures program.

\bibliographystyle{abbrv}
\setlength{\bibsep}{1pt plus 0.3ex}
\bibliography{bibli_law_large_numbers_mullers_ratchet}

\end{document}